\documentclass[a4paper,11pt]{article}

\usepackage[margin=2.5cm]{geometry}
\usepackage{amsmath,amsthm,amssymb,amsfonts,mathrsfs,graphicx,color,bbm,bm}
\usepackage{stmaryrd}
\usepackage{enumitem}

\usepackage[usenames,dvipsnames]{xcolor}
\usepackage[english]{babel}
\usepackage[textsize=small]{todonotes}

\usepackage[bookmarksopen,hidelinks,destlabel]{hyperref}

\newcommand{\E}{\mathbb{E}}
\newcommand{\R}{\mathbb{R}}
\newcommand{\C}{\mathbb{C}}
\newcommand{\N}{\mathbb{N}}
\newcommand{\Z}{\mathbb{Z}}
\renewcommand{\P}{\mathbb{P}}

\let\Re\relax \DeclareMathOperator{\Re}{Re}

\newcommand{\condparentheses}[2]{\left(\left.#1\,\right\vert#2\right)}

\newcommand{\condP}[2]{\mathbb{P}\condparentheses{#1}{#2}}

\DeclareMathOperator{\Var}{Var}
\DeclareMathOperator{\Cov}{Cov}

\newcommand{\set}[1]{\left\{#1\right\}}
\newcommand{\abs}[1]{\left\vert#1\right\vert}
\newcommand{\norm}[1]{\left\Vert#1\right\Vert}

\newcommand{\bigset}[1]{\big\{#1\big\}}

\newcommand{\shortset}[1]{\{#1\}}
\newcommand{\shortabs}[1]{\vert#1\vert}
\newcommand{\bigabs}[1]{\big\vert#1\big\vert}

\newcommand{\ocinterval}[1]{\left(#1\right]} % to help with bracket matching
\newcommand{\cointerval}[1]{\left[#1\right)} % to help with bracket matching

\newcommand{\indicatorofset}[1]{\mathbbm{1}_{#1}}
\newcommand{\indicator}[1]{\indicatorofset{\set{#1}}}

\newcommand{\mat}[1]{\begin{pmatrix}#1\end{pmatrix}}
\newcommand{\smallmat}[1]{\left(\begin{smallmatrix}#1\end{smallmatrix}\right)}

\newcommand{\bigunion}{\bigcup}

\newcommand{\decreasesto}{\searrow}

\renewcommand{\phi}{\varphi}

\let\epsilon\varepsilon

\newcommand{\blank}[1]{}

\newcommand{\Poisson}{\mathrm{Poisson}}
\newcommand{\Binomial}{\mathrm{Binomial}}

\newcommand{\Geometric}{\mathrm{Geometric}}

\newtheorem{theorem}{Theorem}
\newtheorem{prop}[theorem]{Proposition}
\newtheorem{lemma}[theorem]{Lemma}
\newtheorem{coro}[theorem]{Corollary}

\theoremstyle{definition}

\newtheorem{example}[theorem]{Example}

\newcommand{\textandreference}[2]{\texorpdfstring{\hyperref[#2]{#1\ref*{#2}}}{#1\ref*{#2}}}

\newcommand{\lbsect}[1]{\label{s:#1}}
\newcommand{\refsect}[1]{\textandreference{Section~}{s:#1}}

\newcommand{\lbsubsect}[1]{\label{ss:#1}}
\newcommand{\refsubsect}[1]{\textandreference{Section~}{ss:#1}}

\newcommand{\lbsubsubsect}[1]{\label{sss:#1}}
\newcommand{\refsubsubsect}[1]{\textandreference{Section~}{sss:#1}}

\newcommand{\lbthm}[1]{\label{T:#1}}
\newcommand{\refthm}[1]{\textandreference{Theorem~}{T:#1}}

\newcommand{\lbprop}[1]{\label{P:#1}}
\newcommand{\refprop}[1]{\textandreference{Proposition~}{P:#1}}

\newcommand{\lblemma}[1]{\label{L:#1}}
\newcommand{\reflemma}[1]{\textandreference{Lemma~}{L:#1}}

\newcommand{\lbcoro}[1]{\label{C:#1}}
\newcommand{\refcoro}[1]{\textandreference{Corollary~}{C:#1}}

\newcommand{\lbexample}[1]{\label{ex:#1}}
\newcommand{\refexample}[1]{\textandreference{Example~}{ex:#1}}

\numberwithin{equation}{section}

\usepackage[numbers]{natbib}

\newcommand{\MLE}{\mathrm{MLE}}
\newcommand{\CLT}{\mathrm{CLT}}

\DeclareMathOperator{\rank}{rank}

\DeclareMathOperator{\trace}{tr}
\DeclareMathOperator{\interior}{int}
\DeclareMathOperator{\Log}{Log}

\DeclareMathOperator*{\argmax}{argmax}

\DeclareMathOperator{\sech}{sech}

\newcommand{\ii}{\mathrm{i}}

\newcommand{\ones}{\mathbbm{1}}

\newcommand{\thetadomain}{\mathcal{R}}

\let\grad\nabla

\newcommand{\xdim}{d}
\newcommand{\latentdim}{\ell}

\newcommand{\lbappendix}[1]{\label{a:#1}}
\newcommand{\refappendix}[1]{\textandreference{Appendix~}{a:#1}}

\newtheorem*{remarknonumber}{Remark}

\begin{document}

\title{Asymptotic accuracy of the saddlepoint approximation for maximum likelihood estimation}
\author{Jesse Goodman\footnote{Department of Statistics, University of Auckland, Private Bag 92019, Auckland 1142, New Zealand}}
\date{\today}

\maketitle

\begin{abstract}
~
The saddlepoint approximation gives an approximation to the density of a random variable in terms of its moment generating function.
When the underlying random variable is itself the sum of $n$ unobserved i.i.d.\ terms, the basic classical result is that the relative error in the density is of order $1/n$.
If instead the approximation is interpreted as a likelihood and maximised as a function of model parameters, the result is an approximation to the maximum likelihood estimate (MLE) that can be much faster to compute than the true MLE.
This paper proves the analogous basic result for the approximation error between the saddlepoint MLE and the true MLE: subject to certain explicit identifiability conditions, the error has asymptotic size $O(1/n^2)$ for some parameters, and $O(1/n^{3/2})$ or $O(1/n)$ for others.
In all three cases, the approximation errors are asymptotically negligible compared to the inferential uncertainty.

The proof is based on a factorisation of the saddlepoint likelihood into an exact and approximate term, along with an analysis of the approximation error in the gradient of the log-likelihood.
This factorisation also gives insight into alternatives to the saddlepoint approximation, including a new and simpler saddlepoint approximation, for which we derive analogous error bounds.
As a corollary of our results, we also obtain the asymptotic size of the MLE error approximation when the saddlepoint approximation is replaced by the normal approximation.
\end{abstract}

\section{Introduction}\lbsect{Intro}

Let $X$ be a random variable with density function $f(x)$, $x\in\R$.
Define
\begin{equation}
M(s) = \E(e^{sX}), \qquad K(s) = \log M(s),
\end{equation}
the moment generating function (MGF) and cumulant generating function (CGF), respectively, associated to $X$.
Given $x\in\R$, let $\hat{s}$ be the solution to 
\begin{equation}
K'(\hat{s}) = x
\end{equation}
and set
\begin{equation}\label{BasicSPA1d}
\hat{f}(x) = \frac{\exp(K(\hat{s}) - \hat{s}x)}{\sqrt{2\pi K''(\hat{s})}}.
\end{equation}
We call $\hat{f}(x)$ the \emph{saddlepoint approximation} to the density function $f(x)$.

In the statistical context, an important use of the saddlepoint approximation has been to analyse sampling distributions, with $X$ an estimator or a related statistic  to be understood via its sampling density $f(x)$.
In this setting, the natural way to assess the saddlepoint approximation is to measure how well $\hat{f}(x)$ approximates $f(x)$ as a function of $x$ -- for instance, by determining how fast the absolute error $\shortabs{\hat{f}(x)-f(x)}$ or the relative error $\shortabs{\hat{f}(x)/f(x)-1}$ decay in a suitable limit, and whether this convergence is uniform.

The most prominent results of this kind concern what we will call the \emph{standard asymptotic regime}, in which the observed value $X$ is the sample average of $n$ i.i.d.\ values.
In this setup, the saddlepoint approximation has simple $n$-dependence and can be computed in constant time, whereas the true density $f(x)$ commonly becomes intractable.
The classical error estimate \cite{Daniels1954} for the standard asymptotic regime states that $\hat{f}(x)/f(x) = 1+O(1/n)$ as $n\to\infty$.
Under stronger assumptions, the ratio $\hat{f}(x)/f(x)$ may remain uniformly bounded even in the tails, see \cite[Theorem~4.6.1]{Kolassa2006SeriesApproxMethodsStats} and \cite{Jensen1988,B-NKlu1999}, whereas other density approximations such as normal approximations (see \refappendix{SaddlepointVsCLT}) or Edgeworth expansions quickly lose relative accuracy away from the mean.
For many applications based on densities, these error bounds are more than enough to justify using the saddlepoint approximation $\hat{f}(x)$ as a readily computable substitute for $f(x)$.

In this paper, we shift perspectives and consider \eqref{BasicSPA1d} as an approximation to the \emph{likelihood}.
In this viewpoint, $X$ represents the raw data obtained from an experiment, modelled by the parameter $\theta$, and the observation $x$ is fixed.
We write the density and CGF as $f(x;\theta)$, $K(s;\theta)$ to emphasise their dependence on the parameter, and set $L(\theta;x) = f(x;\theta)$.
Instead of \eqref{BasicSPA1d} we form 
\begin{equation}\label{BasicSPA1dL}
\hat{L}(\theta;x) = \frac{\exp(K(\hat{s};\theta) - \hat{s}x)}{\sqrt{2\pi K''(\hat{s};\theta)}},
\end{equation}
the \emph{saddlepoint approximation to the likelihood}, considered as a function of $\theta$.
Note that the saddlepoint $\hat{s}=\hat{s}(\theta,x)$ is a function of both the parameter $\theta$ and the observed value $x$, defined implicitly by
\begin{equation}
K'(\hat{s}(\theta,x); \theta) = x
,
\end{equation}
and the derivatives $K'$, $K''$ are with respect to $s$.

Maximum likelihood inference involves maximising $L(\theta; x)$ with respect to $\theta$ to produce the maximum likelihood estimate (MLE)
\begin{equation}\label{TrueMLE}
\theta_{\MLE}(x) = \argmax_\theta L(\theta;x),
\end{equation}
if it exists.
In cases where the true likelihood $L(\theta; x)$ is intractable, but the saddlepoint likelihood $\hat{L}(\theta; x)$ can be computed, it is natural to ask what are the consequences for inference in maximising the saddlepoint likelihood in place of the true likelihood. 
We term the resulting estimate the \emph{saddlepoint MLE}
\begin{equation}\label{SPMLE}
\hat{\theta}_{\MLE}(x) = \argmax_\theta \hat{L}(\theta;x),
\end{equation}
if it exists.
Our interest is in the error introduced by this approach: specifically, the error in the saddlepoint MLE as an approximation to the true MLE, $\bigabs{\hat{\theta}_{\MLE}(x) - \theta_{\MLE}(x)}$. 
Investigation of approximation error has been a central issue when saddlepoint MLEs have been applied in the literature \cite[see Examples~\ref{ex:ZhaBraFew}--\ref{ex:PedDavFok} in \refappendix{Examples}]{ZhaBraFew2019,DavHauKraParameterLinearBirthDeath,PedDavFok2015}, but no general theory is currently available.
The classical saddlepoint error estimate $\hat{L}(\theta;x)/L(\theta;x)=1+O(1/n)$ does not provide easily interpretable guidance about whether $\hat{\theta}_\MLE(x)$ is close to $\theta_\MLE(x)$.
Instead, the key question is whether the \emph{gradient} $\grad_\theta \log\hat{L}(\theta;x)$ provides a good approximation to the true gradient $\grad_\theta\log L(\theta;x)$ \cite{Ogden2017,OgdenErrorLaplaceHighDimension}.
The main theorems of this paper will provide sharp asymptotic bounds for $\bigabs{\grad_\theta \log\hat{L}(\theta;x)-\grad_\theta\log L(\theta;x)}$ and hence for the size of the MLE approximation error $\bigabs{\hat{\theta}_{\MLE}(x) - \theta_{\MLE}(x)}$ under general conditions.

\paragraph*{Outline of the paper}

\refsubsect{Setup} introduces further notation for the multivariate saddlepoint approximation, including conventions for row and column vectors and gradients.
In \refsubsect{SAR}, we formulate the standard asymptotic regime as an explicit limiting framework relating the distribution $X_\theta$, its CGF $K$, the observed value $x$, and the parameter $n$.
\refsubsect{IntroExample} introduces a class of examples in this asymptotic framework.
The main results, Theorems~\ref{T:GradientError}--\ref{T:IntegerValued}, are stated in \refsubsect{MainResults}, and their interpretation is discussed in \refsubsect{Discussion}.
Several examples from theory and the literature are discussed in detail in \refappendix{Examples}; see \refsubsect{ExamplesGuide} for a brief summary.

\refsubsect{TwoSteps} expresses the saddlepoint procedure as a combination of an exact step (tilting) and an approximation step, and introduces a factorisation and reparametrisation of the likelihood that underlie the rest of the paper.
As a by-product, we obtain in \refsubsect{LowerOrder} a ``lower-order'' version of the saddlepoint approximation, satisfying analogues of Theorems~\ref{T:GradientError}--\ref{T:SamplingMLEWellSpecified} with a different power of $n$: see \refthm{LowerOrderSaddlepoint}.

The proofs of Theorems~\ref{T:GradientError}--\ref{T:MLEerror} are given in \refsect{Proofs}, along with a summary in \refsubsect{DerivativesSummary} of gradients of quantities related to the saddlepoint approximation.
Further proofs, examples and technical details appear in Appendices~\ref{a:Invariance}--\ref{a:Examples}.

\refsect{Conclusion} includes a summary, additional discussion, and directions for further inquiry.

\section{Main results}

\subsection{Setup and notation}\lbsubsect{Setup}

\subsubsection{Moment and cumulant generating functions}

We consider a vector-valued random variable $X$ of dimension $\xdim$ depending on a parameter $\theta$ of dimension $p$, and write $X=X_\theta$ when we wish to emphasise the dependence.
We consider the values of $X_\theta$ and $\theta$ to be column vectors, i.e., $\xdim\times 1$ or $p\times 1$ matrices, which we express as $X_\theta\in\R^{\xdim\times 1}$, $\theta\in\thetadomain\subset\R^{p\times 1}$, where $\thetadomain$ is an open subset of $\R^{p\times 1}$.
The multivariate MGF and CGF are
\begin{equation}\label{MGFCGF}
M(s;\theta) = \E(e^{sX_\theta}), \qquad K(s;\theta) = \log M(s;\theta).
\end{equation}
On those occasions when we consider several random variables, we will write $M_X$, $M_Y$ and so on to distinguish the respective generating functions.
 
In \eqref{MGFCGF}, $s$ is called the \emph{dual variable} to $X$, and we interpret it as a row vector, a $1\times \xdim$ matrix, so that $sX$ is a scalar or $1\times 1$ matrix.
This convention emphasises that $s$ and $X$, despite being vectors of the same dimension, play quite different roles and are not interchangeable; rather, the space of row vectors is the natural dual space to the space of column vectors.
This convention also avoids excessive use of transposes and explicit inner products.

We wish to consider $M$ and $K$ for complex-valued $s$, and to this end we must take care of convergence issues in \eqref{MGFCGF}.
Let 
\begin{equation}
\mathcal{S}_\theta=\set{s\in\R^{1\times\xdim}\colon\E(e^{sX_\theta})<\infty}, 
\qquad
\mathcal{S} = \set{(s,\theta)\colon s\in\mathcal{S}_\theta}.
\end{equation}
Writing $\Re(z)$ for the real part of the complex number $z$, we have $\abs{e^z} = e^{\Re(z)}$ for $z\in\C$, so the expectation in \eqref{MGFCGF} converges absolutely whenever $\Re(s)\in\mathcal{S}_\theta$.
We take the domain of $M$ to be $\set{(s,\theta)\in\C^{1\times\xdim}\times\thetadomain\colon\Re(s)\in\mathcal{S}_\theta}$, where $\Re(s)$ is interpreted coordinatewise for each of the $\xdim$ complex entries of $s$.
Note that for certain distributions $X_\theta$, $\mathcal{S}_\theta$ may reduce to the single point $0$ or otherwise become degenerate, but our assumptions will rule this out.
As soon as the interior $\interior\mathcal{S}_\theta$ is non-empty, $M(s;\theta)$ is analytic as a function of $s\in\interior\mathcal{S}_\theta$.

\subsubsection{Gradients, Hessians, moments and cumulants}\lbsubsubsect{GradHessMomCum}

We interpret gradient operators $\grad_s$, $\grad_\theta$ to have the shape of the transposes of the variables.
Thus $\grad_\theta K$ is a $1\times p$ row vector and $\grad_s K$, which we write as $K'$, is a $\xdim\times 1$ column vector.
We apply this convention also to vector-valued functions; in particular, $\grad_s \grad_\theta K$ is the $(\xdim\times p)$-matrix-valued function with $i,j$ entry $\frac{\partial^2 K}{\partial s_i\partial\theta_j}$.
We can write the Hessians of a scalar-valued function $f(s,\theta)$ as $\grad_s\grad_s^T f$ and $\grad_\theta^T\grad_\theta f$, with $i,j$ entries $\frac{\partial^2 f}{\partial s_i\partial s_j}$ and $\frac{\partial^2 f}{\partial\theta_i\partial\theta_j}$.
When $f=K$ we write the $s$-Hessian as $K''=\grad_s\grad_s^T K$.
For other derivative conventions, see \refappendix{DerivativesDerivation}.

When $0\in\interior\mathcal{S}_\theta$, the derivatives $M'(0;\theta)$, $M''(0;\theta)$ and $K'(0;\theta)$, $K''(0;\theta)$ give moments and cumulants of $X$:
\begin{equation}\label{M'M''K'K''0}
\begin{aligned}
M'(0;\theta) &= \E(X_\theta), & M''(0;\theta) &= \E(X_\theta X_\theta^T), \\ 
K'(0;\theta) &= \E(X_\theta), & K''(0;\theta) &= \Cov(X_\theta,X_\theta).
\end{aligned}
\end{equation}
In particular, $K''(0;\theta)$ must be positive semi-definite.
More generally, as we shall see in \refsubsect{TwoSteps}, $K''(s;\theta)$ has an interpretation as a covariance matrix for all $s\in\interior\mathcal{S}_\theta$.
It is natural to exclude the case where this covariance matrix is singular, and indeed our hypotheses will imply that
\begin{equation}\label{K''PosDef}
K''(s;\theta)\text{ is positive definite for all }s\in\interior\mathcal{S}_\theta.
\end{equation}
As a consequence, $K$ is strictly convex as a function of $s$.

\subsubsection{Multivariate saddlepoint approximation}\lbsubsubsect{MultivariateSPA}

With these preparations we can state the multivariate saddlepoint approximation.
For $x\in\R^{\xdim\times 1}$, we form the \emph{saddlepoint equation}
\begin{equation*}\tag{SE}\label{SaddlepointEquation}
K'(\hat{s};\theta) = x
\end{equation*}
for $\hat{s}\in\mathcal{S}_\theta$.
The strict convexity of $K$ implies that if equation~\eqref{SaddlepointEquation} has a solution, then the solution is unique and we call it the \emph{saddlepoint} $\hat{s}=\hat{s}(\theta,x)$.
We write
\begin{equation}\label{mathcalXDefinition}
\begin{gathered}
\mathcal{X}_\theta = \set{x\in\R^{\xdim\times 1}\colon \exists s\in\mathcal{S}_\theta\text{ solving }K'(s;\theta)=x}, 
\qquad 
\mathcal{X} = \set{(x,\theta)\colon x\in\mathcal{X}_\theta},
\\
\mathcal{X}^o = \set{(x,\theta)\in\mathcal{X}\colon (\hat{s}(\theta,x), \theta) \in\interior\mathcal{S}}
.
\end{gathered}
\end{equation}
We will not discuss under what conditions the saddlepoint equation \eqref{SaddlepointEquation} has a solution; see for instance \cite[section~2.1]{Jensen1995Saddlepoint} or \cite[Corollary~9.6]{B-N19782014InfoExpFamilies}.
We merely remark that in many common examples, we can solve \eqref{SaddlepointEquation} for all $x$ in the interior of the convex hull of the support of $X$, but that this may fail if, for instance, $X$ is non-negative with finite mean and infinite variance.

For $x\in\interior\mathcal{X}_\theta$, we can define the \emph{saddlepoint approximation to the likelihood},
\begin{equation*}\tag{SPA}\label{SPA}
\hat{L}(\theta;x) = \frac{\exp\left( K(\hat{s}(\theta,x);\theta) - \hat{s}(\theta,x) x \right)}{\sqrt{\det(2\pi K''(\hat{s}(\theta,x);\theta))}},
\end{equation*}
the multivariate analogue of \eqref{BasicSPA1dL}.
As in \eqref{SPMLE}, the saddlepoint MLE $\hat{\theta}_\MLE(x)$ is the value of $\theta$ that maximises $\hat{L}(\theta;x)$, if one exists, for a given observed vector $x$.

We will compare $\hat{L}(\theta;x)$ and $\hat{\theta}_\MLE(x)$ with the true likelihood $L(\theta;x)$ and true MLE $\theta_\MLE(x)$.
We are assuming that $X$ has an absolutely continuous distribution, so that the true likelihood $L(\theta;x)$ should be taken to coincide with the density function for $X$.
Complications can arise if there is ambiguity in the choice of density function -- for instance, if the density function has jumps -- and later we will impose decay bounds on $M$ that will imply that $X$ has a continuous, and therefore essentially unique, density function.
Note however that the saddlepoint approximation can be applied whether or not $X$ has a density function, and indeed even when $X$ has a discrete distribution: see \refthm{IntegerValued}.

Throughout the paper we will use the symbol $\hat{\,}$ to indicate saddlepoint approximations, rather than estimators based on observations.
Thus $\theta_\MLE$ and $\hat{\theta}_\MLE$ denote two deterministic functions of the formal argument $x$, whose nature depends on our chosen parametric model.
Although we will continue to describe $x$ as the observed value of $X$, we will think of $x$ as the arbitrary input value to the functions $\theta_\MLE$ and $\hat{\theta}_\MLE$, rather than as the result of a random experiment or sampling procedure.
When we turn to sampling distributions in Theorems~\ref{T:SamplingMLE}--\ref{T:SamplingMLEWellSpecified}, we will introduce further notation to encode any randomness in the observed value.

\subsection{The standard asymptotic regime}\lbsubsect{SAR}

Often, an approximation is given theoretical justification by proving that the approximation error becomes negligible in some relevant limit.
For the saddlepoint approximation, the most commonly treated and mathematically tractable limiting framework is to assume that $X$ is the sum of $n$ unobserved i.i.d.\ terms,
\begin{equation}\label{SARasSum}
X = \sum_{i=1}^n Y^{(i)},
\end{equation}
where $Y^{(1)},Y^{(2)},\dotsc$ are i.i.d.\ copies of a random variable $Y_\theta$ whose parametric distribution does not depend on $n$.
We also scale the observed value $x$ in a matching way.
Thus, we take
\begin{equation*}\tag{SAR}\label{SAR}
\begin{aligned}
M(s;\theta) &= M_0(s;\theta)^n, &&& K(s;\theta) &= n K_0(s;\theta), &&& x &= ny, &&& n\to\infty,
\end{aligned}
\end{equation*}
where $M_0$ and $K_0$, the MGF and CGF corresponding to $Y_\theta$, are fixed.
Throughout the paper, we will assume the relation $x=ny$ implicitly.
It can be helpful to interpret the value $y=x/n$ as the sample mean implied by an observed value $x$.
We think of $y$ as fixed (or varying within a small neighbourhood) in the limit $n\to\infty$, so that both $x$ and $X$ will be of order $n$, and constraints on the observed value $x$ will usually be expressed as constraints on $y$.
We refer to this limiting framework as the \emph{standard asymptotic regime} for the saddlepoint approximation.

In the standard asymptotic regime, the saddlepoint equation \eqref{SaddlepointEquation} simplifies to 
\begin{equation*}\tag{$\mathrm{SE}_{\mathrm{SAR}}$}\label{SESAR}
K_0'(\hat{s};\theta) = y.
\end{equation*}
Write $\hat{s}_0(\theta,y)$ for the function that maps $y\in\mathcal{Y}_\theta$ to the solution of \eqref{SESAR}, where $\mathcal{Y}_\theta,\mathcal{Y},\mathcal{Y}^o$ are defined as in \eqref{mathcalXDefinition} with $x,K$ replaced by $y,K_0$.
Note that when the relations \eqref{SAR} hold, the saddlepoint $\hat{s}(\theta,x)$ does not depend on $n$, with
\begin{equation}\label{ss0Relation}
\hat{s}(\theta,x) = \hat{s}_0(\theta,y)
.
\end{equation}
We sometimes write $\hat{s}$ for the common value in \eqref{ss0Relation} when the distinction is immaterial. 
The domains of $\hat{s}$ and $\hat{s}_0$ are related by $\mathcal{X}_\theta=\set{ny\colon y\in\mathcal{Y}_\theta}$, and similarly for $\mathcal{X},\mathcal{X}^o,\mathcal{Y},\mathcal{Y}^o$.

In the standard asymptotic regime, the saddlepoint approximation \eqref{SPA} becomes
\begin{equation*}\tag{$\mathrm{SPA}_{\mathrm{SAR}}$}\label{SPASAR}
\hat{L}(\theta;x) = \frac{\exp\left( n[K_0(\hat{s}_0(\theta,y);\theta) - \hat{s}_0(\theta,y) y] \right)}{\sqrt{\det(2\pi nK_0''(\hat{s}(\theta,y);\theta))}}.
\end{equation*}
The basic error estimate for the saddlepoint approximation states that, in the standard asymptotic regime and subject to certain technical assumptions, the relative error in the likelihood is of order $1/n$:
\begin{equation}\label{LErrorRough}
\frac{\hat{L}(\theta;x)}{L(\theta;x)} = 1+O(1/n) \qquad\text{as $n\to\infty$, for fixed $(y,\theta)\in\mathcal{Y}^o$.}
\end{equation}
Recall the convention that $x$ and $y$ are implicitly related as in \eqref{SAR}, and note that the $n$-dependence of $L$ and $\hat{L}$ is omitted from the notation.
See \eqref{LErrorPrecise} for a more precise statement.
See also \refappendix{SaddlepointVsCLT}, where we compare \eqref{LErrorRough} to its analogue for normal approximations.

\begin{remarknonumber}
The standard asymptotic regime supposes that the observed value $X$ is the sum of $n$ unobserved i.i.d.\ terms.
That is, the likelihoods $L(\theta;x),\hat{L}(\theta;x)$ and MLEs $\theta_\MLE(x),\hat{\theta}_\MLE(x)$ pertain to a single observation, $X=x$, rather than $n$ observations of the summands $Y^{(1)},\dotsc,Y^{(n)}$.
For this reason, the parameter $n$ \textbf{should not be interpreted as a sample size} in the usual sense.

A model with $k$ i.i.d.\ observations can be adapted to the framework of \eqref{SAR} by giving each observation its own entry in the vector $X$; this setup is discussed further in \refsubsubsect{MultipleSamples} and \refappendix{Examples}, \refexample{MultipleSamplesWorkings}.
Note however that the results of this paper apply for $n\to\infty$, with the sample size $k$ fixed; the joint limit $n\to\infty,k\to\infty$ is excluded from consideration.
\end{remarknonumber}

\subsection{A class of examples}\lbsubsect{IntroExample}

One application of saddlepoint MLEs has been to analyse certain \emph{latent identity} models \cite{ZhaBraFew2019}.
Each individual in a population of size $N$ is assigned one of $\ell$ latent identities, with $U_j$ the number of individuals having latent identity $j$.
The counts $U_1,\dotsc,U_\ell$ are not observed directly, and instead we observe $d$ partial totals $X_1,\dotsc,X_d$, $d<\ell$.
The latent identities are constructed sufficiently richly to make the relationship between latent identities and partial totals deterministic, and we can define a $d\times\ell$ matrix $A$ by $A_{ij}=1$ if individuals with latent identity $j$ contribute to the total $X_i$, and $A_{ij}=0$ otherwise.
Thus the vectors of counts are related by $X=AU$, and the MGF for $X$ can be computed simply by $M_X(s;\theta)=M_U(sA;\theta)$.
However, the likelihood is difficult to compute because the value of $U$ cannot be recovered from observing the value of $X$.

In many cases there is no natural model for the observed count vector $X$ itself.
However, we can naturally model the latent vector $U$ by assuming that $N\sim\Poisson(\lambda)$ and that, given $N$, each individual's latent identity is chosen independently.
The main parameter of interest is $\lambda$, the population size intensity, and the probabilities of different latent identities are determined by other model-specific parameters.

To match this model to the setup of \eqref{SAR}, we make the change of variables $\lambda=n\tilde{\lambda}$.
Note that $N\sim\Poisson(\lambda)$ has the same distribution as the sum of $n$ i.i.d.\ $\Poisson(\tilde{\lambda})$ random variables, so that $U$ and hence $X$ can be written as a sum of $n$ i.i.d.\ terms.
The scaling parameter $n$ is chosen so that $y=x/n$ is of order 1 (if the observed data vector $x$ is given, in the limit of large observed counts) or so that $\tilde{\lambda}=\lambda/n$ is of order 1 (if we consider sampling distributions, in the limit of large population size intensities).
In either case, $x$, $n$ and $\lambda$ will be all of the same asymptotic order.
For further details, see Examples~\ref{ex:X=AU} and \ref{ex:ZhaBraFew} in \refappendix{Examples}.

\subsection{Main results}\lbsubsect{MainResults}

\refthm{GradientError} states a general asymptotic error bound for the gradients of the true and saddlepoint log-likelihoods.
The error bound for the MLE depends on the structure of the model, and we distinguish two cases, which we call fully identifiable (Theorems~\ref{T:MLEerror}--\ref{T:SamplingMLEWellSpecified}) and partially identifiable (\refthm{MLEerrorPartiallyIdentifiable}).
\refthm{MLEerrorNormalApprox} states the corresponding error bounds for normal approximations.
Almost the same results apply to integer-valued random variables, and \refthm{IntegerValued} shows how the assumptions should be modified for this setting.

\subsubsection{Approximation error in the log-gradient}\lbsubsubsect{GradientResult}

Because we wish to understand the true and approximate likelihoods as functions of $\theta$, our central objects of study will be $\grad_\theta \log L$ and $\grad_\theta\log\hat{L}$ rather than $L$ and $\hat{L}$.
We therefore prove a general bound analogous to \eqref{LErrorRough}.

As we will see in \refsubsect{TwoSteps}, $L$ and its derivatives can be expressed as integrals involving $M_0(s+\ii\phi;\theta)$ and its derivatives, where $s$ is fixed and $\phi$ is integrated over $\R^{1\times\xdim}$.
To ensure that these integrals converge, we make the following technical assumptions on the growth or decay of $M_0$ and its derivatives:
\begin{align}
\label{DecayBound}
&\text{there is a continuous function $\delta\colon\interior\mathcal{S}\to(0,\infty)$ such that}
\\
&\quad \abs{\frac{M_0(s+\ii\phi;\theta)}{M_0(s;\theta)}} \leq \left( 1+\delta(s,\theta)\abs{\phi}^2 \right)^{-\delta(s,\theta)}
\quad \text{for all }\phi\in\R^{1\times\xdim}, (s,\theta)\in\interior\mathcal{S}
,
\notag
\\%\intertext{and}
\label{GrowthBound}
&\text{there is a continuous function $\gamma\colon\interior\mathcal{S}\to(0,\infty)$ such that}
\\
&\quad \abs{\frac{\partial^{k+\ell} M_0}{\partial\theta_{i_1}\dotsb\partial\theta_{i_k}\partial s_{j_1}\dotsb\partial s_{j_\ell}}(s+\ii\phi;\theta)} \leq \gamma(s,\theta) (1+\abs{\phi})^{\gamma(s,\theta)}
\notag\\
&\qquad\quad\text{for all $\phi\in\R^{1\times\xdim}$ and $(s,\theta)\in\interior\mathcal{S}$},
\notag\\
&\qquad\quad\text{for $k\in\set{0,1}$, $1\leq k+\ell\leq 6$ and for $k=2$, $0\leq\ell\leq 2$},
\notag\\
&\qquad\quad\text{and each of these partial derivatives is continuous in all its variables}
.
\notag
\end{align}
These assumptions are relatively mild: \eqref{DecayBound} asserts that $\phi\mapsto\abs{M_0(s+\ii\phi;\theta)/M_0(s;\theta)}$, which always attains its maximum at $\phi=0$, has a non-degenerate critical point at $\phi=0$ (and in particular, \eqref{DecayBound} implies \eqref{K''PosDef}) and decays at least polynomially as $\abs{\phi}\to\infty$.
Similarly, \eqref{GrowthBound} asserts that the partial derivatives grow at most polynomially in $\phi$.
Note that the expression in \eqref{DecayBound} means $\frac{\partial^\ell M_0}{\partial s_1\dotsb\partial s_\ell}$ when $k=0$, and similarly when $\ell=0$.

\begin{theorem}[Gradient error bound]\lbthm{GradientError}
If \eqref{SAR} and \eqref{DecayBound}--\eqref{GrowthBound} hold, then
\begin{equation}\label{GradientErrorEquation}
\grad_\theta \log\hat{L}(\theta;x) = \grad_\theta \log L(\theta;x) + O(1/n) \qquad\text{as }n\to\infty
\end{equation}
for $(y,\theta)\in\interior\mathcal{Y}^o$ fixed.
Moreover, given a compact subset $C\subset\interior\mathcal{Y}^o$, there exists $n_0\in\N$ such that the bound in the term $O(1/n)$ is uniform over $n\geq n_0$ and $(y,\theta)\in C$.
\end{theorem}
Note that \refthm{GradientError} includes the implicit assertion that the likelihoods and their gradients exist for $n$ sufficiently large.
However, the likelihoods may be ill-behaved for small $n$: see \refexample{QPlusGamma} in \refappendix{Examples}.
Note also our convention from \eqref{SAR} that $x=ny$.

In the rest of our results, we apply the gradient error bound from \refthm{GradientError} (or rather, its more precise analogues: see \refcoro{gradlogPhatsError} and equation \eqref{gradDifferenceAsq6}) in the neighbourhood of a local maximiser.

\subsubsection{MLE error, posterior and sampling distributions in the fully identifiable case}\lbsubsubsect{FullyIdentifiableResults}

In the limit $n\to\infty$, the asymptotic behaviour of the MLE, and of other quantities derived from the likelihood, depends on the asymptotic shape of the likelihood function near its maximum.
The simplest case occurs when this maximum occurs due a non-degenerate maximum for the leading-order exponential factor in \eqref{SPASAR}.
Concretely, if we set $x=ny_0$, $y=y_0$ and take the limit $n\to\infty$ with $y_0$ fixed, then the leading-order behaviour of the log-likelihood comes from the function
\begin{equation}\label{RateFunctionImplicit}
\theta \mapsto K_0(\hat{s}_0(\theta,y_0);\theta)-\hat{s}_0(\theta,y_0) y_0.
\end{equation}
The results in this section apply when \eqref{RateFunctionImplicit} has a non-degenerate local maximum at $\theta_0$.

\begin{theorem}[MLE error bound -- fully identifiable case]\lbthm{MLEerror}
Let $s_0,\theta_0,y_0$ be related by 
\begin{equation}\label{y0K0's0}
\begin{aligned}
y_0 &= K_0'(s_0;\theta_0)\text{ with }(s_0,\theta_0)\in\interior\mathcal{S}, & \text{or equivalently}
\\[-0.2\baselineskip]
s_0 &= \hat{s}_0(\theta_0,y_0)\text{ with }(y_0,\theta_0)\in\mathcal{Y}^o
,
\end{aligned}
\end{equation}
and suppose that \eqref{SAR} and \eqref{DecayBound}--\eqref{GrowthBound} hold.
Suppose also that
\begin{gather}
\label{RateFunctionImplicitCriticalPoint}
\grad_\theta K_0(s_0;\theta_0) = 0 \quad\text{and that}
\\
\label{RateFunctionImplicitHessianDefinite}
\begin{aligned}
H=\grad_\theta^T\grad_\theta K_0(s_0;\theta_0) - (\grad_s\grad_\theta K_0(s_0;\theta_0))^T K_0''(s_0;\theta_0)^{-1} & (\grad_s\grad_\theta K_0(s_0;\theta_0)) 
\\[-0.2\baselineskip]
&\text{ is negative definite}
.
\end{aligned}
\end{gather}
Then there exist $n_0\in\N$ and neighbourhoods $U\subset\thetadomain$ of $\theta_0$ and $V\subset\R^{\xdim\times 1}$ of $y_0$ such that, for all $n\geq n_0$ and $y\in V$, the functions $\theta\mapsto \hat{L}(\theta;x)$ and $\theta\mapsto L(\theta;x)$ have unique local maximisers in $U$.
Moreover, writing these local maximisers as $\hat{\theta}_{\MLE\,\mathrm{in}\,U}(x)$ and $\theta_{\MLE\,\mathrm{in}\,U}(x)$,
\begin{equation}
\abs{\hat{\theta}_{\MLE\,\mathrm{in}\,U}(x) - \theta_{\MLE\,\mathrm{in}\,U}(x)} = O(1/n^2) \qquad\text{as }n\to\infty,
\end{equation}
uniformly over $n\geq n_0,y\in V$.
\end{theorem}

The assumptions of \refthm{MLEerror} can be understood as follows.
We shall show that, when \eqref{y0K0's0} holds, the expressions in \eqref{RateFunctionImplicitCriticalPoint}--\eqref{RateFunctionImplicitHessianDefinite} are the gradient and Hessian, respectively, of \eqref{RateFunctionImplicit}: see the differentiation formulas in \eqref{L*0hatL*0}--\eqref{L*0Gradients} and \eqref{loghatL*Gradient}--\eqref{loghatL*Hessian}.
Thus \eqref{RateFunctionImplicitCriticalPoint}--\eqref{RateFunctionImplicitHessianDefinite} state that $\theta_0$ should be a non-degenerate local maximiser of \eqref{RateFunctionImplicit}.

More directly, we can interpret $\theta_0$ as the limiting local maximiser for the likelihood of an observation $x=ny_0$, in the limit $n\to\infty$ with $y_0$ fixed. 
In such a limit, an observed value $x=ny_0$ for $X$ amounts to an observed value $y_0$ for the true mean of the summands $Y^{(1)},\dotsc,Y^{(n)}$ from \eqref{SARasSum}.
Thus \eqref{y0K0's0}--\eqref{RateFunctionImplicitHessianDefinite} state that we should be able to recover the asymptotic (local) MLE $\theta_0$ based solely on the implied sample mean $y_0$.
We might describe \eqref{y0K0's0}--\eqref{RateFunctionImplicitHessianDefinite} as saying that the model is \emph{fully identifiable at the level of the sample mean}.

In particular, \refthm{MLEerror} applies in the well-specified case where the observed value is itself drawn according to the model with true parameter $\theta_0$, where $(0,\theta_0)\in\interior\mathcal{S}$.
Then, setting $s_0=0$ and $y_0=K_0'(0;\theta_0)=\E(Y_{\theta_0})$, the conditions \eqref{y0K0's0}--\eqref{RateFunctionImplicitCriticalPoint} hold identically, while \eqref{RateFunctionImplicitHessianDefinite} reduces to the condition that the $\xdim\times p$ matrix $\grad_s\grad_\theta K_0(0;\theta_0)$ should have rank $p$.
Because $X_{\theta_0}/n\to y_0$ and $\P(X_{\theta_0}/n\in V)\to 1$ as $n\to\infty$, the conclusions of \refthm{MLEerror} will hold with high probability with $x$ replaced by $X_{\theta_0}$.

To place the MLE approximation error $O(1/n^2)$ in context, we can compare it to the inferential uncertainty inherent in the model.
The next three theorems give the asymptotics, either Bayesian or frequentist, that result from taking $n\to\infty$.

In a Bayesian framework, let the parameter $\Theta$ be drawn according to a prior $\pi_\Theta$ on $\thetadomain$.
For a given observed value $x$, we will consider the posterior distribution $\pi_{\Theta\,\vert\,U,x}$ on a neighbourhood $U\subset\thetadomain$ with $\pi_\Theta(U)>0$, defined by the Radon-Nikodym derivative
\begin{equation}
\frac{d\pi_{\Theta\,\vert\,U,x}}{d\pi_\Theta}(\theta) = \frac{L(\theta;x)\indicator{\theta\in U}}{C}
,\qquad
C=C_{U,x}=\int_U L(\theta;x) d\pi_\Theta(\theta).
\end{equation}
We construct the \emph{saddlepoint posterior distribution} on $U$, $\hat{\pi}_{\Theta\,\vert\,U,x}$, by replacing $L$ with $\hat{L}$:
\begin{equation}\label{hatpiFormula}
\frac{d\hat{\pi}_{\Theta\,\vert\,U,x}}{d\pi_\Theta}(\theta) = \frac{\hat{L}(\theta;x)\indicator{\theta\in U}}{\hat{C}}
, \qquad 
\hat{C}=\int_U \hat{L}(\theta;x) d\pi_\Theta(\theta).
\end{equation}

\begin{theorem}[Posterior distributions]\lbthm{BayesianError}\lbthm{BAYESIANERROR}
Let $(s_0,\theta_0,y_0)$ be related as in \eqref{y0K0's0}--\eqref{RateFunctionImplicitCriticalPoint}, and suppose that \eqref{SAR}, \eqref{DecayBound}--\eqref{GrowthBound} and \eqref{RateFunctionImplicitHessianDefinite} hold.
Suppose also that the prior distribution $\pi_\Theta$ has a probability density function that is continuous and positive at $\theta_0$.
Fix $y=y_0$, $x=ny_0$.
Then there exists a neighbourhood $U\subset\thetadomain$ of $\theta_0$ such that 
\begin{equation}
\text{under $\pi_{\Theta\,\vert\,U,x}$ or $\hat{\pi}_{\Theta\,\vert\,U,x}$,} \quad \sqrt{n}\left( \Theta-\theta_0 \right) \overset{d}{\to} \mathcal{N}(0,-H^{-1}) \quad\text{as }n\to\infty,
\end{equation}
where $H$ is the negative definite matrix from \eqref{RateFunctionImplicitHessianDefinite}.
\end{theorem}

In particular, \refthm{BayesianError} shows that both the true and saddlepoint likelihoods lead to the same asymptotic posterior.
The proof will follow from a stronger statement, \refprop{BayesianErrorGeneral} in \refappendix{BayesianErrorProof}, that removes the assumption $y=y_0$.

Theorems~\ref{T:MLEerror} and~\ref{T:BayesianError} concern the deterministic functions that map an observed value $x$ to the corresponding MLE or posterior distribution, via either the true likelihood or the saddlepoint approximation.
In this description, the observed value $x$ has been treated as deterministic, separate from any consideration of the random process that might have generated this observation.
The next theorem describes the sampling distribution when the observation is itself a random variable $\chi_n$.

\begin{theorem}[Sampling distributions]\lbthm{SamplingMLE}\lbthm{SAMPLINGMLE}
Let $(s_0,\theta_0,y_0)$ be related as in \eqref{y0K0's0}--\eqref{RateFunctionImplicitCriticalPoint}, and suppose that \eqref{SAR}, \eqref{DecayBound}--\eqref{GrowthBound} and \eqref{RateFunctionImplicitHessianDefinite} hold.
Let $U$ be the neighbourhood of $\theta_0$ given by \refthm{MLEerror}.
Suppose also that $\chi_n\in\R^{\xdim\times 1}$ are random variables satisfying
\begin{equation}\label{chinCLT}
\frac{\chi_n - n y_0}{\sqrt{n}} \overset{d}{\longrightarrow} \mathcal{N}(0,\Sigma) \quad\text{as }n\to\infty,
\end{equation}
where $\Sigma\in\R^{\xdim\times\xdim}$ is a positive semi-definite matrix.
Then:
\begin{enumerate}
\item\label{item:SamplingGeneral}
The joint sampling distribution of the true and saddlepoint MLEs satisfies
\begin{equation}
\left( \sqrt{n}\left( \big. \theta_{\MLE\,\mathrm{in}\,U}(\chi_n) - \theta_0 \right), \sqrt{n}\bigl( \hat{\theta}_{\MLE\,\mathrm{in}\,U}(\chi_n)-\theta_0 \bigr) \right)
\overset{d}{\longrightarrow} (Z, Z)\quad\text{as }n\to\infty
\end{equation}
with
\begin{equation}
Z \sim \mathcal{N}\left( 0, H^{-1} B^T A^{-1} \Sigma A^{-1} B H^{-1} \right)
,
\end{equation}
where we have abbreviated $A=K''_0(s_0;\theta_0)$, $B=\grad_s\grad_\theta K_0(s_0;\theta_0)$, and $H$ is the negative definite matrix from \eqref{RateFunctionImplicitHessianDefinite}.

\item\label{item:SamplingWellSpecified}
If in addition $s_0=0$, $y_0=K'_0(0;\theta_0)$ and $\Sigma = K''_0(0;\theta_0)$, then the limiting distribution has
\begin{equation}
Z \sim \mathcal{N}\left( 0, -H^{-1} \right).
\end{equation}
\end{enumerate}
\end{theorem}

Theorems~\ref{T:MLEerror} and \ref{T:SamplingMLE}\ref{item:SamplingWellSpecified} apply in particular in the well-specified case where the observed data are drawn according to the model distribution $X_{\theta_0}$, with $s_0=0$ and $y_0=\E(Y_{\theta_0})$:

\begin{theorem}[Sampling distribution in the well-specified case]\lbthm{SamplingMLEWellSpecified}\lbthm{SAMPLINGMLEWELLSPECIFIED}
Let $\theta_0\in\mathcal{R}$ be such that $(0,\theta_0)\in\interior\mathcal{S}$, and set $s_0=0$.
Suppose that \eqref{SAR} and \eqref{DecayBound}--\eqref{GrowthBound} hold and that 
\begin{equation}\label{MeanGradientFullRank}
B=\grad_s\grad_\theta K_0(0;\theta_0) \quad\text{has rank $p$.}
\end{equation}
Then:
\begin{enumerate}
\item\label{item:HReduces}
The matrix $H$ from \eqref{RateFunctionImplicitHessianDefinite} reduces to $H=- B^T K_0''(0;\theta_0)^{-1} B$ and is negative definite.

\item\label{item:ConsistentAN}
With observed data $X_{\theta_0}$, both $\theta_{\MLE\,\mathrm{in}\,U}(X_{\theta_0})$ and $\hat{\theta}_{\MLE\,\mathrm{in}\,U}(X_{\theta_0})$ are consistent and asymptotically normal estimators of $\theta_0$ in the limit $n\to\infty$, with
\begin{equation}
\left( \sqrt{n}\left( \big. \theta_{\MLE\,\mathrm{in}\,U}(X_{\theta_0}) - \theta_0 \right), \sqrt{n}\bigl( \hat{\theta}_{\MLE\,\mathrm{in}\,U}(X_{\theta_0})-\theta_0 \bigr) \right)
\overset{d}{\longrightarrow} (Z, Z)\quad\text{as }n\to\infty,
\end{equation}
where $Z\sim \mathcal{N}\left( 0, -H^{-1} \right)$.
Moreover
\begin{equation}
\abs{\hat{\theta}_{\MLE\,\mathrm{in}\,U}(X_{\theta_0}) - \theta_{\MLE\,\mathrm{in}\,U}(X_{\theta_0})} = O_\P(1/n^2) \qquad\text{as }n\to\infty.
\end{equation}

\item\label{item:HEstimator}
With $\Theta=\theta_{\MLE\,\mathrm{in}\,U}(X_{\theta_0})$ or $\Theta=\hat{\theta}_{\MLE\,\mathrm{in}\,U}(X_{\theta_0})$, the Hessians $\frac{1}{n}\grad_\theta^T\grad_\theta\log L(\Theta;X_{\theta_0})$, $\frac{1}{n}\grad_\theta^T\grad_\theta\log\hat{L}(\Theta;X_{\theta_0})$ and $- (\grad_s\grad_\theta K_0(0;\Theta))^T K_0''(0;\Theta)^{-1} (\grad_s\grad_\theta K_0(0;\Theta))$ are consistent estimators of $H$.
\end{enumerate}
\end{theorem}

A key conclusion from these results is that the approximation error in using the saddlepoint MLE in place of the true MLE is negligible, in the limit $n\to\infty$ as in \eqref{SAR}, compared to the underlying inferential uncertainty.
Namely, according to \refthm{MLEerror}, the difference between the true and saddlepoint MLEs is of order $1/n^2$.
Asymptotically, this approximation error is much smaller than the spatial scale $1/\sqrt{n}$ corresponding to either sampling variability of the MLE (in the frequentist setup of Theorems~\ref{T:SamplingMLE}--\ref{T:SamplingMLEWellSpecified}) or posterior uncertainty of the parameter (in the Bayesian setup of \refthm{BayesianError}).
To the extent that the assumptions of \eqref{SAR} and Theorems~\ref{T:GradientError}--\ref{T:SamplingMLEWellSpecified} apply in a given application, the saddlepoint likelihood and saddlepoint MLE may therefore be appropriate as readily-calculated substitutes for the true likelihood and MLE.

\subsubsection{MLE error in the partially identifiable case}\lbsubsubsect{PartiallyIdentifiableResult}

Theorems~\ref{T:MLEerror}--\ref{T:SamplingMLEWellSpecified} apply to models that are fully identifiable at the level of the sample mean.
However, many reasonable models lack this property, notably when some parameters affect the variance only.
Then lower-order contributions to \eqref{SPASAR} become relevant, and the scaling of the MLE approximation error changes.
The following theorem is the analogue of \refthm{MLEerror} in this case, for the well-specified setting (see \refsubsubsect{WellSpecified}) where $s_0=0$.

\begin{theorem}[MLE error bound -- partially identifiable case]\lbthm{MLEerrorPartiallyIdentifiable}\lbthm{MLEERRORPARTIALLYIDENTIFIABLE}
Suppose we can split the parameter vector as
\begin{equation}\label{thetaSplitting}
\theta=\mat{\omega\\ \nu} \quad\text{such that}\quad K_0'(0;\theta)=\E(Y_\theta)\text{ depends only on $\omega$}
,
\end{equation}
where $\omega\in\R^{p_1\times 1}$, $\nu\in\R^{p_2\times 1}$, $p_1+p_2=p$.
Let $\theta_0=\smallmat{\omega_0\\ \nu_0}$ be such that $(0,\theta_0)\in\interior\mathcal{S}$, and suppose that \eqref{SAR} and \eqref{DecayBound}--\eqref{GrowthBound} hold.
Suppose further that the $\xdim\times p_1$ matrix $B_\omega = \grad_s\grad_\omega K_0(0;\theta_0)$ has rank $p_1$.

Introduce the ``partially linearised'' model in which $\Xi_{w,\nu}\in\R^{\xdim\times 1}$ is normally distributed with mean vector $B_\omega w$ and covariance matrix $K_0''\left( 0;\smallmat{\omega_0\\ \nu} \right)$, where the parameters are $w\in\R^{p_1\times 1}$ and $\nu\in\R^{p_2\times 1}$.
Consider $\xi_0\in\R^{\xdim\times 1}$, and suppose that $(w,\nu)=(w_0,\nu_0)$ is a non-degenerate local MLE for the observation $\Xi_{w,\nu}=\xi_0$.
Make the change of variables
\begin{equation}\label{PIChangeOfVariables}
x = n K_0'\left( 0;\smallmat{\omega'\\ \nu_0} \right) + \sqrt{n} \xi, \qquad y = x/n = K_0'\left( 0;\smallmat{\omega'\\ \nu_0} \right) + \xi/\sqrt{n},
\end{equation}
where $\omega'\in\R^{p_1\times 1}$, $\xi\in\R^{\xdim\times 1}$.
Then:
\begin{enumerate}
\item\label{item:PI1/n}
There exist $n_0\in\N$ and neighbourhoods $U,U'\subset\thetadomain$ of $\theta_0$ and $\tilde{V}\subset\R^{\xdim\times 1}$ of $\xi_0$ such that, whenever $n\geq n_0$, $\smallmat{\omega'\\ \nu_0}\in U'$, $\xi\in\tilde{V}$, and \eqref{PIChangeOfVariables} holds, the functions $\theta\mapsto \hat{L}(\theta;x)$ and $\theta\mapsto L(\theta;x)$ have unique local maximisers in $U$.
Moreover, writing these local maximisers as $\hat{\theta}_{\MLE\,\mathrm{in}\,U}(x) = \smallmat{\hat{\omega}_{\MLE\,\mathrm{in}\,U}(x)\\ \hat{\nu}_{\MLE\,\mathrm{in}\,U}(x)}$ and $\theta_{\MLE\,\mathrm{in}\,U}(x) = \smallmat{\omega_{\MLE\,\mathrm{in}\,U}(x)\\ \nu_{\MLE\,\mathrm{in}\,U}(x)}$,
\begin{equation}\label{omeganuErrors}
\begin{aligned}
\abs{\hat{\omega}_{\MLE\,\mathrm{in}\,U}(x) - \omega_{\MLE\,\mathrm{in}\,U}(x)} &= O(1/n^{3/2}),
\\
\abs{\hat{\nu}_{\MLE\,\mathrm{in}\,U}(x) - \nu_{\MLE\,\mathrm{in}\,U}(x)} &= O(1/n),
\end{aligned}
\qquad\text{as }n\to\infty,
\end{equation}
uniformly over $n\geq n_0$, $\xi\in\tilde{V}$.

\item\label{item:PI1/n^2}
Suppose in addition that
\begin{equation}\label{PIOffDiagonalVanishes}
\tilde{B}_\omega^T \tilde{A}^{-1} \frac{\partial K_0''}{\partial\nu_j}(0;\theta) \tilde{J} = 0
\quad\text{for all $\theta$ and for }j=1,\dotsc,p_2
,
\end{equation}
where we have abbreviated $\tilde{A}=K_0''(0;\theta)^{-1}$, $\tilde{B}_\omega^T=\grad_s\grad_\omega K_0''(0;\theta)$ and $\tilde{J} = \tilde{A}^{-1} - \tilde{A}^{-1} \tilde{B}_\omega (\tilde{B}_\omega^T \tilde{A}^{-1} \tilde{B}_\omega)^{-1} \tilde{B}_\omega^T \tilde{A}^{-1}$.
Suppose also that \eqref{GrowthBound} holds for $k\leq 2$, $1\leq k+\ell\leq 7$ and for $k=3, \ell\leq 4$.
Then
\begin{equation}
\abs{\hat{\omega}_{\MLE\,\mathrm{in}\,U}(x) - \omega_{\MLE\,\mathrm{in}\,U}(x)} = O(1/n^2) \qquad\text{as }n\to\infty,
\end{equation}
uniformly over $n\geq n_0$, $\xi\in\tilde{V}$.
\end{enumerate}
\end{theorem}

In \refthm{MLEerrorPartiallyIdentifiable}, we are assuming, roughly, that $\omega$ is identifiable at the level of the sample mean but $\nu$ is not.
Thus with $\nu=\nu_0$ fixed, the function $(s,\omega)\mapsto K_0(s; \smallmat{\omega\\ \nu_0})$ satisfies the assumptions of Theorems~\ref{T:MLEerror}--\ref{T:SamplingMLEWellSpecified} with $s_0=0$, $y_0=K_0'(0;\theta_0)$, whereas for the function $(s,\nu)\mapsto K_0(s;\smallmat{\omega_0\\ \nu})$ the analogue of the matrix $H$ from \eqref{RateFunctionImplicitHessianDefinite} vanishes identically.
As a result, the inferential uncertainty for $\nu$ need not decrease as $n\to\infty$; see \refexample{Normal} in \refappendix{Examples} for an instance of this.
However, \refthm{MLEerrorPartiallyIdentifiable} shows that the MLE approximation error is still negligible in the limit $n\to\infty$.

The change of variables \eqref{PIChangeOfVariables} is the same one used in Edgeworth expansions, with a separation of scales between the mean, parametrised by $\omega'$, and lower-order fluctuations, parametrised by $\xi$.
The partially linearised model $\Xi_{w,\nu}$ arises from $X_\theta$ under this rescaling in the limit $n\to\infty$, and in \eqref{xi0Constraint}--\eqref{EijFormula} we give explicit conditions equivalent to the MLE assumption from \refthm{MLEerrorPartiallyIdentifiable}.

\subsubsection{Normal approximations}\lbsubsubsect{NormalApproxResult}

Given a complicated parametric model $X_\theta$, a different and more elementary approach is to replace $X_\theta$ by the normal random vector $\tilde{X}_\theta$ having the same mean and variance as functions of $\theta$.
Note that the combination of \eqref{SAR} and \eqref{PIChangeOfVariables} is the standard asymptotic setting in which to apply the (Local) Central Limit Theorem, suggesting that this approximation is reasonable.
We might therefore expect the MLE for observing $\tilde{X}_\theta=x$ to be a reasonable approximation to the MLE for observing $X_\theta=x$.
The following theorem states the asymptotic size of the resulting MLE approximation error.

\begin{theorem}[MLE error bound -- normal approximations]\lbthm{MLEERRORNORMALAPPROX}\lbthm{MLEerrorNormalApprox}
Let $\tilde{L}(\theta;x)$ denote the likelihood function for the normal approximation model $\tilde{X}_\theta \sim \mathcal{N}(K'(0;\theta), K''(0;\theta))$.
Then, under the hypotheses of \refthm{MLEerrorPartiallyIdentifiable}\ref{item:PI1/n}, there exist $n_0\in\N$ and neighbourhoods $U,U'\subset\thetadomain$ of $\theta_0$ and $\tilde{V}\subset\R^{\xdim\times 1}$ of $\xi_0$ such that, whenever $n\geq n_0$, $\smallmat{\omega'\\ \nu_0}\in U'$, $\xi\in\tilde{V}$, and \eqref{PIChangeOfVariables} holds, the function $\theta\mapsto \tilde{L}(\theta;x)$ has a unique local maximiser in $U$.
Moreover, writing this local maximiser as $\tilde{\theta}_{\MLE\,\mathrm{in}\,U}(x) = \smallmat{\tilde{\omega}_{\MLE\,\mathrm{in}\,U}(x)\\ \tilde{\nu}_{\MLE\,\mathrm{in}\,U}(x)}$,
\begin{equation}
\begin{aligned}
\abs{\tilde{\omega}_{\MLE\,\mathrm{in}\,U}(x) - \omega_{\MLE\,\mathrm{in}\,U}(x)} &= O(1/n),
\\
\abs{\tilde{\nu}_{\MLE\,\mathrm{in}\,U}(x) - \nu_{\MLE\,\mathrm{in}\,U}(x)} &= O(1/\sqrt{n}),
\end{aligned}
\qquad\text{as }n\to\infty,
\end{equation}
uniformly over $n\geq n_0$, $\xi\in\tilde{V}$.
\end{theorem}

In both Theorems~\ref{T:MLEerrorPartiallyIdentifiable} and \ref{T:MLEerrorNormalApprox}, the change of variables \eqref{PIChangeOfVariables} requires the implied sample mean $y$ to lie in a narrow ``tube'' of diameter of order $1/\sqrt{n}$ around the $p_1$-dimensional surface of possible model means, parametrised by $\omega'\mapsto K_0'(0;\smallmat{\omega'\\ \nu_0})$.
In \refthm{MLEerrorPartiallyIdentifiable}, this restriction does not seem entirely essential, and we conjecture that a similar result holds for the case $s_0\neq 0$.
\refthm{MLEerrorNormalApprox}, by contrast, can have no such analogue, even in the fully identifiable case $p_1=p$, $p_2=0$.
Indeed, away from their shared surface of possible model means, the models $\theta\mapsto X_\theta$ and $\theta\mapsto \tilde{X}_\theta$ may be completely different, and there is no reason to expect any relationship between their MLEs; see for instance \refexample{Poisson} in \refappendix{Examples}.

\subsubsection{The integer-valued case}\lbsubsubsect{IntegerValuedResults}

Finally all of these results apply to integer-valued random variables -- although, as we shall discuss in \refsubsubsect{IntegerValued}, it would be natural to make different and more flexible assumptions in the integer-valued case.

\begin{theorem}\lbthm{INTEGERVALUED}\lbthm{IntegerValued}
Let $X_\theta$ have values in $\Z^{\xdim\times 1}$ and set $L(\theta;x)=\P(X_\theta = x)$, with the restriction $x\in\Z^{\xdim\times 1}$.
Then the results of Theorems~\ref{T:GradientError}--\ref{T:MLEerrorNormalApprox} hold, with the assumption \eqref{DecayBound} replaced by the assumption that $\abs{M_0(s+\ii\phi;\theta)}< M_0(s;\theta)$ for all $(s,\theta)\in\interior\mathcal{S}$ and $\phi\in[-\pi,\pi]^{1\times\xdim}\setminus\set{0}$.
\end{theorem}

In \refthm{IntegerValued}, the assumption $\abs{M_0(s+\ii\phi;\theta)}< M_0(s;\theta)$ can be interpreted as a ``non-lattice'' condition on the distribution $Y_\theta$.
For instance, the condition fails if $Y_\theta$ has an entry with only odd values.
On the other hand, the condition holds if, for some $n\in\N$, $\P(X_\theta=x)>0$ for all $x\in\set{0,1}^{\xdim\times 1}$.
Meanwhile, because $M_0(s+\ii\phi;\theta)$ is periodic in $\phi$, \eqref{GrowthBound} can be verified by showing that the relevant partial derivatives are continuous.

\subsection{Discussion}\lbsubsect{Discussion}

\subsubsection{Application and scope of the results}\lbsubsubsect{ApplicationAndScope}

The results in this paper show that the saddlepoint MLE offers a high degree of asymptotic accuracy, considered as a substitute for the true MLE.
Most notably, when we consider the limit $n\to\infty$, the approximation error in using the saddlepoint MLE, of size $O(1/n^2)$, $O(1/n^{3/2})$ or $O(1/n)$, is negligible compared to the inferential uncertainty, of order $1/\sqrt{n}$ or larger, inherent in the MLE.

The scope for applying saddlepoint methods is reasonably wide: it is enough to know the moment generating function $M$ (either exactly or numerically to high precision on a computer).
Provided we know $M$, the saddlepoint approximation can be computed quickly, uniformly in $n$, whereas the true likelihood often becomes increasingly intractable for larger $n$.
The model need not even fall into the standard asymptotic regime described here, as in \cite[see \refexample{PedDavFok} from \refappendix{Examples}]{PedDavFok2015}; indeed, we have followed Butler in thinking of \eqref{SPASAR} as the special case of \eqref{SPA} where $X$ is a sum of $n$ i.i.d.\ terms, rather than thinking of \eqref{SPA} as the special case of \eqref{SPASAR} where $n=1$; see \cite[section~2.2.2]{Butler2007}.
In all cases, the usual suite of likelihood-based approaches can be applied, using the saddlepoint likelihood as a substitute for the true likelihood.

The encouraging asymptotic results from Theorems~\ref{T:GradientError}--\ref{T:IntegerValued} should, like many limiting statements, be interpreted with some caution in practice.
Whereas the theorems apply when, for instance, we have a sequence of observed values $x_n$ with $y_n=x_n/n\to y_0$ as $n\to\infty$, a typical application yields a single observed value $x$, with $n$ fixed.
Even if we interpret the observed value as being part of an infinite sequence, we may not have access to the limiting implied sample mean $y_0$ (as in Theorems~\ref{T:MLEerror}--\ref{T:SamplingMLEWellSpecified}) or the limiting rescaled deviation $\xi_0$ (as in Theorems~\ref{T:MLEerrorPartiallyIdentifiable}--\ref{T:MLEerrorNormalApprox}), nor the corresponding parameter $\theta_0$.
Thus, in practice, rather than verifying that the Hessians from \eqref{RateFunctionImplicitHessianDefinite} or \eqref{EijFormula} are non-singular at a specified base point $\theta_0$, it may be more relevant to enquire whether these Hessians are nearly singular near the computed saddlepoint MLE $\hat{\theta}_\MLE(x)$.
However, such complications are not specific to the saddlepoint approximation: the same dilemma applies whenever we appeal to an asymptotic result to interpret a fixed dataset.

Note also that this paper is concerned primarily with approximation accuracy.
Theorems~\ref{T:GradientError}--\ref{T:IntegerValued} guarantee under broad conditions that the saddlepoint log-likelihood, its gradient, and the resulting MLE are close to the true values, provided only that $n$ is large.
However, large $n$ does not guarantee that the true or saddlepoint MLE values will be close to an underlying true parameter (if one exists) in the partially identifiable case from Sections~\ref{sss:PartiallyIdentifiableResult}--\ref{sss:NormalApproxResult}.
Such a guarantee would typically come from having $k$ i.i.d.\ observations, with $k$ large.
\refsubsubsect{MultipleSamples} outlines how repeated observations can be implemented in our notation, but the theorems consider the number of i.i.d.\ observations to be fixed.
This paper excludes from consideration the double limit $n\to\infty,k\to\infty$.

Finally, we note that our results concern local maxima and local neighbourhoods in parameter space.
The true likelihood $L(\theta;x)$ may have a complicated global structure as a function of $\theta$, with multiple local maxima, and the saddlepoint approximation cannot do better than faithfully replicating this complicated structure.
Moreover, the saddlepoint approximation might have greater error in distant parts of parameter space, so that the saddlepoint MLE might fail to exist globally even if the true likelihood has a global maximum: see \refexample{MLEWrong} in \refappendix{Examples}.
This possibility does not usually arise in practice but seems difficult to rule out \emph{a priori}.
We note however that the uniformity assertions in Theorems~\ref{T:GradientError}--\ref{T:IntegerValued} robustly handle variability in the observed values: if we assume that $x$ is drawn from the distribution $X_{\theta_0}$, the Law of Large Numbers means that $y=X_{\theta_0}/n$ will lie in a neighbourhood of $y_0=\E(Y_{\theta_0})$ for $n$ large enough.

\subsubsection{Integer-valued versus continuous distributions}\lbsubsubsect{IntegerValued}

On the face of it, the discrete \refthm{IntegerValued} parallels quite closely the continuous Theorems~\ref{T:GradientError}--\ref{T:MLEerrorNormalApprox}: the conclusions are identical, and the hypotheses are quite similar.
There is however an important contextual difference: in the discrete case, the restriction to $\interior\mathcal{S}$ and $\mathcal{Y}^o$ excludes values of interest.
For instance, suppose $X$ represents count data, $X\in\Z_+^{\xdim\times 1}$.
If we observe a count of 0 (or an observed vector including one or more zero counts) then we will be unable to find the saddlepoint $\hat{s}$.
Even if we circumvent this issue by interpreting $\hat{s}$ as the limit $s\to-\infty$, the resulting saddlepoint approximation will still diverge, and no MLE can be computed.

This distinction arises in part because, as we shall discuss in \refsubsect{TwoSteps}, the saddlepoint approximation is fundamentally a (normal) density approximation, and it can be problematic when applied as an approximation for a probability mass function.
It is the author's intention to return to this topic in future research.

\subsubsection{Heuristic for the true MLE and saddlepoint MLE}\lbsubsubsect{ImpliesHeuristic}

We give a heuristic argument for how the size of the gradient error from \refthm{GradientError} leads to the size of the MLE error in \refthm{MLEerror}.
Namely, fix $x=ny_0$ and assume as a simplification that 
\begin{itemize}
\item
the function $\theta \mapsto K_0(\hat{s}_0(\theta,y_0);\theta)-\hat{s}_0(\theta,y_0) y_0$ from \eqref{RateFunctionImplicit} -- which by \eqref{RateFunctionImplicitCriticalPoint}--\eqref{RateFunctionImplicitHessianDefinite} has a non-degenerate local maximum at $\theta=\theta_0$ -- is purely quadratic around its maximum value, say $\theta\mapsto a+\tfrac{1}{2}(\theta-\theta_0)^T H (\theta-\theta_0)$ where $H$ is negative definite;
\item
the function $\theta\mapsto -\tfrac{1}{2}\log\det K_0''(\hat{s}(\theta,y_0);\theta)$ is purely affine, say $\theta\mapsto b+u (\theta-\theta_0)$ for some fixed $u\in\R^{1\times p}$; and
\item
the difference $\grad_\theta\log\hat{L}(\theta;x)-\grad_\theta\log L(\theta;x)$ from \refthm{GradientError} has the form $\frac{1}{n}v$ for some fixed $v\in\R^{1\times p}$.
\end{itemize}
Under these assumptions
\begin{equation}\label{logLhatHeuristic}
\log\hat{L}(\theta;x) = na + \tfrac{1}{2} n (\theta-\theta_0)^T H (\theta-\theta_0) - \tfrac{\xdim}{2}\log (2\pi n) + b + u (\theta-\theta_0),
\end{equation}
where $\xdim$ is the dimension of the vectors $X,x$.
We can complete the square to find
\begin{align}
\log\hat{L}(\theta;x) = \tfrac{1}{2} n \bigl( \theta-\theta_0 &+ \tfrac{1}{n}H^{-1}u^T \bigr)^T H \left( \theta-\theta_0+\tfrac{1}{n}H^{-1}u^T \right)
\notag\\&\qquad
- \tfrac{1}{2n} u H^{-1} u^T + na - \tfrac{\xdim}{2}\log (2\pi n) + b.
\end{align}
Thus the saddlepoint MLE comes to
\begin{equation}
\hat{\theta}_\MLE(x) = \theta_0 - \tfrac{1}{n}H^{-1}u^T.
\end{equation}
For the true likelihood, define the constants $c_n=\log\hat{L}(\theta_0;x)-\log L(\theta_0;x)$.
Then 
\begin{multline}\label{logLHeuristic}
\log L(\theta;x) = na + \tfrac{1}{2} n (\theta-\theta_0)^T H (\theta-\theta_0) - \tfrac{\xdim}{2}\log (2\pi n)
\\
 + b + u (\theta-\theta_0) - c_n - \tfrac{1}{n} v (\theta-\theta_0).
\end{multline}
Comparing \eqref{logLHeuristic} with \eqref{logLhatHeuristic}, we see that changing from $\hat{L}$ to $L$ amounts to replacing $u$ by $u-\tfrac{1}{n}v$ and subtracting a constant term $c_n$.
We can again complete the square to find
\begin{align}
\log L(\theta;x) = \tfrac{1}{2} n \bigl( \theta-\theta_0 &+ \tfrac{1}{n}H^{-1}u^T - \tfrac{1}{n^2}H^{-1}v^T \bigr)^T H \left( \theta-\theta_0+\tfrac{1}{n}H^{-1}u^T - \tfrac{1}{n^2}H^{-1}v^T \right)
\notag\\&\qquad
- \tfrac{1}{2n} (u - \tfrac{1}{n}v) H^{-1} (u - \tfrac{1}{n}v)^T + na - \tfrac{\xdim}{2}\log (2\pi n) + b - c_n
\end{align}
leading to
\begin{equation}
\theta_\MLE(x) = \theta_0 - \tfrac{1}{n}H^{-1}u^T + \tfrac{1}{n^2}H^{-1}v^T,
\end{equation}
with an extra term of order $1/n^2$ in accordance with \refthm{MLEerror}.

Note that order $1/n^2$ is better than what follows from the likelihood error bound \eqref{LErrorRough} alone: knowing only that $\log\hat{L}$ and $\log L$ differ by $O(1/n)$, we could conclude at best that $\hat{\theta}_\MLE$ and $\theta_\MLE$ differ by $O(1/n)$, since that is the size of the region in which the functions remain within $O(1/n)$ of their maximum even in the fully identifiable case.
In the heuristic calculation above, however, the size of $c_n$ (the log-likelihood approximation error at $\theta=\theta_0$) was irrelevant to the size of the MLE error.
So indeed was the term $-\tfrac{\xdim}{2}\log(2\pi n)$ in $\log\hat{L}$.
In fact, even if we drop all terms arising from the factor $(\det(2\pi K''(\hat{s}(\theta,x);\theta)))^{-1/2}$ in the saddlepoint approximation $\hat{L}$, the heuristic suggests that the resulting MLE approximation would still be within $O(1/n)$ of the true MLE.
This intuition is correct: see \refthm{LowerOrderSaddlepoint}.

\subsubsection{The well-specified case}\lbsubsubsect{WellSpecified}

In Theorems~\ref{T:MLEerror}--\ref{T:SamplingMLE}, some simplification occurs if 
\begin{equation}\label{y0Matches}
y_0=\E(Y_{\theta_0}),
\end{equation}
i.e., if the limiting implied sample mean matches with the model mean for some parameter value.
In this case we might say that the model is ``well-specified at the level of the mean.''
By the Law of Large Numbers, this condition holds under the usual assumption of well-specifiedness where the observed value $x$ is itself drawn randomly with the distribution $X_{\theta_0}$, because then $X_{\theta_0}/n\to\E(Y_{\theta_0})$ as $n\to\infty$.

If \eqref{y0Matches} holds and $(0,\theta_0)\in\interior\mathcal{S}$, we see from \eqref{M'M''K'K''0} that $s_0=0$ is the solution of the saddlepoint equation $K_0'(s_0)=y_0$.
Since $K_0(0;\theta)=0$ for all $\theta$, it follows that
\begin{equation}
\grad_\theta K_0(0;\theta_0) = 0, \qquad \grad_\theta^T\grad_\theta K_0(0;\theta_0) = 0.
\end{equation}
Thus the condition \eqref{RateFunctionImplicitCriticalPoint} holds automatically, and the matrix $H$ from \eqref{RateFunctionImplicitHessianDefinite} simplifies to
\begin{equation}
H = - (\grad_s\grad_\theta K_0(0;\theta_0))^T K_0''(0;\theta_0)^{-1} (\grad_s\grad_\theta K_0(0;\theta_0)). 
\end{equation}
The matrix $K_0''(0;\theta_0)$ is already positive definite by assumption, see \eqref{K''PosDef}, so the condition \eqref{RateFunctionImplicitHessianDefinite} reduces to \eqref{MeanGradientFullRank}.
Note that the $\xdim\times p$ matrix $\grad_s\grad_\theta K_0(0;\theta_0)$ from \eqref{MeanGradientFullRank} is the gradient of the mapping
\begin{equation}
\theta\mapsto \E(Y_\theta)
\end{equation}
(evaluated at $\theta_0$) so the condition \eqref{RateFunctionImplicitHessianDefinite}/\eqref{MeanGradientFullRank} is equivalent to saying that the linear approximation (at $\theta_0$) to the mapping $\theta\mapsto\E(Y_\theta)$ is one-to-one.

Heuristically, if the implied sample mean $y_0$ matches with the model at parameter value $\theta_0$, and if the mapping $\theta\mapsto\E(Y_\theta)$ is one-to-one, then by the Law of Large Numbers, $y_0$ is an unlikely observation under any other parameter value $\theta\neq\theta_0$.
This matches the observation that the function \eqref{RateFunctionImplicit} will vanish at $\theta=\theta_0$ and must be strictly negative elsewhere.

Conversely, if the gradient of $\theta\mapsto \E(Y_\theta)$ has rank $p_1<p$, there will be a $p_2$-dimensional hyperplane, $p_2=p-p_1$, along which the mapping $\theta\mapsto\E(Y_\theta)$ is constant to first order.
Hence, for all $\theta$ along this surface, $\hat{s}(\theta,y_0)\approx 0$ continues to be an approximate solution of the saddlepoint equation, the leading-order coefficient $K_0(\hat{s}(\theta,y_0);\theta)-\hat{s}(\theta,y_0)y_0$ remains zero to first order, and the heuristic from \refsubsubsect{ImpliesHeuristic} fails.
In such a case it may still be possible to apply \refthm{MLEerrorPartiallyIdentifiable}, with the coordinate $\nu$ representing the position within $p_2$-dimensional level surfaces along which $\E(Y_\theta)$ is constant.

\subsubsection{Observing multiple samples}\lbsubsubsect{MultipleSamples}

As remarked in \refsubsect{SAR}, the parameter $n$ should not be interpreted as a sample size in the traditional sense since the summands $Y^{(1)},\dotsc,Y^{(n)}$ of \eqref{SARasSum} are not observed.
A model with i.i.d.\ observations $X^{(1)},\dotsc,X^{(k)}\in\R^{\xdim_0\times 1}$ can be fit into the framework of this paper by concatenating them into a vector $\vec{X}$ of dimension $\xdim=k\xdim_0$.
Likewise, the observed data values $x^{(1)},\dotsc,x^{(k)}\in\R^{\xdim_0\times 1}$ are concatenated into a vector $\vec{x}\in\R^{k\xdim_0\times 1}$.
We then study the likelihood $L_{\vec{X}}(\theta;\vec{x})$ and MLE $\theta_\MLE(\vec{x})=\argmax_\theta L_{\vec{X}}(\theta;\vec{x})$ for a single combined observation $\vec{x}$ by applying Theorems~\ref{T:GradientError}--\ref{T:IntegerValued} to $\vec{X},\vec{x}$ instead of $X,x$.

When we apply the saddlepoint approximation to the concatenated vector $\vec{X}$ of dimension $k\xdim_0$, the covariance matrix $K_{\vec{X}}''(\vec{s})$ will be block-diagonal. 
Other quantities appearing in Theorems~\ref{T:GradientError}--\ref{T:IntegerValued} also take a special form: see \refexample{MultipleSamplesWorkings} in \refappendix{Examples} for further details.
For present purposes we note that the saddlepoint approximation for $\vec{X}$ factors as a product of $k$ $\xdim_0$-dimensional saddlepoint approximations for $X^{(1)},\dotsc,X^{(k)}$, each of which is a saddlepoint approximation applied to the distribution $X_\theta$.

This paper considers only the limit $n\to\infty$ along which the saddlepoint approximation becomes more accurate.
In particular, the number $k$ of i.i.d.\ observations is considered to be fixed throughout.

\subsubsection{Sufficient statistics of exponential families}\lbsubsubsect{ExpFamily}

Consider the case where $X$ is a sufficient statistic for a full exponential family of distributions with natural parameter $\eta\in\R^{1\times\xdim}$.
That is, we assume that $p=\xdim$ and that $\eta=\eta(\theta)$ is a reparametrisation of $\theta\in\thetadomain\subset\R^{\xdim\times 1}$, i.e., the mapping $\theta\mapsto\eta(\theta)$ and its inverse are smooth, and $\eta$ varies over an open subset of $\R^{1\times\xdim}$.
Then we can write
\begin{equation}
L(x;\theta) = f(x;\theta) = h(x) \exp\left( \eta x - \rho(\eta) \right)
,\qquad
K_X(s;\theta) = \rho(\eta+s)-\rho(\eta)
,
\end{equation}
for scalar-valued functions $h,\rho$ with $\rho$ convex.
The saddlepoint equation \eqref{SaddlepointEquation} reduces to
\begin{equation}\label{SEExpFamily}
\rho'(\eta+\hat{s}) = x.
\end{equation}
In particular, the quantity $\hat{\eta}=\eta+\hat{s}$ depends on $x$ alone and is fixed as a function of $\eta$, provided \eqref{SEExpFamily} has a solution.
The saddlepoint approximation can be written as
\begin{equation}\label{SPAExpFamily}
\hat{f}(x;\theta) = \frac{\exp\left( \big. \smash{\rho(\hat{\eta}) - \hat{\eta} x } \right)}{\sqrt{\big.\smash{\det(2\pi\rho''(\hat{\eta}))}}} \exp\left( \eta x - \rho(\eta) \right).
\end{equation}
The first factor need not coincide with $h(x)$, so the saddlepoint approximation need not be exact, but because the first factor depends on $x$ only, \emph{the saddlepoint MLE is exact for an exponential family} provided that the saddlepoint approximation itself is well-defined.
Indeed, in terms of the natural parameter $\eta$, the MLE is precisely the quantity $\hat{\eta} = \eta+\hat{s}$ solving \eqref{SEExpFamily}, which we already find in the course of computing the saddlepoint.

\subsection{Guide to examples}\lbsubsect{ExamplesGuide}

\refappendix{Examples} contains a number of examples based on theory and on the literature.
\refexample{X=AU} discusses the case $X=AU$, where $A$ is a constant matrix and $U$ is a random vector with $K_U$ known.
\refexample{MultipleSamplesWorkings} gives the details for the construction of \refsubsubsect{MultipleSamples}, in which $\vec{X}$ is a concatenation of $k$ independent random vectors each formed in accordance with \eqref{SARasSum} and \eqref{SAR}.

Examples~\ref{ex:ZhaBraFew}--\ref{ex:PedDavFok} examine applications of saddlepoint MLEs in the literature \cite{ZhaBraFew2019,DavHauKraParameterLinearBirthDeath,PedDavFok2015}.
\refexample{ZhaBraFew} gives further detail for certain models mentioned in \refsubsect{IntroExample}; it illustrates the fully identifiable case of Theorems~\ref{T:MLEerror}--\ref{T:SamplingMLEWellSpecified}, and also the setup of \refexample{X=AU}.
\refexample{DavHauKra} illustrates the partially identifiable case of Theorems~\ref{T:MLEerrorPartiallyIdentifiable}--\ref{T:MLEerrorNormalApprox}, and also the setup of \refexample{MultipleSamplesWorkings}.
\refexample{PedDavFok} shows a model that falls outside the setup of \eqref{SAR}.

Examples~\ref{ex:Poisson}--\ref{ex:DecayFails} explore families of distributions for which some direct calculations are possible, including the normal, Poisson and Gamma families; families where the true likelihood has different global behaviour than the saddlepoint likelihood; and ill-behaved distributions showing how the regularity conditions \eqref{DecayBound}--\eqref{GrowthBound} may hold or fail.

\section{Structure of the saddlepoint approximation}\lbsect{SPStructure}

In \refsubsect{TwoSteps} we break down the saddlepoint approximation into two steps: an exact step based on tilting, and an approximation step based on the normal distribution.
Understanding the saddlepoint approximation via tilting is not a new idea, cf.\ for instance \cite[section~2]{Reid1988}, but here we use it to motivate a novel factorisation of the likelihood into an exact factor, which encodes the effect of tilting and is shared between the true and approximate likelihoods; and a correction term, a normal approximation of which leads to the saddlepoint likelihood.
This factorisation establishes the framework in which the proofs will take place.

As a natural by-product of the factorisation, we introduce in \refsubsect{LowerOrder} a simpler but less accurate alternative to the saddlepoint approximation, which satisfies results similar to Theorems~\ref{T:GradientError}--\ref{T:IntegerValued}.

\subsection{Tilting and the saddlepoint approximation}\lbsubsect{TwoSteps}

Calculating the saddlepoint approximation splits naturally into two steps.
For given $\theta,x$, we first compute the saddlepoint $\hat{s}(\theta,x)$ by solving \eqref{SaddlepointEquation}, and then this value is substituted into the expression from \eqref{SPA}.
As we now explain, the first step can be understood in terms of tilting, and this will clarify the nature of the approximation made in the second step.

If $X$ has density function $f$, the MGF $M(\ii\phi;\theta)$ along the imaginary axis gives the Fourier transform of $f$.
Consequently we can use the inverse Fourier transform to recover $f$:
\begin{equation}\label{MGFInversion}
f(x;\theta) = \int_{\R^{1\times\xdim}} M(\ii\phi;\theta) e^{-\ii\phi x} \frac{d\phi}{(2\pi)^\xdim} = \int_{\R^{1\times\xdim}} \exp\left( K(\ii\phi;\theta) - \ii\phi x \right) \frac{d\phi}{(2\pi)^\xdim}.
\end{equation}
In practice, either of the factors $M(\ii\phi;\theta)$ or $e^{-\ii\phi x}$ may be highly oscillatory, and will not typically cancel with each other.
To make the integral more manageable, we \emph{exponentially tilt} the distribution of $X$.
Define
\begin{equation}
f_{s_0}(x;\theta) = \frac{e^{s_0 x}f(x;\theta)}{M(s_0;\theta)}
\end{equation}
for all $s_0\in\mathcal{S}_\theta$.
Then $f_{s_0}$ is still a density function, corresponding to a tilted distribution $X_\theta^{(s_0)}$ having the same support as $X_\theta$, and we compute
\begin{equation}
M_{X^{(s_0)}}(s;\theta) = \frac{M(s_0 + s;\theta)}{M(s_0;\theta)}, \qquad K_{X^{(s_0)}}(s;\theta) = K(s_0+s;\theta)-K(s_0;\theta).
\end{equation}
Recalling \eqref{M'M''K'K''0}, we note that $K'(s_0;\theta) = \E(X_\theta^{(s_0)})$ and $K''(s_0;\theta)=\Cov(X^{(s_0)},X^{(s_0)})$ can themselves be interpreted as means and covariance matrices, respectively.
We remark also that, since $X_\theta$ and $X_\theta^{(s_0)}$ have the same support, \eqref{K''PosDef} is unaffected by the choice of $s\in\interior\mathcal{S}_\theta$.

If we apply the inversion formula \eqref{MGFInversion} to $X_\theta^{(s)}$, we can solve to find
\begin{align}
f(x;\theta) &= M(s;\theta)e^{-s x} f_{s}(x;\theta) 
\notag\\&
= \exp\left( K(s;\theta)-sx \right) \int_{\R^{1\times\xdim}} \exp\left( K(s+\ii\phi;\theta)-K(s;\theta) - \ii\phi x \right) \frac{d\phi}{(2\pi)^\xdim}.
\label{MGFInversionTilted}
\end{align}
Note that the assumption \eqref{DecayBound} implies that the integral in \eqref{MGFInversionTilted} converges absolutely and thus defines a continuous density function $x\mapsto f(x;\theta)$ whenever $(s,\theta)\in\interior\mathcal{S}$ and $n\delta(s,\theta)>1$, and this justifies our use of the Fourier inversion formula.
Subject to this condition, we can choose $s\in\mathcal{S}_\theta$ arbitrarily.
To make the integral more tractable, we wish to choose $s$ so that the linear term $-\ii\phi x$ cancels with $K(s+\ii\phi;\theta)-K(s;\theta)$ to leading order.
That is, we choose $s=\hat{s}(\theta;x)$, the solution of \eqref{SaddlepointEquation}, provided that $x\in\mathcal{X}_\theta$.
Replacing $x$ by $K'(\hat{s}(\theta,x);\theta)$, 
\begin{multline}\label{CGFInversionTilted}
f(x;\theta) = \exp\left( K(\hat{s};\theta)-\hat{s} K'(\hat{s};\theta) \right) 
\\
\cdot
\int_{\R^{1\times\xdim}} \exp\left( K(\hat{s}+\ii\phi;\theta)-K(\hat{s};\theta) - \ii\phi K'(\hat{s};\theta) \right) \frac{d\phi}{(2\pi)^\xdim}.
\end{multline}
At this point, explicit $x$ dependence has been virtually eliminated from \eqref{CGFInversionTilted}; all $x$ dependence on the right-hand side is now carried implicitly by $\hat{s}=\hat{s}(\theta,x)$.
We introduce new notation to exploit this feature.
Along with this new notation, we again switch from considering the density $f$ to the likelihood $L$.

Define
\begin{equation}\label{L*Pformula}
\begin{aligned}
L^*(s,\theta) &= \exp\left( \big. K(s;\theta) - s K'(s;\theta) \right),\\
P(s,\theta) &= \int_{\R^{1\times\xdim}} \exp\left( \big. K(s+\ii\phi;\theta)-K(s;\theta) - \ii\phi K'(s;\theta) \right) \frac{d\phi}{(2\pi)^\xdim},
\end{aligned}
\end{equation}
for all $s\in\interior\mathcal{S}_\theta$.
Setting $x=K'(s;\theta)$, \eqref{CGFInversionTilted} can be reformulated as
\begin{equation}\label{LL*P}
L(\theta; K'(s;\theta)) = L^*(s,\theta) P(s,\theta).
\end{equation}
We can recognise $\log L^*(s,\theta)$ as the negative of the relative entropy (or Kullback-Leibler divergence) for the distribution of $X^{(s)}_\theta$ relative to $X_\theta$.
We can also give a probabilistic interpretation of $P(s,\theta)$: it is the density of the tilted distribution $X^{(s)}_\theta$ at its mean $K'(s;\theta)$.

We emphasise at this point that all the calculations so far are exact.
The factor $L^*$ measures the probabilistic ``cost'' of shifting from the original distribution $X_\theta$ to the tilted distribution $X^{(s)}_\theta$, but no distributional information is lost in the tilting step.

The natural next step will be to make an approximation to $P(s,\theta)$.
Define
\begin{equation}\label{hatP}
\hat{P}(s,\theta) = \frac{1}{\sqrt{\det(2\pi K''(s;\theta))}}.
\end{equation}
Then the saddlepoint approximation \eqref{SPA} becomes
\begin{equation}\label{hatLL*P}
\hat{L}(\theta;K'(s;\theta)) = L^*(s,\theta) \hat{P}(s,\theta).
\end{equation}
We recognise $\hat{P}(s,\theta)$ as the density (at its mean) of a normal random variable with covariance matrix $K''(s;\theta)$.
Comparing \eqref{LL*P} and \eqref{hatLL*P}, we see that the saddlepoint approximation amounts to approximating $P(s,\theta)$, the density at its mean of the tilted distribution, by $\hat{P}(s,\theta)$, the density at its mean of the normal random variable with the same covariance matrix.

In the remainder of the paper, we will study the two-variable functions $L^*(s,\theta)$, $P(s,\theta)$ and $\hat{P}(s,\theta)$ and their gradients.
Since the factor $L^*(s,\theta)$ appears in both \eqref{LL*P} and \eqref{hatLL*P}, it will suffice to compare the gradients of $P$ and $\hat{P}$.
Ultimately our interest will be in the case $s=\hat{s}(\theta,x)$, and we will therefore rewrite \eqref{LL*P} and \eqref{hatLL*P} as
\begin{equation}\label{LL*Phats}
\begin{aligned}
L(\theta; x) &= L^*(\hat{s}(\theta,x),\theta) P(\hat{s}(\theta,x),\theta),\\
\hat{L}(\theta;x) &= L^*(\hat{s}(\theta,x),\theta) \hat{P}(\hat{s}(\theta,x),\theta),
\end{aligned}
\qquad\text{for }x\in\interior\mathcal{X}_\theta.
\end{equation}

\begin{remarknonumber}
A key advantage to studying $L^*(s,\theta)$, $P(s,\theta)$ and $\hat{P}(s,\theta)$ is that, because $P$ and $\hat{P}$ both represent densities at the mean, they vary less dramatically as a function of their arguments than the full likelihoods $L(\theta;x)$ and $\hat{L}(\theta;x)$, even in the limit $n\to\infty$ from \eqref{SAR}.
As we shall see, both $P$ and $\hat{P}$ and their gradients behave asymptotically as powers of $n$.
The only factor that decays exponentially in $n$ is $L^*$, and since it is a common factor of $L(\theta;x)$ and $\hat{L}(\theta;x)$ we can circumvent its effects.

The representation \eqref{LL*Phats} is also useful for numerical calculation.
Recent work by Lunde, Kleppe \& Skaug \cite{LunKleSkaPreprintSaddlepointInversion} calculates densities and likelihoods to high relative precision, even in the tails, using saddlepoint methodology.
Expressed in our notation, they evaluate $P(s,\theta)$ by applying a quadrature rule to the integral in \eqref{L*Pformula}.
Because $P(s,\theta)$ is a density at the mean, the integral defining $P(s,\theta)$ is much more amenable to numerical integration, in general, than the inverse Fourier integral \eqref{MGFInversion}.

Likewise, the fact that $P$ and $\hat{P}$ both represent densities at the mean helps to explain why the saddlepoint approximation can be so accurate even in the tails.
Given $x$ close to the boundary of $\mathcal{X}_\theta$ (or tending to infinity), the corresponding saddlepoint $\hat{s}$ will be close to the boundary of $\mathcal{S}_\theta$ (or will tend to infinity).
In such a limit, there is no reason in general to expect the ratio $\hat{P}/P$ to converge to 1; but often the tilted density at the tilted mean remains on the order of the inverse of the tilted standard deviation, so that $\hat{P}/P$ may remain bounded.
By contrast, normal approximations and other similar techniques rely on extrapolating a density far from the mean, and such an extrapolation can only be accurate if the true density happens to have the same tail behaviour far from the mean; see \refappendix{SaddlepointVsCLT}.
\end{remarknonumber}

\paragraph*{$L^*$ and $P$ in the integer-valued case}

When $X\in\Z^{\xdim\times 1}$, the analogue of \eqref{MGFInversion} is
\begin{equation}
\P(X_\theta=x) = \int_{[-\pi,\pi]^{1\times\xdim}} M(\ii\phi;\theta) e^{-\ii\phi x}\frac{d\phi}{(2\pi)^\xdim} \qquad\text{for }x\in\Z^{\xdim\times 1}.
\end{equation}
We can repeat the tilting argument above, leading us to define
\begin{equation}\label{PintFormula}
P_{\mathrm{int}}(s,\theta) = \int_{[-\pi,\pi]^{1\times\xdim}} \exp\left( \big. K(s+\ii\phi;\theta)-K(s;\theta) - \ii\phi K'(s;\theta) \right) \frac{d\phi}{(2\pi)^\xdim}.
\end{equation}
At integer values $x\in\Z^{\xdim\times 1}$, we take the likelihood function to be $L(\theta;x)=\P(X_\theta=x)$, and we will have
\begin{equation}\label{LL*Pint}
L(\theta; x) = L^*(\hat{s}(\theta,x),\theta) P_{\mathrm{int}}(\hat{s}(\theta,x),\theta).
\end{equation}
However, $P_{\mathrm{int}}(s,\theta)$ is defined whenever $s\in\interior\mathcal{S}_\theta$.
Thus we can use \eqref{LL*Pint} as a definition of $L(\theta;x)$ even when $x$ is non-integer, although $L(\theta;x)$ may not represent a probability in that case.

\subsection{A lower-order saddlepoint approximation}\lbsubsect{LowerOrder}

Approximating $P$ using first and second moments is a natural and reasonable step, but is not the only possible approach.
As we have seen, the standard saddlepoint approximation selects a density from the family of normal distributions after matching first and second moments.
An alternative is to use a different reference family of distributions: this is one approach to non-Gaussian saddlepoint approximations, originally developed by Wood, Booth \& Butler~\cite{WooBooBut1993} in a different context and with different tools.

An even simpler alternative, however, is to ignore $P$ altogether.
Define
\begin{equation}\label{hatL*}
\hat{L}^*(\theta;x) = L^*(\hat{s}(\theta,x),\theta)
\end{equation}
for $x\in\mathcal{X}_\theta$.
We could describe $\hat{L}^*(\theta;x)$ as the ``zeroth-order'' saddlepoint approximation to the likelihood.
We saw the quantity $\frac{1}{n}\log\hat{L}^*(\theta;ny_0)$ and its derivatives in \eqref{RateFunctionImplicit} and \eqref{RateFunctionImplicitCriticalPoint}--\eqref{RateFunctionImplicitHessianDefinite}, and $I_\theta(x) = -\log\hat{L}^*(\theta;x)$ is the large deviations rate function from Cram\'{e}r's theorem applied to $X_\theta$.  
Equivalently, $-\log\hat{L}^*(\theta;\cdot)$ is the Legendre transform of $K(\cdot;\theta)$; see for instance \cite[sections~I.4 and V.1]{dH2000} or \cite[section~1.2]{Jensen1995Saddlepoint}.

Form the corresponding maximum likelihood estimator
\begin{equation}
\hat{\theta}^*_\MLE(x) = \argmax_\theta \hat{L}^*(\theta;x)
\end{equation}
when it exists, and likewise define $\hat{\pi}^*_{\Theta\,\vert\,U,x}$ as in \eqref{hatpiFormula} with $\hat{L}$ replaced by $\hat{L}^*$.
We remark that in the standard asymptotic regime \eqref{SAR} where $x=ny$, the MLE $\hat{\theta}^*_\MLE(x)$ depends on $y$ but not on $n$; the same is true if we maximise over $\theta$ restricted to a neighbourhood $U$.

Applied to $\hat{L}^*$, the conclusions of Theorems~\ref{T:GradientError}--\ref{T:SamplingMLEWellSpecified} and \ref{T:IntegerValued} are almost unchanged, except that error terms have powers of $n$ changed by one:
\begin{theorem}[Zeroth-order saddlepoint approximation -- fully identifiable case]\lbthm{LowerOrderSaddlepoint}
~
\begin{enumerate}
\item\label{item:LOGradient}
Under the hypotheses of \refthm{GradientError}, 
\begin{equation}
\grad_\theta \log\hat{L}^*(\theta;x) = \grad_\theta \log L(\theta;x) + O(1) \qquad\text{as }n\to\infty
\end{equation}
for $(y,\theta)\in\mathcal{Y}^o$, with uniformity if $(y,\theta)$ is restricted to a compact subset of $\mathcal{Y}^o$.

\item\label{item:LOMLE}
Under the hypotheses of \refthm{MLEerror}, there exist $n_0\in\N$ and neighbourhoods $U\subset\thetadomain$ of $\theta_0$ and $V\subset\R^{\xdim\times 1}$ of $y_0$ such that, for all $n\geq n_0$ and $y\in V$, the function $\theta\mapsto \hat{L}^*(\theta;x)$ has a  unique local maximiser in $U$ and, writing this local maximiser as $\hat{\theta}^*_{\MLE\,\mathrm{in}\,U}(x)$,
\begin{equation}
\abs{\hat{\theta}^*_{\MLE\,\mathrm{in}\,U}(x) - \theta_{\MLE\,\mathrm{in}\,U}(x)} = O(1/n) \qquad\text{as }n\to\infty.
\end{equation}

\item\label{item:LOBayesian}
Under the hypotheses of \refthm{BayesianError}, with $x=ny_0$, the distribution of the rescaled parameter $\sqrt{n}\left( \Theta-\theta_0 \right)$ under $\hat{\pi}^*_{\Theta\,\vert\,U,x}$ converges as $n\to\infty$ to the same limiting distribution $\mathcal{N}(0,-H^{-1})$, where $H$ is the negative definite matrix from \eqref{RateFunctionImplicitHessianDefinite}.

\item\label{item:LOSampling}
Under the hypotheses of Theorems~\ref{T:SamplingMLE} or \ref{T:SamplingMLEWellSpecified}, the sampling distributions of the rescaled approximate MLEs $\sqrt{n}\bigl( \hat{\theta}^*_{\MLE\,\mathrm{in}\,U}(\chi_n)-\theta_0 \bigr)$ or $\sqrt{n}\bigl( \hat{\theta}^*_{\MLE\,\mathrm{in}\,U}(X_{\theta_0})-\theta_0 \bigr)$, respectively, converge as $n\to\infty$ to $Z$, where $Z$ is as in Theorems~\ref{T:SamplingMLE} or \ref{T:SamplingMLEWellSpecified}, respectively.
These convergences occur jointly with the convergences from Theorems~\ref{T:SamplingMLE} or \ref{T:SamplingMLEWellSpecified}, with the same limiting random variable $Z$ in each case.
Furthermore \refthm{SamplingMLEWellSpecified}\ref{item:HEstimator} applies also with $\Theta=\hat{\theta}^*_{\MLE\,\mathrm{in}\,U}(X_{\theta_0})$.

\item\label{item:LOIntegerValued}
The above results also apply under the hypotheses of \refthm{IntegerValued}.

\item\label{item:LOExpFamily}
Suppose $X$ is the sufficient statistic for an exponential family indexed by the natural parameter $\eta$.
Let $\hat{\eta}^*_\MLE(x)$ be the parameter obtained by maximising $\hat{L}^*$, provided the maximiser exists.
Then $\hat{\eta}^*_\MLE(x)$ coincides with the true MLE.
\end{enumerate}
\end{theorem}

It is notable that maximizing $\hat{L}^*$ results in an approximation to the MLE whose error as $n\to\infty$ is still smaller than the spatial scale $1/\sqrt{n}$, even though $\hat{L}^*$ is no longer an adequate approximation to the likelihood itself in the sense that $\hat{L}^*/L\to \infty$ as $n\to\infty$.

The proofs of \ref{item:LOGradient}--\ref{item:LOIntegerValued} are given along with the corresponding proofs of Theorems~\ref{T:GradientError}--\ref{T:SamplingMLEWellSpecified} and \ref{T:IntegerValued}.
Part~\ref{item:LOExpFamily} holds by the same reasoning as in \refsubsubsect{ExpFamily}: in the first factor in \eqref{SPAExpFamily}, we now remove the denominator, and the resulting factor still does not depend on $\eta$.

\section{Proofs of Theorems~\ref{T:GradientError}--\ref{T:MLEerror}}\lbsect{Proofs}

The proof of \refthm{GradientError} is based on two key results, \refprop{PddtPError} and \refcoro{gradlogPhatsError}, that refine the statement of \refthm{GradientError} using the factorisation from \refsubsect{TwoSteps}.
\refthm{MLEerror} follows using a scaling analysis and the Implicit Function Theorem.
We begin with some derivative formulas; see \refappendix{DerivativesDerivation} for their derivation.

\subsection{Summary of saddlepoint derivatives}\lbsubsect{DerivativesSummary}

Under the scaling of \eqref{SAR}, the quantities $\log L^*(s,\theta)$, $\log L^*(\hat{s}(\theta,x),\theta)$ and their gradients are proportional to $n$.
Define
\begin{equation}\label{L*0hatL*0}
\begin{gathered}
L^*_0(s,\theta) = \exp\left( K_0(s;\theta) - s K'_0(s;\theta) \right),
\\
\hat{L}^*_0(\theta;y) = L^*_0(\hat{s}_0(\theta,y),\theta) = \exp\left( K_0(\hat{s}_0(\theta,y);\theta) - \hat{s}_0(\theta,y) y \right),
\end{gathered} 
\end{equation}
so that
\begin{equation}\label{L*L*0}
\begin{aligned}
\log L^*(s,\theta) &= n\log L^*_0(s,\theta), 
\quad
\grad_\theta \log\hat{L}^*(\theta;x) = n \grad_\theta\log\hat{L}^*_0(\theta;y)
\end{aligned}
\end{equation}
and so on, where
\begin{equation}\label{L*0Gradients}
\begin{aligned}
\grad_\theta\log L^*_0(s,\theta) &= \grad_\theta K_0(s;\theta) - s \grad_s\grad_\theta K_0(s;\theta),\\
\grad_s \log L^*_0(s,\theta) &= -K_0''(s;\theta)s^T,\\
\grad_\theta \log\hat{L}^*_0(\theta;y) &= \grad_\theta K_0(\hat{s}_0(\theta,y);\theta),\\
\grad_\theta^T\grad_\theta \log\hat{L}^*_0(\theta;y) &= \grad_\theta^T\grad_\theta K_0(\hat{s};\theta) - \left( \grad_s\grad_\theta K_0(\hat{s};\theta) \right)^T K''_0(\hat{s};\theta)^{-1} \grad_s\grad_\theta K_0(\hat{s};\theta)
\end{aligned}
\end{equation}
By contrast, $\hat{s}, \grad_\theta\hat{s}^T$ and $\frac{\partial}{\partial t} \log\hat{P}$ do not depend on $n$:
\begin{equation}\label{s0hatPGradients}
\begin{aligned}
\hat{s}(\theta,x) &= \hat{s}_0(\theta,y),\\
\grad_\theta\hat{s}^T(\theta,x) &= \grad_\theta\hat{s}_0^T(\theta,y) = - K''_0(\hat{s}(\theta,x);\theta)^{-1} \grad_s\grad_\theta K_0(\hat{s}(\theta,x);\theta),\\
\frac{\partial}{\partial t}\log\hat{P}(s,\theta) &= -\frac{1}{2}\trace\left( K''_0(s;\theta)^{-1} \frac{\partial K''_0}{\partial t}(s;\theta) \right).
\end{aligned}
\end{equation}
For the remainder of the proofs, we will emphasise that $P(s,\theta)$, $P_{\mathrm{int}}(s,\theta)$, $L(\theta;x)$, $\hat{L}(\theta;x)$, $\theta_\MLE(x)$, $\hat{\theta}_\MLE(x)$ depend on $n$ by writing them as $P_n(s,\theta)$, $P_{\mathrm{int},n}(s,\theta)$, $L_n(\theta;x)$, $\hat{L}_n(\theta;x)$, $\theta_\MLE(x,n)$, $\hat{\theta}_\MLE(x,n)$.
The $n$-dependences of $M(s;\theta)$, $K(s;\theta)$, $\log L^*(s,\theta)$ and $\log\hat{L}^*(\theta;x)$ are simple in form and we will handle them by directly substituting the formulas in \eqref{SAR} and \eqref{L*L*0}.
We remark that the expression in \eqref{RateFunctionImplicit} reduces to $\log\hat{L}^*_0(\theta;y_0)$, and \eqref{RateFunctionImplicitCriticalPoint}--\eqref{RateFunctionImplicitHessianDefinite} amount to the assertion that $\theta_0$ is a non-degenerate local maximum for $\theta\mapsto\log\hat{L}^*_0(\theta;y_0)$.
Note that $\hat{P}_n(s,\theta)$ also depends on $n$ in a simple way,
\begin{equation}\label{hatPn}
\hat{P}_n(s,\theta) = \frac{1}{\sqrt{\det(2\pi n K_0''(s;\theta))}} = \frac{n^{-\xdim/2}}{\sqrt{\det(2\pi K_0''(s;\theta))}},
\end{equation}
but by \eqref{s0hatPGradients} the gradients $\grad_s\log\hat{P}(s,\theta),\grad_\theta\log\hat{P}(s,\theta)$ do not depend on $n$ and we will omit the subscript in those cases.

\subsection{Proof of \refthm{GradientError}}\lbsubsect{MainPropStatement}

Unlike $L^*$ and $\hat{P}_n$, the quantity $P_n$ has no closed form and is instead given as an integral as in \eqref{L*Pformula}.
In the standard asymptotic regime given by \eqref{SAR}, we can substitute $K=nK_0$ to obtain
\begin{equation}\label{PformulaK}
P_n(s,\theta) = \int_{\R^{1\times\xdim}} \exp\left( n[K_0(s+\ii\phi;\theta) - K_0(s;\theta) - \ii\phi K_0'(s;\theta)] \right)\frac{d\phi}{(2\pi)^\xdim}.
\end{equation}
Note that the integrand in \eqref{PformulaK} has the form $h(\phi)e^{ng(\phi)}$ with $g(0)=0$ and $g'(0)=0$.
This is the standard setup for applying the multivariate Laplace method, see \cite{Wong2001,BleiNor1986}.
Thus the limiting framework of \eqref{SAR}, in which $X$ is the sum of $n$ i.i.d.\ terms, leads us to expect that $P_n(s,\theta)$, after suitable rescaling by a power of $n$, will have an asymptotic series expansion in powers of $1/n$.

Let the scalar $t$ denote one of the coordinates $\theta_i$ or $s_j$.
As we will show in the proof of \refprop{PddtPError}, for large enough $n$ we may differentiate under the integral sign:
\begin{align}
\frac{\partial P_n}{\partial t}(s,\theta) &= \int_{\R^{1\times\xdim}} \exp\left( n[K_0(s+\ii\phi;\theta) - K_0(s;\theta) - \ii\phi K_0'(s;\theta)] \right) 
\notag\\&\qquad\qquad\cdot
n\left[ \frac{\partial K_0}{\partial t}(s+\ii\phi;\theta) - \frac{\partial K_0}{\partial t}(s;\theta) - \ii\phi \frac{\partial K'_0}{\partial t}(s;\theta) \right] \frac{d\phi}{(2\pi)^\xdim}
.
\label{dPdtFormulaK}
\end{align}
A key observation motivating \refthm{GradientError} is that $\frac{\partial P_n}{\partial t}(s,\theta)$ scales according to the same $n$-dependent factor as $P_n(s,\theta)$.
Indeed, the values of both integrals arise primarily from the region where $\phi$ is of order $1/\sqrt{n}$, and the extra factor in \eqref{dPdtFormulaK} is of order 1 in that region.
We could summarise by saying that taking gradients of $P_n(s,\theta)$ with respect to $\theta$ or $s$ does not substantially change the nature of the $n$-dependence.

When we turn to MLEs, we will need to control the dependence of $P_n(s,\theta)$ (and its gradients) on $s$, $\theta$ and $n$ simultaneously.
Specifically, the proof of \refthm{MLEerror} is based on the Implicit Function Theorem, which requires continuous differentiability.
We will therefore prove the following result, which is more precise than is necessary for \refthm{GradientError}.

\begin{prop}\lbprop{PDDTPERROR}\lbprop{PddtPError}
Under the hypotheses of \refthm{GradientError}, with $t$ being one of the entries $\theta_i$ or $s_j$, there are continuously differentiable functions $q_1(s,\theta,\epsilon),q_2(s,\theta,\epsilon)$, defined on an open set $\mathcal{Q}$ containing $\set{(s,\theta,0)\colon (s,\theta)\in\interior\mathcal{S}}$, such that $q_1(s,\theta,0)=q_2(s,\theta,0)=0$ and
\begin{equation}\label{PddtPErrorFormula}
\begin{aligned}
P_n(s,\theta) &= \hat{P}_n(s,\theta)\left( 1 + q_1(s,\theta,1/n) \right), 
\\
\frac{\partial P_n}{\partial t}(s,\theta) &= \hat{P}_n(s,\theta) \left( \frac{\partial}{\partial t}\log\hat{P}(s,\theta) + q_2(s,\theta,1/n) \right)
\end{aligned}
\end{equation}
whenever $n$ is large enough that $(s,\theta,1/n)\in\mathcal{Q}$.
\end{prop}

As with \refthm{GradientError}, \refprop{PddtPError} includes the implicit assertion that all the quantities in \eqref{PddtPErrorFormula} exist when $(s,\theta,1/n)\in\mathcal{Q}$.

An almost identical statement holds in the integer-valued case:

\begin{prop}\lbprop{PDDTPERRORINT}\lbprop{PddtPErrorInt}
Under the hypotheses of \refthm{IntegerValued}, the conclusions of \refprop{PddtPError} hold with $P_n(s,\theta)$ replaced by $P_{\mathrm{int},n}(s,\theta)$.
\end{prop}

The proofs of Propositions~\ref{P:PddtPError}--\ref{P:PddtPErrorInt} use many of the elements of standard proofs of Laplace's method.
Additional care is needed to ensure that $q_1,q_2$ are continuously differentiable, and we defer the details to \refappendix{MainPropProof}.

Everything we will use from Propositions~\ref{P:PddtPError}--\ref{P:PddtPErrorInt} can be encapsulated in the following corollary, or its analogue for $P_{\mathrm{int},n}(s,\theta)$.

\begin{coro}\lbcoro{GRADLOGPHATSERROR}\lbcoro{gradlogPhatsError}
Under the hypotheses of \refthm{GradientError}, there is a continuously differentiable function $q_3(\theta,y,\epsilon)$, with values in $\R^{1\times p}$, defined on an open set $\mathcal{Q}'$ containing $\set{(\theta,y,0)\colon (y,\theta)\in\mathcal{Y}^o}$, such that $q_3(\theta,y,0)=0$ and
\begin{equation}\label{gradlogDifferenceshat}
\grad_\theta \left( \Big. \log P_n(\hat{s}_0(\theta,y),\theta) \right) = \grad_\theta \left( \log\hat{P}(\hat{s}_0(\theta,y),\theta) \right) + q_3(\theta,y,1/n)
\end{equation}
whenever $(\theta,y,1/n)\in\mathcal{Q}'$.
\end{coro}

The proof is again deferred to \refappendix{MainPropProof}.
\refthm{GradientError} now follows immediately.

\begin{proof}[Proof of \refthm{GradientError}]
Use \eqref{LL*Phats} and cancel the term $\log L^*$ to obtain
\begin{align}
\grad_\theta\log L(\theta;x) - \grad_\theta\log\hat{L}(\theta;x) &= \grad_\theta \left( \Big. \log P_n(\hat{s}_0(\theta;y),\theta) \right) - \grad_\theta \left( \Big. \log \hat{P}(\hat{s}_0(\theta;y),\theta) \right)
\notag\\&
= q_3(\theta,y,1/n)
.
\label{gradDifferenceAsq6}
\end{align}
By \refcoro{gradlogPhatsError}, $q_3$ is continuously differentiable and $q_3(\theta,y,0)=0$ in a neighbourhood of $(\theta_0,y_0)$.
An application of the Mean Value Theorem, see for instance \cite[Theorem~5.10]{RudinPoMA}, completes the proof.
\end{proof}

Although it is not necessary to our study of MLEs, we note that the same argument yields
\begin{equation}\label{LErrorPrecise}
\log L(\theta;x) - \log\hat{L}(\theta;x) = \tilde{q}(\theta,y,1/n),
\end{equation}
where $\tilde{q}(\theta,y,\epsilon)$ is a continuously differentiable function defined on $\mathcal{Q}'$ with $\tilde{q}(\theta,y,0)=0$.
This is a more precise form of the basic saddlepoint error estimate \eqref{LErrorRough}.
\refthm{LowerOrderSaddlepoint}\ref{item:LOGradient} also follows easily:

\begin{proof}[Proof of \refthm{LowerOrderSaddlepoint}\ref{item:LOGradient}]
Since \refthm{GradientError} already gives a bound on $\grad_\theta\log\hat{L}-\grad_\theta\log L$, it suffices to show that $\grad_\theta\log\hat{L}-\grad_\theta\log\hat{L}^*=O(1)$.
From \eqref{LL*Phats} and \eqref{s0hatPGradients}, $\grad_\theta\log(\hat{L}(\theta;x)/\hat{L}^*(\theta;x))=\grad_\theta\log\hat{P}(\hat{s}_0(\theta,y),\theta)$ is constant with respect to $n$, so it is $O(1)$ in the limit $n\to\infty$.
Since it also depends continuously on $\theta$ and $y$, uniformity follows.
\end{proof}

\subsection{Proof of \refthm{MLEerror}}\lbsubsect{MLEerrorProof}

To study the MLE $\theta_{\MLE\,\mathrm{in}\,U}(x,n)$, we will show that the function $\theta\mapsto \log L_n(\theta;x)$ has a unique maximum when $\theta$ is restricted to lie in a suitably chosen neighbourhood $U$.
In fact it will be more convenient to consider the rescaled function
\begin{equation}
R_{x,n}(\theta) = \tfrac{1}{n}\log L_n(\theta;x) = \log L^*_0(\hat{s}_0(\theta,y),\theta) + \tfrac{1}{n} \log P_n(\hat{s}_0(\theta,y),\theta),
\end{equation}
where we have substituted \eqref{LL*Phats} and \eqref{L*L*0}.
Use \eqref{s0hatPGradients} and \refcoro{gradlogPhatsError} to compute
\begin{align}
\grad_\theta R_{x,n}(\theta) &= \grad_\theta K_0(\hat{s}_0(\theta,y);\theta) + \tfrac{1}{n} q_3(\theta,y,1/n) + \tfrac{1}{n} \grad_\theta \log\hat{P}(\hat{s}_0(\theta,y),\theta) 
\notag
\\&\quad
- \tfrac{1}{n} \grad_s^T \log\hat{P}(\hat{s}_0(\theta,y),\theta) K_0''(\hat{s}_0(\theta,y);\theta)^{-1} \grad_s\grad_\theta K_0(\hat{s}_0(\theta,y);\theta) 
.
\end{align}
We will define a function $F(s^T,\theta;y,\epsilon)$ such that, with the substitution $\epsilon=1/n$, a solution of $F=0$ corresponds to a critical point of $R_{x,n}$.

\begin{proof}[Proof of \refthm{MLEerror}]
Define the functions
\begin{gather}\label{F1F2FDefinition}
F(s^T,\theta;y,\epsilon) = 
\mat{
F_1(s^T,\theta;y)
\\
F_2(s^T,\theta;y,\epsilon)
}
,\qquad
F_1(s^T,\theta;y) = K_0'(s;\theta) - y,
\\
\begin{aligned}
F_2(s^T,\theta;y,\epsilon) &= \grad_\theta^T K_0(s;\theta) + \epsilon q_3(\theta,y,\epsilon)^T
\notag\\&\quad
+ \epsilon \left( \grad_\theta^T \log \hat{P}(s,\theta) - \left( \grad_s\grad_\theta K_0(s;\theta) \right)^T K''_0(s;\theta)^{-1}\grad_s \log \hat{P}(s,\theta) \right) 
.
\end{aligned}
\end{gather}
We think of $F_1,F_2,F$ as column-vector-valued functions of column-vector arguments, with $(s^T,\theta;y,\epsilon)$ and $F(s^T,\theta;y,\epsilon)$ interpreted as column vectors expressed in block form, of sizes $(2\xdim+p+1)\times 1$ and $(\xdim+p)\times 1$ respectively.
We will show that we can solve $F=0$ to define $\theta$ and $s$ implicitly as functions of $y$ and $\epsilon$; to indicate this, we will merge the column vectors $s^T$ and $\theta$ and write $\smallmat{s^T\\ \theta}=G(y,\epsilon)\in\R^{(\xdim+p)\times 1}$ such that $F(G(y,\epsilon);y,\epsilon)=0$.

By \refcoro{gradlogPhatsError}, the function $F$ is continuously differentiable with respect to all its parameters.
Our assumptions imply that $F(s_0^T,\theta_0;y_0,0)=0$ and $\grad_{s^T,\theta} F(s_0^T,\theta_0;y_0,0)$ is non-singular; see \refappendix{FandMLEproof}.
We can therefore apply the Implicit Function Theorem (see for instance \cite[Theorem~9.28]{RudinPoMA}) to find neighbourhoods $U,V,W$ of $\theta_0,y_0,s_0$, a neighbourhood $[-1/n_0,1/n_0]$ of $0$, and a continuously differentiable function $G(y,\epsilon)$ defined on $V\times [-1/n_0,1/n_0]$ such that, for all $y\in V$, $\epsilon\in[-1/n_0,1/n_0]$, the point $\left( \begin{smallmatrix}s^T\\ \theta\end{smallmatrix} \right)=G(y,\epsilon)$ is the unique solution in $W\times U$ of $F(s^T,\theta;y,\epsilon)=0$.

As outlined above, when $\epsilon=1/n$, the solution of $F=0$ corresponds to the MLE:
\begin{lemma}\lblemma{F=0ISMLE}
Possibly after shrinking $U,V$ and increasing $n_0$, we have
\begin{equation}
G(y,1/n) = \mat{\hat{s}_0^T(\theta_{\MLE\,\mathrm{in}\,U}(x,n), y) \\ \theta_{\MLE\,\mathrm{in}\,U}(x,n)} 
\end{equation}
for all $y\in V$ and $n\geq n_0$, including the assertion that the maximum of $L_n(\theta;x)$, restricted to $\theta\in U$, is attained uniquely.
\end{lemma}

We defer the proof to \refappendix{FandMLEproof}.

We now turn to the saddlepoint MLE, which amounts to omitting the term $\epsilon q_3$:
\begin{align}
\hat{F}_2(s^T,\theta;y,\epsilon) &= \grad_\theta^T K_0(s;\theta) + \epsilon \grad_\theta^T \log \hat{P}(s,\theta) 
\notag\\&\quad
- \epsilon \left( \grad_s\grad_\theta K_0(s;\theta) \right)^T K''_0(s;\theta)^{-1} \grad_s \log \hat{P}(s,\theta)
, 
\notag\\
\hat{F}(s^T,\theta;y,\epsilon) &= 
\mat{
F_1(s^T,\theta;y)
\\
\hat{F}_2(s^T,\theta;y,\epsilon)
}.
\end{align}
Then $\hat{F}$ and its gradients agree with $F$ and its gradients at $(s_0^T,\theta_0;y_0,0)$, so that (after shrinking $U$, $V$, $W$ and increasing $n_0$ if necessary) the Implicit Function Theorem again produces a function $\hat{G}(y,\epsilon)$ giving the unique solution $\left( \begin{smallmatrix}s^T\\ \theta\end{smallmatrix} \right)$ in $U\times W$ of $\hat{F}(s^T,\theta;y,\epsilon)=0$.
Moreover the analogue of \reflemma{F=0ISMLE} applies, with the same proof, so that
\begin{equation}
\hat{G}(y,1/n) = \mat{\hat{s}_0^T\bigl( \hat{\theta}_{\MLE\,\mathrm{in}\,U}(x,n), y \bigr) \\ \hat{\theta}_{\MLE\,\mathrm{in}\,U}(x,n)}.
\end{equation}
For later convenience write $\hat{G}$ in block form as $\hat{G}=\left( \begin{smallmatrix}\hat{G}^T_s\\ \hat{G}_\theta\end{smallmatrix} \right)$.

To compare $G(y,\epsilon)$ with $\hat{G}(y,\epsilon)$, note that $F_2$ and $\hat{F}_2$ are close:
\begin{equation}
F_2(s^T,\theta;y,\epsilon) = \hat{F}_2(s^T,\theta;y,\epsilon) + \epsilon q_3(\theta,y,\epsilon).
\end{equation}
In particular, $F(\hat{G}(y,\epsilon);y,\epsilon)$ is almost zero:
\begin{equation}
F(\hat{G}(y,\epsilon);y,\epsilon)=\mat{F_1(\hat{G}(y,\epsilon);y,\epsilon) \\ \hat{F}_2(\hat{G}(y,\epsilon);y,\epsilon)} + \mat{0\\ \epsilon q_3(\hat{G}_\theta(y,\epsilon), y,\epsilon)}.
\end{equation}
The first term in the right-hand side vanishes by definition.
In the second term, note that $q_3(\hat{G}_\theta(y,\epsilon),y,\epsilon)$ is a continuously differentiable function of $(y,\epsilon)$ that vanishes whenever $\epsilon=0$ (since $q_3$ has the same property by \refcoro{gradlogPhatsError}).
We can therefore conclude that
\begin{equation}\label{FofhatG}
F\bigl( \hat{G}(y,\epsilon);y,\epsilon \bigr) = \epsilon^2 q_4(y,\epsilon)
\end{equation}
where $\abs{q_4(y,\epsilon)}\leq C$ for $(y,\epsilon)$ in a suitable neighbourhood of $(y_0,0)$.

To make use of \eqref{FofhatG}, we define an augmented version of $F$ that is locally invertible. 
Let 
\begin{equation}
\widetilde{F}(s^T,\theta;y,\epsilon) = \mat{F(s^T,\theta;y,\epsilon)\\y\\ \epsilon}
,
\quad\text{so that}\quad
\grad_{s^T,\theta,y,\epsilon}\widetilde{F} = 
\mat{\grad_{s^T,\theta} F & \grad_y F & \grad_\epsilon F \\
0 & I_{m\times m} & 0\\
0 & 0 & 1
}
\end{equation}
in block form.
Thus $\grad_{s^T,\theta,y,\epsilon}\widetilde{F}(s_0^T,\theta_0;y_0,0)$ is an invertible $(2\xdim+p+1)\times(2\xdim+p+1)$ matrix, and by the Inverse Function Theorem \cite[Theorem~9.24]{RudinPoMA}, after shrinking the domain of $\widetilde{F}$ if necessary, $\widetilde{F}$ has a continuously differentiable inverse function $\widetilde{G}(u;y,\epsilon)$.
As above, we may shrink the domain further to make the partial derivatives of $\widetilde{G}$ uniformly bounded.

The inverse function $\widetilde{G}$ is related to the implicit function $G$ by $\widetilde{G}(0;y,\epsilon) = \left( \begin{smallmatrix}G(y,\epsilon)\\y\\ \epsilon\end{smallmatrix} \right)$.
From \eqref{FofhatG} and the definition of $\widetilde{F}$ we have
\begin{equation}
\widetilde{F}\bigl( \hat{G}(y,\epsilon); y,\epsilon \bigr) = \mat{
\epsilon^2 q_4(y,\epsilon)\\
y\\
\epsilon
}
\quad\text{and so}\quad
\widetilde{G}(\epsilon^2 q_4(y,\epsilon);y,\epsilon) = \mat{\hat{G}(y,\epsilon) \\ y \\ \epsilon}.
\end{equation}
Thus, setting $\epsilon=1/n$ for $n$ sufficiently large,
\begin{align}
\bigabs{\hat{\theta}_{\MLE\,\mathrm{in}\,U}(x,n) - \theta_{\MLE\,\mathrm{in}\,U}(x,n)} 
&\leq
\bigabs{\hat{G}(y,1/n) - G(y,1/n)}
\notag\\&
=\bigabs{ \widetilde{G}(n^{-2} q_4(y,1/n);y,1/n) - \widetilde{G}(0;y,1/n) }.
\end{align}
The boundedness of $q_4$ and of the partial derivatives of $\widetilde{G}$ imply that this upper bound is $O(1/n^2)$, uniformly over $y$ in a suitable neighbourhood of $y_0$.
\end{proof}

The corresponding assertion from \refthm{LowerOrderSaddlepoint} has a similar proof:

\begin{proof}[Proof of \refthm{LowerOrderSaddlepoint}\ref{item:LOMLE}]
Note that $\theta=\hat{\theta}^*_{\MLE\,\mathrm{in}\,U}(x,n)$ and $s=\hat{s}(\theta,x)$ are the solutions of $F(s^T,\theta;y,0)=0$.
(The proof is the same as for \reflemma{F=0ISMLE}, see \refappendix{FandMLEproof}.)
Then
\begin{equation}
\bigabs{\hat{\theta}^*_{\MLE\,\mathrm{in}\,U}(x,n) - \theta_{\MLE\,\mathrm{in}\,U}(x,n)} 
\leq
\abs{G(y,0) - G(y,1/n)}
\end{equation}
and the conclusion follows from the fact that $G$ is continuously differentiable.
\end{proof}

\section{Conclusion}\lbsect{Conclusion}

A long-established result tells us that the saddlepoint approximation gives a relative error of order $1/n$.
That is, applied to a random variable given as a sum of $n$ i.i.d.\ terms, the saddlepoint approximation estimates the values of the density (or likelihood) up to a factor of the form $1+O(1/n)$: see \eqref{LErrorRough} or \refprop{PddtPError}.
Very commonly, however, we are not interested in the likelihood for its own sake but rather as a step towards computing the MLE.
This paper gives the analogous basic results, Theorems~\ref{T:MLEerror} and \ref{T:MLEerrorPartiallyIdentifiable}, for the approximation error between the true MLE and saddlepoint MLE: under certain explicit identifiability conditions, it is $O(1/n)$, $O(1/n^{3/2})$, or $O(1/n^2)$.

It is worth noting that this MLE error estimate is sharper than what we obtain from the basic likelihood error estimate.
In the fully identifiable case, knowing that $\log\hat{L}(\theta;x)$ has a maximum at $\theta=\hat{\theta}_\MLE(x)$ and knowing that the true log-likelihood satisfies $\bigabs{\log L(\theta;x)-\log\hat{L}(\theta;x)}=O(1/n)$, we conclude only that $\bigabs{\theta_\MLE(x)-\hat{\theta}_\MLE(x)}=O(1/n)$: see \refsubsubsect{ImpliesHeuristic}.
Although an MLE error bound of size $O(1/n)$ is small compared to the scale of the inferential uncertainty in estimating $\theta$, see Theorems~\ref{T:BayesianError}--\ref{T:SamplingMLEWellSpecified} and the remarks at the end of \refsubsubsect{FullyIdentifiableResults}, it is still a significant overestimate compared to the true MLE error $O(1/n^2)$.
The results in this paper help to explain why saddlepoint MLEs in practice often turn out to be so much more accurate than expected.

A key point in the analysis is to ask how well the saddlepoint approximation captures the \emph{shape} of the log-likelihood as a function of the parameter $\theta$. 
Specifically, it is error bounds on the \emph{gradient} of the log-likelihood, as in \refthm{GradientError} and \refcoro{gradlogPhatsError}, that control the size of the MLE approximation error and lead to sharp results (see \cite{Ogden2017,OgdenErrorLaplaceHighDimension} for a similar finding in another context).
The same logic underpins the finding of \refsubsubsect{ExpFamily} that saddlepoint MLEs are exact for exponential families: errors in the saddlepoint approximation to the log-likelihood are irrelevant if the sizes of the errors do not depend on the parameter.

Seen in this light, it is less surprising that a lower-order saddlepoint approximation, $\hat{L}^*(\theta;x)$ from \refsubsect{LowerOrder}, can give a good approximation to the MLE (see \refthm{LowerOrderSaddlepoint}) despite being a poor approximation to the likelihood with $\hat{L}^*(\theta;x)/L(\theta;x)\to\infty$ as $n\to\infty$.
We remark that since $\hat{L}^*(\theta;x)$ is even less computationally demanding than the usual saddlepoint approximation, it may be a useful tool for high-dimensional and computationally intensive applications, or for initialising the search for a true or saddlepoint MLE.

In the other direction, refinements of the saddlepoint approximation that improve likelihood accuracy may be less effective than anticipated when applied to MLEs.
For instance, normalising the saddlepoint approximation $\hat{f}(x;\theta)$ to make it a density (as a function of $x$) often brings the saddlepoint density values closer to the true density (see for instance \cite{Butler2007}).
However, this operation is slow and is often not pursued: cf.\ \cite{Daniels1982}.  
Using the viewpoint developed in this paper, we can reframe the issue by asking: does normalising bring the saddlepoint log-gradient closer to the true log-gradient?
The general answer is far from clear.
Indeed, for an exponential family such as the Poisson family, \refexample{Poisson} in \refappendix{Examples}, for which the saddlepoint MLE is already exact, normalising will actually make the MLE worse.

This paper has considered likelihoods and MLEs in a particularly tractable limiting framework, the standard asymptotic regime \eqref{SAR} in which the observation is a sum of $n$ i.i.d.\ terms.
In particular, we have seen that the basic likelihood accuracy estimate does not directly lead to the correct MLE error estimate, which is markedly better.
It would be of interest to extend these results in several directions, including non-Gaussian saddlepoint approximations, as in \cite{WooBooBut1993}; non-i.i.d.\ sums; integer-valued random variables; uniform upper bounds; large-sample-size limits; and applications based on approximating tail probabilities rather than likelihoods.

\section*{Acknowledgements}

The author thanks Rachel Fewster and Joey Wei Zhang, whose use of saddlepoint MLEs sparked the author's interest in this question; Rachel Fewster for helpful comments on drafts of this paper; and Godrick Oketch for suggesting the model of \refexample{GammaFI} in \refappendix{Examples} as a useful one for explicit comparisons.
Useful comments and questions from the anonymous referees helped to improve this paper.

This work was supported in part by funding from the Royal Society of New Zealand.

\newpage

\centerline{\Large\bf Appendices}

\begin{appendix}

\begin{itemize}
\item
\refappendix{Invariance} states the behaviour of the saddlepoint approximation under affine transformations of the random variables.
\item
\refappendix{DerivativesDerivation} explains in greater detail the conventions on row and column vectors, gradients, and Hessians, and collects differentiation formulas for quantities related to the saddlepoint approximation.
\item
\refappendix{MainPropProof} proves \refprop{PddtPError} and \refcoro{gradlogPhatsError}, as well as other technical results used in those proofs.
\item
\refappendix{IntegerValuedProof} explains the adaptations to the proofs from \refappendix{MainPropProof} needed to prove \refprop{PddtPErrorInt} and \refthm{IntegerValued}.
\item
\refappendix{FandMLEproof} proves \reflemma{F=0ISMLE} and verifies other assertions from the proof of \refthm{MLEerror}.
\item
\refappendix{BayesianErrorProof} states and proves \refprop{BayesianErrorGeneral}, which generalises the result of \refthm{BayesianError} and \refthm{LowerOrderSaddlepoint}\ref{item:LOBayesian}.
\item
\refappendix{SamplingMLEProof} proves Theorems~\ref{T:SamplingMLE}--\ref{T:SamplingMLEWellSpecified}.
\item
\refappendix{MLEePIProof} proves Theorems~\ref{T:MLEerrorPartiallyIdentifiable}--\ref{T:MLEerrorNormalApprox}, including a discussion of the scaling assumptions in those results.
\item
\refappendix{SaddlepointVsCLT} compares the saddlepoint approximation to the normal approximation when both are considered as approximations to the density, as functions of $x$.
\item
\refappendix{Examples} gives detailed discussions of a variety of examples, including two general classes of models; specific models studied in the literature; and theoretical examples that illustrate facets of the saddlepoint approximation.
\end{itemize}

\section{Invariance properties of the saddlepoint approximation}\lbappendix{Invariance}

If $b\in\R^{\xdim\times 1}$ and $Y=X+b$, the corresponding saddlepoint approximations are related by
\begin{equation}\label{hatLY=X+b}
\hat{L}_Y(\theta;y) = \hat{L}_X(\theta;y-b),
\end{equation}
with the same saddlepoint $\hat{s}$ in both approximations.

If $A\in\R^{k\times\xdim}$ and $Z_\theta=AX_\theta$, then 
\begin{equation}
K_Z(s;\theta) = K_X(sA;\theta),
\end{equation}
and it follows that 
\begin{equation}
K_Z''(s;\theta) = A K_X''(s;\theta) A^T.
\end{equation}
In particular, if $A$ is $\xdim\times\xdim$ and invertible, then $\det(K_Z''(s)) = \det(A)^2 \det(K_X''(s))$ and 
\begin{equation}\label{hatLZ=AX}
\hat{L}_Z(\theta;z) = \frac{\hat{L}_X(\theta;A^{-1}z)}{\abs{\det(A)}},
\end{equation}
with the saddlepoints related by 
\begin{equation}
\hat{s}_Z = \hat{s}_X A^{-1}.
\end{equation}

Equations~\eqref{hatLY=X+b} and \eqref{hatLZ=AX} show that, under invertible affine transformations of the random variables, the saddlepoint approximation transforms in the same way that densities do.
In particular, such transformations will not affect the MLE or saddlepoint MLE, and in the standard asymptotic regime we might equivalently consider $\overline{Y}=\frac{1}{n}X$, the sample mean of the $n$ unobserved values $Y^{(1)},\dotsc,Y^{(n)}$, in place of $X$.
We have chosen to consider $X$ rather than $\overline{Y}$, in part because the saddlepoint for $\overline{Y}$ scales with $n$ for a fixed value of $y$, whereas the saddlepoint for $X$ does not.

\section{Derivatives of saddlepoint quantities}\lbappendix{DerivativesDerivation}

\subsection{Conventions for row and column vectors, gradients, and Hessians}

Given a row vector argument $s=\mat{s_1 & s_2 & \dotsb & s_\xdim}\in\R^{1\times\xdim}$, we interpret the gradient with respect to $s$ as a $\xdim\times 1$ matrix, so that, when applied to a function $f=f_{\mathrm{row}}\colon \R^{1\times\xdim}\to \R^{1\times k}$ with row-vector values, the result is the $\xdim\times k$ matrix whose $i,j$ entry is $\frac{\partial f_j}{\partial s_i}$.
Thus
\begin{equation}
\grad_s = \mat{\frac{\partial}{\partial s_1} \\ \vdots \\ \frac{\partial}{\partial s_\xdim}}
,\qquad
\grad_s f_{\mathrm{row}}(s) = \mat{\frac{\partial f_1}{\partial s_1} & \dotsb & \frac{\partial f_k}{\partial s_1}\\ 
\vdots & & \vdots \\ 
\frac{\partial f_1}{\partial s_\xdim} & \dotsb & \frac{\partial f_k}{\partial s_\xdim}
}
.
\end{equation}
If instead $\theta\in\R^{p\times 1}$ is a column vector argument and $f=f_{\mathrm{col}}\colon\R^{p\times 1}\to\R^{k\times 1}$, then $\grad_\theta$ has the form of a $1\times p$ matrix and $\grad_\theta f_{\mathrm{col}}(\theta)$ is the $k\times p$ matrix whose $i,j$ entry is $\frac{\partial f_i}{\partial x_j}$.
For either row or column vectors we have
\begin{equation}
\grad_{x^T} = (\grad_x)^T, \qquad \grad_x f = \left( \grad_{x^T} (f^T)  \right)^T.
\end{equation}

We remark that these conventions match with the interpretation of gradients as linear operators, which act via matrix multiplication either on the right or on the left.
Thus the linear approximations to a function $f$ can be written as 
\begin{equation}\label{LinearApproximationForm}
s\mapsto f(s_0) + (s-s_0) \grad_s f(s_0)  \qquad\text{or}\qquad \theta\mapsto f(\theta_0) + \grad_\theta f(\theta_0) (\theta-\theta_0)
\end{equation}
for row vectors or column vectors, respectively.
(Strictly speaking, $\grad_\theta$ should operate on $f_{\mathrm{col}}$ from the right if we wish to preserve the usual matrix-shape conventions.)

For scalar-valued functions, $k=1$, the Hessian matrix of second partial derivatives can be written as
\begin{equation}
\grad_s \grad^T_s f(s) \qquad\text{or}\qquad \grad_\theta^T\grad_\theta f_{\mathrm{col}}(x),
\end{equation}
the $\xdim\times\xdim$ or $p\times p$ symmetric matrices with $i,j$ entry $\frac{\partial^2 f}{\partial s_i\partial s_j}$ or $\frac{\partial^2 f}{\partial\theta_i\partial\theta_j}$.
We do not apply these conventions to higher-order derivatives, for which matrix notation ceases to be appropriate.
We write $K'=\grad_s K$, $K''=\grad_s\grad_s^T K$, and similarly for other CGFs and MGFs such as $M_0$ and $K_0$.

The functions $\hat{s}(\theta,x)$ and $\hat{s}_0(\theta,y)$ take column vector arguments and return row vector values.
To match with the above conventions, we first convert the values to column vectors and write the gradients as $\grad_\theta \hat{s}^T(\theta,x)$, $\grad_x \hat{s}^T(\theta,x)$ and so on, which are $\xdim\times p$ and $\xdim\times\xdim$ matrices, respectively.

In gradient and matrix form, the chain rule comes in two slightly different forms for row vector and column vector arguments.
If $f_{\mathrm{row}}(s)\colon \R^{1\times\xdim}\to\R^{1\times k}$ and $g_{\mathrm{row}}(v)\colon \R^{1\times j}\to\R^{1\times\xdim}$ are differentiable then $h_{\mathrm{row}}(v)=f_{\mathrm{row}}(g_{\mathrm{row}}(v))$ has
\begin{equation}\label{ChainRuleRow}
\grad_v h_{\mathrm{row}}(v) = \grad_v g_{\mathrm{row}}(v) \grad_s f_{\mathrm{row}}(g_{\mathrm{row}}(v)).
\end{equation}
Similarly, if $f_{\mathrm{col}}(\theta)\colon \R^{p \times 1}\to\R^{k\times 1}$ and $g_{\mathrm{col}}(w)\colon \R^{j\times 1}\to\R^{p\times 1}$ are differentiable then $h_{\mathrm{col}}(w)=f_{\mathrm{col}}(g_{\mathrm{col}}(w))$ has
\begin{equation}\label{ChainRuleCol}
\grad_w h_{\mathrm{col}}(w) = \grad_\theta f_{\mathrm{col}}(g_{\mathrm{col}}(w)) \grad_w g_{\mathrm{col}}(w).
\end{equation}
(Comparing with the linear approximations in \eqref{LinearApproximationForm}, we see that for either row or column vectors, $\grad g$ acts on $s-s_0$ or $\theta-\theta_0$ before $\grad f$ does.)
In expressions such as $K(\hat{s}(\theta,x);\theta)$ containing both row and column vectors, we interpret $K$ as a function of two column vectors $s^T$ and $\theta$, so that applying \eqref{ChainRuleCol} gives
\begin{equation}\label{gradthetaKhats}
\begin{aligned}
\grad_\theta\left( \Big. K(\hat{s}(\theta,x);\theta) \right)
&=\grad_\theta K(\hat{s}(\theta,x);\theta) + \grad_{s^T} K(\hat{s}(\theta,x);\theta) \grad_\theta\hat{s}^T(\theta,x)
\\
&=\grad_\theta K(\hat{s}(\theta,x);\theta) + \grad_s^T K(\hat{s}(\theta,x);\theta) \grad_\theta\hat{s}^T(\theta,x)
\end{aligned}
\end{equation}
and similarly for other functions of $s$ and $\theta$.
Note our convention for gradients and composite functions: in \eqref{gradthetaKhats}, an expression such as $\grad_\theta K(\hat{s}(\theta,x);\theta)$ means the $1\times p$ row vector of partial derivatives $\frac{\partial K}{\partial\theta_j}$ evaluated with $s$ replaced by $\hat{s}(\theta,x)$, whereas $\grad_\theta\left( \big. K(\hat{s}(\theta,x);\theta) \right)$ means the gradient of the composite function $\theta\mapsto \grad_s K(\hat{s}(\theta,x);\theta)$.
We make an exception for gradients of logarithms, as in \eqref{GradientErrorEquation}, where $\grad\log\hat{L}$, $\grad\log L$ mean the gradients of the composite functions even without parentheses.

\subsection{Derivative formulas}

We now proceed to differentiate the quantities $\log L^*$, $\hat{s}$, $\log \hat{L}^*$ and $\hat{P}$.
The first of these is simple since $L^*$ is exact and explicitly computable: 
\begin{equation}\label{logL*Gradients}
\begin{aligned}
\grad_\theta\log L^*(s,\theta) &= \grad_\theta K(s;\theta) - s \grad_s\grad_\theta K(s;\theta),\\
\grad_s \log L^*(s,\theta) &= K'(s;\theta) - K'(s;\theta) - K''(s;\theta)s^T = -K''(s;\theta)s^T.
\end{aligned}
\end{equation}

For $\hat{s}(\theta,x)$, suppose $(s_0,\theta_0)\in\interior\mathcal{S}$ with $K'(s_0;\theta_0)=x_0$.
The function $(s,\theta,x)\mapsto K'(s;\theta)-x$ is continuously differentiable and its gradient with respect to $s^T$, i.e., $K''(s;\theta)$, is positive definite and hence non-singular.
Hence the Implicit Function Theorem applied to \eqref{SaddlepointEquation} implies that $\hat{s}(\theta,x)$ is uniquely defined and continuously differentiable for $\theta,x$ in a neighbourhood of $\theta_0,x_0$.
(Note that continuity may fail without the assumption $(s_0,\theta_0)\in\interior\mathcal{S}$, and note in this connection that $\interior\mathcal{X}$ may be strictly larger than $\mathcal{X}^o$.)
The strict convexity of $K$ for fixed $\theta$ implies that $\hat{s}(\theta,x)$ is globally unique (not merely locally unique, as follows from the Implicit Function Thoerem) whenever it is defined.

We now find the gradients of $\hat{s}(\theta,x)$.
Applying the analogue of \eqref{gradthetaKhats} with $K$ replaced by $\grad_s K$,
\begin{equation}\label{SEgradtheta}
\begin{aligned}
0 = \grad_\theta x &= \grad_\theta\left( \Big. \grad_s K(\hat{s}(\theta,x);\theta) \right)
\\&
=\grad_s\grad_\theta K(\hat{s}(\theta,x);\theta) + \grad_s\grad_s^T K(\hat{s}(\theta,x);\theta) \grad_\theta\hat{s}^T(\theta,x)
.
\end{aligned}
\end{equation}
In \eqref{SEgradtheta}, $\grad_s\grad_s^T K = K''$ is non-singular as in \eqref{K''PosDef}, so we can solve to find
\begin{equation}\label{gradthetashat}
\grad_\theta\hat{s}^T(\theta,x) = - K''(\hat{s}(\theta,x);\theta)^{-1} \grad_s\grad_\theta K(\hat{s}(\theta,x);\theta)
\end{equation}
Similarly, with $I_{\xdim\times\xdim}$ representing the $\xdim\times\xdim$ identity matrix,
\begin{equation}
I_{\xdim\times\xdim} = \grad_x x = \grad_x\left( \Big. \grad_s K(\hat{s}(\theta,x);\theta) \right) = \grad_s\grad_s^T K(\hat{s}(\theta,x);\theta) \grad_x \hat{s}^T(\theta,x)
\end{equation}
so that
\begin{equation}
\grad_x\hat{s}^T(\theta,x) = K''(\hat{s}(\theta,x);\theta)^{-1}.
\end{equation}

Turning to $\log\hat{L}^*(\theta;x) = \log L^*(\hat{s}(\theta,x),\theta) = K(\hat{s}(\theta,x);\theta)-\hat{s}(\theta,x)K'(\hat{s}(\theta,x);\theta)$, recall that $K'(\hat{s}(\theta,x);\theta)=x$ by definition.
Noting that $\hat{s} x = x^T\hat{s}^T$, we find using \eqref{gradthetaKhats} that
\begin{align}
\grad_\theta \log\hat{L}^*(\theta;x) 
&= \grad_\theta\left( \big. K(\hat{s}(\theta,x);\theta) - x^T\hat{s}^T(\theta,x) \right)
\notag\\&
= \grad_\theta K(\hat{s}(\theta,x);\theta) + \grad_s^T K(\hat{s}(\theta,x);\theta) \grad_\theta\hat{s}^T(\theta,x) - x^T \grad_\theta\hat{s}^T(\theta,x)
\notag\\&
= \grad_\theta K(\hat{s}(\theta,x);\theta)
\label{loghatL*Gradient}
\end{align}
since $\grad_s^T K(\hat{s}) = (\grad_s K(\hat{s}))^T = x^T$.
(Alternatively, use the relations \eqref{logL*Gradients}.)
Transposing then applying \eqref{gradthetaKhats} again,
\begin{align}
\grad_\theta^T\grad_\theta \log\hat{L}^*(\theta;x) 
&= \grad_\theta \left( \big. \grad_\theta K(\hat{s}(\theta,x);\theta) \right)^T
\notag\\&
= \grad_\theta^T\grad_\theta K(\hat{s}(\theta,x)) + \grad_\theta^T\grad_s^T K(\hat{s}(\theta,x);\theta) \grad_\theta\hat{s}^T(\theta,x)
\notag\\&
= \grad_\theta^T\grad_\theta K(\hat{s}) - \left( \grad_s\grad_\theta K(\hat{s}) \right)^T K''(\hat{s})^{-1} \grad_s\grad_\theta K(\hat{s}).
\label{loghatL*Hessian}
\end{align}
Equations~\eqref{loghatL*Gradient}--\eqref{loghatL*Hessian} verify the assertion before \refthm{MLEerror} that the quantities in \eqref{RateFunctionImplicitCriticalPoint}--\eqref{RateFunctionImplicitHessianDefinite} are the gradient and Hessian of the function in \eqref{RateFunctionImplicit}.

For $\log\hat{P}$, we first show how to differentiate a determinant.
Suppose that $A(t)\colon \R\to\R^{\xdim\times\xdim}$ is a differentiable matrix-valued function of a scalar parameter $t$.
Then
\begin{equation}\label{logdetDerivative}
\frac{d}{dt} \log\det A(t) = \trace(A(t)^{-1} A'(t));
\end{equation}
see for instance \cite[Theorem~8.2]{MagnusNeudeckerMatrixDifferentialCalc}.
Applying \eqref{logdetDerivative} to $\log\hat{P}=-\tfrac{1}{2}\log\det(2\pi K'')$,
\begin{equation}\label{loghatPDerivative}
\frac{\partial}{\partial t}\log\hat{P}(s,\theta) = -\frac{1}{2}\trace\left( K''(s;\theta)^{-1} \frac{\partial K''}{\partial t}(s;\theta) \right).
\end{equation}
We apply \eqref{loghatPDerivative} with the scalar $t$ being one of the coordinates $\theta_i$ or $s_j$.

\section{Proof of \refprop{PddtPError} and \refcoro{gradlogPhatsError}}\lbappendix{MainPropProof}

The proof of \refprop{PddtPError} is based on an asymptotic expansion of $P_n(s,\theta)$ and its derivatives when $n$ is large.
We will need to justify the interchange of integration and differentiation in \eqref{PformulaK}--\eqref{dPdtFormulaK}, and we state a sufficient condition in \reflemma{DiffUnderInt}.

In addition, passing from the discrete parameter $n$ to the continuous parameter $\epsilon$ will require us to define an analogue of \eqref{PformulaK} for non-integer $n$.
Some care is required because, away from the real axis, $K_0(s+\ii\phi;\theta)=\log M_0(s+\ii\phi;\theta)$ may have branch cuts at zeros of $M_0$.
These make no difference in \eqref{PformulaK} provided that $n$ is an integer, because the combination $\exp(nK_0(s+\ii\phi;\theta))$ can be unambiguously interpreted as $M_0(s+\ii\phi;\theta)^n$.
We will circumvent this by modifying $K_0$ away from the real axis.
This modification introduces error terms indexed by $n\in\N$, which will prove to be negligible as $n\to\infty$.
We will use Lemmas~\ref{L:SequenceToC1}--\ref{L:FunctionsToC1} to show that these error terms can be extended to be indexed by $\epsilon\geq 0$ in a continuously differentiable way.

\begin{lemma}\lblemma{DiffUnderInt}
Let $U\subset\R^{k\times 1}$ be open, let $\mu$ be a measure on a space $\mathcal{Y}$, let $h(y)\geq 0$ be a measurable function with $\int h(y)\, d\mu(y)<\infty$, and let $g(x,y):U\times \mathcal{Y}\to\R$ be such that
\begin{equation}
f(x) = \int_{\mathcal{Y}} g(x,y) \, d\mu(y)
\end{equation}
converges for all $x\in U$.
\begin{enumerate}
\item\label{item:ContinuityIntegral}
If $\abs{g(x,y)}\leq h(y)$ for all $x\in U$ and if $g(\cdot,y)$ is continuous at $x_0$ for $\mu$-almost every $y$, then $f$ is continuous at $x_0$.
\item\label{item:DerivativeIntegral}
Fix an entry $x_i$ of $x$ and suppose that $g(x,y)$ is continuously differentiable as a function of $x_i$, for all fixed values of $y$ and $x_j, j\neq i$.
Suppose in addition that $\frac{\partial g}{\partial x_i}(x,y)$ is jointly measurable and $\bigabs{\tfrac{\partial g}{\partial x_i}(x,y)} \leq h(y)$ for all $x\in U$.
Then
\begin{equation}\label{DerivOfIntegral}
\frac{\partial f}{\partial x_i}(x) = \int_{\mathcal{Y}} \frac{\partial g}{\partial x_i}(x,y) \, d\mu(y)
\end{equation}
for all $x\in U$.
\item\label{item:ContinuityGradientIntegral}
If in addition $\frac{\partial f}{\partial x_i}(\cdot,y)$ is continuous at $x_0$ for $\mu$-almost every $y$, then $\frac{\partial f}{\partial x_i}$ is continuous at $x_0$.
\end{enumerate}\end{lemma}

\begin{proof}
For part~\ref{item:ContinuityIntegral}, the integrand $g(x,y)$ converges $\mu$-a.e.\ to $g(x_0,y)$ as $x\to x_0$ and is bounded by the integrable function $h$.
So the Dominated Convergence Theorem gives the result.

For part~\ref{item:DerivativeIntegral}, let $e_i$ denote the unit vector in the $i^\text{th}$ coordinate direction.
Then given $x\in U$, we can write
\begin{equation}\label{DiffAsIntOfDeriv}
\frac{f(x+te_i)-f(x)}{t} = \int_{\mathcal{Y}} \int_0^1 \frac{\partial g}{\partial x_i}(x+tue_i, y) \, du \, d\mu(y)
\end{equation}
for all $t\neq 0$ sufficiently small.
The integrand in \eqref{DiffAsIntOfDeriv} is bounded by the integrable function $h(y)\indicatorofset{[0,1]}(u)$ and converges pointwise as $t\to 0$, so \eqref{DerivOfIntegral} follows by the Dominated Convergence Theorem.

Finally part~\ref{item:ContinuityGradientIntegral} follows by applying part~\ref{item:ContinuityIntegral} to \eqref{DerivOfIntegral}.
\end{proof}

The next result explains when we can interpolate a given sequence by a continuously differentiable function.
Here and elsewhere, continuous differentiability on a set with boundary points includes differentiability at the boundary (interpreting derivatives as one-sided derivatives as appropriate) and the derivatives are required to be continuous up to and including the boundary.

\begin{lemma}\lblemma{SequenceToC1}
Let $\epsilon_n$ be a sequence with $0<\epsilon_{n+1}<\epsilon_n\in[0,1]$ for all $n\in\N$ and $\epsilon_n\to 0$ as $n\to\infty$, and let $r_n$ be a real-valued sequence defined for $n\in\N, n\geq n_0$.
Then
\begin{equation}\label{SequenceToC1Condition}
\lim_{n\to\infty} r_n=0, \qquad \lim_{n\to\infty} \frac{r_n-r_{n+1}}{\epsilon_n-\epsilon_{n+1}} = 0,
\end{equation}
is a necessary and sufficient condition for the existence of a continuously differentiable function $f:[0,1]\to\R$ satisfying $f(0)=0$, $f'(0)=0$ and $f(\epsilon_n) = a_n$ for all $n\geq n_0$.

If in addition 
\begin{equation}\label{SequenceToCkCondition}
\lim_{n\to\infty} \frac{r_n-r_{n+1}}{(\epsilon_n-\epsilon_{n+1})^k} = 0
\end{equation}
where $k\in\N$, $k\geq 2$, then $f$ can be taken to be $k$ times continuously differentiable.
\end{lemma}

Note that \eqref{SequenceToCkCondition} is not sharp, unlike \eqref{SequenceToC1Condition}, but will be sufficient for our purposes: we will have $\epsilon_n=1/n$, so that $\epsilon_n-\epsilon_{n+1}$ is of order $1/n^2$, while $r_n$ will decay exponentially in $n$.

\begin{proof}
For necessity, $r_n\to 0$ follows by taking $f(\epsilon)\to f(0)=0$ along $\epsilon=\epsilon_n$.
We can write $\frac{r_n-r_{n+1}}{\epsilon_n-\epsilon_{n+1}} = \int_0^1 f'(u\epsilon_n + (1-u)\epsilon_{n+1})\,du$, and the integrand tends uniformly to $f'(0)=0$ as $n\to\infty$.

For sufficiency, fix a smooth function $\eta:\R\to\R$ whose support is a compact subset of $(0,1)$ and such that $\int_0^1 \eta(u)du = 1$.
Define
\begin{equation}\label{fepsilonFormula}
h(\epsilon)=\sum_{n=n_0}^\infty \frac{r_n-r_{n+1}}{\epsilon_n-\epsilon_{n+1}} \eta\left( \frac{\epsilon-\epsilon_{n+1}}{\epsilon_n-\epsilon_{n+1}} \right), 
\qquad
f(\epsilon) = \int_0^\epsilon h(y) \, dy.
\end{equation}
We claim that $f$ has the required properties.
Note that the $n^\text{th}$ term contributing to $h(\epsilon)$ is supported in a compact subset of $(\epsilon_{n+1},\epsilon_n)$.
Hence around each fixed $\epsilon\in\ocinterval{0,1}$ there is a neighbourhood in which at most one summand contributes to $h$, and it follows that $h$ is continuous on $\ocinterval{0,1}$ and can moreover be differentiated term-by-term.
Moreover $h(0)=0$, and $\frac{r_n-r_{n+1}}{\epsilon_n-\epsilon_{n+1}}\to 0$ as $n\to\infty$ implies that $h(\epsilon)\to 0$ as $\epsilon\to 0$.
(Here we have used the disjointness of the intervals $(\epsilon_{n+1},\epsilon_n)$ and the boundedness of $\eta$.)
Thus $h$ is continuous and bounded on $[0,1]$, implying that $f$ is continuously differentiable.

We must next verify that $f(\epsilon_n)=a_n$ or equivalently that $\int_0^{\epsilon_n} h(\epsilon) \, d\epsilon = r_n$, for all $n\geq n_0$.
By the choice of $\eta$, $\int_{\epsilon_{n+1}}^{\epsilon_n} h(\epsilon) \, d\epsilon = r_n-r_{n+1}$, and a finite sum implies $\int_{\epsilon_{N+1}}^{\epsilon_n} h(\epsilon) \, d\epsilon = r_n-r_{N+1}$.
Take $N\to\infty$.
The Dominated Convergence Theorem applies, so the integral converges to $\int_0^{\epsilon_n} h(\epsilon) \, d\epsilon$, while $r_{N+1}\to 0$ by assumption.

Finally we can differentiate term-by-term to find
\begin{equation}
f^{(j)}(\epsilon) = h^{(j-1)}(\epsilon) = \sum_{n=n_0}^\infty \frac{r_n-r_{n+1}}{(\epsilon_n-\epsilon_{n+1})^j}\eta^{(j-1)}\left( \frac{\epsilon-\epsilon_{n+1}}{\epsilon_n-\epsilon_{n+1}} \right)
\end{equation}
for $\epsilon>0$.
This function is continuous on $\ocinterval{0,1}$ by the same argument as for $h$.
If \eqref{SequenceToCkCondition} holds and $j\leq k$, it follows that $\lim_{\epsilon\decreasesto 0}f^{(j)}(\epsilon)=0$, and by repeated application of the Mean Value Theorem it follows that $f^{(j)}=0$ and so $f^{(j)}$ is continuous at $\epsilon=0$.
\end{proof}

\begin{lemma}\lblemma{FunctionsToC1}
Let $\epsilon_n$ be a sequence with $0<\epsilon_{n+1}<\epsilon_n\in[0,1]$ for all $n\in\N$ and $\epsilon_n\to 0$ as $n\to\infty$.
Let $S\subset\R^p$ and let $r_n\colon S_n\to\R$, $n\in\N$, be $k$ times continuously differentiable functions, $k\geq 1$, defined on subsets $S_n\subset S$.
Suppose that for each $x_0\in\interior S$, there exist a neighbourhood $G\subset S$ of $x_0$ and $n_0\in\N$ such that $G\subset S_n$ for all $n\geq n_0$ and
\begin{equation}\label{rnxConditions}
\begin{gathered}
\lim_{n\to\infty}\sup_{x\in G} r_n(x)=0, \qquad \lim_{n\to\infty}\sup_{x\in G} \frac{r_n(x)-r_{n+1}(x)}{(\epsilon_n-\epsilon_{n+1})^k} = 0,
\\
\lim_{n\to\infty} \sup_{x\in G} \frac{\frac{\partial^j}{\partial x_{i_1}\dotsb\partial x_{i_j}} (r_n(x)-r_{n+1}(x))}{(\epsilon_n-\epsilon_{n+1})^{k-j}} = 0, \quad i_j\in\set{1,\dotsc,p}, j=1,\dotsc,k.
\end{gathered}
\end{equation}
Then there exist a subset $\mathcal{Q}\subset S\times[0,1]$, open relative to $S\times[0,1]$ and containing $\set{(x,0)\colon x\in\interior S}$, and a $k$ times continuously differentiable function $f(x,\epsilon)$ defined on $\mathcal{Q}$, such that $f(x,0)=0$, $\frac{\partial^j f}{\partial\epsilon^j}(x,0)=0$ for $j=1,\dotsc,k$, and $f(x,\epsilon_n) = r_n(x)$ whenever $(x,\epsilon_n)\in\mathcal{Q}$.
\end{lemma}

\begin{proof}
Let $\tilde{r}_n(x)$ be a $k$ times continuously differentiable function defined on $\R^p$ that agrees with $r_n(x)$ on $\set{x\colon d(x, S_n^c)\geq 1/n}$.
Such a function can be constructed by, for instance, taking an infinitely differentiable function $\eta(t)$ with $\eta(t)=1$ if $t\geq 1$ and $\eta(t)=0$ if $t\leq 1/2$, and setting $\tilde{r}_n(x)=r_n(x)\eta(nd(x,S_n^c))$.

Let $x_0\in\interior S$ with the corresponding neighbourhood $G$ and $n_0\in\N$.
We can choose an open ball $\tilde{G}_{x_0}\subset G$ containing $x_0$ and $n_1\in\N$ such that every $x\in\tilde{G}_{x_0}$ satisfies $d(x,G^c)\geq 1/n_1$.
Write $\tilde{n}(x_0)=\max\set{n_0,n_1}$.
Then, for all $n\geq\tilde{n}(x_0)$, $r_n$ is defined on $\tilde{G}_{x_0}$ and agrees with $\tilde{r}_n$ there.
In particular, \eqref{rnxConditions} remains true when we replace $r_n$, $r_{n+1}$ and $G$ by $\tilde{r}_n$, $\tilde{r}_{n+1}$ and $\tilde{G}_{x_0}$.

For each $x\in\interior S$, apply \reflemma{SequenceToC1} with $r_n$ replaced by $\tilde{r}_n(x)$ and $n_0$ replaced by 1, noting that the assumptions apply because of \eqref{rnxConditions}, to produce functions $\epsilon\mapsto f(x,\epsilon)$ and $\epsilon\mapsto h(x,\epsilon)$.
By construction, $f(x,0)=0$ and $\frac{\partial f}{\partial\epsilon}(x,0)=0$. 
Moreover $f(x,\epsilon_n) = \tilde{r}_n(x)$ for all $n\in\N$, and therefore $f(x,\epsilon_n)=r_n(x)$ for all $n\geq\tilde{n}(x)$.
Thus if we set $\mathcal{Q}=\bigunion_{x\in\interior S} \tilde{G}_x \times \cointerval{0, \epsilon_{\tilde{n}(x)}}$ then $f(x,\epsilon_n)=r_n(x)$ whenever $(x,\epsilon_n)\in\mathcal{Q}$, as required.

It remains to show the continuity, in both variables, of $\frac{\partial^{j+j'} f}{\partial x_{i_1}\dotsb\partial x_{i_j}\partial\epsilon^{j'}}$, where $j+j'\leq k$.
If $j\geq 1$ we have
\begin{equation}
\frac{\partial^{j+j'} f}{\partial x_{i_1}\dotsb\partial x_{i_j}\partial\epsilon^{j'}}(x,\epsilon) = \frac{\partial^{j+j'-1} h}{\partial x_{i_1}\dotsb\partial x_{i_j}\partial\epsilon^{j'-1}}(x,\epsilon).
\end{equation}
As in the proof of \reflemma{SequenceToC1}, around any $\epsilon_0>0$ there is a neighbourhood $(\epsilon_1,\epsilon_2)$ containing $\epsilon_0$ such that at most one term contributes to the series defining $h(x,\epsilon)$ (cf.\ \eqref{fepsilonFormula}) throughout the region $\interior S\times(\epsilon_1,\epsilon_2)$.
Therefore $\frac{\partial^{j+j'-1} h}{\partial x_{i_1}\dotsb\partial x_{i_j}\partial\epsilon^{j'-1}}$ is continuous in the region $\interior S\times\ocinterval{0,1}$, and can be computed by termwise differentiation:
\begin{equation}
\frac{\partial^{j+j'-1} h}{\partial x_{i_1}\dotsb\partial x_{i_j}\partial\epsilon^{j'-1}}(x,\epsilon) 
= 
\sum_{n=1}^\infty \frac{\frac{\partial^j}{\partial x_{i_1}\dotsb\partial x_{i_j}} (r_n(x)-r_{n+1}(x))}{(\epsilon_n-\epsilon_{n+1})^{j'}}\eta^{(j'-1)}\left( \frac{\epsilon-\epsilon_{n+1}}{\epsilon_n-\epsilon_{n+1}} \right)
\end{equation}
By assumption $j' \leq k-j$, so \eqref{rnxConditions} and the argument from the proof of \reflemma{SequenceToC1} imply that $\frac{\partial^{j+j'-1} h}{\partial x_{i_1}\dotsb\partial x_{i_j}\partial\epsilon^{j'-1}}(x,\epsilon)\to 0$ as $\epsilon\decreasesto 0$, uniformly over $x$ in suitable neighbourhoods.
On the other hand, $\frac{\partial^{j'} h}{\partial\epsilon^{j'}}(x,0)$ vanishes identically by \reflemma{SequenceToC1}, and so therefore do its partial derivatives with respect to $x$.
We have therefore shown that $\frac{\partial^{j+j'-1} h}{\partial x_{i_1}\dotsb\partial x_{i_j}\partial\epsilon^{j'-1}}$, and therefore also $\frac{\partial^{j+j'} f}{\partial x_{i_1}\dotsb\partial x_{i_j}\partial\epsilon^{j'}}$, are continuous in both variables for $j'\geq 1$.
Finally for $j'=0$ we can repeatedly apply \reflemma{DiffUnderInt} to obtain
\begin{equation}
\frac{\partial^j f}{\partial x_{i_1}\dotsb\partial x_{i_j}}(x,\epsilon)
=\int_0^\epsilon \frac{\partial^j h}{\partial x_{i_1}\dotsb\partial x_{i_j}}(x,y)\, dy
= \int_0^1 \frac{\partial^j h}{\partial x_{i_1}\dotsb\partial x_{i_j}}(x,y) \indicator{y\leq\epsilon}.
\end{equation}
For any $(x_0,\epsilon_0)$, the function $(x,\epsilon)\mapsto \frac{\partial^j h}{\partial x_{i_1}\dotsb\partial x_{i_j}}(x,y) \indicator{y\leq\epsilon}$ is continuous at $(x_0,\epsilon_0)$ for almost every $y$, so a further application of \reflemma{DiffUnderInt} gives continuity in both variables for this case too.
\end{proof}

We are now ready to prove \refprop{PddtPError}.
\begin{proof}[Proof of \refprop{PddtPError}]
We first claim that we can find a continuous function $\delta(s,\theta)$ defined on $\mathcal{S}$ such that for all $(s,\theta)\in\mathcal{S}$, 
\begin{equation}\label{ReM0BoundedBelow}
\Re M_0(s+\ii\phi;\theta)>\delta(s,\theta) \qquad\text{for all $\phi\in\R^{1\times\xdim}$ with $\abs{\phi}\leq 2\delta(s,\theta)$}
\end{equation}
and
\begin{equation}\label{K''BoundedBelow}
v\Re K''_0(s+\ii\phi;\theta) v^T\geq \delta(s,\theta)\abs{v}^2\qquad\text{for all $v,\phi\in\R^{1\times\xdim}$ with $\abs{\phi}\leq 2\delta(s,\theta)$,}
\end{equation}
as well as the bounds \eqref{DecayBound}--\eqref{GrowthBound}.
In particular, we can locally define $K_0(s+\ii\phi;\theta)=\Log M_0(s+\ii\phi;\theta)$ for $\abs{\phi}\leq 2\delta(s,\theta)$, where $\Log$ denotes the standard branch of the logarithm with a branch cut along the negative real axis.
(Note that we did not need to specify this choice of branch cut for the logarithm in order to state \eqref{K''BoundedBelow} because $K'_0(s+\ii\phi;\theta)$ is given unambiguously as $K'_0 = M_0'/M_0$ in all cases.)
The function $K_0$ thus defined has the same smoothness properties as $M_0$.

For \eqref{ReM0BoundedBelow}, the Mean Value Theorem and \eqref{GrowthBound} with $k=0,\ell=1$ show that if $\abs{\phi}\leq \min\set{M_0(s;\theta)/4\gamma(s,\theta) 3^{\gamma(s,\theta)}, 1}$ then $\Re M_0(s+\ii\phi;\theta)\geq\frac{1}{2}M_0(s;\theta)$.
In all of the assertions, we are free to replace $\delta(s,\theta)$ by a smaller positive quantity, so the first part of the claim follows by replacing the function $\delta(s,\theta)$ from \eqref{DecayBound}--\eqref{GrowthBound} by 
\begin{equation}
\min\set{\delta(s,\theta), M_0(s;\theta)/4\gamma(s,\theta) 3^{\gamma(s,\theta)}, 1, \tfrac{1}{2}M_0(s;\theta)}.
\end{equation}
The bound \eqref{K''BoundedBelow} follows by a similar argument using \eqref{GrowthBound} with $k=0$ and $\ell\in\set{1,2,3}$.

Fix an infinitely differentiable function $\eta(s,\phi,\theta)$ with values in $[0,1]$ and satisfying $\eta(s,\phi,\theta)=0$ if $\abs{\phi}\geq 2\delta(s,\theta)$ and $\eta(s,\phi,\theta)=1$ if $\abs{\phi}\leq \delta(s,\theta)$.
We may assume moreover that $\phi\mapsto\eta(s,\phi,\theta)$ is radially symmetric for each fixed $(s,\theta)\in\mathcal{S}$.
(For instance, since $\delta$ is continuous, we may find an infinitely differentiable function $\tilde{\delta}(s,\theta)$ defined on $\interior\mathcal{S}$ such that $\frac{2}{3}\delta(s,\theta)\leq \tilde{\delta}(s,\theta)\leq \delta(s,\theta)$.
Then set $\eta(s,\phi,\theta) = \tilde{\eta}(\abs{\phi}/\tilde{\delta}(s,\theta))$ where $\tilde{\eta}$ is infinitely differentiable with $\tilde{\eta}(r)=0$ for $r\geq 2$ and $\tilde{\eta}(r)=1$ for $r\leq\frac{3}{2}$.)
In addition, all derivatives of $\eta$ are supported in the region $\delta(s,\theta)\leq\abs{\phi}\leq 2\delta(s,\theta)$.

In the rest of the proof, we will construct the functions $q_1,q_2$.
In order to show continuous differentiability, we will focus our attention on a fixed but arbitrary point $(s_0,\theta_0)\in\interior\mathcal{S}$, and suitable chosen neighbourhoods $U$ of $\theta_0$ and $W$ of $s_0$.
This will not entail any loss of generality because continuous differentiability is a local property, while, as we shall see, the construction of $q_1,q_2$ does not depend on the choice of $(s_0,\theta_0)$, $U$ or $W$.

Let $(s_0,\theta_0)\in\interior\mathcal{S}$ be given.
Since $\delta(s,\theta)$ and $\gamma(s,\theta)$ are continuous, we may choose neighbourhoods $U$, $W$ and constants $\delta_0>0$, $\gamma_0<\infty$ such that 
\begin{equation}\label{delta0gamma0}
\delta(s,\theta)\geq\delta_0, \quad \gamma(s,\theta)\leq\gamma_0 \qquad\text{for all $s\in W$, $\theta\in U$}.
\end{equation}
The numbers $\delta_0,\gamma_0$ and the neighbourhoods $U,W$ will now be fixed for the remainder of the proof.
We remark that the bounds \eqref{DecayBound}--\eqref{GrowthBound} and \eqref{ReM0BoundedBelow}--\eqref{K''BoundedBelow} are monotone in $\delta$ and $\gamma$, so they hold uniformly over $U,W$ with $\delta(s,\theta)$, $\gamma(s,\theta)$ replaced by $\delta_0,\gamma_0$.

Recall that $K_0(s+\ii\phi)$ can be defined for $\abs{\phi}\leq 2\delta(s,\theta)$, but may not be defined globally.
We will therefore rewrite the integrals in \eqref{PformulaK}--\eqref{dPdtFormulaK} by defining
\begin{equation}\label{I1I2Formula}
\begin{aligned}
w(s,\phi,\theta) &= \frac{M_0(s+\ii\phi;\theta)e^{-\ii\phi K_0'(s;\theta)}}{M_0(s)},
\\
I_1(s,\theta,n) &= \int_{\R^{1\times\xdim}} w(s,\phi,\theta)^n \frac{d\phi}{(2\pi)^\xdim},
\\
I_2(s,\theta,n) &= \int_{\R^{1\times\xdim}} w(s,\phi,\theta)^{n-1} n\frac{\partial w}{\partial t}(s,\phi,\theta) \frac{d\phi}{(2\pi)^\xdim}
\end{aligned}
\end{equation}
Let $\tau$ denote another entry $\theta_i$ or $s_j$ (possibly the same entry as $t$) and define
\begin{equation}\label{I3Formula}
I_3 = \int_{\R^{1\times\xdim}} w(s,\phi,\theta)^{n-2} \left[ n(n-1)\frac{\partial w}{\partial\tau}(s,\phi,\theta)\frac{\partial w}{\partial t}(s,\phi,\theta) + n \frac{\partial^2 w}{\partial\tau\partial t}(s,\phi,\theta)\right] \frac{d\phi}{(2\pi)^\xdim}.
\end{equation}
Thus the integrands for $I_2,I_3$ are $\frac{\partial}{\partial t}$ and $\frac{\partial^2}{\partial\tau\partial t}$ of the integrand for $I_1$.
Throughout the proof we will think of $t$ and $\tau$ as fixed but arbitrary; thus our conclusions about $I_2$ will apply \emph{mutatis mutandis} to the integral with $\frac{\partial w}{\partial t}$ replaced by $\frac{\partial w}{\partial\tau}$.

It is readily verified that $\frac{\partial w}{\partial t}$, $\frac{\partial w}{\partial\tau}$ and $\frac{\partial^2 w}{\partial\tau\partial t}$ can be expressed as multivariate polynomials in the arguments $\phi$, $M_0(s+\ii\phi)/M_0(s)$, $\frac{\partial M_0}{\partial t}(s+\ii\phi)$, $\frac{\partial M_0}{\partial\tau}(s+\ii\phi)$ and $\frac{\partial^2 M_0}{\partial\tau\partial t}(s+\ii\phi)$; the same quantities evaluated at $\phi=0$; $\frac{\partial M_0'}{\partial t}(s)$, $\frac{\partial M_0'}{\partial\tau}(s)$ and $\frac{\partial^2 M_0'}{\partial\tau\partial t}(s)$; and $1/M_0(s)$.
By \eqref{GrowthBound}, it follows that 
\begin{equation}\label{wDerivativeBounds}
\max\bigset{\abs{\tfrac{\partial w}{\partial t}}, \abs{\tfrac{\partial w}{\partial\tau}}, \bigabs{\tfrac{\partial^2 w}{\partial\tau\partial t}}} \leq C(1+\abs{\phi})^\alpha
\end{equation}
for some $C,\alpha<\infty$ that depend only on $\delta_0$, $\gamma_0$.
On the other hand, \eqref{DecayBound} implies
\begin{equation}\label{wBound}
\abs{w(s,\phi,\theta)}\leq (1+\delta_0\abs{\phi}^2)^{-\delta_0}
\end{equation}
since the upper bound in \eqref{DecayBound} is monotone in $\delta$.
So the integrands in \eqref{I1I2Formula}--\eqref{I3Formula} can be bounded by 
\begin{equation}\label{IntegrandBoundphi}
n^2 (C^2+C+1)(1+\abs{\phi})^{2\alpha} (1+\delta_0\abs{\phi}^2)^{-\delta_0(n-2)}
\end{equation}
uniformly in $s,\theta$, and the function in \eqref{IntegrandBoundphi} is integrable (for fixed $n$) as soon as $2\delta_0(n-2)-2\alpha>m$.
Thus if we define $n_0=3+(m+2\alpha)/2\delta_0$ then the integrands in \eqref{I1I2Formula}--\eqref{I3Formula} are dominated by integrable functions for each $n\geq n_0$.
In particular, by \reflemma{DiffUnderInt}, for all $n\geq n_0$, 
\begin{equation}\label{PartialsOfI12}
I_1(s,\theta,n)=P_n(s,\theta), \quad I_2=\frac{\partial I_1}{\partial t}=\frac{\partial P_n}{\partial t}, \quad I_3=\frac{\partial I_2}{\partial\tau}
\end{equation}
with $I_1,I_2,I_3$ continuous for fixed $n$.

Define 
\begin{multline}
h(s,\phi,\theta)=\eta(s,\phi,\theta)\left[ K_0(s+\ii\phi;\theta) - K_0(s;\theta) - \ii\phi K_0'(s;\theta) \right] 
\\
+ (1-\eta(s,\phi,\theta)) \left[ - \tfrac{1}{2}\phi K_0''(s;\theta)\phi^T \right].
\end{multline}
Note that $h$ is well-defined since by construction $\eta$ is supported in the region where $K_0$ is well-defined.
We remark that $h$ behaves quadratically in $\phi$ for $\abs{\phi}$ small.
To emphasise this, apply Lagrange's form of the remainder term,
\begin{equation}\label{LagrangeForm2}
f(1) - f(0) - f'(0) = \int_0^1 (1-u) f''(u)\, du,
\end{equation}
with the function 
\begin{equation}\label{fForLagrangeForm}
f(u)=\eta(s,\phi,\theta)K_0(s+\ii u\phi;\theta)-\tfrac{1}{2}u^2(1-\eta(s,\phi,\theta))\phi K_0''(s;\theta)\phi^T.
\end{equation}
Thus, writing
\begin{equation}
g(s,\phi,\theta,u) = \eta(s,\phi,\theta) K_0''(s+\ii u\phi;\theta) + (1-\eta(s,\phi,\theta)) K_0''(s;\theta), 
\end{equation}
we can rewrite $h$ as
\begin{equation}\label{hAsphigphi}
h(s,\phi,\theta) = - \int_0^1 (1-u)\phi g(s,\phi,\theta,u) \phi^T \, du
.
\end{equation}
Then \eqref{K''BoundedBelow} implies 
\begin{equation}\label{ghBound}
v \Re g(s,\phi,\theta,u)v^T \geq \delta_0\abs{v}^2, \qquad \Re h(s,\phi,\theta) \leq -\tfrac{1}{2}\delta_0\abs{\phi}^2.
\end{equation}

Since $K_0'' = M_0''/M_0 - (M_0'/M_0)^2$, the partial derivatives $\frac{\partial g}{\partial t}$, $\frac{\partial g}{\partial\tau}$ and $\frac{\partial^2 g}{\partial\tau\partial t}$ can be expressed as multivariate polynomials in the arguments $\frac{\partial M_0}{\partial t}(s+\ii u\phi)$, $\frac{\partial M_0}{\partial\tau}(s+\ii u\phi)$, $\frac{\partial^2 M_0}{\partial\tau\partial t}(s+\ii u\phi)$, $\frac{\partial M_0'}{\partial t}(s+\ii u\phi)$, $\frac{\partial M_0'}{\partial\tau}(s+\ii u\phi)$, $\frac{\partial^2 M_0'}{\partial\tau\partial t}(s+\ii u\phi)$, $\frac{\partial M_0''}{\partial t}(s+\ii u\phi)$, $\frac{\partial M_0''}{\partial\tau}(s+\ii u\phi)$, $\frac{\partial^2 M_0''}{\partial\tau\partial t}(s+\ii u\phi)$, and $1/M_0(s+\ii u\phi)$; the same quantities evaluated at $\phi=0$; and $\eta(s,\phi,\theta)$ and its derivatives.
Note that $1/M_0(s+\ii u\phi)$ may become unbounded if $M_0$ has zeros, but powers $1/M_0(s+\ii u\phi)^k$ only appear in combination with a factor $\eta(s,\phi,\theta)$ or one of its derivatives, all of which are supported in the region $\abs{\phi}\leq 2\delta(s,\theta)$ where \eqref{ReM0BoundedBelow} applies.
It follows that each such combination is uniformly bounded.
Hence, by \eqref{GrowthBound}, $g$ and its partial derivatives (and hence, by \reflemma{DiffUnderInt}, $h$ and its partial derivatives) grow at most polynomially in $\phi$: 
\begin{equation}\label{partialghBounds}
\max\bigset{\abs{\tfrac{\partial g}{\partial t}}, \abs{\tfrac{\partial g}{\partial\tau}}, \bigabs{\tfrac{\partial^2 g}{\partial\tau\partial t}}, \abs{\tfrac{\partial h}{\partial t}}, \abs{\tfrac{\partial h}{\partial\tau}}, \bigabs{\tfrac{\partial^2 h}{\partial\tau\partial t}}} \leq C'(1+\abs{\phi})^{\alpha'}
\end{equation}
for some $C',\alpha'<\infty$ that depend only on $\delta_0$, $\gamma_0$.
The same argument (involving up to 3 $s$-derivatives of $M_0''$ and $\frac{\partial M_0''}{\partial t}$) shows that
\begin{equation}\label{gradphigBounds}
\max\bigset{\bigabs{\tfrac{\partial g}{\partial\phi_i}}, \bigabs{\tfrac{\partial^2 g}{\partial\phi_i\partial\phi_j}}, \bigabs{\tfrac{\partial^3 g}{\partial\phi_i\partial\phi_j\partial\phi_k}}, \bigabs{\tfrac{\partial^2 g}{\partial\phi_i\partial t}}, \bigabs{\tfrac{\partial^3 g}{\partial\phi_i\partial\phi_j\partial t}}, \bigabs{\tfrac{\partial^4 g}{\partial\phi_i\partial\phi_j\partial\phi_k\partial t}}} \leq C''(1+\abs{\phi})^{\alpha''}
\end{equation}
for some $C'',\alpha''<\infty$ that also depend only on $\delta_0$, $\gamma_0$.

By analogy with \eqref{I1I2Formula}--\eqref{I3Formula}, define
\begin{equation}\label{J1J2J3Formula}
\begin{aligned}
J_1(s,\theta,\epsilon) &= \int_{\R^{1\times\xdim}} \exp\left( \epsilon^{-1} h(s,\phi,\theta) \right) \frac{d\phi}{(2\pi)^\xdim}
,
\\
J_2(s,\theta,\epsilon) &= \int_{\R^{1\times\xdim}} \exp\left( \epsilon^{-1} h(s,\phi,\theta) \right) \epsilon^{-1} \frac{\partial h}{\partial t}(s,\phi,\theta) \frac{d\phi}{(2\pi)^\xdim}
,
\\
J_3(s,\theta,\epsilon) &= \int_{\R^{1\times\xdim}} \exp\left( \epsilon^{-1} h(s,\phi,\theta) \right) \left[ \epsilon^{-1} \frac{\partial^2 h}{\partial\tau\partial t}(s,\phi,\theta) 
\right.
\\&\qquad\qquad\qquad\qquad\qquad\qquad\qquad
\left.
+ \, \epsilon^{-2} \frac{\partial h}{\partial\tau}(s,\phi,\theta)\frac{\partial h}{\partial t}(s,\phi,\theta) \right] \frac{d\phi}{(2\pi)^\xdim}.
\end{aligned}
\end{equation}
The bounds \eqref{ghBound}--\eqref{partialghBounds} imply that the integrands in \eqref{J1J2J3Formula} can be bounded by 
\begin{equation}\label{IntegrandBoundphih}
\epsilon^{-2}((C')^2+C'+1)(1+\abs{\phi})^{2\alpha'}\exp\left( -\tfrac{1}{2}\epsilon^{-1}\delta_0\abs{\phi}^2 \right)
\end{equation}
uniformly in $s,\theta$, which is integrable for any fixed $\epsilon>0$.
So \reflemma{DiffUnderInt} applies and shows 
\begin{equation}\label{PartialsOfJ12}
J_2=\frac{\partial J_1}{\partial t}, \quad J_3=\frac{\partial J_2}{\partial\tau} \qquad\text{for all }\epsilon>0.
\end{equation}
Furthermore $J_1,J_2,J_3$ are continuous in all three variables in the region $\epsilon>0$.

We will next show that, for $k=1,2,3$,
\begin{equation}\label{rnkFormula}
r_n^{(k)}(s,\theta) = \frac{I_k(s,\theta,n) - J_k(s,\theta,1/n)}{\hat{P}_n(s,\theta)}
\end{equation}
satisfies $r_n^{(k)}(s,\theta) = o(n^{-2})$ uniformly over $s\in W,\theta\in U$.
Note that the integrands defining $I_k(s,\theta,n)$ and $J_k(s,\theta,1/n)$ agree over the region $\abs{\phi}\leq\delta_0$ because $\eta(s,\phi,\theta)=1$ and $\exp(h)=w$ there.
When $n\geq n_0$ and $\abs{\phi}\geq\delta_0$, the quantities in \eqref{IntegrandBoundphi} and \eqref{IntegrandBoundphih} (with $\epsilon^{-1}$ replaced by $n$) are bounded by
\begin{multline}\label{ExpDecayingBound}
C''' (1+\abs{\phi})^{2\alpha'''} \left( (1+\delta_0\abs{\phi}^2)^{-\delta_0(n_0-2)} + \exp(-\tfrac{1}{2} n_0\delta_0\abs{\phi}^2) \right) 
\\
\cdot n^2 \left( (1+\delta_0^3)^{-\delta_0(n-n_0)} + \exp(-\tfrac{1}{2} (n-n_0)\delta_0^3) \right)
\end{multline}
uniformly in $s,\theta$.
Integrating over $\phi$ leads to the conclusion that $I_k(s,\theta,n) - J_k(s,\theta,1/n)$ decays exponentially in $n$, uniformly in $s,\theta$.
Since $\hat{P}_n(s,\theta)^{-1}=O(n^{\xdim/2})$, uniformly in $s,\theta$, we have
\begin{equation}\label{rnkon-2}
r_n^{(k)}(s,\theta) = o(n^{-2})\text{ uniformly over }s\in W,\theta\in U,
\end{equation}
as asserted.

Next we apply \reflemma{FunctionsToC1} to $r_n^{(k)}(s,\theta)$ for $k=1,2$.
We set $\epsilon_n = 1/n$, so that $\epsilon_n-\epsilon_{n+1} \sim n^{-2}$.
Recalling \eqref{PartialsOfI12} and \eqref{PartialsOfJ12}, the identity
\begin{equation}\label{ZoverhatP}
\frac{\partial}{\partial\tau} \left( \frac{Z}{\hat{P}_n(s,\theta)} \right) = \frac{\frac{\partial Z}{\partial\tau}}{\hat{P}_n(s,\theta)} - \frac{Z}{\hat{P}_n(s,\theta)} \frac{\partial}{\partial\tau}\log\hat{P}(s,\theta)
\end{equation}
with $Z=I_k-J_k$ allows us to express the gradients $\grad_s r_n^{(k)}$ and $\grad_\theta r_n^{(k)}$ for $k=1,2$ in terms of $r_n^{(2)}$ (or rather, the analogue of $r_n^{(2)}$ with $t$ replaced by $\tau$) and $r_n^{(3)}$.
Since $\frac{\partial}{\partial\tau}\log\hat{P}(s,\theta)$ is bounded and continuous in $s,\theta$, and does not depend on $n$, it follows that $\grad_s r_n^{(k)}$ and $\grad_\theta r_n^{(k)}$ are continuous and $o(n^{-2})$, uniformly in $s,\theta$, for $k=1,2$.
By \reflemma{FunctionsToC1} we can find continuously differentiable functions $\tilde{f}_1(s,\theta,\epsilon)$, $\tilde{f}_2(s,\theta,\epsilon)$ such that
\begin{equation}\label{tildefPJhatP}
\tilde{f}_1(s,\theta,1/n) = \frac{P_n(s,\theta) - J_1(s,\theta,1/n)}{\hat{P}_n(s,\theta)}, \qquad \tilde{f}_2(s,\theta,1/n) = \frac{\frac{\partial P_n}{\partial t}(s,\theta) - J_2(s,\theta,1/n)}{\hat{P}_n(s,\theta)}
\end{equation}
for $s\in W$, $\theta\in U$, $n\geq n_0$.

We claim that $J_k(s,\theta,1/n)/\hat{P}_n(s,\theta)$, $k=1,2$, can be extended to be continuously differentiable functions of $s,\theta,\epsilon=1/n$, including at $\epsilon=0$.
We have already defined $J_k$ for all $\epsilon>0$, and we can extend $\hat{P}_n(s,\theta)$ by declaring that \eqref{hatPn} holds whether or not $n>0$ is an integer.
To handle the extension to $\epsilon=0$, substitute $\phi=\psi\sqrt{\epsilon}$, $d\phi=\epsilon^{\xdim/2}d\psi$ in \eqref{J1J2J3Formula} and use \eqref{hAsphigphi}:
\begin{align}
\frac{J_1(s,\theta,\epsilon)}{\hat{P}_{1/\epsilon}(s,\theta)} &= 
\int_{\R^{1\times\xdim}} \exp\left( \epsilon^{-1} h(s,\phi,\theta) \right) \frac{\sqrt{\det(2\pi K_0''(s;\theta))}}{(2\pi)^\xdim \epsilon^{\xdim/2}} d\phi
\notag\\&
= \int_{\R^{1\times\xdim}} \frac{\exp\bigl( -\int_0^1 (1-u)\psi g(s,\psi\sqrt{\epsilon},\theta,u) \psi^T \, du \bigr)}{\sqrt{\det(2\pi K_0''(s;\theta)^{-1})}} d\psi
.
\label{J1Pratio}
\end{align}
Similar considerations hold for $J_2$ and $J_3$.
Define
\begin{align}
f_1(s,\theta,\epsilon) &= \int_{\R^{1\times\xdim}} \frac{\exp\bigl( -\int_0^1 (1-u)\psi g(s,\psi\sqrt{\epsilon},\theta,u) \psi^T \, du \bigr)}{\sqrt{\det(2\pi K_0''(s;\theta)^{-1})}} d\psi
,
\notag\\
f_2(s,\theta,\epsilon) &= \int_{\R^{1\times\xdim}} \frac{\exp\bigl( -\int_0^1 (1-u)\psi g(s,\psi\sqrt{\epsilon},\theta,u) \psi^T \, du \bigr)}{\sqrt{\det(2\pi K_0''(s;\theta)^{-1})}} 
\notag\\&\qquad\cdot
\left( -\int_0^1 (1-u)\psi \frac{\partial g}{\partial t}(s,\psi\sqrt{\epsilon},\theta,u) \psi^T \, du \right) d\psi
,
\notag\\
f_3(s,\theta,\epsilon) &= \int_{\R^{1\times\xdim}} \frac{\exp\bigl( -\int_0^1 (1-u)\psi g(s,\psi\sqrt{\epsilon},\theta,u) \psi^T \, du \bigr)}{\sqrt{\det(2\pi K_0''(s;\theta)^{-1})}} 
\notag\\&\qquad\cdot
\left( \int_0^1 (1-v)\psi \frac{\partial g}{\partial\tau}(s,\psi\sqrt{\epsilon},\theta,v) \psi^T \, dv \int_0^1 (1-u)\psi \frac{\partial g}{\partial t}(s,\psi\sqrt{\epsilon},\theta,u) \psi^T \, du 
\right.
\notag\\&\qquad\quad
\left.
- \int_0^1 (1-u)\psi \frac{\partial^2 g}{\partial\tau\partial t}(s,\psi\sqrt{\epsilon},\theta,u) \psi^T \, du \right) d\psi
\label{f1f2f3Formula}
\end{align}
for $\epsilon\geq 0$.
\reflemma{DiffUnderInt} applied repeatedly to \eqref{hAsphigphi} confirms that
\begin{equation}\label{fkJkhatP}
f_k(s,\theta,\epsilon) = \frac{J_k(s,\theta,\epsilon)}{\hat{P}_{1/\epsilon}(s,\theta)}
\end{equation}
for all $\epsilon>0$.
The integrands in \eqref{f1f2f3Formula} are continuous in all of their variables, including at $\epsilon=0$, and by \eqref{ghBound}--\eqref{partialghBounds} are bounded by
\begin{equation}\label{IntegrandBoundpsi}
((C')^2+C'+1)(1+\abs{\psi}^2)^{2\alpha'+2}\exp\left( -\tfrac{1}{2}\delta_0\abs{\psi}^2 \right)
\end{equation}
for all $\epsilon\in[0,1]$.
The function in \eqref{IntegrandBoundpsi} is integrable, so \reflemma{DiffUnderInt} implies that $f_1,f_2,f_3$ are continuous.
By another application of \eqref{ZoverhatP} with $Z=f_k$, $k=1,2$, the continuity of $f_2,f_3$ implies that $\frac{\partial f_1}{\partial\tau},\frac{\partial f_2}{\partial\tau}$ are continuous.

To complete the claim that $f_1,f_2$ are continuously differentiable, it suffices to show that $\frac{\partial f_1}{\partial\epsilon}$ and $\frac{\partial f_2}{\partial\epsilon}$ are continuous.
This is less clear: the integrands in \eqref{f1f2f3Formula} are typically not differentiable at $\epsilon=0$ because of the $\sqrt{\epsilon}$.
However, in a Taylor expansion, the $\sqrt{\epsilon}$ term is odd as a function of $\psi$ and therefore makes no contribution to the integral.
To see this, note that the integrands defining $f_1,f_2$ can be written in the form $H_k(\psi, \psi\sqrt{\epsilon})$, where 
\begin{equation}\label{H1H2Formula}
\begin{aligned}
H_1(v,\phi) &= \frac{\exp\bigl( -\int_0^1 (1-u)v g(s,\phi,\theta,u) v^T \, du \bigr)}{\sqrt{\det(2\pi K_0''(s;\theta)^{-1})}}, 
\\
H_2(v,\phi) &= - H_1(v,\phi) \int_0^1 (1-u)v \frac{\partial g}{\partial t}(s,\phi,\theta,u) v^T \, du.
\end{aligned}
\end{equation}
(For convenience we omit the dependence on $s,\theta$.)
Apply \eqref{LagrangeForm2} with $f(u)=H_k(v,u\phi)$ to find
\begin{equation}
H_k(\psi, \psi\sqrt{\epsilon}) = H_k(\psi,0) + \sqrt{\epsilon} \psi\grad_\phi H_k(\psi,0) + \epsilon \int_0^1 (1-u)\psi \grad_\phi\grad_\phi^T H_k(\psi, u\psi\sqrt{\epsilon}) \psi^T \, du.
\end{equation}
The quantities $H_k(v,\phi)$ are even as functions of $v$ for each fixed $\phi$, so that $\grad_\phi H_k(v,\phi)$ is also even as a function of $v$.
Hence $\psi\grad_\phi H_k(\psi,0)$ is odd as a function of $\psi$.
From \eqref{ghBound} and \eqref{gradphigBounds} it is straightforward to verify that $H_k(v,\phi)$, $\frac{\partial H_k}{\partial\phi_i}(v,\phi)$, $\frac{\partial^2 H_k}{\partial\phi_i\phi_j}(v,\phi)$, $\frac{\partial^3 H_k}{\partial\phi_i\phi_j\phi_\ell}(v,\phi)$, for $k=1,2$ and $i,j,\ell=1,\dotsc,\xdim$, can all be bounded by
\begin{equation}\label{HkBound}
C'''' (1+\abs{\phi}+\abs{v})^{\alpha''''} \exp\left( -\tfrac{1}{2}\delta_0\abs{v}^2 \right)
\end{equation}
for some constants $C'''',\alpha''''$ depending only on $\delta_0,\gamma_0$.
In particular, the term $\psi\grad_\phi H_k(\psi,0)$ is integrable, so that oddness implies that its integral is zero.
Thus
\begin{equation}\label{fWithoutOdd}
f_k(\theta,s,\epsilon) = \int_{\R^{1\times\xdim}} \left( H_k(\psi,0) + \epsilon \int_0^1 (1-u)\psi \grad_\phi\grad_\phi^T H_k(\psi, u\psi\sqrt{\epsilon}) \psi^T \, du \right) d\psi
.
\end{equation}
In \eqref{fWithoutOdd}, the integrand is now continuously differentiable even at $\epsilon=0$, and \eqref{HkBound} implies that its derivative with respect to $\epsilon$ is bounded by an integrable function.
Thus \reflemma{DiffUnderInt} applies and verifies that $\frac{\partial f_1}{\partial\epsilon}$ and $\frac{\partial f_2}{\partial\epsilon}$ are continuous, as claimed, and hence $f_1,f_2$ are continuously differentiable.

Finally define $q_k(s,\theta,\epsilon)=\tilde{f}_k(s,\theta,\epsilon)+f_k(s,\theta,\epsilon) - f_k(s,\theta,0)$ for $\epsilon\in[0,1/n_0]$, so that $q_1,q_2$ are continuously differentiable and $q_k(s,\theta,0)=0$ by construction. 
Combining \eqref{tildefPJhatP} and \eqref{fkJkhatP},
\begin{equation}\label{Pnfq}
\begin{aligned}
\frac{P_n(s,\theta)}{\hat{P}_n(s,\theta)} &= \frac{I_1(s,\theta,n)}{\hat{P}_n(s,\theta)} = \tilde{f}_1(s,\theta,1/n) + f_1(s,\theta,1/n) = f_1(s,\theta,0) + q_1(s,\theta,1/n)
,
\\
\frac{\frac{\partial P_n}{\partial t}(s,\theta)}{\hat{P}_n(s,\theta)} &= \frac{I_2(s,\theta,n)}{\hat{P}_n(s,\theta)} = \tilde{f}_2(s,\theta,1/n) + f_2(s,\theta,1/n) = f_2(s,\theta,0) + q_2(s,\theta,1/n)
.
\end{aligned}
\end{equation}
We need to extend $q_k$ to be defined for $\epsilon\in[-1/n_0,1/n_0]$ (this is for convenience only, since the Implicit Function Theorem typically assumes that the base point belongs to the interior of the domain) and this can be done arbitrarily as long as $q_k$ remains continuously differentiable.
Since $q_k(s,\theta,0)$ vanishes identically, one way to make this extension is to require $q_k$ to be an odd function of $\epsilon$.
Lastly \eqref{PddtPErrorFormula} follows from \eqref{Pnfq} by noting $g(s,0,\theta,u) = K_0''(s;\theta)$, $\frac{\partial g}{\partial t}(s,0,\theta,u) = \frac{\partial K_0''}{\partial t}(s;\theta)$ and calculating the Gaussian integrals
\begin{align*}
f_1(s,\theta,0) &= \int_{\R^{1\times\xdim}} \frac{\exp\bigl( -\frac{1}{2}\psi K_0''(s;\theta) \psi^T \bigr)}{\sqrt{\det(2\pi K_0''(s;\theta)^{-1})}} d\psi = 1,
\\
f_2(s,\theta,0) &= -\int_{\R^{1\times\xdim}} \frac{\exp\bigl( -\frac{1}{2}\psi K_0''(s;\theta) \psi^T \bigr)}{\sqrt{\det(2\pi K_0''(s;\theta)^{-1})}} 
\frac{1}{2} \psi \frac{\partial K_0''}{\partial t}(s,\phi) \psi^T d\psi 
\\
&= -\frac{1}{2}\trace\left( K_0''(s;\theta)^{-1} \frac{\partial K_0''}{\partial t}(s;\theta) \right) = \frac{\partial}{\partial t}\log\hat{P}(s,\theta).
\qedhere
\end{align*}
\end{proof}

To prove \refcoro{gradlogPhatsError}, we first state an intermediate result.

\begin{coro}\lbcoro{ddtlogPError}
Under the hypotheses of \refthm{GradientError}, there are continuously differentiable functions $q_5(s,\theta,\epsilon),q_6(s,\theta,\epsilon)$, with values in $\R^{1\times p}$ and $\R^{\xdim\times 1}$ respectively, defined on an open set $\mathcal{Q}''$ containing $\set{(s,\theta,0)\colon (s,\theta)\in\interior\mathcal{S}}$, such that $q_5(s,\theta,0)=0$, $q_6(s,\theta,0)=0$ and
\begin{equation}\label{gradlogDifference}
\begin{aligned}
\grad_\theta \log P_n(s,\theta) &= \grad_\theta \log\hat{P}(s,\theta) + q_5(s,\theta,1/n)
,
\\
\grad_s \log P_n(s,\theta) &= \grad_s \log\hat{P}(s,\theta) + q_6(s,\theta,1/n)
\end{aligned}
\end{equation}
whenever $(s,\theta,1/n)\in\mathcal{Q}''$.
\end{coro}

\begin{proof}
Let $t$ denote one of the entries $\theta_i$ or $s_j$ and consider $\frac{\partial}{\partial t}\log P_n(s,\theta)$.
Comparing with the ratio of the two quantities in \eqref{PddtPErrorFormula}, set
\begin{align}
q_t(s,\theta,\epsilon) &= \frac{\frac{\partial}{\partial t}\log\hat{P}(s,\theta) + q_2(s,\theta,\epsilon)}{1 + q_1(s,\theta,\epsilon)} - \frac{\partial}{\partial t}\log\hat{P}(s,\theta)
\notag\\&
=\frac{q_2(s,\theta,\epsilon) - q_1(s,\theta,\epsilon)\frac{\partial}{\partial t}\log\hat{P}(s,\theta)}{1 + q_1(s,\theta,\epsilon)}
.
\end{align}
We have immediately that $q_t(s,\theta,0)=0$.
By \refprop{PddtPError} and \eqref{s0hatPGradients}, $q_1,q_2$ and $\frac{\partial}{\partial t}\log\hat{P}$ are continuously differentiable, so $q_t$ will be as well provided that $q_1(s,\theta,\epsilon)\neq -1$.
Since $q_1(s,\theta,0)=0$ and $q_1$ is continuous, we can rule out $q_1(s,\theta,\epsilon)=-1$ by shrinking $\mathcal{Q}''$ to be a suitable subset of $\mathcal{Q}$.
The functions $q_5,q_6$ are obtained by performing this construction $p+\xdim$ times, once for each of the entries of $\theta$ and $s$.
\end{proof}

\begin{proof}[Proof of \refcoro{gradlogPhatsError}]
By the assumption $(y_0,\theta_0)\in\mathcal{Y}^o$, we have $\hat{s}_0(\theta_0,y_0)\in\interior\mathcal{S}$, so \refcoro{ddtlogPError} implies that $q_5,q_6$ are defined and continuously differentiable in a suitable neighbourhood.
From \eqref{ChainRuleCol} we have
\begin{align}
&\grad_\theta \left( \Big. \log P_n(\hat{s}_0(\theta,y),\theta) \right) 
\notag\\&\quad
= \left( \big. \grad_\theta\log P_n \right) (\hat{s}_0(\theta,y),\theta) 
+ \left[ \left( \big. \grad_s^T \log P_n \right) (\hat{s}_0(\theta,y),\theta) \right] \grad_\theta \hat{s}_0^T(\theta,y)
\end{align}
and similarly for $\grad_\theta \left( \log \hat{P}_n(\hat{s}_0(\theta,y),\theta) \right)$.
Thus if we set
\begin{equation}
q_3(\theta,y,\epsilon) = q_5(\hat{s}_0(\theta,y),\theta,\epsilon) + q_6(\hat{s}_0(\theta,y),\theta,\epsilon)^T \grad_\theta \hat{s}_0^T(\theta,y)
\end{equation}
then the relation \eqref{gradlogDifferenceshat} holds, and the relation $q_3(\theta,y,0)=0$ and the continuous differentiability of $q_3$ follow from the corresponding properties for $q_5,q_6$ and the fact that $\grad_\theta\hat{s}_0^T$ is continuously differentiable, see \eqref{s0hatPGradients}.
\end{proof}

\section{Proofs of \refprop{PddtPErrorInt} and \refthm{IntegerValued}}\lbappendix{IntegerValuedProof}

We begin by verifying the assertion following \refthm{IntegerValued} that if $\P(X_\theta=x)>0$ for all $x\in\set{0,1}^{\xdim\times 1}$, then $\abs{M_0(s+\ii\phi;\theta)}< M_0(s;\theta)$ holds for all $\phi\in[-\pi,\pi]^{1\times\xdim}\setminus\set{0}$.
(We remark that if $Y$ satisfies the same property, then so does $X$: given $x\in\set{0,1}^{\xdim\times 1}$, there is a positive probability that $Y^{(1)}=x$ and $Y^{(i)}=0$ for $2\leq i\leq n$.)
To this end, note that if $\phi\in[-\pi,\pi]^{1\times\xdim}\setminus\set{0}$, then there exist $x^{(1)},x^{(2)}\in\set{0,1}^{\xdim\times 1}$ such that $e^{\ii\phi(x^{(2)}-x^{(1)})}\neq 1$.
Namely, since $\phi\neq 0$ there is some coordinate $i\in\set{1,\dotsc,m}$ such that $\phi_i\neq 0$, and it suffices to choose $x^{(1)},x^{(2)}$ so that they differ by 1 in the $i^\text{th}$ coordinate and are equal elsewhere.
In particular, if we set $z=e^{\ii \phi}$ then $\abs{z^x}=1$ for all $x\in\Z^{\xdim\times 1}$ and $z^{x^{(2)}-x^{(1)}}\neq 1$.
(Here and elsewhere, $z=e^{\ii\phi}$ is interpreted coordinatewise and $z^x$ means $\prod_{i=1}^\xdim z_i^{x_i}$ in multi-index notation.)
Hence
\begin{align}
\abs{M(s+\ii\phi;\theta)} &= \abs{\sum_{x\in\Z^{\xdim\times 1}} \P(X=x) e^{sx+\ii\phi x}}
\notag\\&
\leq \abs{\P(X=x^{(1)})e^{sx^{(1)}}z^{x^{(1)}} + \P(X=x^{(2)})e^{sx^{(2)}}z^{x^{(2)}}} 
\notag\\&\quad
+ \sum_{\substack{x\in\Z^{\xdim\times 1},\\ x\notin\shortset{x^{(1)},x^{(2)}}}} \abs{\P(X=x) e^{sx}z^x}
\notag\\&
= \abs{z^{x^{(1)}}} \abs{\P(X=x^{(1)})e^{sx^{(1)}} + \P(X=x^{(2)})e^{sx^{(2)}}z^{x^{(2)}-x^{(1)}} } 
\notag\\&\quad
+ \sum_{\substack{x\in\Z^{\xdim\times 1},\\ x\notin\shortset{x^{(1)},x^{(2)}}}} \P(X=x) e^{sx}
\notag\\&
< \P(X=x^{(1)})e^{s^T x^{(1)}} + \P(X=x^{(2)})e^{sx^{(2)}} + \sum_{\substack{x\in\Z^{\xdim\times 1},\\ x\notin\shortset{x^{(1)},x^{(2)}}}} \P(X=x) e^{sx}
\end{align}
because, for $a,b>0$ and $\abs{w}=1$, the inequality $\abs{a+bw} \leq a+b$ is an equality if and only if $w=1$.
That is, $\abs{M(s+\ii\phi;\theta)}<M(s;\theta)$, and the same result holds for $M_0$ because of the relation $\abs{M(s;\theta)}=\abs{M_0(s;\theta)}^n$.

The proof of \refprop{PddtPErrorInt} is almost identical to that of \refprop{PddtPError}, and we outline where and how they differ.

\begin{proof}[Proof of \refprop{PddtPErrorInt}]
The local bounds \eqref{ReM0BoundedBelow}--\eqref{K''BoundedBelow} use only \eqref{K''PosDef} and \eqref{GrowthBound}, so they continue to hold (after possibly decreasing $\delta(s,\theta)$) under the hypotheses of \refthm{IntegerValued}.
We may in addition assume that
\begin{equation}
\delta(s,\theta) < \tfrac{1}{2}\pi
\end{equation}
so that the region $\abs{\phi}\leq 2\delta(s,\theta)$, and therefore the support of $\eta$, are subsets of the region $(-\pi,\pi)^{1\times\xdim}$.

The bound \eqref{wDerivativeBounds} depends only on \eqref{GrowthBound} and therefore remains valid; indeed, as remarked in \refsubsubsect{IntegerValued}, the quantity $w$ is $2\pi$-periodic in each entry of $\phi$, so it is enough to know that $w$ is continuous.
We must establish an analogue of the bound \eqref{wBound}, which cannot be valid for all $\phi$ because of periodicity.
From \eqref{K''BoundedBelow} and the argument that led to \eqref{ghBound}, we can conclude that, uniformly over $s$ and $\theta$ in suitable neighbourhoods $W$ and $U$,
\begin{equation}
\abs{w(s,\phi,\theta)} \leq \exp\left( -\tfrac{1}{2}\delta_0\abs{\phi}^2 \right) \quad\text{for }\abs{\phi}\leq 2\delta(s,\theta).
\end{equation}
On the other hand, by the assumption in \refthm{IntegerValued}, 
\begin{equation}
\max_{\phi\in[-\pi,\pi]^{1\times\xdim}\colon \abs{\phi}\geq 2\delta(s,\theta)} \abs{w(s,\phi,\theta)} < 1 \quad\text{for all }(s,\theta)\in\interior\mathcal{S},
\end{equation}
and by continuity and compactness, this maximum is continuous as a function of $(s,\theta)$.
It follows that, shrinking $\delta_0$ if necessary, $w$ satisfies the bound \eqref{wBound} for all $\phi\in[-\pi,\pi]^{1\times\xdim}$.
Let $\tilde{w}(s,\phi,\theta)=w(s,\phi,\theta)\indicatorofset{[-\pi,\pi]^{1\times\xdim}}(\phi)$.
Differentiation with respect to $t$ or $\tau$ does not affect the indicator, so $\tilde{w}$ satisfies the bounds \eqref{wDerivativeBounds}--\eqref{wBound}.
With these bounds given, we can define $\tilde{I}_1,\tilde{I}_2,\tilde{I}_3$ as in \eqref{I1I2Formula}--\eqref{I3Formula} with $w$ replaced by $\tilde{w}$, and the same argument shows that \eqref{PartialsOfI12} holds with $I_k$ and $P_n$ replaced by $\tilde{I}_k$ and $P_{\mathrm{int},n}$.

The rest of the argument proceeds as in the proof of \refprop{PddtPError}: the construction and bounding of $h$ and $g$ use only \eqref{ReM0BoundedBelow}--\eqref{K''BoundedBelow}; the exponentially decaying bound \eqref{ExpDecayingBound} still follows from \eqref{wDerivativeBounds}--\eqref{wBound}, applied to $\tilde{w}$ and $h$; and the definition and analysis of the functions $f_1$, $f_2$, $f_3$, and of $q_1$, $q_2$, are unchanged.
\end{proof}

\begin{proof}[Proof of Theorems~\ref{T:IntegerValued} and \ref{T:LowerOrderSaddlepoint}\ref{item:LOIntegerValued}]
Analogues of Corollaries~\ref{C:gradlogPhatsError} and~\ref{C:ddtlogPError} for $P_{\mathrm{int}}(s,\theta)$ follow from \refprop{PddtPErrorInt} by the same argument as in \refappendix{MainPropProof}.
The analogues of Theorems~\ref{T:GradientError}--\ref{T:SamplingMLE} and \ref{T:LowerOrderSaddlepoint}\ref{item:LOGradient}--\ref{item:LOSampling} are identical since, as we shall see, all of those proofs use only the result of \refcoro{gradlogPhatsError}.
The proof of \refthm{MLEerrorPartiallyIdentifiable}\ref{item:PI1/n^2} uses an extension of \refcoro{gradlogPhatsError}, but this extension applies equally in the setting of \refthm{IntegerValued}.
\end{proof}

\section{Derivatives of \texorpdfstring{$F$}{F} and proof of \reflemma{F=0ISMLE}}\lbappendix{FandMLEproof}

In this section we verify that $F(s_0^T,\theta_0;y_0,0)=0$ and $\grad_{s^T,\theta} F(s_0^T,\theta_0;y_0,0)$ is non-singular, as asserted in the proof of \refthm{MLEerror}, and we prove \reflemma{F=0ISMLE}.

Since $\epsilon=0$, the assertions $F_1(s_0^T,\theta_0;y_0,0)=0$ and $F_2(s_0^T,\theta_0;y_0,0)=0$ reduce to \eqref{y0K0's0} and \eqref{RateFunctionImplicitCriticalPoint}, respectively.
The gradient of $F$ with respect to its first two variables (which we combine into a single column vector by concatenating $s^T$ and $\theta$) has the block decomposition
\begin{align}
\grad_{s^T,\theta} F(s_0^T,\theta_0;y_0,0) &= \mat{\grad_s^T F_1(s_0^T,\theta_0;y_0) & \grad_\theta F_1(s_0^T,\theta_0;y_0) \\ \grad_s^T F_2(s_0^T,\theta_0;y_0,0) & \grad_\theta F_2(s_0^T,\theta_0;y_0,0)}
\notag\\&
=\mat{K''_0(s_0;\theta_0) & \grad_s\grad_\theta K_0(s_0;\theta_0) \\ \left( \grad_s\grad_\theta K_0(s_0;\theta_0) \right)^T & \grad_\theta^T\grad_\theta K_0(s_0;\theta_0)}
,
\end{align}
with the diagonal blocks symmetric.
Abbreviate this as 
\begin{equation}\label{gradFBlockForm}
\grad_{s^T,\theta} F(s_0^T,\theta_0;y_0,0) = \mat{A & B\\B^T & D}.
\end{equation}
Then, with $I$ and $0$ denoting the identity matrix and zero matrix of the indicated sizes, 
\begin{equation}\label{gradFRowColOps}
\mat{I_{\xdim\times\xdim} & 0_{\xdim\times p} \\-B^T A^{-1} & I_{p\times p}} \mat{A & B\\B^T & D} \mat{I_{\xdim\times\xdim} & -A^{-1} B \\0_{p\times\xdim} & I_{p\times p}} =\mat{A & 0_{\xdim\times p}\\0_{p\times\xdim} & D-B^T A^{-1}B}.
\end{equation}
The matrix $A=K''_0(s_0;\theta_0)$ is positive definite, and we recognise $D-B^T A^{-1}B$ as the negative definite matrix $H$ from \eqref{RateFunctionImplicitHessianDefinite}.
In particular, both are non-singular, and hence $\grad_{s^T,\theta} F(s_0^T,\theta_0;y_0,0)$ is non-singular also, as claimed.

\begin{proof}[Proof of \reflemma{F=0ISMLE}]
Consider a solution $s=s_1,\theta=\theta_1$ of $F(s^T,\theta;y,1/n)=0$.
The condition $F_1(s^T,\theta;y)=0$ reduces to the saddlepoint equation \eqref{SESAR}, so that $s_1=\hat{s}_0(\theta,y)$.
Note that
\begin{equation}\label{R'xnFromF2}
\grad_\theta^T R_{x,n}(\theta) = F_2(\hat{s}_0^T(\theta,y),\theta;y,1/n)
\end{equation}
so the condition $F_2(s^T,\theta;y,1/n)=0$ shows that $\theta_1$ is a critical point of $R_{x,n}$.

To complete the proof it suffices to show that, possibly after shrinking $U,V$ and increasing $n_0$, the Hessian $\grad_\theta^T\grad_\theta R_{x,n}(\theta)$ is negative definite for all $\theta\in U,y\in V,n\geq n_0$.
From \eqref{R'xnFromF2} and \eqref{s0hatPGradients} we have
\begin{align}
\grad_\theta^T\grad_\theta R_{x,n}(\theta) &= \grad_\theta \left( F_2(\hat{s}_0^T(\theta,y),\theta;y,1/n) \right)
\notag\\
&= \grad_\theta F_2(\hat{s}_0^T(\theta,y),\theta;y,1/n) 
\notag\\&\quad
- \grad_s^T F_2(\hat{s}_0^T(\theta,y),\theta;y,1/n) K''_0(\hat{s}_0(\theta,y);\theta)^{-1} \grad_s\grad_\theta K_0(\hat{s}_0(\theta,y);\theta)
.
\label{RxnHessian}
\end{align}
We note that 
\begin{equation}\label{RxnHessianepsilon}
h(s,\theta,y,\epsilon) = \grad_\theta F_2(s^T,\theta;y,\epsilon) - \grad_s^T F_2(s^T,\theta;y,\epsilon) K''_0(s;\theta)^{-1} \grad_s\grad_\theta K_0(s;\theta)
\end{equation}
is a continuous function.
Moreover $h(s_0,\theta_0,y_0,0)$ reduces to the negative definite matrix $D-B^T A^{-1} B = H$ from \eqref{RateFunctionImplicitHessianDefinite} and \eqref{gradFRowColOps}.
By continuity, it follows that $h(s,\theta,y,\epsilon)$ has only negative eigenvalues for $s,\theta,y,\epsilon$ in a suitable neighbourhood of $s_0,\theta_0,y_0,0$.
From \eqref{RxnHessian} we have $\grad_\theta^T\grad_\theta R_{x,n}(\theta)=h(\hat{s}_0(\theta,y),\theta,y,1/n)$ which is therefore negative definite for all $\theta,y$ in suitable neighbourhoods and all $n$ large enough, as required.
\end{proof}

The proof of \refthm{LowerOrderSaddlepoint}\ref{item:LOMLE} uses also the fact that $\theta=\hat{\theta}^*_{\MLE\,\mathrm{in}\,U}(x,n)$ and $s=\hat{s}(\theta,x)$ are the solutions of $F(s^T,\theta;y,0)=0$.
This follows from the same proof as \reflemma{F=0ISMLE}, with the relation $\grad_\theta^T \log\hat{L}^*_0(\theta;y)=\hat{F}_2(\hat{s}_0^T(\theta,y),\theta;y,0)$ in place of \eqref{R'xnFromF2}.

\section{Proof of \refthm{BayesianError}}\lbappendix{BayesianErrorProof}

\refthm{BayesianError} contains two separate assertions, about $\pi_{\Theta\,\vert\,U,x}$ and about $\hat{\pi}_{\Theta\,\vert\,U,x}$.
We prove a stronger statement that removes the restriction $y=y_0$ and also applies to $\hat{\pi}^*_{\Theta\,\vert\,U,x}$ from \refthm{LowerOrderSaddlepoint}\ref{item:LOBayesian}.
To state it, define 
\begin{equation}
\hat{\theta}^*_{\MLE\,\mathrm{in}\,U; \, 0}(y) = \argmax_{\theta\in U} \hat{L}^*_0(\theta;y)
\end{equation}
and note that 
\begin{equation}\label{hattheta*xny}
\hat{\theta}^*_{\MLE\,\mathrm{in}\,U}(x,n) = \hat{\theta}^*_{\MLE\,\mathrm{in}\,U; \, 0}(y)
\end{equation}
depends only on $y$ and not $n$; see \eqref{L*L*0}.

\begin{prop}\lbprop{BayesianErrorGeneral}
Under the assumptions of \refthm{BayesianError}, there exist neighbourhoods $U\subset\thetadomain$ of $\theta_0$ and $V\subset \R^{\xdim\times 1}$ of $y_0$ such that, for all $y\in V$,
\begin{align}
&\text{under $\pi_{\Theta\,\vert\,U,x}$, $\hat{\pi}_{\Theta\,\vert\,U,x}$ or $\hat{\pi}^*_{\Theta\,\vert\,U,x}$,}
\notag\\&\qquad
\left( \sqrt{n}\left( \big. \Theta-\theta_{\MLE\,\mathrm{in}\,U}(x,n) \right), \sqrt{n}\bigl( \Theta-\hat{\theta}_{\MLE\,\mathrm{in}\,U}(x,n) \bigr), \sqrt{n}\bigl( \Theta-\hat{\theta}^*_{\MLE\,\mathrm{in}\,U}(x,n) \bigr) \right)
\notag\\&\qquad\qquad
\overset{d}{\to} (Z, Z, Z) \quad\text{as $n\to\infty$ with $y$ fixed}
,
\end{align}
where 
\begin{equation}
Z \sim \mathcal{N}( 0,-H(y)), \qquad H(y) = \grad_\theta^T\grad_\theta\log\hat{L}^*_0(\hat{\theta}^*_{\MLE\,\mathrm{in}\,U; \, 0}(y);y).
\end{equation}
\end{prop}

Note that $H(y)$ generalises the Hessian $H$ from \eqref{RateFunctionImplicitHessianDefinite}, which corresponds to $H(y_0)$.

\begin{proof}
Start with $U$ and $V$ satisfying the conclusions of Theorems~\ref{T:MLEerror} and \ref{T:LowerOrderSaddlepoint}\ref{item:LOMLE}.

Let $f_\Theta(\theta)$ denote the prior density for $\Theta$.
Since $f_\Theta(\theta_0)>0$ and $f_\Theta$ is continuous, we can shrink $U$ if necessary so that $\log f_\Theta(\theta)$ is bounded and continuous over $\theta\in U$.

Choose $\tilde{\pi}_n$ to be one of the posteriors $\pi_{\Theta\,\vert\,U,x}$, $\hat{\pi}_{\Theta\,\vert\,U,x}$ or $\hat{\pi}^*_{\Theta\,\vert\,U,x}$, let $\tilde{L}_n(\theta;x)$ be the corresponding choice out of $L_n(\theta;x), \hat{L}_n(\theta;x)$ or $\hat{L}^*_n(\theta;x)$, respectively, and let $\tilde{\theta}(x,n)$ denote the corresponding choice out of $\theta_{\MLE\,\mathrm{in}\,U}(x,n)$, $\hat{\theta}_{\MLE\,\mathrm{in}\,U}(x,n)$ or $\hat{\theta}^*_{\MLE\,\mathrm{in}\,U}(x,n)$, respectively.
Let $Q_n(\theta;y)$ be the corresponding choice out of the functions $P_n(\hat{s}_0(\theta,y),\theta)$, $\hat{P}_n(\hat{s}_0(\theta,y),\theta)$ or the constant function 1, respectively, so that
\begin{equation}\label{tildeLL*Q}
\log\tilde{L}_n(\theta;x) = n\log\hat{L}^*_0(\theta;y) + \log Q_n(\theta;y).
\end{equation}

Note from \refthm{MLEerror}, \refthm{LowerOrderSaddlepoint}\ref{item:LOMLE} and \eqref{hattheta*xny} that
\begin{equation}\label{tildethetaLimit}
\lim_{n\to\infty} \tilde{\theta}(x,n) = \hat{\theta}^*_{\MLE\,\mathrm{in}\,U; \, 0}(y)
\end{equation}
for fixed $y\in V$ and all three choices for $\tilde{\theta}$.

The Hessian matrix $H(y_0)=\grad_\theta^T\grad_\theta\log\hat{L}_0^*(\theta_0;y_0)$ is negative definite by \eqref{RateFunctionImplicitHessianDefinite}, so continuity allows us to shrink $U$ if necessary and choose $\delta>0$ small enough to ensure that 
\begin{equation}\label{L0*HessianBound}
v^T \grad_\theta^T\grad_\theta\log\hat{L}_0^*(\theta;y) v \leq -\delta \abs{v}^2
\end{equation}
for all $\theta\in U$, $y\in V$ and $v\in\R^{p\times 1}$.

Shrinking $U$ and $V$ further if necessary, we may assume that $U$ is a bounded convex neighbourhood and that the Hessians $\grad_\theta^T\grad_\theta\log Q_n$ and $\grad_\theta^T\grad_\theta\log\hat{L}^*_0$ are bounded and uniformly continuous over $\theta\in U$ and $n$ sufficiently large.
This is evident if $Q_n(\theta;y)$ is the constant 1 or the function $\hat{P}_n(\hat{s}_0(\theta,y),\theta)$ since both have gradients that do not depend on $n$ and are continuously differentiable functions of $\theta$.
For $P_n(\hat{s}_0(\theta,y),\theta)$ the statement follows from \refcoro{gradlogPhatsError}.

Let $\tilde{f}_{Z_n}(z)$ denote the posterior density function for $Z_n=\sqrt{n}(\Theta-\tilde{\theta}(x,n))$ under $\tilde{\pi}_n$.
By construction, $\tilde{f}_{Z_n}$ is chosen to have the form
\begin{equation}
c_n \tilde{L}_n\bigl( \tilde{\theta}(x,n)+z/\sqrt{n};x \bigr) f_\Theta\bigl( \tilde{\theta}(x,n)+z/\sqrt{n} \bigr) \indicator{z\in\mathrm{supp}Z_n},
\end{equation}
provided it is possible to choose $c_n=c_n(y,n)$ to make $\int \tilde{f}_{Z_n}(z)dz =1$.

Apply Lagrange's form of the Taylor series remainder term, 
\begin{equation}\label{LagrangeFormg}
g(1) = g(0) + g'(0) + \int_0^1 (1-u) g''(u)du,
\end{equation}
with $g(u) = \log\tilde{L}_n(\tilde{\theta}(x,n)+uz/\sqrt{n};x)$.
By construction, $\tilde{\theta}(x,n)$ is a critical point of $\theta\mapsto\tilde{L}_n(\theta;x)$, so that $g'(0)=0$, while 
\begin{equation}\label{g''zTHz}
g''(u) = n^{-1} z^T \grad_\theta^T\grad_\theta\log\tilde{L}_n\left( \tilde{\theta}(x,n)+\frac{z}{\sqrt{n}};x \right) z.
\end{equation}
Combining \eqref{tildeLL*Q} and \eqref{LagrangeFormg}--\eqref{g''zTHz}, we compute
\begin{align}
&\log\tilde{L}_n\left( \tilde{\theta}(x,n)+\frac{z}{\sqrt{n}};x \right)
\notag\\&\quad
= \log\tilde{L}_n\bigl( \tilde{\theta}(x,n);x \bigr) 
+ \int_0^1 (1-u) z^T \grad_\theta^T\grad_\theta\log \hat{L}^*_0\left( \tilde{\theta}(x,n)+\frac{uz}{\sqrt{n}};y \right) z \, du
\notag\\&\qquad
+ \frac{1}{n} \int_0^1 (1-u) z^T \grad_\theta^T\grad_\theta\log Q_n\left( \tilde{\theta}(x,n)+\frac{uz}{\sqrt{n}};y \right) z \, du
.
\label{logtildeLnFormula}
\end{align}
In the right-hand side of \eqref{logtildeLnFormula}, the first term is constant with respect to $z$ and does not affect $\tilde{f}_{Z_n}$.
By \eqref{tildethetaLimit} and uniform continuity, the second and third term converge to $\frac{1}{2}z^T H(y) z$ and $0$, respectively.
Since $\tilde{\theta}(x,n)$ converges to an interior point of $U$, it follows that $\indicator{z\in\mathrm{supp}Z_n}$ and $f_\Theta\bigl( \tilde{\theta}(x,n)+z/\sqrt{n} \bigr)$ converge pointwise to 1 and $f_\Theta(\hat{\theta}^*_{\MLE\,\mathrm{in}\,U; \, 0}(y))$, respectively.
Thus
\begin{equation}
\tilde{L}_n\bigl( \tilde{\theta}(x,n)+z/\sqrt{n};x \bigr) f_\Theta\bigl( \tilde{\theta}(x,n)+z/\sqrt{n} \bigr) \indicator{z\in\mathrm{supp}Z_n} / \tilde{L}_n\bigl( \tilde{\theta}(x,n);x \bigr)
\end{equation} 
converges pointwise to $c(y)\exp\left( \frac{1}{2}z^T H(y) z \right)$, while the bound \eqref{L0*HessianBound} and the boundedness of $f_\Theta$ and $\grad_\theta^T\grad_\theta\log Q_n$ show that it is bounded by $C\exp\left( (-\frac{1}{2}\delta+C/n)\abs{z}^2 \right)$.
By the Dominated Convergence Theorem, it follows that $\tilde{f}_{Z_n}(z)$ converges to the limiting density $c'\exp\left( \frac{1}{2} z^T H(y) z \right)$ and $Z_n$ converges under $\tilde{\pi}_n$ to the corresponding distribution, which is $\mathcal{N}(0,-H^{-1})$ as claimed.

Finally the joint convergence follows because, by Theorems~\ref{T:MLEerror} and \ref{T:LowerOrderSaddlepoint}\ref{item:LOMLE}, the differences between the three MLEs are $O(1/n)$ or $O(1/n^2)$ and are negligible compared to the $\sqrt{n}$ scaling.
\end{proof}

\begin{proof}[Proof of Theorems~\ref{T:BayesianError} and \ref{T:LowerOrderSaddlepoint}\ref{item:LOBayesian}]
These are the special case $y=y_0$ from \refprop{BayesianErrorGeneral}.
We note that both $\hat{\theta}^*_{\MLE\,\mathrm{in}\,U}(x,n)$ and $\hat{\theta}^*_{\MLE\,\mathrm{in}\,U; \, 0}(y)$ reduce to $\theta_0$ when $y=y_0$, so that $\grad_\theta^T\grad_\theta\log\hat{L}^*_0(\hat{\theta}^*_{\MLE\,\mathrm{in}\,U; \, 0}(y_0);y_0)$ reduces to the matrix $H$ from \eqref{RateFunctionImplicitHessianDefinite}.
\end{proof}

\section{Proof of Theorems~\ref{T:SamplingMLE}--\ref{T:SamplingMLEWellSpecified}}\lbappendix{SamplingMLEProof}

\refthm{SamplingMLE} follows by an application of the delta method to the function $G$ from \refsubsect{MLEerrorProof}.
Note that in the setup of \refthm{SamplingMLE}, $x$ and $y$ from \eqref{SAR} are replaced by $\chi_n$ and $\tfrac{1}{n}\chi_n$ respectively.
Write 
\begin{equation}
\bar{\zeta}_n = \frac{1}{n}\chi_n
\end{equation}
so that \eqref{chinCLT} is the assertion
\begin{equation}\label{zetanCLT}
\sqrt{n}\left( \bar{\zeta}_n - y_0 \right) \overset{d}{\longrightarrow} \mathcal{N}(0,\Sigma).
\end{equation}

\begin{proof}[Proof of Theorems~\ref{T:SamplingMLE} and \ref{T:LowerOrderSaddlepoint}\ref{item:LOSampling}]
The assumptions give a neighbourhood $V$ of $y_0$ and $n_0\in\N$ such that $\hat{\theta}_{\MLE\,\mathrm{in}\,U}(x,n)$ and $\theta_{\MLE\,\mathrm{in}\,U}(x,n)$ exist for all $y\in V$ and $n\geq n_0$.
Shrinking $U,V$ and increasing $n_0$ if necessary, we can assume that the constructions involving $G$ from the proof of \refthm{MLEerror} in \refsubsect{MLEerrorProof}, including \reflemma{F=0ISMLE}, apply for $y\in V,n\geq n_0$.
(Shrinking $U$ will not affect the values $\hat{\theta}_{\MLE\,\mathrm{in}\,U}(x,n)$ or $\theta_{\MLE\,\mathrm{in}\,U}(x,n)$ provided they still lie within $U$.)
For $n\geq n_0$, define the event
\begin{equation}
E_n = \set{\bar{\zeta}_n\in V} = \set{\frac{1}{n}\chi_n\in V}.
\end{equation}
The assumption \eqref{zetanCLT} implies that $\bar{\zeta}_n$ converges in probability to $y_0$, so $\P(E_n)\to 1$ as $n\to\infty$.
On the event $E_n$, \reflemma{F=0ISMLE} gives 
\begin{align}
\mat{\hat{s}^T\bigl( \theta_{\MLE\,\mathrm{in}\,U}(\chi_n,n), \chi_n \bigr) \\ \theta_{\MLE\,\mathrm{in}\,U}(\chi_n,n)} = G(\bar{\zeta}_n,1/n)
. 
\end{align}
Since $G$ is continuously differentiable,
\begin{align}
&\sqrt{n} \mat{\hat{s}^T\bigl( \theta_{\MLE\,\mathrm{in}\,U}(\chi_n,n), \chi_n \bigr) - s_0^T \\ \theta_{\MLE\,\mathrm{in}\,U}(\chi_n,n) - \theta_0} 
\notag\\&\quad
= \sqrt{n}\left( G(\bar{\zeta}_n,1/n) - G(y_0,0) \right)
\notag\\&\quad
= \sqrt{n} \left( \grad_{y,\epsilon} G(y_0,0) \mat{\bar{\zeta}_n - y_0 \\ 1/n} + o_\P\left( \big. \abs{\bar{\zeta}_n-y_0} + 1/n \right) \right)
\notag\\&\quad
= \grad_y G(y_0,0) \left[ \big. \sqrt{n} \left( \bar{\zeta}_n - y_0 \right) \right] + O(1/\sqrt{n}) + o_\P(1)
\end{align}
where $o_\P(B_n)$ denotes a random variable $A_n$ such that $A_n/B_n$ converges to 0 in probability as $n\to\infty$.
Restricting our attention to the second sub-block of this vector,
\begin{align}
&\sqrt{n} \left( \theta_{\MLE\,\mathrm{in}\,U}(\chi_n,n) - \theta_0 \right) 
= \mat{0_{p\times\xdim} & I_{p\times p}} \grad_y G(y_0,0) \left[ \big. \sqrt{n} \left( \bar{\zeta}_n - y_0 \right) \right] + o_\P(1)
.
\end{align}
In particular, $\sqrt{n}\left( \theta_{\MLE\,\mathrm{in}\,U}(\chi_n,n) - \theta_0 \right)$ will converge to a limiting distribution $Z$ that is asymptotically normal with mean 0.

To calculate the covariance matrix of $Z$, note that
\begin{equation}
\begin{aligned}
\grad_y G(y_0,0) &= \grad_{y,\epsilon} G(y_0,0) \mat{I_{\xdim\times\xdim}\\ 0_{1\times\xdim}}
\\
&= \left[ \Big. - (\grad_{s^T,\theta} F(s_0^T,\theta_0;y_0,0))^{-1} \grad_{y,\epsilon}F(s_0^T,\theta_0;y_0,0) \right] \mat{I_{\xdim\times\xdim}\\ 0_{1\times\xdim}}
.
\end{aligned}
\end{equation}
We compute
\begin{equation}
\grad_{y,\epsilon} F(s_0^T,\theta_0;y_0,0) \mat{I_{\xdim\times\xdim}\\ 0_{1\times\xdim}} = - \mat{I_{\xdim\times\xdim}\\ 0_{p\times\xdim}}
\end{equation}
and recalling \eqref{gradFBlockForm}--\eqref{gradFRowColOps} in which $D-B^T A^{-1} B=H$, 
\begin{align}
\mat{0 & I} \grad_y G(y_0,0) &= - \mat{0 & I} \left( \grad_{s^T,\theta} F \right)^{-1} \left[ \left( \grad_{y,\epsilon} F \right) \mat{I\\ 0} \right]
\notag\\&
= - \mat{0 & I} \left[ \mat{I & -A^{-1}B\\ 0 & I} \mat{A^{-1} & 0 \\ 0 & H^{-1}} \mat{I & 0\\ -B^T A^{-1} & I} \right] \left[ - \mat{I\\ 0} \right]
\notag\\&
= \mat{0 & I} \mat{A^{-1} & 0 \\ 0 & H^{-1}} \mat{I \\ -B^T A^{-1}} = - H^{-1} B^T A^{-1}
.
\label{gradGpartial}
\end{align}
From \eqref{zetanCLT} and \eqref{gradGpartial}, the covariance matrix of $Z$ is $H^{-1} B^T A^{-1} \Sigma \left( H^{-1} B^T A^{-1} \right)^T$.  
Since $H$ and $A$ are symmetric, this reduces to the expression in part~\ref{item:SamplingGeneral}.

To handle the joint convergence in Theorems~\ref{T:SamplingMLE}\ref{item:SamplingGeneral} and \ref{T:LowerOrderSaddlepoint}\ref{item:LOSampling}, we again observe that the differences between the three MLEs are negligible in the $\sqrt{n}$ scaling of \refthm{SamplingMLE}.

Finally \refthm{SamplingMLE}\ref{item:SamplingWellSpecified} follows from the observation, discussed in \refsubsubsect{WellSpecified}, that $s_0=0$ gives $D=0$ and $H=-B^T A^{-1} B$.
If in addition $\Sigma=A$ then $H^{-1} B^T A^{-1} \Sigma A^{-1} B H^{-1}$ simplifies to $-H^{-1}$, as claimed.
\end{proof}

\begin{proof}[Proof of \refthm{SamplingMLEWellSpecified}]
Part~\ref{item:HReduces} follows from the observations that $\grad_\theta^T \grad_\theta K_0(0;\theta_0)$ vanishes and $K_0''(0;\theta_0)$ is positive definite.

For part~\ref{item:ConsistentAN}, note by the Central Limit Theorem that $\chi_n=X_{\theta_0}$ satisfies \eqref{chinCLT} with $y_0=K_0'(0;\theta_0)$, $\Sigma=K_0''(0;\theta_0)$, so the consistency and asymptotic normality follow from \refthm{SamplingMLE}\ref{item:SamplingWellSpecified}.
Setting $s_0=0$, we can apply \refthm{MLEerror}, with $H$ negative definite by part~\ref{item:HReduces}.
Let $V$ be the neighbourhood of $y_0$ given in \refthm{MLEerror}.
By the Law of Large Numbers, $X_{\theta_0}/n\overset{a.s}{\to} y_0$ as $n\to\infty$, and in particular $\P(X_{\theta_0}/n \in V)\to 1$, as in the discussion following \refthm{MLEerror}.
Whenever this event occurs, we can use the bound from \refthm{MLEerror}, and we conclude that $\abs{\hat{\theta}_{\MLE\,\mathrm{in}\,U}(X_{\theta_0}) - \theta_{\MLE\,\mathrm{in}\,U}(X_{\theta_0})} = O_\P(1/n^2)$.

Part~\ref{item:HEstimator} follows by continuity.
Specifically, the Hessians mentioned in part~\ref{item:HEstimator} coincide with the continuous function $h(s,\theta,y,\epsilon)$ from the proof of \reflemma{F=0ISMLE}, evaluated at $\theta=\Theta$, $y=X_{\theta_0}/n$, $s=\hat{s}_0(\Theta,X_{\theta_0}/n)$, and $\epsilon=1/n$; the analogue of the function $h$ with $F_2$ replaced by $\hat{F}_2$; and the function $h$ evaluated instead at $\epsilon=0$, respectively.
In all cases the arguments converge as $n\to\infty$, and the limiting value $h(s_0,\theta_0,y_0,0)$ (with $s_0=0$) reduces to $H$.
\end{proof}

\section{Proofs of Theorems~\ref{T:MLEerrorPartiallyIdentifiable}--\ref{T:MLEerrorNormalApprox}}\lbappendix{MLEePIProof}

The proofs of Theorems~\ref{T:MLEerrorPartiallyIdentifiable}--\ref{T:MLEerrorNormalApprox} use a different, and more complicated, scaling compared to the proof of \refthm{MLEerror}.
Extending \eqref{PIChangeOfVariables}, we set
\begin{equation}\label{PIChangeOfVariablesProof}
x = \delta^{-2} y, \quad 
y = K_0'\left( 0;\smallmat{\omega'\\ \nu_0} \right) + \delta \xi, \quad
\theta = \mat{\omega \\ \nu} = \mat{\omega' + \delta w \\ \nu}
,
\end{equation}
where $\delta$ is a scaling parameter.
We will be interested in the case $\delta=1/\sqrt{n}$, but $\delta$ can take any sufficiently small value.

In \eqref{PIChangeOfVariables}, the scaling variable $\xi$ expresses the lower-order ``fluctuations'' of the observed value, centred around a mean parametrised by $\omega'$.
(For instance, if we set $x=X_\theta$ with $\theta=\smallmat{\omega'\\ \nu_0}$, the corresponding $\xi$ values will be asymptotically normal as $n\to\infty$, with no further rescaling needed.)
Correspondingly, in $\omega=\omega'+\delta w$ from \eqref{PIChangeOfVariablesProof}, the correction term $\delta w$ lets us encode the lower-order effect of $\xi$ on inference for the parameter $\omega$.

Define
\begin{equation}
A = K_0''(0;\theta_0), \qquad J = A^{-1} - A^{-1} B_\omega (B_\omega^T A^{-1} B_\omega)^{-1} B_\omega^T A^{-1},
\end{equation}
with \eqref{K''PosDef} and the assumption $\rank(B_\omega)=p_1$ justifying the existence of the matrix $J$, and
\begin{equation}\label{w0Formula}
w_0 = (B_\omega^T A^{-1} B_\omega)^{-1} B_\omega^T A^{-1} \xi_0.
\end{equation}
Then in the condition from Theorems~\ref{T:MLEerrorPartiallyIdentifiable}--\ref{T:MLEerrorNormalApprox} for $(w_0,\nu_0)$ to be a non-degenerate local MLE for observing $\Xi_{w,\nu}=\xi_0$, the value $w_0$ must be the one given by \eqref{w0Formula}.
Using this fixed value, we can express the MLE condition from Theorems~\ref{T:MLEerrorPartiallyIdentifiable}--\ref{T:MLEerrorNormalApprox} explicitly as the requirement that
\begin{equation}\label{xi0Constraint}
\xi_0^T J \frac{\partial K_0''}{\partial\nu_j}(0;\theta_0) J \xi_0 = \trace\left( A^{-1} \frac{\partial K_0''}{\partial\nu_j}(0;\theta_0) \right)
\quad\text{for }j=1,\dotsc,p_2
,
\end{equation}
and that the $p_2\times p_2$ symmetric matrix $E$ with $i,j$ entry
\begin{multline}\label{EijFormula}
E_{ij} = \xi_0^T J\left( \frac{1}{2}\frac{\partial^2 K_0''}{\partial\nu_i\partial\nu_j}(0;\theta_0) - \frac{\partial K_0''}{\partial\nu_i}(0;\theta_0) J \frac{\partial K_0''}{\partial\nu_j}(0;\theta_0) \right) J \xi_0 
\\
- \frac{1}{2} \trace\left( A^{-1} \frac{\partial^2 K_0''}{\partial\nu_i\partial\nu_j}(0;\theta_0) - A^{-2} \frac{\partial K_0''}{\partial\nu_i}(0;\theta_0)\frac{\partial K_0''}{\partial\nu_j}(0;\theta_0) \right)
\end{multline}
should be negative definite.

In the proof of \refthm{MLEerrorPartiallyIdentifiable}, \eqref{xi0Constraint}--\eqref{EijFormula} will arise naturally as first- and second-derivative conditions analogous to \eqref{RateFunctionImplicitCriticalPoint}--\eqref{RateFunctionImplicitHessianDefinite}.
Unlike in \refthm{MLEerror}, these conditions involve the lower-order terms in the saddlepoint approximation, and the derivatives of the log-determinants in $\log\hat{P}_n(\theta,s)$ lead to the traces in \eqref{xi0Constraint}--\eqref{EijFormula}.
The equivalence of \eqref{xi0Constraint}--\eqref{EijFormula} to the assertion about $\Xi$ can be verified by direct calculation or by tracing through the argument that follows with $X$ replaced by $\Xi$ and $n=1, \delta=1$.

\begin{proof}[Proof of \refthm{MLEerrorPartiallyIdentifiable}]
Throughout the proof we assume that \eqref{xi0Constraint}--\eqref{EijFormula} hold and the change of variables \eqref{PIChangeOfVariablesProof} applies.

As in the proof of \refthm{MLEerror}, we will produce a function $F$ such that a solution of $F=0$ corresponds to a critical point of $\theta\mapsto\log L_n(\theta;x)$.
In the proof of \refthm{MLEerror}, we could obtain $F$ by applying a single rescaling to the values of the log-likelihood function, because $\tfrac{1}{n}\log L_n$ converged as $n\to\infty$ to a well-behaved function with non-singular Hessian.
In this proof, by contrast, we will obtain $F$ from the asymmetrically rescaled gradients $\sqrt{n}\grad_\omega\log L_n$ and $\grad_\nu\log L_n$.
Another interpretation of this scaling is that we consider the gradients of $\log L_n$ with respect to $w$ and $\nu$, rather than $\omega$ and $\nu$, with no scaling factor applied to the values of $\log L_n$.

We note in advance that, because of the scaling of $\omega$ in \eqref{PIChangeOfVariablesProof}, the Implicit Function Theorem will tell us that $F=0$ has a solution that is unique when $w$ is constrained to lie in some fixed neighbourhood of $w_0$.
In terms of $\omega$, this amounts to uniqueness within a neighbourhood of size $\delta$ only, so later in the proof we will show that the claimed MLE is actually a maximiser within an $\omega$-neighbourhood of non-vanishing size.

Make the further change of variables
\begin{equation}\label{sigma0Formula}
s = \delta\sigma \quad\text{and define}\quad \sigma_0 = \xi_0^T J .
\end{equation}
We will define the function $F(\sigma^T,w,\nu;\omega',\xi,\delta)$ so that $F(\sigma_0^T,w_0,\nu_0;\omega_0,\xi_0,0)=0$ and, for $\delta=1/\sqrt{n}$, we have $F=0$ exactly when $\theta=\smallmat{\omega' + \delta w\\ \nu}$ is the MLE for the observed value $x=\delta^{-2} K_0'(0;\smallmat{\omega'\\ \nu_0}) + \delta^{-1}\xi$.

For convenience, write $\theta'=\smallmat{\omega'\\ \nu_0}$.
For $\delta\neq 0$, define the functions
\begin{equation}\label{F1FomegaFnuFFormula}
\begin{aligned}
&F_1(\sigma^T,w,\nu;\omega',\xi,\delta) = \delta^{-1}\left( K_0'(\delta\sigma; \theta)-K_0'( 0;\theta') \right)-\xi
,
\\
&F_\omega(\sigma^T,w,\nu;\omega',\xi,\delta) 
\\&\quad
= \delta^{-1}\grad_\omega^T K_0\left( \delta\sigma; \theta \right) 
+ \delta q_{3,\omega}\left( \theta,K_0'(0;\theta')+ \delta\xi,\delta^2 \right)^T
\\&\qquad
+ \delta\grad_\omega^T\log\hat{P}\left( \delta\sigma, \theta \right) 
\\&\qquad
- \delta \left( \grad_s\grad_\omega K_0\left( \delta\sigma;\theta \right) \right)^T K_0''\left( \delta\sigma;\theta \right)^{-1}\grad_s\log\hat{P}\left( \delta\sigma,\theta \right) 
,
\\
& F_\nu(\sigma^T,w,\nu;\omega',\xi,\delta)
\\&\quad
= \delta^{-2}\grad_\nu^T K_0(\delta\sigma;\theta) 
+ q_{3,\nu}\left( \theta,K_0'(0;\theta') + \delta\xi,\delta^2 \right)^T
\\&\qquad
+ \grad_\nu^T\log\hat{P}(\delta\sigma,\theta) 
\\&\qquad
- (\grad_s\grad_\nu K_0(\delta\sigma;\theta))^T K_0''(\delta\sigma;\theta)^{-1}\grad_s\log\hat{P}(\delta\sigma,\theta) 
,
\\
& F(\sigma^T,w,\nu;\omega',\xi,\delta) = \mat{F_1(\sigma^T,w,\nu;\omega',\xi,\delta) \\ F_\omega(\sigma^T,w,\nu;\omega',\xi,\delta) \\ F_\nu(\sigma^T,w,\nu;\omega',\xi,\delta)}
,
\end{aligned}
\end{equation}
where the row vector $q_3$ from \refcoro{gradlogPhatsError} is written in block form as $q_3 = \mat{q_{3,\omega} & q_{3,\nu}}$, and where we have used the relation $\theta=\smallmat{\omega'+\delta w\\ \nu}$ from \eqref{PIChangeOfVariablesProof} throughout.

Since $K_0'(0;\theta)$ does not depend on $\nu$ by assumption, we have $K_0'(0;\theta')=K_0'\left( 0;\smallmat{\omega'\\ \nu_0}\right )=K_0'\left( 0;\smallmat{\omega'\\ \nu}\right )$ and we can write
\begin{align}
F_1(\sigma^T,w,\nu;\omega',\xi,\delta) 
&
= \delta^{-1}\left( K_0'(\delta\sigma; \theta)-K_0'(0;\theta) + K_0'(0;\smallmat{\omega'+\delta w\\ \nu})-K_0'( 0;\smallmat{\omega'\\ \nu}) \right)-\xi
\notag\\&
= \int_0^1 \left( K_0''(u\delta\sigma;\theta) \sigma^T + \grad_s\grad_\omega K_0(0;\smallmat{\omega'+u\delta w\\ \nu})w \right) du -\xi 
.
\end{align}
Applying \reflemma{DiffUnderInt} to this representation and using \refcoro{gradlogPhatsError}, we see that $F_1$ can be extended to be a continuously differentiable function of all its variables in some neighbourhood around $\delta=0$.
Recalling that $\grad_\omega K_0(0;\theta)=0$ for all $\theta$, we can similarly rewrite the first term in $F_\omega$ as
\begin{align}
\delta^{-1}\grad_\omega^T K_0(\delta\sigma;\theta) &= \delta^{-1}\left( \grad_\omega^T K_0(\delta\sigma;\theta) - \grad_\omega^T K_0(0;\theta) \right)
\notag\\&
= \int_0^1 \left( \grad_s\grad_\omega K_0(u\delta\sigma;\theta) \right)^T \sigma^T du 
.
\end{align}
For $F_\nu$, $\grad_\nu K_0(0;\theta)$ vanishes, and now also $\grad_s\grad_\nu K_0(0;\theta) = \grad_\nu K_0'(0;\theta)$ vanishes, because by assumption $K_0'(0;\theta)$ does not depend on $\nu$.
So the first term in $F_\nu$ can be rewritten as
\begin{align}
\delta^{-2}\grad_\nu^T K_0(\delta\sigma;\theta) 
&
= \delta^{-2}\left( \grad_\nu^T K_0(\delta\sigma;\theta) - \grad_\nu^T K_0(0;\theta) - \delta \grad_s^T\grad_\nu^T K_0(0;\theta)\sigma^T \right)
\notag\\&
= \int_0^1 (1-u) \sum_{k,\ell=1}^m \sigma_k\sigma_\ell \grad_\nu^T \frac{\partial^2 K_0}{\partial s_k\partial s_\ell}(u\delta\sigma;\theta) du
.
\label{FnuTermRewritten}
\end{align}
In particular, $F_1,F_\omega,F_\nu,F$ can all be extended as continuously differentiable functions in a neighbourhood of $\delta=0$.

If we set $\delta=0$, we obtain
\begin{equation}\label{F1Fomegadelta0}
\begin{aligned}
F_1(\sigma^T,w,\nu;\omega',\xi,0) &= K_0''(0;\smallmat{\omega'\\ \nu}) \sigma^T + \grad_s\grad_\omega K_0(0;\smallmat{\omega'\\ \nu})w - \xi,
\\
F_\omega(\sigma^T,w,\nu;\omega',\xi,0) &= \left( \grad_s\grad_\omega K_0(0;\smallmat{\omega'\\ \nu}) \right)^T \sigma^T  
\end{aligned}
\end{equation}
and $F_\nu(\sigma^T,w,\nu;\omega',\xi,0)$ is the $p_2\times 1$ column vector whose $i^\text{th}$ entry is 
\begin{align}
e_i^T F_\nu(\sigma^T,w,\nu;\omega',\xi,0) &= \frac{1}{2} \sum_{k,\ell=1}^m \sigma_k\sigma_\ell \frac{\partial^3 K_0}{\partial\nu_i\partial s_k\partial s_\ell}(0;\smallmat{\omega'\\ \nu}) - \frac{1}{2}\frac{\partial}{\partial\nu_i} \log\det(K_0''(0;\smallmat{\omega'\\ \nu}))
\notag\\&
= \frac{1}{2} \sigma \frac{\partial K_0''}{\partial\nu_i}(0;\smallmat{\omega'\\ \nu}) \sigma^T - \frac{1}{2}\trace\left( K_0''(0;\smallmat{\omega'\\ \nu})^{-1} \frac{\partial K_0''}{\partial\nu_i}(0;\smallmat{\omega'\\ \nu}) \right)
\end{align}
by \eqref{loghatPDerivative}; here $e_i$ means the $i^\text{th}$ standard basis vector, interpreted as a column vector of the appropriate size.

Let $\tilde{A}$, $\tilde{B}_\omega$, $\tilde{J}$ denote the analogues of $A$, $B_\omega$, $J$ in which $K_0$ and its derivatives are evaluated at $(0;\smallmat{\omega' \\ \nu})$ rather than at $(0;\theta_0=\smallmat{\omega_0\\ \nu_0})$; thus $\tilde{A}$, $\tilde{B}_\omega$, $\tilde{J}$ reduce to $A$, $B_\omega$, $J$ when $\nu=\nu_0$ and $\omega'=\omega_0$.
By continuity, we may find a neighbourhood in which $\tilde{B}_\omega$ has rank $p_1$ and $\tilde{A}$ is non-singular, so that $\tilde{H}_\omega = -\tilde{B}_\omega^T \tilde{A}^{-1}\tilde{B}_\omega$ (the analogue of the Hessian $H$ from \eqref{RateFunctionImplicitHessianDefinite} when $s_0=0$, applied only to the $\omega$ variable; see \refsubsubsect{WellSpecified}) is negative definite.

With this notation, \eqref{F1Fomegadelta0} gives
\begin{equation}\label{F1Fomegadelta00iff}
\smallmat{F_1(\sigma^T,w,\nu;\omega',\xi,0) \\ F_\omega(\sigma^T,w,\nu;\omega',\xi,0)} = 0 \quad\text{if and only if}\quad \sigma^T= \tilde{J} \xi, \;\; w = (\tilde{B}_\omega^T \tilde{A}^{-1} \tilde{B}_\omega)^{-1} \tilde{B}_\omega^T \tilde{A}^{-1} \xi .
\end{equation}
If we further set $(\nu, \omega', \xi)=(\nu_0,\omega_0,\xi_0)$, then $\sigma, w$ reduce to the values $\sigma_0,w_0$ given by \eqref{sigma0Formula} and \eqref{w0Formula}, and with this choice of $\sigma$ the equation $F_\nu=0$ reduces to \eqref{xi0Constraint}.
Thus we have verified that 
\begin{equation}\label{F0PI}
F(\sigma_0^T, w_0, \nu_0; \omega_0, \xi_0, 0)=0.
\end{equation}

To differentiate $F$, first note that $\grad_s\grad_\omega K_0(0;\theta)$ does not depend on $\nu$ (since it is the gradient of $K_0'(0;\theta)$, which depends on $\omega$ only) and therefore $\grad_\nu F_\omega(\sigma^T,w,\nu;\omega',\xi,0)=0$.
Moreover neither $F_\omega(\sigma^T,w,\nu;\omega',\xi,0)$ nor $F_\nu(\sigma^T,w,\nu;\omega',\xi,0)$ depends on $w$.
Consequently, in block form, 
\begin{equation}\label{gradFBlockFormPI}
\grad_{\sigma^T,w,\nu} F(\sigma^T, w, \nu; \omega, \xi, 0) = 
\mat{\tilde{A} & \tilde{B}_\omega & \tilde{C} \\
\tilde{B}_\omega^T & 0 & 0 \\
\tilde{C}^T & 0 & \tilde{Q}}
\end{equation}
where $\tilde{C}$ is the $\xdim\times p_2$ matrix whose $j^\text{th}$ column is
\begin{equation}\label{CejFormula}
\tilde{C} e_j = \frac{\partial K_0''}{\partial\nu_j}(0;\smallmat{\omega' \\ \nu}) \sigma^T
\end{equation}
and $\tilde{Q}$ is the $p_2\times p_2$ matrix whose $i,j$ entry is
\begin{align}\label{QijFormula}
\tilde{Q}_{ij} &= \frac{1}{2} \sigma \frac{\partial^2 K_0''}{\partial\nu_i\partial\nu_j}(0;\smallmat{\omega' \\ \nu}) \sigma^T 
\notag\\&\quad
- \frac{1}{2} \trace\left( \tilde{A}^{-1} \frac{\partial^2 K_0}{\partial\nu_i\partial\nu_j}(0;\smallmat{\omega' \\ \nu}) - \tilde{A}^{-2} \frac{\partial K_0}{\partial\nu_i}(0;\smallmat{\omega' \\ \nu})\frac{\partial K_0}{\partial\nu_j}(0;\smallmat{\omega' \\ \nu}) \right) 
;
\end{align}
see for instance \cite[Exercise~8.4.1]{MagnusNeudeckerMatrixDifferentialCalc}.

To show the non-singularity of $\grad_{\sigma^T,w,\nu} F$, perform block-wise row and column operations:
\begin{multline}\label{ReducegradsigmawnuF1}
\mat{I & 0 & 0 \\
-\tilde{B}_\omega^T \tilde{A}^{-1} & I & 0 \\
-\tilde{C}^T \tilde{A}^{-1} & 0 & I}
\mat{\tilde{A} & \tilde{B}_\omega & \tilde{C} \\
\tilde{B}_\omega^T & 0 & 0 \\
\tilde{C}^T & 0 & \tilde{Q}}
\mat{I & -\tilde{A}^{-1} \tilde{B}_\omega & -\tilde{A}^{-1} \tilde{C} \\
0 & I & 0 \\
0 & 0 & I}
\\
=
\mat{\tilde{A} & 0 & 0\\
0 & \tilde{H}_\omega & -\tilde{B}_\omega^T \tilde{A}^{-1} \tilde{C}\\
0 & -\tilde{C}^T \tilde{A}^{-1} \tilde{B}_\omega & \; \tilde{Q} - \tilde{C}^T \tilde{A}^{-1} \tilde{C}}
\end{multline}
and
\begin{multline}\label{ReducegradsigmawnuF2}
\mat{I & 0 \\
\tilde{C}^T \tilde{A}^{-1} \tilde{B}_\omega \tilde{H}_\omega^{-1} \; & I}
\mat{\tilde{H}_\omega & -\tilde{B}_\omega^T \tilde{A}^{-1} \tilde{C}\\
-\tilde{C}^T \tilde{A}^{-1} \tilde{B}_\omega \; & \tilde{Q} - \tilde{C}^T \tilde{A}^{-1} \tilde{C}}
\mat{I & \; \tilde{H}_\omega^{-1} \tilde{B}_\omega^T \tilde{A}^{-1} \tilde{C} \\
0 & I}
\\
= 
\mat{\tilde{H}_\omega & 0\\
0 & \tilde{Q} - \tilde{C}^T \tilde{A}^{-1} \tilde{C} - \tilde{C}^T \tilde{A}^{-1} \tilde{B}_\omega \tilde{H}_\omega^{-1} \tilde{B}_\omega^T \tilde{A}^{-1} \tilde{C}
}
=
\mat{\tilde{H}_\omega & 0\\
0 & \tilde{Q} - \tilde{C}^T \tilde{J} \tilde{C}
}
.
\end{multline}
Combining \eqref{CejFormula}--\eqref{QijFormula}, we see that the $i,j$ entry of $\tilde{Q}-\tilde{C}^T \tilde{J} \tilde{C}$ is
\begin{align}
[\tilde{Q}-\tilde{C}^T \tilde{J} \tilde{C}]_{ij} &= \tilde{Q}_{ij} - e_i^T \tilde{C}^T \tilde{J} \tilde{C} e_j
\notag\\&
= \sigma\left( \frac{1}{2}\frac{\partial^2 K_0''}{\partial\nu_i\partial\nu_j}(0;\smallmat{\omega' \\ \nu}) - \frac{\partial K_0''}{\partial\nu_i}(0;\smallmat{\omega' \\ \nu}) \tilde{J} \frac{\partial K_0''}{\partial\nu_j}(0;\smallmat{\omega' \\ \nu}) \right) \sigma^T
\notag\\&\quad
- \frac{1}{2}\trace\left( \tilde{A}^{-1} \frac{\partial^2 K_0''}{\partial\nu_i\partial\nu_j}(0;\smallmat{\omega' \\ \nu}) - \tilde{A}^{-2} \frac{\partial K_0''}{\partial\nu_i}(0;\smallmat{\omega' \\ \nu})\frac{\partial K_0''}{\partial\nu_j}(0;\smallmat{\omega' \\ \nu}) \right)
.
\end{align}
If we now set $(\sigma^T,w,\nu;\omega',\xi)=(\sigma_0^T,w_0,\nu_0;\omega_0,\xi_0)$ and recall the relation $\sigma_0^T = J\xi_0$ from \eqref{sigma0Formula}, we see that this expression reduces to $E_{ij}$ from \eqref{EijFormula}, and thus \eqref{ReducegradsigmawnuF2} reduces to 
\begin{equation}\label{ReducegradsigmawnuF3}
\mat{I & 0 \\
C^T A^{-1} B_\omega H_\omega^{-1} \; &  I}
\mat{H_\omega & -B_\omega^T A^{-1} C\\
-C^T A^{-1} B_\omega \; & Q - C^T A^{-1} C}
\mat{I & \; H_\omega^{-1} B_\omega^T A^{-1} C \\
0 & I}
=
\mat{H_\omega & 0\\
0 & E
}
.
\end{equation}
Since $E$ is negative definite by assumption and therefore non-singular, it follows from \eqref{ReducegradsigmawnuF1} and \eqref{ReducegradsigmawnuF3} that $\grad_{\sigma^T,w,\nu}F(\sigma_0,w_0,\nu_0;\omega_0,\xi_0,0)$ is non-singular too.

By the Implicit Function Theorem, there exist $n_0\in\N$ and neighbourhoods $U_\omega$, $\tilde{U}_w$, $U_\nu$, $\tilde{V}$, $\tilde{W}$ of $\omega_0$, $w_0$, $\nu_0$, $\xi_0$, $\sigma_0$ and a continuously differentiable function $G(\omega',\xi,\delta)$ defined on $U_\omega\times\tilde{V}\times (-1/\sqrt{n_0},1/\sqrt{n_0})$ such that, for all $\omega'\in U_\omega$, $\xi\in\tilde{V}$, $\delta\in(-1/\sqrt{n_0},1/\sqrt{n_0})$, the point $\smallmat{\sigma^T\\ w \\ \nu} = G(\omega',\xi,\delta) = \smallmat{G_\sigma(\omega',\xi,\delta)^T \\ G_w(\omega',\xi,\delta) \\ G_\nu(\omega',\xi,\delta)}$ is the unique solution in $\tilde{W}\times \tilde{U}_w\times U_\nu$ of $F(\sigma^T, w, \nu; \omega',\xi,\delta)=0$.

To explain the relationship, analogous to \reflemma{F=0ISMLE}, between $G$ and the MLE problem, let $\hat{\sigma}(w,\nu,\omega',\xi,\delta)$ denote the solution to $F_1(\sigma^T,w,\nu;\omega',\xi,\delta)=0$.
The existence of such a solution, for $w,\nu,\omega',\xi,\delta$ in suitable neighbourhoods of $w_0,\nu_0,\omega_0,\xi_0,0$, is guaranteed by the Implicit Function Theorem since $\grad_\sigma^T F_1(\sigma_0,w_0,\nu_0;\omega_0,\xi_0,0)=A$, which is non-singular.
We remark that the solution, when it exists, is always unique: for $\delta=0$, we can explicitly solve \eqref{F1Fomegadelta0} because of \eqref{K''PosDef}, and for $\delta\neq 0$ the equation $F_1=0$ is a rescaling of the saddlepoint equation \eqref{SESAR}, so that
\begin{equation}
\hat{\sigma}(w,\nu,\omega',\xi,\delta) = \delta^{-1}\hat{s}_0(\smallmat{\omega'+\delta w\\ \nu}, K_0'(0;\smallmat{\omega'\\ \nu_0}) + \delta\xi) \quad\text{for }\delta\neq 0
.
\end{equation}
Moreover from \eqref{F0PI} we obtain $\hat{\sigma}(w_0,\nu_0,\omega_0,\xi_0,0)=\sigma_0$.
Then the analogue of \eqref{R'xnFromF2} is
\begin{multline}\label{gradwnulogLnFromFomegaFnu}
\grad_{w,\nu}^T\left( \log L_n(\smallmat{\omega'+ w/\sqrt{n}\\ \nu}; nK_0'(0;\smallmat{\omega'\\ \nu_0})+\sqrt{n}\xi) \right) 
\\
= 
\mat{F_\omega(\hat{\sigma}(w,\nu,\omega',\xi,1/\sqrt{n})^T, w, \nu; \omega', \xi, 1/\sqrt{n}) \\ 
F_\nu(\hat{\sigma}(w,\nu,\omega',\xi,1/\sqrt{n})^T, w, \nu; \omega', \xi, 1/\sqrt{n})}
.
\end{multline}

Similarly to \eqref{RxnHessianepsilon}, define
\begin{equation}\label{hFomegaFnuFormula}
h(w,\nu,\omega',\xi,\delta) = 
\mat{\grad_w F_\omega + \grad_\sigma^T F_\omega \grad_w\hat{\sigma}^T & \;\; \grad_\nu F_\omega + \grad_\sigma^T F_\omega \grad_\nu\hat{\sigma}^T \\
\grad_w F_\nu + \grad_\sigma^T F_\nu \grad_w\hat{\sigma}^T & \;\; \grad_\nu F_\nu + \grad_\sigma^T F_\nu \grad_\nu\hat{\sigma}^T}
\end{equation}
where the gradients of $F_\omega$ and $F_\nu$ are evaluated at $(\hat{\sigma}(w,\nu,\omega',\xi,\delta)^T, w, \nu; \omega', \xi, \delta)$ and the gradients of $\hat{\sigma}$ are evaluated at $(w,\nu,\omega',\xi,\delta)$.
Then $h$ is continuous, and we note by the chain rule \eqref{ChainRuleCol} that $h(w,\nu,\omega',\xi,1/\sqrt{n})$ is the Hessian corresponding to the gradient in \eqref{gradwnulogLnFromFomegaFnu}.
Moreover by implicitly differentiating the relation $F_1=0$ when $\delta=0$ (see \eqref{F1Fomegadelta0}) we obtain
\begin{equation}
\grad_w\hat{\sigma}^T(w_0,\nu_0,\omega_0,\xi_0,0) = -A^{-1} B_\omega, \qquad \grad_\nu\hat{\sigma}^T(w_0,\nu_0,\omega_0,\xi_0,0) = -A^{-1} C
\end{equation}
so that
\begin{equation}
h(w_0,\nu_0,\omega'_0,\xi_0,0) = 
\mat{H_\omega & -B_\omega^T A^{-1} C\\
-C^T A^{-1} B_\omega & \; Q - C^T A^{-1} C}
;
\end{equation}
we saw this matrix in \eqref{ReducegradsigmawnuF3}.
Note that $H_\omega=-B_\omega^T A^{-1} B_\omega$ is negative definite by construction since the $\xdim\times p_1$ matrix $B_\omega$ has rank $p_1$, while $E$ is negative definite by assumption.
Therefore \eqref{ReducegradsigmawnuF3} implies that $h(w_0,\nu_0,\omega'_0,\xi_0,0)$ is negative definite.
By continuity, it follows that the Hessian corresponding to the gradient in \eqref{gradwnulogLnFromFomegaFnu} is negative definite whenever $w, \nu, \omega', \xi$ lie in a suitable neighbourhood of $w_0,\nu_0,\omega_0,\xi_0$ and $n$ is sufficiently large.

The upshot of this discussion is that, perhaps after shrinking the neighbourhoods $U_\omega$, $U_\nu$, $\tilde{U}_w$, $\tilde{V}$, $\tilde{W}$ and taking $n_0$ sufficiently large, the point $w_{\max}=G_w(\omega',\xi,1/\sqrt{n})$, $\nu_{\max}=G_w(\omega',\xi,1/\sqrt{n})$ is the unique maximum of $(w, \nu)\mapsto \log L_n(\smallmat{\omega'+ w/\sqrt{n}\\ \nu}; x)$ when $w$ and $\nu$ are restricted to $\tilde{U}_w$ and $U_\nu$; $\omega'$ and $\xi$ are restricted to $U_\omega$ and $\tilde{V}$; $n\geq n_0$; and the change of variables \eqref{PIChangeOfVariables} applies.

To show that the corresponding point $\theta_{\max}=\smallmat{\omega' + w_{\max}/\sqrt{n}\\ \nu_{\max}}$ is the MLE in $U=U_\omega\times U_\nu$, we will need to show that (perhaps after shrinking the neighbourhoods and increasing $n_0$) the function $\theta\mapsto \log L_n(\theta;x)$ cannot have a maximum with $\theta\in U$ but $w = \sqrt{n}(\omega-\omega')\notin\tilde{U}_w$.
To this end, use \eqref{LagrangeForm2} and the fact that $\grad_\nu\log L_n$ and $\grad_\omega\log L_n$ vanish at $\theta_{\max}$ to write
\begin{align}
&\log L_n(\smallmat{\omega\\ \nu};x) 
\notag\\&\quad
= \log L_n(\theta_{\max};x) 
+ \int_0^1 \grad_\nu\log L_n(\smallmat{\omega_{\max}\\ \nu_{\max}+t(\nu-\nu_{\max})};x) (\nu-\nu_{\max}) \, dt
\notag\\&\qquad
+ \int_0^1 (\nu-\nu_{\max})^T \grad_\nu^T\grad_\omega\log L_n(\smallmat{\omega_{\max}\\ \nu_{\max}+t(\nu-\nu_{\max})};x)(\omega-\omega_{\max}) \, dt
\notag\\&\qquad
+ \int_0^1 (1-t) (\omega-\omega_{\max})^T \grad_\omega^T \grad_\omega\log L_n(\smallmat{\omega_{\max}+t(\omega-\omega_{\max})\\ \nu};x)(\omega-\omega_{\max}) \, dt
.
\end{align}
By \eqref{gradwnulogLnFromFomegaFnu}, $\grad_\nu\log L_n(\smallmat{\omega_{\max}\\ \nu};x)$ can be expressed in terms of $F_\nu$ evaluated at $w=w_{\max}$, and is therefore continuous as a function of $\nu$.
In particular, after shrinking the neighbourhoods appropriately, we may assume that $\abs{\grad_\nu\log L_n(\smallmat{\omega_{\max}\\ \nu_{\max}+t(\nu-\nu_{\max})};x)}\leq C_1$ for some constant $C_1<\infty$.
In \eqref{gradwnulogLnFromFomegaFnu}, $w$ and $\omega$ differ by a scaling factor $\sqrt{n}$, with $\grad_\omega = \sqrt{n}\grad_w$, so by a similar continuity argument using \eqref{hFomegaFnuFormula} we may assume that $\norm{\grad_\omega\log L_n(\smallmat{\omega_{\max}\\ \nu_{\max}+t(\nu-\nu_{\max})};x)}\leq C_2\sqrt{n}$ for some constant $C_2<\infty$.
On the other hand, since $H_\omega$ is negative definite, we can repeat the argument from the proof of \reflemma{F=0ISMLE}, taking the Hessian with respect to $\omega$ only, to find $c_1>0$ such that 
\begin{equation}
\frac{1}{n} w^T \grad_\omega^T \grad_\omega\log L_n(\theta;x) w \leq -c_1 \abs{w}^2
\end{equation}
for all $w\in\R^{p_1\times 1}$ and all $\theta$ sufficiently close to $\theta_{\max}$.

Combining all of these elements, 
\begin{multline}\label{logLnomeganuBound}
\log L_n(\smallmat{\omega\\ \nu};x) 
\leq 
\log L_n(\theta_{\max};x) + C_1 \abs{\nu-\nu_{\max}} 
\\
+ C_2 \sqrt{n} \abs{\nu-\nu_{\max}} \abs{\omega-\omega_{\max}} - \frac{1}{2} c_1 n \abs{\omega-\omega_{\max}}^2 
.
\end{multline}
We know that $\theta_{\max}$ is the unique maximiser under the restrictions $w\in \tilde{U}_w$, $\nu\in U_\nu$, and in particular we can find $c_2>0$ such that $\theta_{\max}$ is the unique maximiser under the restrictions $\abs{\omega-\omega_{\max}}\leq c_2/\sqrt{n}$, $\nu\in U_\nu$.
But if we shrink $U_\nu$ so that $\nu\in U_\nu$ implies $\abs{\nu-\nu_{\max}}\leq \min\set{c_1 c_2^2/4C_1, c_1 c_2/4C_2}$, then we can apply \eqref{logLnomeganuBound} to find $\log L_n(\theta;x)<\log L_n(\theta_{\max};x)$ whenever $\theta\in U_\omega\times U_\nu$ and $\abs{\omega-\omega_{\max}}> c_2/\sqrt{n}$, so that the MLE in $U$ must be $\theta_{\max}$ as claimed.
In particular, this argument establishes that
\begin{equation}\label{MLEfromGPI}
\theta_{\MLE\,\mathrm{in}\,U}(x,n) = \mat{\omega' + G_w(\omega',\xi,1/\sqrt{n})/\sqrt{n} \\ G_\nu(\omega',\xi,1/\sqrt{n})}
\end{equation}
for $n$ sufficiently large and $\xi$, $\omega'$ in suitable neighbourhoods, when the relations \eqref{PIChangeOfVariables} hold.

The remainder of the proof of part~\ref{item:PI1/n} follows closely the proof of \refthm{MLEerror}.
All of the reasoning above applies equally to the saddlepoint MLE $\hat{\theta}_{\MLE\,\mathrm{in}\,U}(x,n)$, the only change being that terms involving $q_3$ are absent.
Define the augmented function
\begin{equation}
\tilde{F}(\sigma^T,w,\nu;\omega',\xi,\delta) = \mat{F_1(\sigma^T,w,\nu;\omega',\xi,\delta) \\ F_\omega(\sigma^T,w,\nu;\omega',\xi,\delta) \\ F_\nu(\sigma^T,w,\nu;\omega',\xi,\delta) \\ \omega' \\ \xi \\ \delta}
,
\end{equation}
which has the continuously differentiable inverse function
\begin{equation}
\tilde{G}(u_1, u_\omega, u_\nu; \omega',\xi,\delta) = \mat{\tilde{G}_\sigma(u_1, u_\omega, u_\nu; \omega',\xi,\delta)^T \\ \tilde{G}_w(u_1, u_\omega, u_\nu; \omega',\xi,\delta) \\ \tilde{G}_\nu(u_1, u_\omega, u_\nu; \omega',\xi,\delta) \\ \omega' \\ \xi \\ \delta}
.
\end{equation}
We find
\begin{equation}\label{FofhatGPI}
F(\hat{G}(\omega',\xi,\delta); \omega',\xi,\delta) = \mat{0 \\ \delta q_{3,\omega}\left( \smallmat{\omega'+\delta\hat{G}_w(\omega',\xi,\delta) \\ \hat{G}_\nu(\omega',\xi,\delta)}, K_0'(0;\smallmat{\omega'\\ \nu_0})+\delta\xi , \delta^2 \right)^T \\ q_{3,\nu}\left( \smallmat{\omega'+\delta\hat{G}_w(\omega',\xi,\delta) \\ \hat{G}_\nu(\omega',\xi,\delta)}, K_0'(0;\smallmat{\omega'\\ \nu_0})+\delta\xi , \delta^2 \right)^T}
.
\end{equation}
This quantity has the form $\delta^2 q_7(\omega',\xi,\delta)$ for some continuous function $q_7$.
Using \eqref{MLEfromGPI} and its analogue for $\hat{\theta}_{\MLE\,\mathrm{in}\,U}$, we find as in the proof of \refthm{MLEerror} that
\begin{align}
&\hat{\omega}_{\MLE\,\mathrm{in}\,U}(x,n) - \omega_{\MLE\,\mathrm{in}\,U}(x,n) 
\notag\\&\quad
= n^{-1/2} \bigl( \hat{G}_w(\omega',\xi,1/\sqrt{n}) - G_w(\omega',\xi,1/\sqrt{n}) \bigr)
\notag\\&\quad
= n^{-1/2} \bigl( \tilde{G}_w(n^{-1}q_7(\omega',\xi,1/\sqrt{n});\omega',\xi,1/\sqrt{n}) - \tilde{G}_w(0;\omega',\xi,1/\sqrt{n}) \bigr)
.
\label{omegahat-omegaPI}
\end{align}
Continuity of $q_7$ and continuous differentiability of $\tilde{G}$ imply that this quantity is $O(n^{-3/2})$, uniformly over $\omega'$ and $\xi$ in suitable neighbourhoods.

The argument for $\abs{\hat{\nu}_{\MLE\,\mathrm{in}\,U}(x,n) - \nu_{\MLE\,\mathrm{in}\,U}(x,n)}$ is similar except that no scaling factor $1/\sqrt{n}$ arises from \eqref{MLEfromGPI}, so that the resulting bound is $O(1/n)$.
This completes the proof of part~\ref{item:PI1/n}.

For part~\ref{item:PI1/n^2}, the right-hand side of \eqref{FofhatGPI} can be written more specifically as
\begin{equation}
F(\hat{G}(\omega',\xi,\delta); \omega',\xi,\delta) = \mat{0 \\ \delta^3 q_8(\omega',\xi,\delta) \\ \delta^2 q_9(\omega',\xi,\delta)}
\end{equation}
for continuous functions $q_8,q_9$.
Abbreviating $\delta=1/\sqrt{n}$ for convenience, \eqref{omegahat-omegaPI} becomes
\begin{align}
&\hat{\omega}_{\MLE\,\mathrm{in}\,U}(x,n) - \omega_{\MLE\,\mathrm{in}\,U}(x,n) 
\notag\\&\quad
=
\delta\left( \tilde{G}_w(0,\delta^3 q_8(\omega',\xi,\delta), \delta^2 q_9(\omega',\xi,\delta);\omega',\xi,\delta) - \tilde{G}_w(0,0,0;\omega',\xi,\delta) \right)
\notag\\&\quad
= \delta^4 \int_0^1 \grad_{u_\omega}\tilde{G}_w(0,t\delta^3 q_8(\omega',\xi,\delta), t\delta^2 q_9(\omega',\xi,\delta);\omega',\xi,\delta) q_8(\omega',\xi,\delta) dt
\notag\\&\qquad
+ \delta^3 \int_0^1 \grad_{u_\nu}\tilde{G}_w(0,t\delta^3 q_8(\omega',\xi,\delta), t\delta^2 q_9(\omega',\xi,\delta);\omega',\xi,\delta) q_8(\omega',\xi,\delta) dt
\label{omegahat-omegaPISpecific}
\end{align}
In the right-hand side of \eqref{omegahat-omegaPISpecific}, the first term is already $O(\delta^4)=O(1/n^2)$ with the required uniformity.
To prove part~\ref{item:PI1/n^2} it is therefore enough to show that $\grad_{u_\nu}\tilde{G}_w(0,0,0;\omega',\xi,0)=0$ for all $\omega',\xi$ and that $\grad_{u_\nu}\tilde{G}_w$ is a locally Lipschitz function of its arguments.

To compute $\grad_{u_\nu}\tilde{G}_w$, note that $\grad_{u_1,u_\omega,u_\nu;\omega',\xi,\delta}\tilde{G} = (\grad_{\sigma^T,w,\nu;\omega',\xi,\delta}\tilde{F})^{-1}$ (where the gradient of $\tilde{F}$ is evaluated at the point $\tilde{G}$).
In block form
\begin{equation}
\begin{gathered}
\grad_{\sigma^T,w,\nu;\omega',\xi,\delta}\tilde{F} = \mat{\grad_{\sigma^T,w,\nu}F & \grad_{\omega',\xi,\delta}F \\ 0 & I}
,
\\
(\grad_{\sigma^T,w,\nu;\omega',\xi,\delta}\tilde{F})^{-1} = \mat{(\grad_{\sigma^T,w,\nu}F)^{-1} & \; -(\grad_{\sigma^T,w,\nu}F)^{-1}\grad_{\omega',\xi,\delta}F \\ 0 & I}
,
\end{gathered}
\end{equation}
so $\grad_{u_\nu}\tilde{G}_w$ is given by the $(2,3)$ block entry in the block form of $(\grad_{\sigma^T,w,\nu}F)^{-1}$.

Now set $\tilde{\sigma} = G_\sigma(0, 0, 0; \omega', \xi, 0)$, $\tilde{w} = G_w(0, 0, 0; \omega', \xi, 0)$, $\tilde{\nu} = G_\nu(0, 0, 0; \omega', \xi, 0)$.
Then, by construction, $\tilde{\sigma}$ and $\tilde{w}$ satisfy $F_1(\tilde{\sigma}^T,\tilde{w},\tilde{\nu};\omega',\xi,0)=0$ and $F_\omega(\tilde{\sigma}^T,\tilde{w},\tilde{\nu};\omega',\xi,0)=0$, so that $\tilde{\sigma}^T=\tilde{J}\xi$ by \eqref{F1Fomegadelta00iff}.
By \eqref{CejFormula}, the $j^\text{th}$ column of $\tilde{B}_\omega^T \tilde{A}^{-1} \tilde{C}$ becomes
\begin{equation}
\tilde{B}_\omega^T \tilde{A}^{-1} \tilde{C} e_j = \tilde{B}_\omega^T \tilde{A}^{-1} \frac{\partial K_0''}{\partial\nu_j}(0;\smallmat{\omega' \\ \nu}) \tilde{J} \xi
\end{equation}
which vanishes by \eqref{PIOffDiagonalVanishes}.
In particular, the right-hand side of \eqref{ReducegradsigmawnuF1} is already block-diagonal.
(Moreover its corner entry must reduce to $\tilde{E}$, the analogue of $E$, and by continuity we may restrict to a neighbourhood in which $\tilde{E}$ is non-singular.)
So we can invert to find
\begin{align}
\left( \grad_{\sigma^T,w,\nu}F(\tilde{\sigma},\tilde{w},\tilde{\nu};\omega',\xi,0)  \right)^{-1}
&= 
\mat{I & \tilde{A}^{-1}\tilde{B}_\omega & \tilde{A}^{-1}\tilde{C}\\
0 & I & 0\\
0 & 0 & I}
\mat{\tilde{A}^{-1} & 0 & 0\\
0 & \tilde{H}_\omega^{-1} & 0\\
0 & 0 & \tilde{E}^{-1}}
\mat{I & 0 & 0\\
\tilde{B}_\omega^T \tilde{A}^{-1} & I & 0\\
\tilde{C}^T \tilde{A}^{-1} & 0 & I}
\end{align}
and it can be verified that the $(2,3)$ block entry of the resulting matrix is 0.
We have therefore shown that $\grad_{u_\nu}\tilde{G}_w(0,0,0;\omega',\xi,0)=0$, as required.

Finally to show that $\grad_{u_\nu}\tilde{G}$ is locally Lipschitz, it is enough to show that $\tilde{G}$ is of class $C^2$, i.e., twice continuously differentiable, rather than only $C^1$.
Since matrix inversion is a smooth operation, it is enough to show that $F$ is $C^2$.
Note that $F$ includes partial derivatives $\frac{\partial^{k+\ell} K_0}{\partial\theta_{i_1}\dotsb\partial\theta_{i_k}\partial s_{j_1}\dotsb\partial s_{j_\ell}}$ with $k\leq 1, k+\ell\leq 3$, and the fact that the corresponding terms are $C^2$ follows under our assumptions from \eqref{GrowthBound} with $k\leq 3, k+\ell\leq 5$.
To argue that the other terms -- i.e., the terms involving $q_3$ -- are $C^2$, we show that $q_1$ and $q_2$ from \refprop{PddtPError} are themselves $C^2$; then the statement for $q_3$ follows by the argument from the proof of \refcoro{gradlogPhatsError} and the fact that $\hat{s}_0$ is $C^2$.
To prove the latter assertion, we explain the modifications to the proof of \refprop{PddtPError} needed to extend from $C^1$ to $C^2$.

Recall that $q_k(s,\theta,\epsilon) = \tilde{f}_k(s,\theta,\epsilon)+f_k(s,\theta,\epsilon) - f_k(s,\theta,0)$ for $k=1,2$ and $\epsilon\geq 0$.
The functions $\tilde{f}_1$, $\tilde{f}_2$ are constructed by applying \reflemma{FunctionsToC1} to the functions $r_n^{(1)}$ and $r_n^{(2)}$ from \eqref{rnkFormula}.
To conclude twice differentiability, we must show that $r_n^{(1)}$, $r_n^{(2)}$ and their second-order partial derivatives are $o(n^{-4})$, uniformly, rather than $o(n^{-2})$ as in \eqref{rnkon-2}.
To that end, let $\tilde{\tau}$ denote another entry $\theta_i$ or $s_j$, again possibly the same entry as $t$ or $\tau$.
Define $I_4(s,\theta,n)$, $J_4(s,\theta,\epsilon)$ and $r_n^{(4)}(s,\theta)$ in the same way as $I_k(s,\theta,n)$, $J_k(s,\theta,\epsilon)$ and $r_n^{(k)}(s,\theta)$, $k=1,2,3$, from \eqref{I1I2Formula}--\eqref{I3Formula}, \eqref{J1J2J3Formula} and \eqref{rnkFormula}, with the integrands of $I_4$ and $J_4$ taken to be $\frac{\partial}{\partial\tilde{\tau}}$ of the integrands for $I_3$ and $J_3$.
These integrands can be expressed in terms of the partial derivatives $\frac{\partial^{k+\ell} M_0}{\partial\theta_{i_1}\dotsb\partial\theta_{i_k}\partial s_{j_1}\dotsb\partial s_{j_\ell}}$ with $k\leq 3, k+\ell\leq 5$, to which \eqref{GrowthBound} applies under our assumptions.
We can therefore repeat the argument from the proof of \refthm{MLEerror} to conclude that $r_n^{(k)}(s,\theta)$ decays exponentially in $n$, and in particular $r_n^{(k)}(s,\theta)=o(n^{-4})$, all uniformly in $\theta,s$ in suitable neighbourhoods, for $k=1,2,3,4$.
Applying \eqref{ZoverhatP} twice, the partial derivatives $\frac{\partial^2 r_n^{(k)}}{\partial\tau\partial\tilde{\tau}}$ for $k=1,2$ can be expressed in terms of $r_n^{(k)}$ for $k=3,4$, so these partial derivatives are also $o(n^{-4})$, uniformly.
By \reflemma{FunctionsToC1}, $\tilde{f}_k$ is $C^2$ for $k=1,2$.

Next recall the functions $f_k(s,\theta,\epsilon)$, $k=1,2,3$, from \eqref{f1f2f3Formula}, and define $f_4$ similarly: namely, by differentiating the integrand for $f_3$ with respect to $\tilde{\tau}$, except that the factor $\sqrt{\det(2\pi K_0''(s;\theta)^{-1})}$ is treated as constant.
This integrand involves the same partial derivatives from \eqref{GrowthBound}, with $k\leq 3$ and $k+\ell\leq 5$.
The arguments from the proof of \refprop{PddtPError} can therefore be repeated to show that $f_k$ is continuous, $k=1,2,3,4$.
Further applications of \eqref{ZoverhatP} show that partial derivatives $\frac{\partial^2 f_k}{\partial\tau\partial\tilde{\tau}}$, $k=1,2$, can be expressed in terms of $f_3$, $f_4$ and are therefore continuous.

To show the continuity of $\frac{\partial^2 f_k}{\partial\epsilon^2}$, $k=1,2$, use the representation \eqref{fWithoutOdd} in terms of the functions $H_k(\psi,\phi)$ from \eqref{H1H2Formula}.
(Recall that $H_k(\psi,\phi)$ depend implicitly on $s,\theta$, ultimately through $K_0''$ and $\frac{\partial K_0''}{\partial t}$.)
The problematic part of this representation arises from the quantity $\epsilon \grad_\phi\grad_\phi^T H_k(\psi,u\psi\sqrt{\epsilon})$.
For this to be $C^2$ as a function of $\epsilon$, it is enough if $\grad_\phi\grad_\phi^T H_k(\psi,\phi)$ is $C^2$ as a function of $\phi$, together with bounds analogous to \eqref{HkBound}.
This part of the argument works as in the proof of \refprop{PddtPError}, now using \eqref{GrowthBound} with $k\in\set{0,1}, k+\ell\leq 7$.

Finally, by \eqref{ZoverhatP}, the continuity of $\frac{\partial^2 f_k}{\partial\tau\partial\epsilon}$ will follow from the continuity of $\frac{\partial f_{k+1}}{\partial\epsilon}$.
We have already proved this for $k=1$.
For $k=2$, define $H_3$ by analogy with $H_1$, $H_2$ and use the analogue of \eqref{fWithoutOdd}.
We therefore require $\grad_\phi\grad_\phi^T H_3(\psi,\phi)$ to be $C^1$ as a function of $\epsilon$, together with bounds analogous to \eqref{HkBound}.
This part of the argument also works as in the proof of \refprop{PddtPError}, now using \eqref{GrowthBound} with $k\leq 2, k+\ell\leq 7$.
\end{proof}

\begin{proof}[Proof of \refthm{MLEerrorNormalApprox}]
Replacing $X_\theta$ by $\tilde{X}_\theta$ amounts to replacing $K_0(s;\theta)$ by its second-order Taylor approximation $\tilde{K}_0(s;\theta) = sK_0'(0;\theta) + \tfrac{1}{2}s K_0''(0;\theta) s^T$; cf.\ \cite[section 5]{DavHauKraParameterLinearBirthDeath}.
Then the corresponding integrals in \eqref{PformulaK}--\eqref{dPdtFormulaK} reduce to $\hat{P}_n(s,\theta)$ and $\frac{\partial\hat{P}_n}{\partial t}(s,\theta)$ and the corresponding error terms $q_1,q_2,q_3$ vanish.
In the setup of the proof of \refthm{MLEerrorPartiallyIdentifiable}, the corresponding functions $F_1$, $F_\omega$, $F_\nu$, $F$ from \eqref{F1FomegaFnuFFormula}--\eqref{FnuTermRewritten} are obtained by replacing some, but not all, occurrences of $\delta$ by 0.
To interpolate between the two cases, define
\begin{equation}
\begin{aligned}
&F_{1,\CLT}(\sigma^T,w,\nu;\omega',\xi,\delta,\delta') 
\\&\quad
= \int_0^1 \left( K_0''(u\delta'\sigma;\theta)\sigma^T + \grad_s\grad_\omega K_0(0;\smallmat{\omega'+u\delta\\ \nu}) \right) du - \xi
\\
&F_{\omega,\CLT}(\sigma^T,w,\nu;\omega',\xi,\delta,\delta') 
\\&\quad
= \int_0^1 (\grad_s\grad_\omega K_0(u\delta'\sigma;\theta))^T \sigma^T du + \delta q_{3,\omega}(\theta,K_0'(0;\theta')+\delta\xi,(\delta')^2)^T
\\&\qquad
+ \delta\grad_\omega^T\log\hat{P}\left( \delta'\sigma,\theta \right) 
\\&\qquad
- \delta \left( \grad_s\grad_\omega K_0\left( \delta'\sigma;\theta \right) \right)^T K_0''\left( \delta'\sigma;\theta \right)^{-1}\grad_s\log\hat{P}\left( \delta'\sigma,\theta \right) 
,
\\
& F_{\nu,\CLT}(\sigma^T,w,\nu;\omega',\xi,\delta,\delta')
\\&\quad
= \int_0^1 (1-u) \sum_{k,\ell=1}^m \sigma_k\sigma_\ell \grad_\nu^T \frac{\partial^2 K_0}{\partial s_k\partial s_\ell}(u\delta'\sigma;\theta) du
\\&\qquad
+ q_{3,\nu}\left( \theta,K_0'(0;\theta') + \delta\xi,(\delta')^2 \right)^T
+ \grad_\nu^T\log\hat{P}(\delta'\sigma,\theta) 
\\&\qquad
- (\grad_s\grad_\nu K_0(\delta'\sigma;\theta))^T K_0''(\delta'\sigma;\theta)^{-1}\grad_s\log\hat{P}(\delta'\sigma,\theta) 
.
\end{aligned}
\end{equation}
Then setting $\delta'=\delta$ recovers the functions $F_1$, $F_\omega$, $F_\nu$ from the proof of \refthm{MLEerrorPartiallyIdentifiable}, whereas setting $\delta'=0$ recovers the analogue of $F_1$, $F_\omega$, $F_\nu$ in which $K_0$ is replaced by $\tilde{K}_0$ (including implicitly in $\log\hat{P}_n$ and its gradients).
As in the proof of \refthm{MLEerrorPartiallyIdentifiable}, we merge $F_1$, $F_\omega$, $F_\nu$ into a single function $F$, and then there is a continuously differentiable implicit function $G(\omega',\xi,\delta,\delta')$ solving $F(G(\omega',\xi,\delta,\delta'); \omega',\xi,\delta,\delta')=0$.
Write $G_\sigma$, $G_w$, $G_\nu$ for the sub-blocks of $G$.
Then, similar to \eqref{MLEfromGPI},
\begin{equation}
\begin{aligned}
\theta_{\MLE\,\mathrm{in}\,U}(x,n) &= \mat{\omega' + G_w(\omega',\xi,1/\sqrt{n},1/\sqrt{n})/\sqrt{n} \\ G_\nu(\omega',\xi,1/\sqrt{n},1/\sqrt{n})}
,
\\
\tilde{\theta}_{\MLE\,\mathrm{in}\,U}(x,n) &= \mat{\omega' + G_w(\omega',\xi,1/\sqrt{n},0)/\sqrt{n} \\ G_\nu(\omega',\xi,1/\sqrt{n},0)}
,
\end{aligned}
\end{equation}
so that
\begin{equation}
\tilde{\omega}_{\MLE\,\mathrm{in}\,U}(x,n) - \omega_{\MLE\,\mathrm{in}\,U}(x,n) = n^{-1/2}\left( G_w(\omega',\xi,1/\sqrt{n},0) - G_w(\omega',\xi,1/\sqrt{n},1/\sqrt{n})\right)
\end{equation}
and this quantity is $O(1/n)$, with the required uniformity, by the continuous differentiability of $G$.
The same reasoning applies to $\nu$, except that the extra factor $n^{-1/2}$ is absent.
\end{proof}

\section{Saddlepoint approximations compared to normal approximations}\lbappendix{SaddlepointVsCLT}

In this section we discuss the analogue of \eqref{LErrorRough} when a normal approximation is used in place of the saddlepoint approximation.
For the purposes of this discussion, we fix $\theta$ and revert to considering the density $f(x;\theta)$ and its saddlepoint approximation $\hat{f}(x;\theta)$ in place of $L(\theta;x)$ and $\hat{L}(\theta;x)$.

We remark that the normal approximation has one notable advantage over the saddlepoint approximation: the normal approximation corresponds to a stochastic approximation in which one distribution is approximated by another.
Thus normal approximations preserve our ability to reason probabilistically about the distribution of interest.
By contrast, the saddlepoint approximation to the density, $\hat{f}(x;\theta)$, need not be a density function itself.
In principle it may be possible to normalise $\hat{f}(x;\theta)$ by computing the integral $\int_{\R^{\xdim\times 1}} \hat{f}(x;\theta)\,dx$, but this integral is usually not susceptible to exact calculation and therefore requires computationally intensive numerical integration.

Apart from this, however, we shall see that the normal approximation is less accurate than the saddlepoint approximation in general, outside a relatively narrow window near the mean.
Indeed, because of the strong bound \eqref{LErrorPrecise}, it will suffice to compare the saddlepoint approximation and the normal approximation.
We remark that when $x=\E(X)$, so that $\hat{s}=0$, the saddlepoint approximation and the normal approximation coincide.

Suppose that $0\in\interior\mathcal{S}_\theta$, which implies that $X_\theta$ has mean $K'(0;\theta)$ and covariance matrix $K''(0;\theta)$.
Then we can write the normal density approximation as
\begin{equation}
\hat{f}_\CLT(x;\theta) = \frac{\exp\left( -\tfrac{1}{2}(x-K'(0;\theta))^T K''(0;\theta)^{-1} (x-K'(0;\theta)) \right)}{\sqrt{\det(2\pi K''(0;\theta))}}
.
\end{equation}
Similar to \eqref{hatP}--\eqref{hatLL*P}, we introduce a factorisation and change variables from $x$ to $s$ by setting $x=K'(s;\theta)$.
Define
\begin{equation}
\begin{aligned}
L^*_\CLT(s,\theta) &= \exp\left( -\tfrac{1}{2}(K'(s;\theta)-K'(0;\theta))^T K''(0;\theta)^{-1} (K'(s;\theta)-K'(0;\theta)) \right)
,
\\
\hat{P}_\CLT(s,\theta) &= \frac{1}{\sqrt{\det(2\pi K''(0;\theta))}} = \hat{P}(0,\theta)
,
\end{aligned}
\end{equation}
so that $\hat{f}_\CLT(K'(s;\theta);\theta) = L^*_\CLT(s,\theta)\hat{P}_\CLT(s,\theta)$ and
\begin{equation}
\frac{\hat{f}_\CLT(K'(s;\theta);\theta)}{\hat{f}(K'(s;\theta);\theta)} = \exp\left( \log L^*_\CLT(s,\theta) - \log L^*(s,\theta) + \log\frac{\hat{P}_\CLT(s,\theta)}{\hat{P}(s,\theta)} \right)
.
\end{equation}

We now turn to the standard asymptotic regime \eqref{SAR}, and we write $\hat{f}_\CLT=\hat{f}_{\CLT,n}$ and $\hat{f}=\hat{f}_n$ to make the $n$-dependence explicit.
Note that, similar to \eqref{L*L*0} and \eqref{s0hatPGradients}, both $\log L^*_\CLT$ and $\log L^*$ are proportional to $n$, whereas the powers $n^{-d/2}$ in $\hat{P}_\CLT/\hat{P}$ cancel, so that
\begin{equation}
\frac{\hat{f}_{\CLT,n}(K'(s;\theta);\theta)}{\hat{f}_n(K'(s;\theta);\theta)} = \exp\left( n g(s;\theta) + h(s;\theta) \right) \quad\text{where} \quad g(0;\theta)=h(0;\theta)=0.
\end{equation}
We will analyse the ratio $\hat{f}_{\CLT,n}/\hat{f}_n$ for values of $x$ near the mean by expanding $g$ and $h$ as Taylor series in $s$ near $s=0$, and comparing to the size of the saddlepoint error term, which is $O(1/n)$ by \eqref{LErrorPrecise}.
To simplify the exposition, we restrict our attention to the univariate case $\xdim=1$, so that we can write
\begin{equation}
\begin{aligned}
g(s;\theta) &= -\frac{(K'_0(s;\theta)-K'_0(0;\theta))^2}{2 K''_0(s;\theta)} - K_0(s;\theta) + s K'_0(s;\theta)
,
\\
h(s;\theta) &= \frac{1}{2}\log\frac{K''_0(s;\theta)}{K''_0(0;\theta)}
.
\end{aligned}
\end{equation}

Specifically, suppose that $0\in\interior\mathcal{S}_\theta$ and the assumptions of \refthm{GradientError} apply.
(Note that $0\in\interior\mathcal{S}_\theta$ implies that all moments of $Y_\theta$ are finite, which is stronger than necessary for the Central Limit Theorem to apply.)
For convenience, write $s_0=0$ and $y_0=K_0'(0;\theta)$, where $\theta_0=\theta$ is fixed in the present discussion.
(Since we are ignoring dependence on $\theta$, we could pass to the model where $K_0(s;\theta)=K_0(s;\theta_0)$ for all $\theta$; this means that in addition to $0\in\interior\mathcal{S}_\theta$ we have $(0,\theta)\in\interior\mathcal{S}$ and there is no difficulty in applying the results from \refsubsect{MainPropStatement}.)
Since $s_0=0$ is an interior point of $\mathcal{S}_\theta$, $(y_0,\theta)$ is an interior point of the open set $\mathcal{Y}^o$, and there is some $\delta>0$ such that the bound \eqref{LErrorPrecise} holds uniformly over the region $\abs{y-y_0}\leq\delta$.
The mapping $s\mapsto K'_0(s;\theta)$ is continuously differentiable with continuously differentiable inverse mapping $y\mapsto \hat{s}_0(\theta;y)$, so under this change of variables we can alternatively say that \eqref{LErrorPrecise} holds uniformly over the region $\abs{s}\leq\delta'$ for some $\delta'>0$.
By shrinking $\delta,\delta'$ if necessary, we can ensure that in the Taylor expansions that follow, the first non-zero terms are the dominant ones.

Now compute derivatives (with respect to $s$) of $g$ and $h$ at $s=0$.
We find
\begin{equation}
h'(s;\theta) = \frac{K'''_0(s;\theta)}{2K''_0(s;\theta)}, \qquad h''(s;\theta) = \frac{K''''_0(s;\theta)}{2K''_0(s;\theta)} - \frac{K'''_0(s;\theta)^2}{2K''_0(s;\theta)^2}
.
\end{equation}
A lengthier calculation leads to $g'(0;\theta)=g''(0;\theta)=0$ and
\begin{align}
g'''(0;\theta) &= - K'''_0(0;\theta) & g''''(0;\theta) &= -K''''_0(0;\theta) - 3\frac{K'''_0(0;\theta)^2}{K''_0(0;\theta)}.
\end{align}
We obtain different results depending on which of $K'''_0(0;\theta)$ and $K''''_0(0;\theta)$ are non-zero.
We will ignore the case where $K'''_0(0;\theta)=K''''_0(0;\theta)=0$, noting that for random variables with this property (including normal distributions) both the saddlepoint approximation and the normal approximation are more accurate than usual.

\textit{Case 1: $K'''_0(0;\theta)=0$ but $K''''_0(0;\theta)\neq 0$.}

If $Y$ (or equivalently $X$) has zero skew but non-zero excess of kurtosis, then $g'''(0;\theta)=h'(0;\theta)=0$ but $g''''(0;\theta)\neq 0$, $h''(0;\theta)\neq 0$.
Thus
\begin{equation}\label{SPAvsCLTNoSkew}
\frac{\hat{f}_{\CLT,n}(K'(s;\theta);\theta)}{\hat{f}_n(K'(s;\theta);\theta)} = \exp\left( O(n s^4) + O(s^2) \right).
\end{equation}
The error terms $O(n s^4)$ and $O(s^2)$ both become of order $1/n$ once $\abs{s}\approx n^{-1/2}$, or equivalently once $\abs{y-y_0}\approx n^{-1/2}$, $\abs{x-ny_0}\approx n^{1/2}$.
Thus in this case, the scale $\abs{y-y_0}\approx n^{-1/2}$, $\abs{x-ny_0}\approx n^{1/2}$ is the scale at which the normal approximation starts to have larger relative error than the saddlepoint approximation.

Note that at this scale, the normal approximation is still asymptotically reasonably accurate: the relative error (which, by \eqref{SPAvsCLTNoSkew} and \eqref{LErrorPrecise}, is $O(1/n) + O(ns^4) + O(s^2)$) remains $O(1)$ until $s\approx n^{-1/4}$, corresponding to $\abs{y-y_0}\approx n^{-1/4}$, $\abs{x-ny_0}\approx n^{3/4}$.
At this scale, the density itself is small, which is why the absolute error in the normal approximation remains reasonably small even uniformly.

\textit{Case 2: $K'''_0(0;\theta)\neq 0$.}

If $Y$, or equivalently $X$, exhibits skew, then $g$ and $h$ are controlled by lower order Taylor terms:
\begin{equation}\label{SPAvsCLTSkew}
\frac{\hat{f}_{\CLT,n}(K'(s;\theta);\theta)}{\hat{f}_n(K'(s;\theta);\theta)} = \exp\left( O(n s^3) + O(s) \right).
\end{equation}
Now the normal approximation starts to have larger relative error as soon as $\abs{s}\approx n^{-1}$, corresponding to $\abs{y-y_0}\approx n^{-1}$, $\abs{x-ny_0}\approx O(1)$.
In particular, for values in the range $\abs{y-y_0}\approx n^{-1/2}$, $\abs{x-ny_0}\approx n^{1/2}$ -- i.e., for the values of $x$ and $y$ predicted to occur by the Central Limit Theorem -- the relative accuracy of the normal approximation is $O(n^{-1/2})$ compared to $O(n^{-1})$ for the saddlepoint approximation.

\section{Examples}\lbappendix{Examples}

\refexample{X=AU} examines models where the observed data are formed by applying a deterministic many-to-one linear mapping to unobserved, or latent, variables, so that $X_\theta=AU_\theta$ for a fixed matrix $A$.
The hypotheses \eqref{SAR} and \eqref{DecayBound}--\eqref{GrowthBound}, and \eqref{RateFunctionImplicitHessianDefinite}/\eqref{MeanGradientFullRank} in the well-specified case, can be simplified by expressing them in terms of the latent variables $U_\theta$ and their dependence on the parameter $\theta$.

\refexample{MultipleSamplesWorkings} considers models where multiple independent samples are observed (possibly with different values of $n$) and gives simplified expressions for the hypotheses of Theorems~\ref{T:GradientError}--\ref{T:MLEerrorNormalApprox} in that case.

\refexample{ZhaBraFew} considers two related families of models studied by Zhang, Bravington \& Fewster \cite{ZhaBraFew2019} for population size estimation based on ambiguous count data.
The model implements the ambiguity by expressing the observed counts $X_\theta$ as a non-invertible linear function of latent counts $U_\theta$ that encode the underlying unambiguous counts, so that $X_\theta=A U_\theta$ as in \refexample{X=AU}.
These models fall into the framework of \eqref{SAR}, and are fully identifiable at the level of the sample mean in the well-specified case, so that Theorems~\ref{T:MLEerror}--\ref{T:SamplingMLE} apply when $s_0=0$ and $\theta_0$ is arbitrary.
The authors use saddlepoint approximations to compute MLEs, and we compare their findings with the results of this paper.

\refexample{DavHauKra} considers a branching process model used by Davison, Hautphenne \& Kraus \cite{DavHauKraParameterLinearBirthDeath} to estimate birth and death rates from a time series of population size observations.
Because of the branching property, the framework of \eqref{SAR} applies.
Their model is partially identifiable at the level of the sample mean in the well-specified case, so that Theorems~\ref{T:MLEerrorPartiallyIdentifiable}--\ref{T:MLEerrorNormalApprox} apply when $s_0=0$ and $\theta_0$ is arbitrary.
The authors compare saddlepoint MLEs with other estimators, including estimators based on normal approximations, and we compare their findings with the results of this paper.

\refexample{PedDavFok} discusses an autoregressive model used by Pedeli, Davison \& Fokianos \cite{PedDavFok2015} to model integer-valued autoregressive processes.
Their model cannot be expressed as a sum of $n$ i.i.d.\ terms, and we indicate why and to what extent \eqref{SAR} fails.

\refexample{Poisson} gives the explicit calculations for applying the saddlepoint approximation to the Poisson family of distributions.
The findings illustrate the assertions of \eqref{LErrorRough} and \refsubsubsect{ExpFamily}.
For comparison purposes, we can also compute approximate MLEs based on the normal approximation to the Poisson family.
The resulting exact calculations match the assertions of \refthm{MLEerrorNormalApprox}, which are sharp in this case.
These calculations also illustrate a potentially surprising pitfall: if applied to a negative observed value, the normal approximation likelihood leads to a well-defined MLE, in contrast to saddlepoint-based approximate likelihoods which correctly indicate that negative values are impossible.

Examples~\ref{ex:Gamma}--\ref{ex:GammaPI} discuss the Gamma family and various subfamilies, including two models with full identifiability (\refexample{GammaFI}) and partial identifiability (\refexample{GammaPI}) where explicit calculation shows the asymptotic sharpness of bounds from Theorems~\ref{T:GradientError}--\ref{T:LowerOrderSaddlepoint}.
The exact calculations give another indication of the different uniformity considerations: the normal approximation likelihood depends on the observed data values in a qualitatively different way than both the true likelihood and all the saddlepoint-based approximate likelihoods, and the functional forms only become approximately equal in a narrow tube around the mean values predicted by the model.

Examples~\ref{ex:Normal}--\ref{ex:NormalSquare} discuss the Normal family in two settings, depending on whether we observe a univariate or bivariate sufficient statistic.
The resulting likelihood function is exact in only one of these cases, but the corresponding MLEs, when defined, are exact in both cases.
\refexample{Normal} contains two sub-families illustrating that increasing $n$ affects the inferential uncertainty differently between fully identifiable and partially identifiable models.

Examples~\ref{ex:MLEWrong}--\ref{ex:DecayFails} use artificial examples to illustrate various technical features of saddlepoint MLEs: \refexample{MLEWrong} shows a case where the saddlepoint approximation gives incorrect global information about the shape of the likelihood function; \refexample{LhatVsL*Different} shows that a global maximum for $\hat{L}^*$ need not be a global maximum for $\hat{L}$; \refexample{QPlusGamma} gives a distribution with an ill-behaved density that nevertheless satisfies the regularity conditions \eqref{DecayBound}--\eqref{GrowthBound}; and \refexample{DecayFails} gives an example where \eqref{DecayBound} fails.

\subsection{Two general classes of examples}

\begin{example}[Linear mapping]\lbexample{X=AU}
Let $X=AU$, where $A\in\R^{\xdim\times\latentdim}$ is a fixed matrix and $U=U_\theta\in\R^{\latentdim\times 1}$ is an unobserved random vector whose moment generating function is known.
We are interested in the case $\latentdim > \xdim$, with $A$ of maximal rank $\xdim$, so that $X$ has non-singular covariance matrix if $U$ does.
Since $A$ is not invertible, observing $X$ does not allow us to identify the value of $U$, and a straightforward approach to computing the likelihood for $X$ requires integrating, or summing, over a $(\latentdim-\xdim)$-dimensional subspace of $U$ values compatible with a given $X$ value.
However, the moment generating function of $X$ can be readily calculated by $M_X(s;\theta)=M_U(sA;\theta)$, making the saddlepoint approximation an attractive alternative that avoids lengthy approximate integration.
A common case is that the entries of $U$ are independent, so that columns of $A$ with several non-zero entries induce complicated dependencies between the entries of $X$; however, $M_U$ will have a simple product form.

In the context of our results, assuming that $X_\theta$ is the sum of $n$ i.i.d.\ terms $Y^{(i)}_\theta$, as in \eqref{SAR}, amounts to the assumption that $U_\theta$ is itself the sum of $n$ i.i.d.\ terms $V^{(i)}_\theta$.
Provided that $A$ has rank $\xdim$, the assumptions \eqref{DecayBound}--\eqref{GrowthBound} for the CGF of $Y_\theta$ follow if the CGF of $V_\theta$ satisfies the corresponding bounds.
Moreover, in the well-specified case, the condition \eqref{RateFunctionImplicitHessianDefinite}/\eqref{MeanGradientFullRank} simplifies to the condition \eqref{MeanGradientFullRankX=AU}; see below.

In greater detail, write $K_U(t;\theta)=\E(e^{t U_\theta})$ for the MGF of $U$, with domain $\mathcal{T}_\theta = \shortset{t\in\R^{1\times\latentdim}\colon \E(e^{t U})<\infty}$.
As remarked, we have $K(s;\theta) = K_U(s A;\theta)$, which follows from the relation $sX = s(AU) = (sA)U$.
(Here we write $K$ for $K_X$.)

First, suppose that $s\mapsto sA$ maps $\interior\mathcal{S}_\theta$ into $\interior\mathcal{T}_\theta$, and that $K_U''(t;\theta)$ is non-singular for all $t\in\interior\mathcal{T}_\theta$.
Then \eqref{K''PosDef} holds, i.e., $K''(s;\theta)$ is non-singular for all $s\in\interior\mathcal{S}_\theta$.
To see this, differentiate twice and note that if $t=sA$, then $\frac{\partial t_a}{\partial s_i} = A_{ia}$.
Thus
\begin{equation}
\frac{\partial^2 K}{\partial s_i\partial s_j}(s;\theta) = \sum_{a,b=1}^\latentdim \frac{\partial^2 K_U}{\partial t_a\partial t_b}(sA;\theta) A_{ia} A_{jb} = \sum_{a,b=1}^\latentdim A_{ia} \frac{\partial^2 K_U}{\partial t_a\partial t_b}(sA;\theta) A^T_{bj} 
,
\end{equation}
which we recognise as the $i,j$ entry of the $\xdim\times\xdim$ matrix $A K''_U(sA;\theta) A^T$.
Since $K''_U(sA;\theta)$ is positive definite and $A$ has rank equal to the number of its rows, it follows that that $K''(s;\theta)=A K''_U(sA;\theta) A^T$ is non-singular as claimed.

Next, suppose $U=\sum_{i=1}^n V^{(i)}_\theta$, where $V^{(i)}_\theta$ are i.i.d.\ copies of $V_\theta$.
Then $X$ is the sum of $n$ i.i.d.\ terms $Y^{(i)}_\theta = A V^{(i)}_\theta$, as in \eqref{SAR}, and the MGFs $M_0(s;\theta)$ and $M_V(t;\theta)$ are related by $M_0(s;\theta) = M_V(sA;\theta)$.
Suppose that $A$ has rank $\xdim$, and that the analogues of \eqref{DecayBound}--\eqref{GrowthBound} hold in which $M_0$ replaced by $M_V$ and $\delta(s,\theta),\gamma(s,\theta)$ are replaced by $\tilde{\delta}(t,\theta)$ and $\tilde{\gamma}(t,\theta)$.
Then \eqref{DecayBound}--\eqref{GrowthBound} follow.
To see this, note that since $\tilde{s}A\neq 0$ for all $\tilde{s}\neq 0$, there must exist $c>0$ such that $\abs{\tilde{s}A}\geq c\abs{\tilde{s}}$ for all $\tilde{s}\in\R^{1\times\xdim}$.
Then 
\begin{align}
\abs{\frac{M_0(s+\ii\phi;\theta)}{M_0(s;\theta)}} &= \abs{\frac{M_V(sA + \ii\phi A;\theta)}{M_V(sA;\theta)}} 
\notag\\&
\leq (1+\tilde{\delta}(sA, \theta)\abs{\phi A})^{-\tilde{\delta}(sA,\theta)} 
\notag\\&
\leq (1+\tilde{\delta}(sA, \theta)c\abs{\phi})^{-\tilde{\delta}(sA,\theta)}
\label{M0DecayFromMVDecay}
\end{align}
so \eqref{DecayBound} holds with $\delta(s,\theta) = \tilde{\delta}(sA,\theta)\min(c,1)$.
In \eqref{GrowthBound}, each of the partial derivatives can be expressed as finite sums of partial derivatives of $M_V$ times products of entries of $A$.
So each of the partial derivatives in \eqref{GrowthBound} will be continuous and will grow at most polynomially in $\phi$ by an argument similar to \eqref{M0DecayFromMVDecay}, this time using the bound $\abs{\tilde{s}A} \leq C\abs{\tilde{s}}$ (for some constant $C<\infty$) in place of $\abs{\tilde{s}A} \geq c\abs{\tilde{s}}$.

Finally $X=AU$ gives a simple relation between the means, $\E(X_\theta) = A\E(U_\theta)$.
It follows that in the well-specified case, the condition for full identifiability simplifies: \eqref{MeanGradientFullRank} is equivalent to assuming that
\begin{equation}\label{MeanGradientFullRankX=AU}
A\grad_s\grad_\theta K_V(0;\theta_0)\text{ has rank }p.
\end{equation}
Equivalently, \eqref{MeanGradientFullRank} reduces to the condition that that the vectors $A \frac{\partial}{\partial \theta_i}\E(V_\theta)$, $i=1,\dotsc,p$, should be linearly independent (when evaluated at $\theta=\theta_0$).
\end{example}

\begin{example}[Multiple samples]\lbexample{MultipleSamplesWorkings}
Let $\vec{X}_\theta$ be a vector of dimension $\xdim=k \xdim_0$ formed by concatenating $k$ independent sub-blocks $X^{(1)}_\theta,\dotsc,X^{(k)}_\theta \in\R^{\xdim_0\times 1}$.
If the $X^{(j)}_\theta$'s are i.i.d.\ (as in the discussion in \refsubsubsect{MultipleSamples}) and if their common distribution satisfies the assumptions of \eqref{SAR}, then so will $\vec{X}_\theta$, with $\vec{Y}_\theta$ being a concatenation of $k$ independent copies of $Y_\theta$.

More generally, we will consider a situation in which $X^{(j)}_\theta$ is the sum of $n_j$ i.i.d.\ copies of $Y_\theta$, where the numbers $n_j$ are permitted to vary  but where all summands have a common parametric distribution.
To bring this setup closer to the setup of Theorems~\ref{T:GradientError}--\ref{T:IntegerValued}, make the change of variables 
\begin{equation}\label{vecnnvecbeta}
\vec{n}=n\vec{\beta}, \quad n_j=n\beta_j\text{ for $j=1,\dotsc,k$,} \quad\text{where $\vec{\beta}\in(0,\infty)^k$.}
\end{equation}
This scaling envisages a situation in which the $n_j$'s, though not necessarily equal, are all large on a common scale given by $n$.
See \refexample{DavHauKra} for an instance of this setup.
(Note however that $k$ is fixed in this discussion, and we exclude from consideration the joint limit $k\to\infty,n\to\infty$.)
Then we can write
\begin{equation}\label{KvecXFormula}
\vec{K}(\vec{s};\theta) = n \sum_{j=1}^k \beta_j K_0(s_j;\theta)
\end{equation}
for the CGF of $\vec{X}$.
Here $s_j$ denotes the $j^\text{th}$ sub-block of $\vec{s}$, with similar notation for other vector quantities.
For consistency of notation, we also write $\vec{x}$, $\vec{y}$ instead of $x$, $y$, with the relationship $\vec{x}=n\vec{y}$ holding implicitly.
Note that $y_j=x_j/n$, the $j^\text{th}$ sub-block of $\vec{y}$, no longer represents an implied sample mean; it is instead the scaled quantity $x_j/n_j = x_j/(n\beta_j) = y_j / \beta_j$ that has this interpretation.

If $\beta_j$ is not an integer and the distribution of $Y_\theta$ is not infinitely divisible, then $s_j\mapsto \beta_j K_0(s_j;\theta)$ need not be the CGF of a random variable.
Apart from this, however, 
\begin{equation}\label{Kvec0Formula}
\vec{K}_0(\vec{s};\theta) = \sum_{j=1}^k \beta_j K_0(s_j;\theta)
\end{equation}
behaves very much like $K_0$ and still has all the properties necessary for the proofs (for instance, the bounds \eqref{DecayBound}--\eqref{GrowthBound} for $\vec{M}_0$ follow easily from those for $M_0$).
Hence the results of Theorems~\ref{T:GradientError}--\ref{T:IntegerValued} still apply when $K$, $K_0$ are replaced by $\vec{K}$, $\vec{K}_0$, all implicitly depending on $\vec{\beta}$ as in \eqref{vecnnvecbeta}--\eqref{Kvec0Formula}, for $\vec{\beta}$ restricted to some neighbourhood of a given base point $\vec{\beta}_0$.
For instance, to prove \refthm{MLEerror}, we replace the function $F_2$ by 
\begin{multline}
\vec{F}_2(\vec{s}\,^T,\theta;\vec{y},\epsilon,\vec{\beta}) = 
\sum_{j=1}^k \beta_j \grad_\theta^T K_0(s_j;\theta) + \epsilon q_3(\theta,y_j,\epsilon/\beta_j)^T 
\\
+ \epsilon\left( \grad_\theta^T \log\hat{P}(s_j,\theta) - \left( \grad_s\grad_\theta K_0(s_j;\theta) \right)^T K_0''(s_j;\theta)^{-1}\grad_s\log\hat{P}(s_j,\theta) \right)
\end{multline}
in which $1/n$ is replaced by $\epsilon$ and $1/n_j$ is replaced by $\epsilon/\beta_j$.
Then $\vec{F}_2$ is still continuously differentiable in all its variables and the rest of the proof carries through as before.
In particular, since all of our uniform estimates are based on continuous differentiability, they remain valid if $\beta$ is allowed to vary over a compact subset of $(0,\infty)^k$.

Because $\vec{K}$ and $\vec{K}_0$ are sums of functions applied to each $s_j$ separately, their derivatives have simpler block forms.
The saddlepoint equation \eqref{SESAR} becomes 
\begin{equation}\label{SEMultipleSamples}
\beta_j K_0'(\hat{s}_j;\theta) = y_j, \qquad j=1,\dotsc,k;
\end{equation}
we note that $\hat{s}_j$ has the effect of parametrising the implied sample mean for the $j^\text{th}$ sub-block.
The condition \eqref{RateFunctionImplicitCriticalPoint} becomes 
\begin{equation}\label{vecKCriticalPoint}
\sum_{j=1}^k \beta_{0,j} \grad_\theta K_0(s_{0,j};\theta_0) = 0.
\end{equation}
Write $\vec{A}=\vec{K}_0''(\vec{s}_0;\theta_0)$ for the full $\xdim\times\xdim$ matrix of second derivatives with $\vec{\beta}=\vec{\beta}_0$, and write $A_j=K_0''(s_{0,j};\theta_0)$ for the $\xdim_0\times\xdim_0$ matrix corresponding to the $j^\text{th}$ sub-block.
Similarly, write $\vec{B}$, $\vec{H}$ for the quantities from Theorems~\ref{T:MLEerror}--\ref{T:SamplingMLE} based on $\vec{K}_0$ evaluated at $\vec{s}_0$, and write $B_j$, $H_j$ for the analogues of these quantities based on $K_0$ evaluated at $s_{0,j}$.
Then
\begin{equation}\label{vecAvecBvecH}
\begin{gathered}
\vec{A} = 
\mat{\beta_{0,1} A_1 & 0 & \dotsb & 0 \\ 
0 & \beta_{0,2} A_2 & \dotsb & 0 \\ 
\vdots & \vdots & \ddots & \vdots \\
0 & 0 & \dotsb & \beta_{0,k} A_k}
,
\quad
\vec{B} = 
\mat{\beta_{0,1} B_1 \\ 
\beta_{0,2} B_2 \\ 
\vdots\\
\beta_{0,k} B_k}
,
\\
\begin{aligned}
\vec{H} &= \grad_\theta^T \grad_\theta \vec{K}_0(\vec{s}_0;\theta_0) - \vec{B}^T \vec{A}^{-1} \vec{B} 
\\&
= \sum_{j=1}^k \beta_{0,j} \grad_\theta^T \grad_\theta K_0(s_{0,j};\theta_0) - \sum_{j=1}^k \beta_{0,j} B_j^T A_j^{-1} B_j 
= \sum_{j=1}^k \beta_{0,j} H_j.
\end{aligned}
\end{gathered}
\end{equation}

We now restrict our attention to the well-specified case $\vec{s}_0=0$, as in \refsubsubsect{WellSpecified}.
This amounts to the assumption that, in the limit $n\to\infty$, our observed vector $\vec{x}$ will be such that every sub-block's implied sample mean, $y_j/\beta_j$, has approximately the same value, and that this common value coincides with $\E(Y_{\theta_0})$ for some $\theta_0\in\thetadomain$.
Then the matrices $A_j$, $B_j$, $H_j$ in \eqref{vecAvecBvecH} all reduce to the same values $A$, $B$, $H$.
Furthermore, as in the discussion in \refsubsubsect{WellSpecified}, the condition \eqref{RateFunctionImplicitCriticalPoint}/\eqref{vecKCriticalPoint} holds automatically, and the condition \eqref{RateFunctionImplicitHessianDefinite} that $\vec{H}$ (or equivalently $H$) should be negative definite reduces to the condition \eqref{MeanGradientFullRank} that $B=\grad_s \grad_\theta K_0(0;\theta_0)$ should have rank $p$.  

Now suppose that the sub-block model is partially identifiable at the level of the sample mean, i.e., the parameter vector can be split so that \eqref{thetaSplitting} holds for the sub-block distributions $Y_\theta$ and the CGF $K_0$.
Then the full model is also partially identifiable at the level of the sample mean (i.e., \eqref{thetaSplitting} holds with $K_0$ replaced by $\vec{K}_0$ and $\E(Y_\theta)$ replaced by $\frac{1}{n}\E(\vec{X}_\theta)$) because $\E(\vec{X}_\theta)$ depends on $\omega$ and $\vec{\beta}$, but not on $\nu$.

The partially linearised model $\vec{\Xi}_{w,\nu}$ from \refthm{MLEerrorPartiallyIdentifiable} is also formed by concatenating $k$ independent sub-blocks $\Xi^{(j)}_{w,\nu}$ with distributions $\mathcal{N}(\beta_j B_\omega w, \beta_j K''_0(0;\smallmat{\omega_0\\ \nu}))$.
To continue the analysis, we use the explicit conditions \eqref{xi0Constraint}--\eqref{EijFormula} that are equivalent to the MLE condition from \refthm{MLEerrorPartiallyIdentifiable}.
We find $\vec{B}_\omega^T \vec{A}^{-1} \vec{B}_\omega = B_\omega^T A^{-1} B_\omega \sum_{j=1}^k \beta_j$ and
\begin{equation}
\vec{J} = \vec{A}^{-1} + \frac{1}{\sum_{j=1}^k \beta_{0,j}}
\mat{J-A^{-1} & J-A^{-1} & \dotsb & J-A^{-1} \\
J-A^{-1} & J-A^{-1} & \dotsb & J-A^{-1} \\
\vdots & \vdots & \ddots & \vdots \\
J-A^{-1} & J-A^{-1} & \dotsb & J-A^{-1}}
.
\end{equation}
Abbreviate $Q_j=\frac{\partial K_0''}{\partial\nu_j}(0;\theta_0)$ and note that its vector analogue $\vec{Q}_j=\frac{\partial\vec{K}_0''}{\partial\nu_j}$ has the same block-diagonal structure as $\vec{A}$, with diagonal blocks $\beta_{0,1} Q_j,\dotsc,\beta_{0,k} Q_j$.
After some calculation, we find that the condition \eqref{xi0Constraint} reduces to the condition
\begin{multline}\label{xi0ConstraintReduced}
\sum_{\ell=1}^k \frac{\xi_{0,\ell}^T A^{-1} Q_j A^{-1} \xi_{0,\ell}}{\beta_{0,\ell}} 
+ \frac{\bigl( \sum_{\ell=1}^k \xi_{0,\ell} \bigr)^T \left( J Q_j J - A^{-1} Q_j A^{-1} \right) \bigl( \sum_{\ell=1}^k \xi_{0,\ell} \bigr)}{\sum_{\ell=1}^k\beta_{0,\ell}} 
\\
=
k \trace\left( A^{-1} Q_j \right)
,\quad
j=1,\dotsc,p_2,
\end{multline}
while the entries $\vec{E}_{ij}$ from \eqref{EijFormula} reduce to
\begin{multline}\label{EijFormulaReduced}
\vec{E}_{ij} = 
\sum_{\ell=1}^k \frac{\xi_{0,\ell}^T \left( \tfrac{1}{2} A^{-1} Q_{ij} A^{-1} - A^{-1} Q_i A^{-1} Q_j A^{-1} \right) \xi_{0,\ell}}{\beta_{0,\ell}} 
\\
+ \frac{\bigl( \sum_{\ell=1}^k \xi_{0,\ell} \bigr)^T \left( \tfrac{1}{2} (J Q_{ij} J - A^{-1} Q_{ij} A^{-1}) - J Q_i J Q_j J + A^{-1} Q_i A^{-1} Q_j A^{-1} \right) \bigl( \sum_{\ell=1}^k \xi_{0,\ell} \bigr)}{\sum_{\ell=1}^k\beta_{0,\ell}} 
\\
- \frac{k}{2} \trace\left( A^{-1} Q_{ij} - A^{-2} Q_i Q_j \right)
,\quad
i,j=1,\dotsc,p_2.
\end{multline}
Further calculation and cancellation leads to
\begin{equation}
\vec{B}_\omega^T \vec{A}^{-1} \vec{Q}_j \vec{J} = \mat{B_\omega^T A^{-1} Q_j J \; & B_\omega^T A^{-1} Q_j J \; & \dotsb & B_\omega^T A^{-1} Q_j J}
.
\end{equation}
This calculation does not rely on any special features of $\theta_0$, so it applies equally to the vector analogue of the quantity in \eqref{PIOffDiagonalVanishes}.
In particular, unlike \eqref{xi0Constraint}/\eqref{xi0ConstraintReduced} and \eqref{EijFormula}/\eqref{EijFormulaReduced}, the condition in \refthm{MLEerrorPartiallyIdentifiable}\ref{item:PI1/n^2} is the same for all $k$.

Finally we remark that if $\xdim_0=p_1=1$, then $A$ and $B_\omega$ are both scalars, say $A=a$ and $B_\omega=b$, and the scalar $J$ vanishes by commutativity.
In particular, the condition \eqref{xi0Constraint}/\eqref{xi0ConstraintReduced} from \refthm{MLEerrorPartiallyIdentifiable}\ref{item:PI1/n^2} always holds.
However, the $\xdim\times\xdim$ matrix $\vec{J}$ does not vanish unless $k=1$.
\end{example}

\subsection{Examples from the literature}\lbappendix{FurtherLiterature}

\begin{example}[Capture-recapture and administrative datasets]\lbexample{ZhaBraFew}
Zhang, Bravington \& Fewster \cite{ZhaBraFew2019} use saddlepoint methods to analyse capture-recapture models with corrupted or incomplete identity records.
In a conventional capture-recapture experiment, an unknown population of animals is partially sampled at each of $r\geq 2$ different capture occasions.
Recorded sightings are compiled into (observed) \emph{capture histories} for each animal seen in the experiment.
The data vector $X$ records the counts for each possible capture history: for instance, one entry of $X$ counts the number of animals seen at the first and third capture occasions but not seen on any other occasion.
From the counts in $X$, we wish to infer the overall population size, including animals not seen on any capture occasion.
This model can also be interpreted for humans: a capture occasion is an administrative dataset, and capture histories encode entries in different datasets that can be traced back to the same individual. 

When $r\geq 3$, this framework can be expanded to consider misidentification and other mechanisms that can prevent observations of a single individual from being correctly linked.
In the animal context, sightings on different capture occasions may fail to be recognised as the same animal; in the human context, when an individual has records in only some of the datasets, it may be inconclusive whether these records belong to a single individual.
Such mechanisms, and similar models outlined in \refsubsect{IntroExample}, can be expressed by giving each individual a latent capture history fully specifying their status in the experiment (for instance, specifying on which capture occasions an animal was seen, not seen, or seen but misidentified).
Then, with $U$ the vector of counts of latent histories, we can write $X=AU$ as in \refexample{X=AU}, where $A$ is a deterministic matrix encoding the relationship between different observed and latent capture histories.
For instance, in the animal context, one entry of $U$ counts the number of animals seen on the first three capture occasions but misidentified on the third occasion.
Each such animal is counted twice in $X$ (as one animal seen on the first two occasions only, and as another ``ghost'' animal seen on the third occasion only) so the corresponding column of $A$ contains two ones in the appropriate rows, with zeroes otherwise.

There is no natural way to model the observed vector of counts, $X$, directly.
For instance, individuals that contribute to more than one count will tend to make the entries of $X$ positively correlated in parameter-dependent ways, and it is difficult to produce simple models for positively correlated integer-valued vectors.
However, the latent vector $U$ can be modelled naturally, whereupon the distribution of $X$ is specified indirectly via $X=AU$.

The authors model $U$ by the Multinomial$(N,\vec{\rho})$ distribution, with the parameter $N$ representing the unknown overall population size and $\vec{\rho}=\vec{\rho}(\theta')$ representing the probabilities of different latent capture histories, in terms of the other underlying model parameters encoded in $\theta'$.
Models of this kind are typically intractable.
Although each entry of $X$ has a Binomial distribution, the multivariate distribution of $X$ itself is not Multinomial; this arises because of entries of $U$ that contribute to more than one entry of $X$, or equivalently because of columns of $A$ with more than one non-zero entry.
In principle, the likelihood $\P(X_\theta=x)$ can be calculated exactly as a finite sum over non-negative-integer vectors $u$ satisfying $Au=x$, but the number of terms grows very rapidly and this option quickly becomes impractical.
The authors therefore compute saddlepoint MLEs and compare their findings with other estimation techniques.
Note that in practice, some adaptation of the saddlepoint method is necessary when $X$ includes some zero counts; see \cite[section~2]{ZhaBraFew2019}.

To relate their findings to the results of this paper, note that $U$ is the sum of $N$ i.i.d.\ indicator vectors encoding the latent capture histories, so that $X$ is also the sum of $N$ i.i.d.\ vectors and the standard asymptotic regime \eqref{SAR} applies with $n=N$. 
Here the value of $N$ is \emph{unknown} -- indeed it is the main parameter whose value is to be inferred from data.
The inferential setup of Theorems~\ref{T:GradientError}--\ref{T:IntegerValued} does not apply to the Multinomial formulation in \cite{ZhaBraFew2019}, because the unknown parameter $N$ is discrete rather than continuous (although in practice it is common to optimise over $N$ as if it were a continuous variable).
Nevertheless, our results suggest that inference on $N$ will have little approximation error, particularly if $N$ is large.

This is consistent with the authors' empirical findings.
In one model they consider, a combinatorial reformulation \cite{ValFewCarPat2014} allows the true likelihood to be computed as a (much smaller) finite sum.
Then the true and saddlepoint MLEs can be compared directly, and the authors find that the true and saddlepoint estimates are virtually indistinguishable \cite[section~3]{ZhaBraFew2019}.
For instance, their Figure~1 illustrates the conclusion of \refthm{SamplingMLE}: the joint sampling distribution of the true and saddlepoint MLE lies so close to the diagonal that the approximation error is invisible compared to the scale of sampling variability.
Their Figure~2, meanwhile, illustrates that the accuracy, or otherwise, of the saddlepoint approximation to the likelihood is not a direct guide to the accuracy of the saddlepoint MLE.
In another model \cite[section~4]{ZhaBraFew2019}, no exact calculation is available, and previous analysis employed an ad-hoc estimator for $N$ and a quasi-likelihood approach for the other model parameters.
The saddlepoint MLE gives virtually the same estimates for the other parameters \cite[Table~1 and Figure~4]{ZhaBraFew2019}, and also yields an estimator for $N$ that shows improved confidence interval coverage properties in a simulation study \cite[Figure~4 and section~4.4]{ZhaBraFew2019}.

We can adapt the model from \cite{ZhaBraFew2019} to the inferential setup of Theorems~\ref{T:GradientError}--\ref{T:IntegerValued} by recasting the discrete parameter $N$ as a random variable with a continuous parameter, $N\sim\Poisson(\lambda)$.
This is a common alternative assumption in capture-recapture models: $\lambda$ represents the population size intensity, and if $\lambda$ is large then $\lambda\approx N$.
Then the entries of $U$ have independent Poisson distributions with parameter vector $\lambda \vec{\rho}$.
If $\lambda$ is large and the entries of $\vec{\rho}$ are bounded away from 0, the counts in $X$ will all be large on a common scale $\lambda$.

Given an observed vector $x$ of counts, we can match the setup of \eqref{SAR} by setting $n=\max_i x_i$, so that $y=x/n$ is of order 1.
Make the substitution $\lambda=n\tilde{\lambda}$.
Following the notation of \refexample{X=AU}, if we set $V$ to be the vector of independent Poisson distributions with parameter vector $\tilde{\lambda}\vec{\rho}$ and set $U$ to be the sum of $n$ i.i.d.\ copies of $V$, then the entries of $U$ will have independent Poisson distributions with parameter vector $n\tilde{\lambda}\vec{\rho}=\lambda \vec{\rho}$, as desired.
In this rescaling, the the relevant values of $\tilde{\lambda}$ will be of order 1, and the parameter vector $\theta=\smallmat{\tilde{\lambda}\\ \theta'}$ does not scale with $n$.
Thus the framework of \eqref{SAR} applies.

For all the models the authors consider, the analogue of \eqref{DecayBound} holds because $\P(X_\theta=x)>0$ for all $x\in\set{0,1}^{\xdim\times 1}$, $\theta\in\thetadomain$, and for all $n$ sufficiently large, as in the discussion following \refthm{IntegerValued}.
Likewise \eqref{GrowthBound} holds because $\vec{\rho}(\theta')$ depends smoothly on the other parameters $\theta'$.
In particular, the integer-valued analogue of \refthm{GradientError} applies.

To show that Theorems~\ref{T:MLEerror}--\ref{T:SamplingMLE} apply, it suffices to verify \eqref{y0K0's0}--\eqref{RateFunctionImplicitHessianDefinite}.
That is, for given $\theta_0$ and $s_0$ satisfying \eqref{RateFunctionImplicitCriticalPoint}, we should determine whether the matrix $H$ from \eqref{RateFunctionImplicitHessianDefinite} is negative definite.
If it is, let $y_0$ be determined by \eqref{y0K0's0}, and then Theorems~\ref{T:MLEerror}--\ref{T:SamplingMLE} apply for the triple $(s_0,\theta_0,y_0)$.

For simplicity, we might choose to restrict our attention to the well-specified case $s_0=0$, as in \refthm{SamplingMLEWellSpecified} and \refsubsubsect{WellSpecified}.
Then it suffices to verify \eqref{MeanGradientFullRank}, and by the discussion in \refexample{X=AU} this amounts to verifying whether the vector $A\vec{\rho}(\theta'_0)$ and the vectors $A\frac{\partial\vec{\rho}}{\partial\theta'_j}(\theta'_0)$ (over all choices of parameter entries $\theta'_j$ other than $\tilde{\lambda}$) are linearly independent.
If they are, then Theorems~\ref{T:MLEerror}--\ref{T:SamplingMLEWellSpecified} apply for this choice of $\theta_0$.

The conditions just outlined are model-specific, and verifying them analytically becomes rather intricate even in the well-specified case and is not attempted here. 
However, the conditions can be investigated numerically, either for specific choices of $s_0$, $\theta_0$ or by searching the parameter space to find problematic examples.
Preliminary calculations suggest that in the models discussed here, \eqref{RateFunctionImplicitHessianDefinite} holds for \emph{all} choices of $\theta_0$ in the well-specified case $s_0=0$, and at least for a reasonably wide range of choices of $(s_0,\theta_0)$ beyond the well-specified case.
For the purposes of the present discussion, we may safely assume that Theorems~\ref{T:MLEerror}--\ref{T:SamplingMLEWellSpecified} apply.

By Theorems~\ref{T:MLEerror}--\ref{T:SamplingMLEWellSpecified}, the difference between the true and saddlepoint MLEs for the rescaled parameter $\tilde{\lambda}$ is of asymptotic order $1/n^2$.
(Strictly speaking, the results in this paper refer only to MLEs within a small neighbourhood, leaving open the possibility that the global true or saddlepoint likelihoods could have multiple competing local maxima or no global maximum.
However, the authors make no mention of any such problem, and we will ignore the issue for the present discussion.)
Reversing the rescaling $\lambda=n\tilde{\lambda}$, we conclude that the relative approximation error for the population size will be of order $1/n^2$.
This is small compared to the inferential uncertainty, which is of order $1/\sqrt{n}$ on the same relative scale.

Finally we remark that in the capture-recapture-type experiments described here, we naturally obtain only a single observation, rather than a sample of many i.i.d.\ observations.
The informational content of the experiment comes from having large observed counts, rather than many counts.
In this context, the natural way to evaluate the model's asymptotic properties would seem to be \eqref{SAR} rather than the usual large-sample limit.
In a similar way, the natural criterion for assessing the model's asymptotic identifiability would seem to be not the usual non-singularity of the Fisher information matrix, but rather the condition \eqref{RateFunctionImplicitHessianDefinite} from Theorems~\ref{T:MLEerror}--\ref{T:SamplingMLE}.
\end{example}

\begin{example}[Branching processes]\lbexample{DavHauKra}
Davison, Hautphenne \& Kraus \cite{DavHauKraParameterLinearBirthDeath} consider a continuous-time birth-death branching process $Z(t)$ that is observed at discrete times $t_0<t_1<\dotsb<t_k$.
The model parameters are the underlying per-individual birth and death rates.
The initial population is treated as fixed, and using the remaining $k$ observations the authors numerically compute true and saddlepoint MLEs for the per-individual birth and death rates.
They also compare the resulting MLEs to parameter estimates based on classical estimators for the mean and variance of Galton-Watson offspring distributions (which apply only when inter-observation intervals have equal length) and to a new estimator obtained via the saddlepoint approximation.

In more detail, the process $Z(t)$ is the continuous-time discrete-state Markov chain which, from state $i\in\Z_+$, either increases by 1 at rate $\lambda i$, or decreases by 1 at rate $\mu i$.
The parameters $\lambda$ and $\mu$ are the birth and death rates, respectively, per individual per unit time, and are encoded in the parameter vector $\theta$ (though we postpone the decision about precisely how).
We require both $\lambda$ and $\mu$ to be positive.
The population size is measured at $k+1$ deterministic times $t_0<t_1<\dotsb<t_k$, resulting in observed data $z_0,z_1,\dotsc,z_k$, modelled by the random $(k+1)$-dimensional vector $(Z_\theta(t_0),Z_\theta(t_1),\dotsc,Z_\theta(t_k))$.
Throughout, the initial population size is taken as fixed, so that all distributions are understood to be conditional on the event $\set{Z_\theta(t_0)=z_0}$, and the observed population sizes are all assumed to be positive.

This model is a Markov chain, and the univariate saddlepoint approximation can be applied to each of the transition probabilities.
To express their univariate setup in our notation, let $Y_{\theta,j}$ have the distribution of $Z_\theta(t_j)$ conditional on $Z_\theta(t_{j-1})=1$.
Let $Y^{(i)}_{\theta,j}$ denote copies of $Y_{\theta,j}$, independently across all $i$ and $j$, and set 
\begin{equation}
X_{\theta,j,n_j}=\sum_{i=1}^{n_j} Y^{(i)}_{\theta,j}.
\end{equation}
Then, by the branching property, $X_{\theta,j,n_j}$ has the distribution of the population size started from $n_j$ individuals.
If we set $n_j=z_{j-1}$, the likelihood can be factored as
\begin{align}
\P_\theta \! \condparentheses{Z(t_j)=z_j,j=1,\dotsc,k}{Z(t_0)=z_0} &= \prod_{j=1}^k \P_\theta \! \condparentheses{Z(t_j)=z_j}{Z(t_{j-1})=z_{j-1}} 
\notag\\&
= \prod_{j=1}^k \P(X_{\theta,j,n_j}=z_j)
.
\label{DHKFactorisation}
\end{align}
The authors apply the saddlepoint approximation to each factor $\P(X_{\theta,j,n_j}=z_j)$ separately.
In our notation, this amounts to forming $\vec{X}_\theta = (X_{\theta,j,n_j})_{j=1}^k$ and setting $\vec{x}=(z_j)_{j=1}^k$, $n_j=z_{j-1}$.
Note that for $1\leq j<k$, each $z_j$ appears twice, once as $n_{j+1}$ and once as $x_j$.
However, this affects neither the correctness of \eqref{DHKFactorisation}, which is an instance of the Markov property, nor the calculation of the MLEs $\theta_{\MLE}(\vec{x}),\hat{\theta}_{\MLE}(\vec{x})$ for a fixed but arbitrary $\vec{x}$.
Indeed, the right-hand side of \eqref{DHKFactorisation} is a well-defined likelihood whether or not $x_j=n_{j+1}$.

The authors conduct extensive numerical calculations to assess the accuracy of their approximations.
They find that the approximation error in the MLE is moderate even with small population sizes, and can be improved further -- to the point that the estimates become virtually indistinguishable -- by approximating instead distributions conditioned on non-extinction.
Furthermore, the saddlepoint approximation delivers a dramatic increase in speed, with computation times 50-150 times faster or more \cite[Table~2]{DavHauKraParameterLinearBirthDeath}.
This increase becomes even more marked when the population size increases: exact calculation becomes increasingly expensive, whereas the saddlepoint approximation can be computed in constant time \cite[Web Figure~2]{DavHauKraParameterLinearBirthDeath}.
The authors also compute approximate MLEs based on the normal approximation, as in \refthm{MLEerrorNormalApprox}.
They show that in the case of equal inter-observation intervals, these approximate MLEs coincide with non-likelihood-based estimators based on classical moment estimators for general Galton-Watson branching processes.
Moreover, even for general inter-observation intervals for which the classical estimators do not apply, they show that approximate MLEs based on the normal approximation satisfies consistency and asymptotic normality in a large-sample limit \cite[section~5, Lemma~2 and Theorem~3]{DavHauKraParameterLinearBirthDeath}.

To match with the setup of this paper, note that \eqref{SAR} applies to each each $X_{\theta,j,n_j}$ separately, with $n$ replaced by $n_j$.
It is interesting to note that the values $n_j$ are determined by the data \emph{values}, not the sampling mechanism or the size of the dataset.

Because the $n_j$ vary, we will use the framework of \refexample{MultipleSamplesWorkings}, with $\xdim_0=1$, $\xdim=k$.
For convenience and to match with \refexample{MultipleSamplesWorkings}, we will henceforth restrict our attention to the case where the inter-observation intervals are equal, $t_j-t_{j-1}=t$ for all $j$, so that all the summands $Y^{(i)}_{\theta,j}$, $j=1,\dotsc,k$, have a common parametric distribution, which we can interpret as an offspring distribution for a Galton-Watson branching process.
(Our analysis, however, should apply without significant change to the general case.)
The scaling \eqref{vecnnvecbeta} amounts to the assumption that all observed population sizes $z_j$ are of the same large order; this is what we would anticipate observing if the initial population size $z_0$ became large, with other quantities remaining fixed.

From \refexample{MultipleSamplesWorkings}, it is enough to study the univariate CGF $K_0$ associated to $Y_{\theta,j}$.
The authors show \cite[section~2]{DavHauKraParameterLinearBirthDeath} that $Y_{\theta,j}$ has a modified Geometric distribution: by time $t$, the population is extinct with probability $\alpha$ and otherwise follows a Geometric distribution (taking values in $\N$) with success probability $1-q$, where
\begin{equation}\label{alphaqFormula}
\alpha=\frac{\mu (e^{(\lambda-\mu) t}-1)}{\lambda e^{(\lambda-\mu) t}-\mu}, \qquad q = \frac{\lambda(e^{(\lambda-\mu)t}-1)}{\lambda e^{(\lambda-\mu) t}-\mu}
.
\end{equation}
Note that $\alpha$, $q$ depend on $\theta$ and $t$, and we suppress this dependence from our notation.
When $\lambda=\mu$, different formulas apply, which amount to extending the expressions in \eqref{alphaqFormula} by continuity at their removable singularities; here and elsewhere, we shall omit this step.

(We remark that the modified Geometric family of distributions is an exponential family.
For this two-parameter family, a bivariate sufficient statistic is given by $(Y, \indicator{Y=0})$, rather than $Y$ alone, so this example does not fall within the scope of \refsubsubsect{ExpFamily}.
Changing the summands from $Y$ to $(Y, \indicator{Y=0})$ would mean that, in addition to knowing the population sizes at each observation time, we would know the number of individuals who, since the previous observation time, died without giving birth to new children.
This is a quite different and simpler inferential setup; cf.\ the discussion in \cite[section~1]{DavHauKraParameterLinearBirthDeath}.)

The corresponding MGF and CGF for the modified Geometric distributions are
\begin{equation}\label{DHKM0K0}
\begin{gathered}
M_0(s;\theta) = \alpha + (1-\alpha)\frac{(1-q)e^s}{1-q e^s}
,
\\
K_0(s;\theta) = \log\left( \alpha+(1-q-\alpha)e^s \right) - \log\left( 1-qe^s \right)
,
\end{gathered}
\end{equation}
defined for $s < -\log q$.
Because $M_0$ depends smoothly on $\theta$, via $\alpha$ and $q$ from \eqref{alphaqFormula}, and because the probabilities $\P(Y_{\theta,j}=0)$ and $\P(Y_{\theta,j}=1)$ are both non-zero, the condition \eqref{GrowthBound} and the analogue of the condition \eqref{DecayBound} needed for \refthm{IntegerValued} hold.
We find $\mathcal{S}_\theta=(-\infty,-\log q)$, $\mathcal{X}_\theta=(0,\infty)$ and $\mathcal{X}^o=\mathcal{X}=(0,\infty)\times\thetadomain$, and indeed the authors give an explicit formula for $\hat{s}(\theta,x)$ defined for all $\theta\in\thetadomain$ and $x\in(0,\infty)$; see \cite[Lemma~1]{DavHauKraParameterLinearBirthDeath}.

As in \refexample{MultipleSamplesWorkings}, it follows that the conclusions of Theorems~\ref{T:GradientError}--\ref{T:IntegerValued} apply to $\vec{X}_\theta$.
\refthm{GradientError} applies without further hypotheses, and to assess the applicability of Theorems~\ref{T:MLEerror}--\ref{T:MLEerrorNormalApprox} we should next determine whether \eqref{y0K0's0}--\eqref{RateFunctionImplicitHessianDefinite}, or \eqref{thetaSplitting} and \eqref{xi0Constraint}--\eqref{EijFormula}, apply for particular choices of $\vec{y}_0$, using the simplifications to these conditions derived in \refexample{MultipleSamplesWorkings}.

We will primarily consider values for $\vec{y}_0$ that correspond to the well-specified case $\vec{s}_0=0$, as in \refsubsubsect{WellSpecified}.
We compute
\begin{equation}\label{DHKMeanFormula}
K_0'(0;\theta) = e^{(\lambda-\mu)t}
\end{equation}
and, as in \eqref{SEMultipleSamples}, the condition $\vec{y}_0=\vec{K}_0'(0;\theta_0)$ from \eqref{y0K0's0} reduces to $y_{0,j}=\beta_{0,j} e^{(\lambda_0-\mu_0)t}$ for all $j$.
Thus $\vec{s}_0=0$ means that, when $n$ is large, we should have
\begin{equation}
e^{(\lambda_0-\mu_0)t} = K_0'(0;\theta_0) = \frac{y_{0,j}}{\beta_{0,j}} \approx \frac{y_j}{\beta_j} = \frac{x_j}{n\beta_j} = \frac{x_j}{n_j} = \frac{z_j}{z_{j-1}} \quad\text{for }j=1,\dotsc,k.
\end{equation}
In other words, in the well-specified case with $n$ large, we should expect observed population sizes that change by a fixed common factor $e^{(\lambda_0-\mu_0)t}$.
Moreover, by choosing $\theta_0$ appropriately, we can arrange for the factor $e^{(\lambda_0-\mu_0)t}$ to take on any positive value.
(More generally, if the inter-observation intervals are not of equal length, the condition \eqref{y0K0's0} with $\vec{s}_0=0$ amounts to the condition that the observed population sizes should grow exponentially in time with some fixed exponential growth parameter.)
It seems reasonable to restrict ourselves to this case: any branching process model with finite offspring mean, even if it is not the birth-death model described here, will produce population sizes following the same pattern in the limit $n\to\infty$.

As discussed in \refexample{MultipleSamplesWorkings}, the condition \eqref{RateFunctionImplicitHessianDefinite} applied to $\vec{K}$ reduces when $\vec{s}_0=0$ to the condition that $\grad_s\grad_\theta K_0(0;\theta_0)$ should have rank $p$.
This condition fails: the matrix $\grad_s\grad_\theta K_0(0;\theta_0)$ is $\xdim_0\times p = 1\times 2$ and therefore cannot have rank $p=2$.
The model is therefore not fully identifiable at the level of the sample mean in the well-specified case.
(However, when $\vec{s}_0$ is not identically zero, a convexity argument shows that $\vec{H}$ from \eqref{vecAvecBvecH} is negative definite, so that \eqref{RateFunctionImplicitHessianDefinite} holds outside of the well-specified case.)

We therefore turn to Theorems~\ref{T:MLEerrorPartiallyIdentifiable}--\ref{T:MLEerrorNormalApprox} and seek to split the parameter vector as in \eqref{thetaSplitting}.
Recalling \eqref{DHKMeanFormula}, we can choose
\begin{equation}
\omega = \lambda - \mu,
\end{equation}
the Malthusian parameter for this model.
Various choices for $\nu$ are possible, such as $\nu=\mu$ or $\nu=\lambda$ or even, when the inter-observation intervals are equal, $\nu=(\lambda+\mu)\frac{e^{\omega t}-1}{\omega}e^{\omega t}$, which gives the variance $K_0''(0;\theta)$ of the offspring distribution.
We shall select 
\begin{equation}
\nu=\lambda+\mu, 
\end{equation}
the total rate of branching events per individual per unit time, which is closely related to the offspring variance.
Note however that there is essentially no freedom in the choice of $\omega$, which must be an invertible function of $K_0'(0;\theta)=e^{(\lambda-\mu)t}$.

With this choice of $\theta=\smallmat{\omega\\ \nu}$, we have $p_1=p_2=1$ and the condition that the rates $\lambda$, $\mu$ are both positive means that we set
\begin{equation}\label{DHKRFormula}
\thetadomain=\set{\theta\in\R^{2\times 1}\colon -\nu < \omega < \nu} = \set{\theta\in\R^{2\times 1}\colon \nu > \abs{\omega}}.
\end{equation}
Using the notation of \refexample{MultipleSamplesWorkings}, we compute
\begin{equation}
A = \nu_0 \frac{e^{\omega_0 t}-1}{\omega_0}e^{\omega_0 t}, \quad B_\omega = te^{\omega_0 t}, \quad Q_1 = \frac{e^{\omega_0 t}-1}{\omega_0}e^{\omega_0 t}, \quad Q_{11} = 0
,
\end{equation}
and as in \refexample{MultipleSamplesWorkings} we find $J=0$ since $\xdim_0=p_1=1$.
In particular, the scalar $B_\omega$ is non-zero; that is, $B_\omega$ is an $\xdim_0\times p_1=1\times 1$ matrix of rank $p_1=1$, as required for Theorems~\ref{T:MLEerrorPartiallyIdentifiable}--\ref{T:MLEerrorNormalApprox}.

The partially linearised model $\vec{\Xi}_{w,\nu}$ has independent entries $\Xi^{(j)}_{w,\nu} \sim \mathcal{N}(c \beta_j w, c' \beta_j \nu )$ for certain constants $c,c'$ depending only on $\omega_0$ and $t$.
Here the effect of the parameters splits, with $w$ affecting only the mean and $\nu$ affecting only the variance.
We can foresee that to have an MLE $(w_0,\nu_0)$ for $\vec{\Xi}_{w,\nu}=\xi_0$, we will need $\nu_0$ to make the model variances match a suitable population variance coming from $\xi_0$.
The factors $\beta_j$ complicate the calculation, however, and to continue we will use the explicit workings from \refexample{MultipleSamplesWorkings}.

Since $Q_1$ is a non-zero scalar we may cancel it from \eqref{xi0ConstraintReduced} to obtain
\begin{equation}
\sum_{\ell=1}^k \frac{\xi_{0,\ell}^2}{\beta_{0,\ell}} - \frac{(\sum_{\ell=1}^k \xi_{0,\ell})^2}{\sum_{\ell=1}^k \beta_{0,\ell}} = kA
\end{equation}
or, after some rearranging,
\begin{equation}\label{DHKxi0ConditionRearranged}
\nu_0 \frac{e^{\omega_0 t}-1}{\omega_0}e^{\omega_0 t} = \frac{\sum_{\ell=1}^k \beta_{0,\ell}}{k} \sum_{i=1}^k \frac{\beta_{0,i}}{\sum_{\ell=1}^k\beta_{0,\ell}} \left( \frac{\xi_{0,i}}{\beta_{0,i}} - \frac{\sum_{j=1}^k \xi_{0,j}}{\sum_{\ell=1}^k \beta_{0,\ell}} \right)^2
.
\end{equation}
This is the condition that the offspring variance should equal a multiple of an appropriately weighted population variance constructed out of the values $\xi_{0,i}/\beta_{0,i}$, $i=1,\dotsc,k$.
If $k>1$ and the values $\xi_{0,i}/\beta_{0,i}$ are not identically constant, then the right-hand side of \eqref{DHKxi0ConditionRearranged} is positive and we can always find $\nu_0>0$ satisfying \eqref{DHKxi0ConditionRearranged}.
However, the point $\theta_0=\smallmat{\omega_0\\ \nu_0}$ will fail to belong to $\thetadomain$, see \eqref{DHKRFormula}, if $\omega_0\neq 0$ and the right-hand side of \eqref{DHKxi0ConditionRearranged} is too small.
(Roughly speaking, for any given non-zero growth rate $\omega_0$, the model posits a non-zero minimum for the possible offspring variances, given by the boundary values $\lambda=\omega_0$, $\mu=0$ if $\omega_0>0$ or $\mu=\abs{\omega_0}$, $\lambda=0$ if $\omega_0<0$.
If the weighted population variance in \eqref{DHKxi0ConditionRearranged} does not exceed this minimal offspring variance, then we will be unable to solve \eqref{DHKxi0ConditionRearranged} with $\theta_0\in\thetadomain$.
The existence of vectors $\xi_0$ for which this occurs indicates that, even in the limit $n\to\infty$, it remains possible to observe data patterns that push the MLE to the boundary of $\thetadomain$.
This possibility would be ruled out, however, if in addition $k\to\infty$.)

Either way, a quick calculation shows that if \eqref{DHKxi0ConditionRearranged} has a solution then the scalar $\vec{E}=E_{11}$ is negative, as required for \refthm{MLEerrorPartiallyIdentifiable}.
Moreover, as discussed in \refexample{MultipleSamplesWorkings}, $\xdim_0=p_1=1$ implies that the condition \eqref{PIOffDiagonalVanishes} holds identically.

All in all, if we can solve \eqref{DHKxi0ConditionRearranged} for given $\xi_0$ with $\theta_0\in\thetadomain$, then the conclusions of Theorems~\ref{T:MLEerrorPartiallyIdentifiable}--\ref{T:MLEerrorNormalApprox} apply whenever $n$ is sufficiently large and the observed values are sufficiently close to the base points, in the manner specified by \eqref{PIChangeOfVariables}.
In particular, the saddlepoint MLE introduces an approximation error that is $O(1/n^2)$ for $\omega$ and $O(1/n)$ for $\nu$, while the approximate MLE based on the normal approximation introduces an error $O(1/n)$ for $\omega$ and $O(1/\sqrt{n})$ for $\nu$.

Restrict attention now to the saddlepoint approximation and return to the rates $\lambda = \tfrac{1}{2}(\nu+\omega)$ and $\mu=\tfrac{1}{2}(\nu-\omega)$.
Note that the approximation error for $\nu$, of size $O(1/n)$, swamps the approximation error for $\omega$, of size $O(1/n^2)$, giving an overall approximation error $O(1/n)$ for both $\lambda$ and $\mu$.
This reflects an intrinsic difference between the parameter $\omega$ and other parameters, rather than a feature of the saddlepoint approximation per se.
Indeed, the authors' theoretical results for estimators based on the normal approximation (or equivalently, as they discuss, based on classical Galton-Watson moment estimators) illustrate that different scaling results apply to estimators for $\omega$ compared to other parameters.
Roughly speaking, they show that the estimator for $\omega$ is asymptotically normal on a spatial scale $1/\sqrt{nk}$, while the estimators for $\lambda$ and $\mu$ are asymptotically normal on a spatial scale $1/\sqrt{k}$; see \cite[Theorems~1--2]{DavHauKraParameterLinearBirthDeath}.
Note that their results, which concern only sampling variability and not approximation error, apply to a different limiting framework than the results of this paper: they consider the limit $k\to\infty$, which for $\omega>0$ implies that also $n\to\infty$ with $n\approx e^{k \omega t}$.
In particular, $n\gg k$, so that the estimator for $\omega$ becomes dramatically more accurate than the separate estimates for $\lambda$ and $\mu$.

We remark that in the limit of long time series, $k\to\infty$, this model exhibits inhomogeneity: the numbers of summands $n_j$, $j=1,\dotsc,k$, can no longer be expected to have a common scale $n$.
Indeed, whenever $\omega\neq 0$, we would expect $\frac{\max_j n_j}{\min_j n_j}\approx e^{k\abs{\omega} t} \to\infty$ as $k\to\infty$.
Our results do not apply in such a limit, which has been excluded from consideration in this paper.
However, the heuristics of the proofs suggest that in such a limit, the common scaling parameter $n$ must be replaced by two distinct scaling parameters corresponding to $\sum_{j=1}^k n_j$ and $\sum_{j=1}^k 1/n_j$.

A situation where some but not all $n_j$'s tend to infinity may nevertheless be tractable.
For instance, the authors argue \cite[Web Appendix~B]{DavHauKraParameterLinearBirthDeath} that if the $n_j$ grow exponentially (which is what we would expect for a supercritical branching process, on the event of survival) then the approximation error in the likelihood remains bounded even in the limit $k\to\infty$.
We remark that results from this paper, particularly \refprop{PddtPError}, provide the missing uniform estimates needed for their argument; see \cite[assumption (A) in Web Appendix~B]{DavHauKraParameterLinearBirthDeath}.

Finally we mention another approximation that the authors consider.
The likelihood in \eqref{DHKFactorisation} can be approximated instead by applying the saddlepoint approximation to the conditional probability $\condP{X_{\theta,j,n_j}=z_j}{X_{\theta,j,n_j}\neq 0}$ (this is possible because both $\P(X_{\theta,j,n_j}=0)$ and the conditional CGF for $X_{\theta,j,n_j}$ conditioned to be non-zero can be calculated, although the resulting saddlepoint equation must be solved numerically).
The authors call the resulting approximation the adjusted saddlepoint approximation.
Heuristically, the adjustment accounts for the fact that $X_{\theta,j,n_j}$ may have an unusual point mass at 0, corresponding to the fact that 0 is an absorbing state for the Markov chain $Z(t)$ (indeed, for $\omega>0$ and $t\to\infty$, the distribution of $Z(t)$ is a non-trivial mixture of a point mass at 0 together with a distribution which, if rescaled by $e^{\omega t}$, converges to a limiting density).
The authors' numerical results confirm that the adjusted saddlepoint approximation gives higher accuracy than the ordinary saddlepoint approximation for this model. 
However, the results of this paper do not apply to the adjusted saddlepoint approximation because conditioning on $X\neq 0$ brings us out of the setting of \eqref{SAR}.
\end{example}

\begin{example}[INAR$(p)$ model]\lbexample{PedDavFok}
Pedeli, Davison \& Fokianos \cite{PedDavFok2015} fit time series data to an integer-valued autoregressive process given by $Z_i\sim\Binomial(Z_{i-1},q)+\epsilon_i$.
That is, the process undergoes Binomial thinning at each step, counteracted by i.i.d.\ integer-valued ``innovations'' $\epsilon_i$, and the parameter vector $\theta$ specifies the thinning probability $q$ as well as any parameters for the distribution of the $\epsilon_i$'s.
As in \refexample{DavHauKra}, the authors apply the saddlepoint approximation to the one-step transition probabilities $\P_\theta \! \condparentheses{Z_i=z_i}{Z_{i-1}=z_{i-1}}$, compute the saddlepoint MLE by maximising the resulting approximate likelihood, and study the accuracy (which is high compared to the sampling variability, see \cite[supplementary materials, Figures~SM10--SM11]{PedDavFok2015}) via simulation.

In our notation, for each step, we take $X$ to have the conditional distribution of $Z_i$ given $Z_{i-1}=z_{i-1}$, and $x=z_i$; and these univariate saddlepoint approximation are multiplied across steps as in \eqref{DHKFactorisation}.
This model does not quite fall into the standard asymptotic regime because of the innovations $\epsilon_i$: each $X$ is the sum of $z_{i-1}+1$ independent summands, but only $z_{i-1}$ of them are identically distributed.
If the innovations were absent (or if $\epsilon_i\sim\Binomial(k,q)$ with the same parameter $q$) then the setup of \eqref{SAR} would apply to each step, with $n=z_{i-1}$ (or $n=z_{i-1}+k$) and $Y$ having the Bernoulli$(q)$ distribution.
As in \refexample{DavHauKra}, we could adopt the setup of \refexample{MultipleSamplesWorkings} by using vectors $\vec{X}$, $\vec{x}$, $\vec{n}$, with $x_i=z_i$ and $n_i=z_{i-1}$.
Note that, as in \refexample{DavHauKra}, these values $n_i$ would be determined by the time series \emph{data values} that are being fit to the model, not the sample size.

In fact, the authors' main interest is in autoregressive processes of order $p\geq 2$, in which $Z_i$ is a sum based on the preceding $p$ values $Z_{i-1},Z_{i-2},\dotsc,Z_{i-p}$, all Binomially thinned with respective probabilities $q_1,\dotsc,q_p$, together with independent innovations.
Although its purpose is different, we might liken this model to that of \refexample{DavHauKra} by describing it as a multitype branching process with immigration.
The likelihood still has a factorisation across steps, similar to \eqref{DHKFactorisation}, and a univariate saddlepoint approximation can be applied at each step.
The random variables $X$ will now be a sum of $p+1$ independent terms: the innovation $\epsilon_i$, and the Binomial contributions from the preceding values (each of which individually follows the setup of \eqref{SAR}, for various values of $n$).
The results in this paper therefore provide a starting point for the analysis of this more elaborate scenario, and it is reasonable to expect similar results to apply.
\end{example}

\subsection{Specific examples from theory}

\begin{example}[Poisson distribution]\lbexample{Poisson}
Let $X_\lambda\sim\Poisson(\lambda)$, corresponding to
\begin{equation}
\begin{gathered}
K(s;\lambda) = \lambda (e^s-1), \quad \hat{s}(\lambda,x) = \log(x/\lambda),
\\
\hat{L}(\lambda;x) = \hat{f}(x;\lambda) = \frac{\exp(\lambda(x/\lambda-1) - x\log(x/\lambda))}{\sqrt{2\pi \lambda(x/\lambda)}} = e^{-\lambda}\lambda^x\frac{(e/x)^x}{\sqrt{2\pi x}}.
\end{gathered}
\end{equation}
Note that $\hat{f}$ is well-defined for $x>0$ even though $X$ does not have a density, and $\hat{f}(x;\lambda)$ is a reasonable approximation to $\P(X_\lambda=x)$ for $x\in\N$, corresponding to Stirling's approximation to $x!$ with relative error of order $1/x$.
Note also that as an approximation to the likelihood, the saddlepoint approximation is essentially exact for all $x>0$ since it has the form $c(x)\P(X_\lambda=x)$ where $c(x)$ depends on $x$ but not $\lambda$.
In particular, the saddlepoint MLE is exact for $x>0$, in accordance with the reasoning in \refsubsubsect{ExpFamily}.

If instead we write $Y_\theta\sim\Poisson(\theta)$ and $X=\sum_{i=1}^n Y^{(i)}_\theta$ in the setup of \eqref{SAR}, then $X\sim\Poisson(n\theta)$.
This amounts to a reparametrisation $\lambda=n\theta$, and if we expect an observation $x$ of order $n$ we can set $x=ny$ with $y>0$ of constant order, as in \eqref{SAR}.
Indeed the relative error in the likelihood, for an observed value $x=ny$, is of order $O(1/x)=O(1/ny)=O(1/n)$ for fixed $y>0$, as in \eqref{LErrorRough}.

The normal approximation for this model is $\tilde{X}_\lambda\sim\mathcal{N}(\lambda, \lambda)$.
To compute the MLE for an observed value $x$ under the normal approximation, as in \refthm{MLEerrorNormalApprox}, we solve for $\lambda>0$ in
\begin{equation}
0 = \frac{\partial}{\partial\lambda}\left( -\frac{(x-\lambda)^2}{2\lambda} - \frac{\log(2\pi\lambda)}{2} \right) = -\frac{\lambda^2+\lambda-x^2}{2\lambda^2},
\end{equation}
leading to
\begin{equation}
\tilde{\lambda}_\MLE(x) = \frac{\sqrt{4x^2+1}-1}{2} = \sqrt{x^2+1/4}-1/2
\end{equation}
provided $x\neq 0$.
If we substitute $x=ny$, $\lambda=n\theta$, then the normal approximation MLE for $\theta$ is $\tilde{\theta}_\MLE(x,n)=\sqrt{y^2+1/4n^2}-1/2n$.
If $y>0$, then the true MLE for $\theta$ is exactly $\theta_\MLE(x,n)=y$, so that the error in the normal approximation MLE is $O(1/n)$ in accordance with \refthm{MLEerrorNormalApprox}.
Note however that the normal approximation also leads to a well-behaved MLE for $x<0$, $y<0$, in complete contrast to the original Poisson model for which values $x<0$, $y<0$ are outside of the convex hull of the support.
This illustrates the remark following \refthm{MLEerrorNormalApprox} that away from the subset of $y$-values corresponding to true model means, the behaviour of the normal approximation MLE may be entirely different to that of the true MLE.
\end{example}

\begin{example}[Gamma distribution]\lbexample{Gamma}
Let $Z$ have the Gamma distribution with shape parameter $\alpha$ and rate parameter $r$.
As is well known, the Gamma family is an exponential family of distributions.
However, the details of the saddlepoint approximation, and the way in which the conclusions of \refsubsubsect{ExpFamily} apply to MLEs, vary depending on the exact choice of $X$.

Write $Y_1=Z$ and let $X_1$ be the sum of $n$ i.i.d.\ copies of $Y_1$. 
Then $X_1$ will have the Gamma$(n\alpha,r)$ distribution, and we find
\begin{equation}\label{MKhatshatLBasicGamma}
\begin{gathered}
M_{Y_1}(s;r,\alpha) = \left( \frac{r}{r-s} \right)^\alpha,
\quad
K_{Y_1}(s;r,\alpha) = \alpha \log r - \alpha\log(r-s)
\quad
\text{for $\Re(s)<r$},
\\
\hat{s}_1 = r-\frac{\alpha}{y}, 
\quad 
\hat{L}_{1,n}(r,\alpha;x) = r^{n\alpha} x^{n\alpha - 1} e^{-rx} \frac{e^{n\alpha}}{(n\alpha)^{n\alpha-1/2}\sqrt{2\pi}}
\quad
\text{for $x,y>0$}
.
\end{gathered}
\end{equation}
As a function of $r$, $\hat{L}_{1,n}(r,\alpha;x)$ is a constant multiple of the true likelihood, so using $\hat{L}_{1,n}$ to approximate the MLE for $r$ will be exact.
However, the last factor in $\hat{L}_{1,n}$ is not a constant multiple of $1/\Gamma(n\alpha)$ -- in fact it is the reciprocal of Stirling's approximation to $\Gamma(n\alpha)$ -- so using $\hat{L}_{1,n}$ to approximate the MLE for $\alpha$ will not be exact.
In the framework of \refsubsubsect{ExpFamily}, this corresponds to the observation that $X_1$ (or $Y_1$) is the sufficient statistic for the sub-family in which $r$ varies but $\alpha$ is fixed.

If we set $Y_2=\log Z$ and again let $X_2$ be the sum of $n$ i.i.d.\ copies of $Y_2$, then
\begin{equation}
M_{Y_2}(t;r,\alpha) = r^{-t}\frac{\Gamma(\alpha+t)}{\Gamma(\alpha)}, \quad K_{Y_2}'(t;r,\alpha) = \psi(\alpha+t) - \log r
\quad
\text{for $\Re(t)>-\alpha$},
\end{equation}
where $\psi(z)=\Gamma'(z)/\Gamma(z)$ denotes the digamma function.
The saddlepoint equation has no elementary solution but can be solved numerically.
Now $Y_2$ is the sufficient statistic for the sub-family in which $\alpha$ varies and $r$ is fixed, so using $\hat{L}_{2,n}$ to find the MLE for $\alpha$ will be exact.

Finally we can set $Y_3 = (Z, \log Z)^T$, the bivariate sufficient statistic for the entire Gamma exponential family:
\begin{equation}
\begin{gathered}
M_{Y_3}(s, t; r, \alpha) = \frac{r^\alpha}{(r-s)^{\alpha+t}}\frac{\Gamma(\alpha+t)}{\Gamma(\alpha)}
\qquad
\text{for $\Re(s)<r$, $\Re(t)>-\alpha$}
,
\\
\frac{\partial K_{Y_3}}{\partial s}(s, t; r, \alpha) = \frac{\alpha+t}{r-s},
\quad
\frac{\partial K_{Y_3}}{\partial t}(s, t; r, \alpha) = \psi(\alpha+t) - \log(r-s).
\end{gathered}
\end{equation}
The saddlepoint equation does not have an elementary solution, but the quantities $\hat{\alpha}=\hat{\alpha}_3 = \alpha + \hat{t}_3$ and $\hat{r}=\hat{r}_3 = r - \hat{s}_3$ will depend only on the observed value $x_3$ and not on $\alpha$ or $r$.
Specifically, if we write $x_3 = ny_3$ with $y_3=(\bar{z},\bar{\ell})^T$, then $\hat{\alpha},\hat{r},\hat{L}_{3,n}$ are determined by
\begin{equation}\label{SEGammaBivariate}
\begin{gathered}
\psi(\hat{\alpha}) - \log\hat{\alpha} = \bar{\ell} - \log\bar{z}, 
\qquad 
\hat{r} = \frac{\hat{\alpha}}{\bar{z}}, 
\\
\hat{L}_{3,n}(r,\alpha;x_3) = 
\left( \frac{r^\alpha e^{\alpha\bar{\ell}-r\bar{z}}}{\Gamma(\alpha)} \right)^n \frac{\hat{\alpha} \left[ \Gamma(\hat{\alpha}) e^{\hat{\alpha}(1-\psi(\hat{\alpha}))} \right]^n}{2\pi n\bar{z}\sqrt{\hat{\alpha}\psi'(\hat{\alpha})-1}}
.
\end{gathered}
\end{equation}
Consulting for instance \cite[equation 8.361.3]{GRToISP2007}, we have
\begin{equation}\label{GRpsiFormula}
\psi(\alpha) - \log\alpha = -\frac{1}{2\alpha} - \int_0^\infty \frac{2t\, dt}{(t^2+\alpha^2)(e^{2\pi t}-1)},
\end{equation}
and in particular we see that the function $\alpha\mapsto \psi(\alpha)-\log\alpha$ is strictly increasing and maps $(0,\infty)$ onto $(-\infty,0)$.
So \eqref{SEGammaBivariate} will have a solution if and only if $\bar{z}>0$ and $\bar{\ell} < \log\bar{z}$.

In practice, this means that the saddlepoint approximation for $X_3$ can only be applied when $n\geq 2$.
For the case $n=1$, the values of $Y_3$ will lie on the curve $\bar{\ell}=\log\bar{z}$ where $\hat{\alpha}$ is undefined, and $Y_3$ does not have a density; likewise if we take $\bar{z}\decreasesto e^{\bar{\ell}}$, we can verify that $\hat{\alpha}\to\infty$ and $\hat{L}_{3,n}\to\infty$.
On the other hand, as soon as $n\geq 2$, $X_3$ is supported in the region where $\bar{\ell}<\log\bar{z}$ (this is Jensen's inequality applied to the summands $Y_3^{(1)},\dotsc,Y_3^{(n)}$) and $\hat{L}_{3,n}$ is finite there.
As in \refsubsubsect{ExpFamily}, as soon as $n\geq 2$, using $\hat{L}_{3,n}$ to find the MLEs for both $r$ and $\alpha$ will be exact.
\end{example}

\begin{example}[Gamma subfamily with full identifiability]\lbexample{GammaFI}
Let $Y_\theta\in\R$ have the Gamma distribution with parameters $\alpha=\theta$, $r=1$, where $\theta\in(0,\infty)=\thetadomain$ and $\mathcal{S}_\theta=(-\infty,1)$ for all $\theta\in\thetadomain$.
Form $X_\theta\in\R$ by summing $n$ i.i.d.\ copies of $Y_\theta$, so that $X_\theta$ has the $\mathrm{Gamma}(n\theta,1)$ distribution.
Let $\vec{X}_\theta\in\R^{k\times 1}$ be the vector whose entries are $k$ i.i.d.\ copies of $X_\theta$.
This is the situation of \refexample{MultipleSamplesWorkings} with $n_j=n$, $\beta_j=1$ for $j=1,\dotsc,k$, and 
\begin{equation}
K_0(s;\theta) = - \theta\log(1-s)
.
\end{equation}
Note that $X$ and $Y$ reduce to $X_1$ and $Y_1$ from the first part of \refexample{Gamma}.
Here, however, $Y$ is not the sufficient statistic corresponding to the sub-family in which $r$ is fixed and $\alpha$ varies, so the logic of \refsubsubsect{ExpFamily} does not apply and, as we shall see, the MLE is not exact.

It is readily seen that $M_0$ and all its derivatives all take the form $C(1-s)^{-c_1}(\log(1-s))^{c_2}$ for various $C,c_1,c_2$ depending smoothly on $\theta$, and the bounds \eqref{DecayBound}--\eqref{GrowthBound} follow.
The Hessian from \eqref{RateFunctionImplicitHessianDefinite} is the $1\times 1$ matrix (i.e., scalar) $H=-1/\theta_0$, which is negative definite for all choices of $\theta_0$.
Consequently this model falls into the well-specified case from Theorems~\ref{T:MLEerror}--\ref{T:SamplingMLE}.

As in \refexample{MultipleSamplesWorkings}, set $x_i=n y_i$ for $i=1,\dotsc,k$, and let $\vec{x}\in\R^{k\times 1}$ be the vector with entries $x_i$.
From \eqref{MKhatshatLBasicGamma} we find that if $x_i,y_i>0$ for all $i$, then
\begin{gather}
\log\hat{L}_{\vec{X}}(\theta;\vec{x}) = nk\theta - nk\theta\log\theta + \frac{k}{2}\log\theta - \frac{k}{2}\log(2\pi n) + \sum_{i=1}^k \left( (n\theta-1)\log y_i - n y_i \right),
\notag\\
\frac{\partial}{\partial\theta}\log\hat{L}_{\vec{X}}(\theta;\vec{x}) = nk\left( -\log\theta + \frac{1}{2n\theta} + \frac{1}{k}\sum_{i=1}^k \log y_i \right)
.
\label{GammaFIgradloghatL}
\end{gather}
The lower-order approximation $\log\hat{L}^*$ from \refsubsect{LowerOrder} omits non-leading terms:
\begin{gather}
\log\hat{L}^*_{\vec{X}}(\theta;\vec{x}) = nk\theta - nk\theta\log\theta + n\sum_{i=1}^k \left( \theta\log y_i - y_i \right)
,
\notag\\
\frac{\partial}{\partial\theta}\log\hat{L}^*_{\vec{X}}(\theta;\vec{x}) = nk\left( -\log\theta + \frac{1}{k}\sum_{i=1}^k \log y_i \right)
.
\label{GammaFIgradloghatL*}
\end{gather}
On the other hand, we can easily compute the true likelihood:
\begin{gather}
\log L_{\vec{X}}(\theta;\vec{x}) = k(n\theta-1)\log n - k\log\Gamma(n\theta) + \sum_{i=1}^k \left( (n\theta-1)\log y_i - n y_i \right)
,
\notag\\
\frac{\partial}{\partial\theta}\log L_{\vec{X}}(\theta;\vec{x}) = nk\left( \log n - \psi(n\theta) + \frac{1}{k}\sum_{i=1}^k \log y_i \right)
,
\end{gather}
where again $\psi(z)=\Gamma'(z)/\Gamma(z)$ denotes the digamma function; using \eqref{GRpsiFormula}, this becomes
\begin{equation}\label{GammaFIgradlogL}
\frac{\partial}{\partial\theta}\log L_{\vec{X}}(\theta;\vec{x}) = nk\left( -\log\theta + \frac{1}{2n\theta} + \int_0^\infty \frac{2t\, dt}{(t^2+n^2\theta^2)(e^{2\pi t}-1)} + \frac{1}{k}\sum_{i=1}^k \log y_i \right)
.
\end{equation}
It is readily verified that the (global) MLEs $\hat{\theta}_\MLE(\vec{x},n)$, $\theta_\MLE(\vec{x},n)$, $\hat{\theta}^*_\MLE(\vec{x},n)$ are given by the unique values of $\theta$ for which the derivatives in \eqref{GammaFIgradloghatL}, \eqref{GammaFIgradloghatL*}, or \eqref{GammaFIgradlogL}, respectively, become zero.
The formulas \eqref{GammaFIgradloghatL} and \eqref{GammaFIgradlogL} lead to implicit equations with no elementary solution, but from \eqref{GammaFIgradloghatL*} we find
\begin{equation}\label{GammaFIhattheta*}
\hat{\theta}^*_\MLE(\vec{x},n) = \left( \prod_{i=1}^k y_i \right)^{1/k},
\end{equation}
the geometric mean of the implied sample means $y_1,\dotsc,y_k$.

Note that the bracketed terms in \eqref{GammaFIgradloghatL} and \eqref{GammaFIgradloghatL*} differ from the bracketed term in \eqref{GammaFIgradlogL} by terms of order $1/n^2$ and $1/n$, respectively.
Using the Implicit Function Theorem it follows that $\hat{\theta}_\MLE(\vec{x},n)$ and $\hat{\theta}^*_\MLE(\vec{x},n)$ differ from the true MLE $\theta_\MLE(\vec{x},n)$ by an error of order $1/n^2$ and $1/n$, respectively.
For instance, we can explicitly compare $\hat{\theta}_\MLE$ with $\hat{\theta}^*_\MLE$: from $-\log\hat{\theta}_\MLE + \frac{1}{2n\hat{\theta}_\MLE} + \log\hat{\theta}^*_\MLE = 0$ we find 
\begin{equation}\label{GammaFIhatthetavshattheta*Explicit}
\hat{\theta}_\MLE = \hat{\theta}^*_\MLE\exp\left( \frac{1}{2n\hat{\theta}^*_\MLE}\exp\left( -\frac{1}{2n\hat{\theta}_\MLE} \right) \right)=\hat{\theta}^*_\MLE + \frac{1}{2n}+O(n^{-2})
,
\end{equation}
with a difference of order $1/n$ as asserted.
These errors match with the assertions of Theorems~\ref{T:MLEerror} and \ref{T:LowerOrderSaddlepoint}\ref{item:LOMLE} and show that those asymptotic bounds are sharp in general.

We can also compare with MLE obtained from the normal approximation, as in \refsubsubsect{NormalApproxResult}.
The normal approximation for this model is the vector $\tilde{\vec{X}}_\theta$ whose $k$ entries are i.i.d.\ drawn from the $\mathcal{N}(n\theta,n\theta)$ distribution.
A computation similar to that of \refexample{Poisson} leads to
\begin{equation}\label{GammaFIMLENormalApprox}
\tilde{\theta}_\MLE(\vec{x},n) = \sqrt{\frac{1}{k} \sum_{i=1}^k y_i^2 +\frac{1}{4n^2}} - \frac{1}{2n}
.
\end{equation}

If we make the rescaling $x_i=\theta' n + \xi_i\sqrt{n}$, $y_i=\theta'+\xi_i/\sqrt{n}$ as in \eqref{PIChangeOfVariables}, with $\theta'\in(0,\infty)$ and $\vec{\xi}$ taken to be of constant order, then we find
\begin{align}
\label{GammaFIMLENormalPICV}
&\tilde{\theta}_\MLE(\vec{x},n) 
= \theta' + \frac{\overline{\xi}}{\sqrt{n}} + \frac{1}{2n\theta'}\left( \overline{\xi^2} - (\overline{\xi})^2 - \theta' \right) - \frac{\overline{\xi}}{2 n^{3/2}(\theta')^2}\left( \overline{\xi^2}-(\overline{\xi})^2 \right) + O(n^{-2})
,
\end{align}
where we have written $\overline{\xi} = \frac{1}{k}\sum_{i=1}^k \xi_i$, $\overline{\xi^2} = \frac{1}{k}\sum_{i=1}^k \xi_i^2$ for the sample moments corresponding to the sample $\xi_1,\dotsc,\xi_k$.
By way of comparison, expanding $\hat{\theta}^*_\MLE(x,n)=\theta'\exp\left( \frac{1}{k}\sum_{i=1}^k\log(1+\xi/\sqrt{n}\theta') \right)$ leads to
\begin{multline}
\hat{\theta}^*_\MLE(\vec{x},n) 
\\
= \theta' + \frac{\overline{\xi}}{\sqrt{n}} - \frac{1}{2n\theta'}\left( \overline{\xi^2} - (\overline{\xi})^2 \right) + \frac{1}{6n^{3/2}(\theta')^2}\left( 2\overline{\xi^3} - 3\overline{\xi} \, \overline{\xi^2} + (\overline{\xi})^3 \right) + O(n^{-2})
,
\end{multline}
where $\overline{\xi^3}=\frac{1}{k}\sum_{i=1}^k \xi_i^3$, and \eqref{GammaFIhatthetavshattheta*Explicit} leads to
\begin{multline}\label{GammaFIMLEPICV}
\theta_\MLE(\vec{x},n) = \hat{\theta}_\MLE(\vec{x},n) + O(n^{-2})
\\
= \theta' + \frac{\overline{\xi}}{\sqrt{n}} - \frac{1}{2n\theta'}\left( \overline{\xi^2} - (\overline{\xi})^2 - \theta' \right) + \frac{1}{6n^{3/2}(\theta')^2}\left( 2\overline{\xi^3} - 3\overline{\xi} \, \overline{\xi^2} + (\overline{\xi})^3 \right) + O(n^{-2})
.
\end{multline}
In particular, in this scaling regime, $\tilde{\theta}_\MLE(x,n)$ differs from the true MLE by a term of order $1/n$ (for fixed $k$, on the assumption that the population variance of the sample $\xi_1,\dotsc,\xi_k$ does not come unusually close to $\theta'$) in accordance with the assertions of \refthm{MLEerrorNormalApprox} about $\omega$.
(We remark that \refthm{MLEerrorNormalApprox} includes the case $p_1=p$, $p_2=0$ where the model is fully identifiable at the level of the sample mean.
So does \refthm{MLEerrorPartiallyIdentifiable}, but its conclusions are weaker than those of \refthm{MLEerror} when both theorems apply.)

Since we have explicit formulas for all values of $k$, we can examine the interplay between the limit $n\to\infty$ from \eqref{SAR} and the large-sample limit $k\to\infty$ for this model.
Suppose as in \refsubsubsect{MultipleSamples} that the entries $x_i$ are themselves random and generated independently according to the distribution $X_{\theta_0}$ with true parameter value $\theta_0$.
(Note that under this assumption, the rescaled entries $y_i$ follow the $\mathrm{Gamma}(n\theta_0,n)$ distribution, not the distribution $Y_{\theta_0}$.)
Then we naturally make the change of variables \eqref{PIChangeOfVariables} with $\theta'=\theta_0$, and in the limit $n\to\infty$ the entries $\xi_i$ will be i.i.d.\ with limiting distribution $\mathcal{N}(0,\Var(Y_{\theta'})) = \mathcal{N}(0,\theta')$.
Then \eqref{GammaFIMLENormalPICV}--\eqref{GammaFIMLEPICV} show that the sampling distributions of $\hat{\theta}_\MLE$, $\hat{\theta}^*_\MLE$, $\tilde{\theta}_\MLE$ and $\theta_\MLE$ all concentrate around the true parameter, with variability on a scale of order $1/\sqrt{nk}$.
Thus both $k\to\infty$ and $n\to\infty$ contribute to reducing the sampling variability of all the estimators.

Consider now the approximation error, where the saddlepoint-based MLE approximations $\hat{\theta}_\MLE$ and $\hat{\theta}^*_\MLE$ show different behaviour to the normal-based MLE approximation $\tilde{\theta}_\MLE$.
The approximation errors for $\hat{\theta}_\MLE$ and $\hat{\theta}^*_\MLE$ arise because of the differences of order $1/n^2$ and $1/n$ between \eqref{GammaFIgradloghatL}/\eqref{GammaFIgradloghatL*} and \eqref{GammaFIgradlogL}.
Taking $k\to\infty$ means that the random term $\frac{1}{k}\sum_{i=1}^k\log y_i$ will have smaller sampling variability, but the asymptotic sizes of the saddlepoint approximation errors will be unaffected and remain of order $1/n^2$ or $1/n$.
In particular, the saddlepoint approximation errors are uniform in $k$.

The normal approximation error improves in the limit $k\to\infty$.
To see this, note that under our assumptions, $\xi_i$ are drawn independently with variance $\Var(Y_{\theta'})=\theta'$.
It follows that $\overline{\xi^2}-(\overline{\xi})^2-\theta'$ is of order $1/\sqrt{k}$ and the approximation error is of order $1/n\sqrt{k}$, with random fluctuations.
However, this result is not robust to mis-specification: if the data do not have sample variances asymptotically equal to their sample means, $\overline{\xi^2}-(\overline{\xi})^2-\theta'$ will converge to a non-zero determinstic constant as $k\to\infty$, and the approximation error will remain of order $1/n$.
Furthermore, the result relies implicitly on the assumption that all the implied sample means $y_1,\dotsc,y_k$ are asymptotically equal to leading order.
This is not apparent in \eqref{GammaFIMLENormalPICV}, which is based on the change of variables $y_i=\theta'+\xi_i/\sqrt{n}$ for which the leading-order term is constant by construction.
Instead, return to \eqref{GammaFIMLENormalApprox}.
We see that $\tilde{\theta}_\MLE(x,n)$ depends only on the root-mean-square average of the entries $y_i$, whereas from \eqref{GammaFIgradloghatL}--\eqref{GammaFIhattheta*} we see that $\hat{\theta}_\MLE(\vec{x},n)$, $\theta_\MLE(\vec{x},n)$, $\hat{\theta}^*_\MLE(\vec{x},n)$ depend only on the geometric mean of the entries $y_i$.
When the $y_i$ are not all asymptotically equal, these two averages will be asymptotically different as $n\to\infty$, and $\tilde{\theta}_\MLE(x,n)$ will converge to a limit that differs from the common limiting value of $\hat{\theta}_\MLE(\vec{x},n)$, $\theta_\MLE(\vec{x},n)$, and $\hat{\theta}^*_\MLE(\vec{x},n)$.
This is another instance of the remark following \refthm{MLEerrorNormalApprox}: away from the curve $\set{\vec{y}_0=(\theta',\dotsc,\theta')^T\colon \theta'\in\thetadomain}$, the normal approximation model $\theta\mapsto\tilde{\vec{X}}_\theta$ differs considerably from the true model $\theta\mapsto\vec{X}_\theta$.
By way of comparison, the saddlepoint approximation errors are uniform in $k$ and $\vec{x}$ so long as the geometric mean of $y_1,\dotsc,y_k$ remains confined to a compact subset of $(0,\infty)$.
\end{example}

\begin{example}[Gamma subfamily with partial identifiability]\lbexample{GammaPI}
Let $Y_\theta\in\R$ have the Gamma distribution with parameters $\alpha=\theta$, $r=\theta$, where $\theta\in(0,\infty)=\thetadomain$ and $\mathcal{S}_\theta=(-\infty,\theta)$.
Form $X_\theta\in\R$ by summing $n$ i.i.d.\ copies of $Y_\theta$, so that $X_\theta$ has the $\mathrm{Gamma}(n\theta,\theta)$ distribution.
Let $\vec{X}_\theta\in\R^{k\times 1}$ be the vector whose entries are $k$ i.i.d.\ copies of $X_\theta$.
This is the situation of \refexample{MultipleSamplesWorkings} with $n_j=n$, $\beta_j=1$ for $j=1,\dotsc,k$, and 
\begin{equation}
K_0(s;\theta) = \theta\log\theta - \theta\log(\theta-s)
.
\end{equation}

By construction, $\E(Y_\theta)=K_0'(0;\theta)=1$ for all $\theta$.
Therefore, in the well-specified case $s_0=0$, this model falls into the partially identifiable case from Theorems~\ref{T:MLEerrorPartiallyIdentifiable}--\ref{T:MLEerrorNormalApprox}, with $p_1=0$, $p_2=1$ and $\nu=\theta$.
(We remark that Theorems~\ref{T:MLEerrorPartiallyIdentifiable}--\ref{T:MLEerrorNormalApprox} both apply in the boundary case $p_1=0,p_2=p$.)

Following the scaling \eqref{PIChangeOfVariables}, set $x_i=n+\xi_i\sqrt{n}$, $y_i=1+\xi_i/\sqrt{n}$ for $i=1,\dotsc,k$, and let $\vec{x}\in\R^{k\times 1}$ be the vector with entries $x_i$.
From \eqref{MKhatshatLBasicGamma} we find
\begin{gather}
\log\hat{L}_{\vec{X}}(\theta;\vec{x}) = \frac{k}{2}\log\theta - \frac{k}{2}\log(2\pi n) + \sum_{i=1}^k \left( (n\theta-1)\log\left( 1+\tfrac{\xi_i}{\sqrt{n}} \right) - \theta\xi_i\sqrt{n} \right),
\notag\\
\frac{\partial}{\partial\theta}\log\hat{L}_{\vec{X}}(\theta;\vec{x}) = \frac{k}{2\theta} + \sum_{i=1}^k n\left( \log\left( 1+\tfrac{\xi_i}{\sqrt{n}} \right) - \tfrac{\xi_i}{\sqrt{n}} \right)
.
\label{GammaPIgradloghatL}
\end{gather}
On the other hand, we can easily compute the true likelihood:
\begin{align}
\log L_{\vec{X}}(\theta;\vec{x}) &= k\left( \big. n\theta\log(n\theta) - n\theta - \log\Gamma(n\theta) \right) - k\log n 
\notag\\&\quad
+ \sum_{i=1}^k \left( (n\theta-1)\log\left( 1+\tfrac{\xi_i}{\sqrt{n}} \right) - \theta\xi_i\sqrt{n} \right)
,
\notag\\
\frac{\partial}{\partial\theta}\log L_{\vec{X}}(\theta;\vec{x}) &= nk\left( \log(n\theta) - \psi(n\theta) \right) + \sum_{i=1}^k n\left( \log\left( 1+\tfrac{\xi_i}{\sqrt{n}} \right) - \tfrac{\xi_i}{\sqrt{n}} \right)
\end{align}
where again $\psi(z)=\Gamma'(z)/\Gamma(z)$ denotes the digamma function; using \eqref{GRpsiFormula}, this becomes
\begin{equation}\label{GammaPIgradlogL}
\frac{\partial}{\partial\theta}\log L_{\vec{X}}(\theta;\vec{x}) = \frac{k}{2\theta} + \int_0^\infty \frac{2nk t\, dt}{(t^2+n^2\theta^2)(e^{2\pi t}-1)} + \sum_{i=1}^k n\left( \log\left( 1+\tfrac{\xi_i}{\sqrt{n}} \right) - \tfrac{\xi_i}{\sqrt{n}} \right)
.
\end{equation}
Finally we can compute the normal approximation likelihood based on $\tilde{X}_i\sim\mathcal{N}(1,n/\theta)$:
\begin{equation}\label{GammaPIgradlogtildeL}
\log\tilde{L}_{\vec{X}}(\theta;\vec{x}) = \frac{k}{2}\log\frac{\theta}{2\pi n} - \frac{\theta}{2} \sum_{i=1}^k \xi_i^2,
\quad
\frac{\partial}{\partial\theta}\log\tilde{L}_{\vec{X}}(\theta;\vec{x}) = \frac{k}{2\theta} - \frac{1}{2}\sum_{i=1}^k \xi_i^2
.
\end{equation}
It is readily verified that, whenever $\vec{\xi}$ is not identically zero, the (global) MLEs $\hat{\theta}_\MLE(\vec{x},n)$, $\theta_\MLE(\vec{x},n)$, $\tilde{\theta}_\MLE(\vec{x},n)$ are given by the unique values of $\theta$ for which the derivatives in \eqref{GammaPIgradloghatL}, \eqref{GammaPIgradlogL}, or \eqref{GammaPIgradlogtildeL}, respectively, become zero.

Note that \eqref{GammaPIgradloghatL} and \eqref{GammaPIgradlogL} differ by a term of order $1/n$.
Using the Implicit Function Theorem it follows that $\hat{\theta}_\MLE(\vec{x},n) - \theta_\MLE(\vec{x},n)$ is of order $1/n$.
Similarly, \eqref{GammaPIgradloghatL} and \eqref{GammaPIgradlogtildeL} differ by a term of order $1/\sqrt{n}$.
Direct computation verifies that $\tilde{\theta}_\MLE(\vec{x},n)-\hat{\theta}_\MLE(\vec{x},n)$, and consequently also $\tilde{\theta}_\MLE(\vec{x},n)-\theta_\MLE(\vec{x},n)$, is of order $1/\sqrt{n}$.
These conclusions match the assertions from Theorems~\ref{T:MLEerrorPartiallyIdentifiable}--\ref{T:MLEerrorNormalApprox} about $\nu$, and show that those assertions are sharp in general.

We remark also that $n$ and $k$ affect the inference problem in non-overlapping ways: taking $n\to\infty$ decreases the approximation error, while taking $k\to\infty$ decreases the inferential uncertainty.
However, increasing $k$ does not affect the approximation error, since the values of the MLEs, and the differences between them, depend only on certain sample means and not on $k$ directly.
Likewise, increasing $n$ does not change the steepness of the likelihood function (in a Bayesian setup) and does not substantially affect the sampling variability (in a frequentist setup) of the sample means on which the MLEs depend.
\end{example}

\begin{example}[Normal distribution]\lbexample{Normal}
Let $X_\theta\sim\mathcal{N}(\mu,\Sigma)\in\R^{\xdim\times 1}$ be multivariate normal with mean vector $\mu\in\R^{\xdim\times 1}$ and symmetric positive-definite covariance matrix $\Sigma\in\R^{\xdim\times\xdim}$.
We allow $\mu=\mu(\theta)$ and $\Sigma=\Sigma(\theta)$ to vary depending on underlying model parameters, which we encode in the vector $\theta$, and we require this dependence to be smooth.
For the initial part of our discussion, however, the nature of this dependence need not concern us and we will suppress it from the notation.

We compute 
\begin{equation}
\begin{gathered}
M(s;\theta) = \exp\left( s \mu + \tfrac{1}{2} s \Sigma s^T \right), 
\qquad
K(s;\theta) = s \mu + \tfrac{1}{2} s \Sigma s^T, 
\\
K'(s;\theta) = \mu + \Sigma s^T,
\qquad
K''(s;\theta) = \Sigma,
\qquad
\hat{s}(x,\theta) = (x-\mu)^T \Sigma^{-1},
\end{gathered}
\end{equation}
with $\mathcal{S}_\theta=\R^{1\times\xdim}$ for all $\theta$.
This leads to
\begin{align}
\hat{s}(x,\theta) \Sigma \hat{s}(x,\theta)^T &= \left( x-\mu \right)^T \Sigma^{-1} \Sigma \Sigma^{-1} \left( x-\mu \right) = \left( x-\mu \right)^T \Sigma^{-1} \left( x-\mu \right)
\end{align}
and
\begin{align}
\hat{L}(\theta;x) 
&= \frac{\exp\left( \hat{s}(\theta,x) \mu + \frac{1}{2} \hat{s}(x,\theta) \Sigma \hat{s}(x,\theta)^T - \hat{s}(\theta,x) x \right)}{\sqrt{\det(2\pi \Sigma)}}
\notag\\&
= \frac{\exp\left( \left( x-\mu \right)^T \Sigma^{-1} \left( \mu - x \right) + \frac{1}{2} \left( x-\mu \right)^T \Sigma^{-1} \left( x-\mu \right) \right)}{\sqrt{\det(2\pi \Sigma)}}
\notag\\&
=\frac{\exp\left( - \frac{1}{2} \left( x-\mu \right)^T \Sigma^{-1}\left( x-\mu \right) \right)}{\sqrt{\det(2\pi \Sigma)}}
.
\label{NormalSaddlepointf}
\end{align}
In particular, the saddlepoint approximation is exact for the normal distribution, for all choices of parameters.

To match the setup of \eqref{SAR}, we can take $X$ to be the sum of $n$ i.i.d.\ copies of $Y_\theta$ with $Y_\theta\sim\mathcal{N}(\tilde{\mu},\tilde{\Sigma})$.
As in \refexample{Poisson}, this amounts to the reparametrisation $\mu=n\tilde{\mu}$, $\Sigma=n\tilde{\Sigma}$, where we assume that $\tilde{\mu}=\tilde{\mu}(\theta)$ and $\tilde{\Sigma}=\tilde{\Sigma}(\theta)$ depend smoothly on $\theta$ and $\tilde{\Sigma}(\theta)$ is positive definite for all $\theta\in\thetadomain$.
It is straightforward to verify that the hypotheses of \refthm{GradientError} apply.
To verify \eqref{DecayBound}, note that
\begin{align}
\abs{\frac{M_0(s+\ii\phi;\theta)}{M_0(s;\theta)}} &= \abs{\exp\left( \ii\phi \tilde{\mu}(\theta) + \ii\phi\tilde{\Sigma}(\theta) s^T - \tfrac{1}{2}\phi \tilde{\Sigma}(\theta) \phi^T \right)} = \exp\left( - \tfrac{1}{2}\phi \tilde{\Sigma}(\theta) \phi^T \right).
\end{align}
Since $\tilde{\Sigma}(\theta)$ is positive definite, there is a number $c(\theta)$ such that $\tfrac{1}{2}\phi \tilde{\Sigma}(\theta) \phi^T \geq c(\theta)\abs{\phi}^2$ for all $\phi\in\R^{1\times\xdim}$, and indeed it suffices to choose $c(\theta)$ to be half the smallest eigenvalue of $\tilde{\Sigma}(\theta)^{-1}$, which is positive and varies continuously as a function of $\tilde{\Sigma}(\theta)$ and hence of $\theta$.
Then the estimate $e^{-q}\leq 1/(1+q)$ for $q>-1$ gives $\abs{\frac{M_0(s+\ii\phi;\theta)}{M_0(s;\theta)}}\leq (1+c(\theta)\abs{\phi}^2)^{-1}$, so that we can take $\delta(s,\theta)=\min(c(\theta),1)$.
For \eqref{GrowthBound}, note that we can consider $K_0(s+\ii\phi;\theta)$ as a single-valued function (without branch cuts arising from the complex logarithm) for all $\phi$.
Since we assumed that $\tilde{\mu}$ and $\tilde{\Sigma}$ depended smoothly on $\theta$, it is immediate that all of the partial derivatives of $K_0(s+\ii\phi;\theta)$ (of order at most 6) are continuous and grow at most polynomially in $\abs{\phi}$, and indeed will be bounded by $C(\theta)(1+\abs{\phi})^2$ for some continuous function $C(\theta)$ given in terms of $\tilde{\mu}(\theta)$, $\tilde{\Sigma}(\theta)$ and their partial derivatives.
The partial derivatives of $M_0$ can be expressed in terms of $M_0$ and the partial derivatives of $K_0$, so \eqref{GrowthBound} follows.
A closer look at this argument shows that $\tilde{\mu}$ and $\tilde{\Sigma}$ need only be twice continuously differentiable.

Consider now the MLE problem for an observation $x$ of $X_\theta$.
The nature of the MLE, if any, will depend on how the distributional parameters $\mu$ and $\Sigma$, or equivalently $\tilde{\mu}$ and $\tilde{\Sigma}$, depend on the model parameters $\theta$.
If the parameter vector $\theta$ allows all values of $\mu$ and $\Sigma$ (which can be arranged by taking $\theta$ in a suitable open subset of $\R^{(\xdim+\xdim(\xdim+1)/2)\times 1}$) then no MLE will exist, for any value of $n$.
This is no surprise, because, for any value of $n$, we observe only one $\xdim$-dimensional vector $x$, which does not contain enough information to identify both $\mu$ and $\Sigma$.
(The summands $Y^{(1)},\dotsc,Y^{(n)}$ that go into $X$ contain a total of $n\xdim$ scalar entries, so they would provide enough information to identify both $\mu$ and $\Sigma$ as soon $n\geq 1+\frac{1}{2}(\xdim+1)$, but this does not help us because these summands are not observed.)
Indeed, the likelihood will approach its supremum $\infty$ in the limit where $\mu=x$ and $\Sigma\to 0$.

\textit{A fully identifiable submodel.}
An MLE may exist if the parametrisation keeps $\Sigma$ bounded away from 0.
A simple example is where $\tilde{\Sigma}$ is fixed and $\theta$ parametrises the unknown mean $\E(Y_\theta)$ of the summands, $\tilde{\mu}(\theta)=\theta$.
Then it is elementary to verify $\theta_\MLE(x,n)=x/n$.
Since the saddlepoint approximation to the likelihood is exact, it follows that the saddlepoint MLE will be exact also, $\hat{\theta}_\MLE(x,n)=\theta_\MLE(x,n)=x/n$.
To match with the setup of Theorems~\ref{T:MLEerror}--\ref{T:SamplingMLEWellSpecified}, compute
\begin{equation}
\begin{gathered}
K_0'(s;\theta) = \theta + \tilde{\Sigma} s^T,
\qquad
\grad_\theta K_0(s;\theta) = s, 
\\
K_0''(s;\theta) = \tilde{\Sigma},
\quad
\grad_\theta^T \grad_\theta K_0(s;\theta) = 0, 
\quad
\grad_s\grad_\theta K_0(s;\theta) = I_{\xdim\times\xdim}
.
\end{gathered}
\end{equation}
Thus the solutions of \eqref{y0K0's0}--\eqref{RateFunctionImplicitCriticalPoint} are $(s_0,\theta_0)=(0,y_0)$ with $y_0\in\R^{\xdim\times 1}$ arbitrary, and we compute $H=-\tilde{\Sigma}$, which is negative definite as required by \eqref{RateFunctionImplicitHessianDefinite}.
The fact that $s_0=0$ indicates that, in the terminology of \refsubsubsect{WellSpecified}, this model is well-specified at the level of the mean for an arbitrary observed value $x=ny_0$.
By the discussion in \refsubsubsect{WellSpecified}, to verify \eqref{RateFunctionImplicitHessianDefinite} it would also suffice to check that the gradient of the mapping $\theta\mapsto\E_\theta(Y_0)$ is non-singular, which is trivial since this is the identity mapping.
Hence this model is fully identifiable at the level of the mean, and the conclusions of Theorems~\ref{T:MLEerror}--\ref{T:SamplingMLEWellSpecified} apply.
Note in particular that the inferential uncertainty in this model decreases to zero in the limit $n\to\infty$: when $n$ is large, even a single observation of $X_\theta$ is highly informative for the value of the mean.

\textit{A partially identifiable submodel.}
Different behaviour can be obtained if the model is partially identifiable.
Let $\theta=\smallmat{\omega\\ \nu}$ with $\omega\in\R$ and $\nu>0$.
Set $\tilde{\mu}(\theta)=\omega\ones_{\xdim\times 1}$ and $\tilde{\Sigma}(\theta)=\nu I_{\xdim\times\xdim}$, where $\ones_{\xdim\times 1}$ denotes the $\xdim\times 1$ vector all of whose entries are 1.
In other words, $Y_\theta$ consists of $\xdim$ i.i.d.\ copies of the univariate $\mathcal{N}(\omega,\nu)$ distribution, and $X_\theta$ consists of $\xdim$ i.i.d.\ copies of the univariate $\mathcal{N}(n\omega,n\nu)$ distribution.
(This is the setup of \refexample{MultipleSamplesWorkings} with $k=\xdim$, $\xdim_0=1$, but for consistency with the rest of this example we do not use the notation $\vec{X}_\theta$, etc.)
Then
\begin{equation}
\begin{gathered}
K'_0(s;\theta) = \omega\ones_{\xdim\times 1} + \nu s^T,
\qquad
\grad_\theta K_0(s;\theta) = \mat{\sum_{i=1}^\xdim s_i & \tfrac{1}{2}\sum_{i=1}^\xdim s_i^2}
,
\end{gathered}
\end{equation}
so that the solutions of \eqref{y0K0's0}--\eqref{RateFunctionImplicitCriticalPoint} are given by $s_0=0$, $y_0=\omega\ones_{\xdim\times 1}$ with $\theta_0$ arbitrary.
However, $\grad_s\grad_\theta K_0(0;\theta_0)=\mat{\ones_{\xdim\times 1} & 0_{\xdim\times 1}}$ has rank 1, so the Hessian from \eqref{RateFunctionImplicitHessianDefinite} is singular and Theorems~\ref{T:MLEerror}--\ref{T:SamplingMLEWellSpecified} do not apply.

Instead, the setup of \refthm{MLEerrorPartiallyIdentifiable} applies, since $\nu$ affects the variance but not the mean.
There is no need to apply \refthm{MLEerrorPartiallyIdentifiable} because the true and saddlepoint MLEs coincide identically, but it is still of interest to note the effect of $n$ on inference for $\omega$ and $\mu$.
We find
\begin{equation}
\omega_\MLE(x) = \frac{1}{\xdim}\sum_{i=1}^\xdim \frac{x_i}{n},
\qquad
\nu_\MLE(x) = \frac{1}{n \xdim}\sum_{i=1}^\xdim \left( x_i-\tfrac{1}{\xdim}\textstyle{\sum_{j=1}^\xdim} x_j \right)^2
.
\end{equation}
In the notation of \eqref{PIChangeOfVariables}, $\omega_\MLE(x)$ is the sample mean of $y_1,\dotsc,y_\xdim$, whereas $\nu_\MLE(x)$ is the population variance corresponding to the lower-order quantities $\xi_1,\dotsc,\xi_\xdim$.
If we assume that $x$ is sampled according to the model distribution $X_{\theta_0}$, as in \refthm{SamplingMLEWellSpecified}, then in the limit $n\to\infty$ each value $y_i$ concentrates around $\omega_0$, whereas each $\xi_1,\dotsc,\xi_\xdim$ follows the $\mathcal{N}(0,\nu)$ distribution regardless of the value of $n$.
Thus $\omega_\MLE(X_{\theta_0})$ becomes increasingly concentrated, whereas $\nu_\MLE(X_{\theta_0})$ retains non-trivial sampling variability and does not exhibit concentration in the limit $n\to\infty$.
We might say that $n\to\infty$ makes an observation of $X$ increasingly informative for inference about $\omega$, but not for inference about $\nu$.
Correspondingly, from
\begin{equation}
\log L(\theta;x) = -\frac{n\xdim}{2\nu}\left( \omega-\omega_\MLE(x) \right)^2 - \frac{\xdim}{2}\left( \frac{\nu_\MLE(x)}{\nu}+\log\nu \right) - \frac{\xdim}{2}\log(2\pi n),
\end{equation}
we see that in the limit $n\to\infty$, the log-likelihood becomes increasingly steep in the $\omega$ direction, but not in the $\nu$ direction.

This phenomenon explains why there is no analogue of Theorems~\ref{T:BayesianError}--\ref{T:SamplingMLEWellSpecified} in the partially identifiable case: the inferential uncertainty for $\nu$ remains of order 1 as $n\to\infty$.
Nevertheless, \refthm{MLEerrorPartiallyIdentifiable} shows that the MLE approximation error continues to vanish as $n\to\infty$ under general conditions.
\end{example}

\begin{example}[Normal distribution and its square]\lbexample{NormalSquare}
Let $Z\sim\mathcal{N}(\mu,\sigma^2)$ have the univariate normal distribution, and set $Y=(Z,Z^2)$.
Set $\theta=\smallmat{\mu\\ \sigma^2}$ with the restriction $\sigma^2>0$.
Then
\begin{equation}
M_0(s;\theta) = \frac{\exp\left( \frac{\mu s_1 + \frac{1}{2}\sigma^2 s_1^2 + \mu^2 s_2}{1-2s_2\sigma^2} \right)}{\sqrt{1-2s_2\sigma^2}}
\qquad
\text{for $\Re(s_2)<\tfrac{1}{2}\sigma^2$}
.
\end{equation}
This expression can be simplified by noting that $Y$ is a sufficient statistic for an exponential family, for which a natural parameter is
\begin{equation}
\eta = (\eta_1, \eta_2) = \left(\tfrac{\mu}{\sigma^2}, -\tfrac{1}{2\sigma^2} \right) \qquad\text{(with $\eta_2<0$)}.
\end{equation}
In terms of $\eta$, 
\begin{equation}
M_0(s;\theta) = \sqrt{\frac{\eta_2}{\eta_2+s_2}} \exp\left( \frac{\frac{s_2\eta_1^2}{4\eta_2} - \frac{\eta_1 s_1}{2} - \frac{s_1^2}{4}}{\eta_2 + s_2} \right)
\quad
\text{for $\Re(s_2)<-\eta_2$}, 
\end{equation}
and some calculation leads to the conclusion that $M_0$ and $K_0$ can be written as
\begin{equation}
\begin{gathered}
M_0(s;\theta)=\frac{e^{\rho_0(\eta+s)}}{e^{\rho_0(\eta)}}, \quad K_0(s;\theta)=\rho_0(\eta+s)-\rho_0(\eta), \quad\text{where}
\\
\rho_0(\eta) = \frac{\eta_1^2}{-4\eta_2} - \frac{1}{2}\log(-\eta_2)
\end{gathered}
\end{equation}
(which may be obtained from $\exp(\rho_0(\eta))=\int_{-\infty}^\infty e^{\eta_1 z + \eta_2 z^2} \frac{dz}{\sqrt{\pi}}$, with $\frac{dz}{\sqrt{\pi}}$ acting as the reference measure).
Then
\begin{equation}
\frac{\partial\rho_0}{\partial\eta_1}(\eta) = \frac{\eta_1}{-2\eta_2}, \qquad \frac{\partial\rho_0}{\partial\eta_2}(\eta) = \left( \frac{\eta_1}{-2\eta_2} \right)^2 + \frac{1}{-2\eta_2},
\end{equation}
and if we follow the notation of \refsubsubsect{ExpFamily} in setting $\hat{\eta} = \eta + \hat{s}$, so that $\hat{\eta}$ is the MLE for the natural parameter $\eta$, then we find
\begin{equation}\label{NormalSquareSaddlepointSolution}
\frac{\hat{\eta}_1}{-2\hat{\eta}_2} = y_1, \quad \frac{1}{-2\hat{\eta}_2} = y_2-y_1^2, \quad \hat{\eta}_1 = \frac{y_1}{y_2-y_1^2} \quad \hat{\eta}_2 = \frac{-1}{2(y_2-y_1^2)},
\end{equation}
provided that $y_2>y_1^2$.
If we think of $x$ as the sum of $n$ pairs $(z(i),z(i)^2)$, then $y=x/n$ is the corresponding implied sample mean and we could denote its entries as $y_1=\overline{z}$, $y_2=\overline{z^2}$.
Then, changing back to the parameters $\mu,\sigma^2$ and their corresponding MLEs $\hat{\mu},\hat{\sigma}^2$, this reduces to the familiar and unsurprising statement that $\hat{\mu}=\overline{z}$ and $\hat{\sigma}^2=\overline{z^2}-\overline{z}^2$, the population mean and variance, respectively, corresponding to the $n$ values $z(i)$.

Turning to the saddlepoint approximation, differentiate $\rho$ then, for convenience, return to the parameters $\mu,\sigma^2$:
\begin{equation}
\begin{gathered}
\rho_0(\eta) = \frac{\mu^2}{2\sigma^2} + \frac{1}{2}\log(2\sigma^2)
,
\qquad
\rho_0(\eta)-\eta\rho'_0(\eta) = \frac{1}{2} + \frac{1}{2}\log(2\sigma^2)
,
\\
\rho''_0(\eta) = 
\mat{\frac{1}{-2\eta_2} & \frac{\eta_1}{2\eta_2^2} \\
\frac{\eta_1}{2\eta_2^2} & \; \frac{\eta_1^2}{-2\eta_2^3} + \frac{1}{2\eta_2^2}}
=
\mat{\sigma^2 & 2\mu\sigma^2 \\
2\mu\sigma^2 & \; 4\mu^2\sigma^2+2\sigma^4}
,
\\
\det(2\pi n\rho''_0(\eta)) = 8\pi^2 n^2 \sigma^6
.
\end{gathered}
\end{equation}
We now apply \eqref{SPAExpFamily} with $\rho=n\rho_0$, noting that $\rho'(\hat{\eta})=x$ by construction.
Recalling that $\hat{\sigma}^2 = y_2-y_1^2 = \frac{x_2}{n}-\frac{x_1}{n^2}$,
\begin{equation}
\begin{aligned}
\hat{L}_n(\theta;x) 
&= \frac{\exp\left( n\left( \rho_0(\hat{\eta}) - \hat{\eta}\rho'_0(\hat{\eta}) \right) \right)}{\sqrt{\det(2\pi n\rho''_0(\hat{\eta}))}} \exp\left( \eta x-n\rho_0(\eta) \right)
\\
&= \frac{\exp\left( \frac{n}{2}+\frac{n}{2}\log(2\hat{\sigma}^2) \right)}{\sqrt{8\pi^2 n^2 \hat{\sigma}^6}} \exp\left( \eta_1 x_1 + \eta_2 x_2 - \frac{n\mu^2}{2\sigma^2} - \frac{n}{2}\log(2\sigma^2) \right)
\\
&= \frac{e^{n/2}2^{n/2}}{2^{3/2}\pi n} \left( \hat{\sigma}^2 \right)^{\frac{n-3}{2}} \frac{\exp\left( \frac{x_1\mu}{\sigma^2} - \frac{x_2}{2\sigma^2} - \frac{n\mu^2}{2\sigma^2} \right)}{2^{n/2}(\sigma^2)^{n/2}}
\\
&= \frac{e^{n/2}}{2^{3/2}\pi n} \left( \frac{x_2}{n}-\frac{x_1^2}{n^2} \right)^{\frac{n-3}{2}} \frac{\exp\left( \frac{x_1\mu}{\sigma^2} - \frac{x_2}{2\sigma^2} - \frac{n\mu^2}{2\sigma^2} \right)}{(\sigma^2)^{n/2}}
.
\end{aligned}
\end{equation}
We note that this expression requires $y_2>y_1^2$, or equivalently $x_2>x_1^2/n$, which is the condition for \eqref{NormalSquareSaddlepointSolution} to give a valid solution to $\rho'_0(\eta)=y$.
In particular, the saddlepoint approximation will not make sense if we try to apply it with $n=1$, i.e., if we try to apply it directly to $Y$, which makes sense because $Y$ is supported on the curve $\set{(z,z^2)\colon z\in\R}$ and does not have a density.

By comparison, for $n\geq 2$, the true likelihood can be directly calculated as
\begin{equation}
L_n(\theta;x) = \frac{\exp\left( -\frac{1}{2n\sigma^2}(x_1-n\mu)^2  \right)}{\sqrt{2\pi n\sigma^2}} \frac{(\frac{x_2-x_1^2/n}{2\sigma^2})^{\frac{n-3}{2}} \exp\left( -\frac{1}{2\sigma^2}\left( x_2-x_1^2/n \right) \right) }{2\sigma^2\Gamma(\frac{n-1}{2})} \indicator{x_2 > x_1^2/n}
\end{equation}
using the observation that $X_1$ and $X_2-X_1^2/n$ are independent with respective distributions $\mathcal{N}(n\mu,n\sigma^2)$ and $\mathrm{Gamma}(\frac{n-1}{2},\frac{1}{2\sigma^2})$.
Taking the ratio of these two expressions leads, after some simplification, to
\begin{equation}
\frac{\hat{L}_n(\theta;x)}{L_n(\theta;x)} = \frac{2^{\frac{n-3}{2}} e^{\frac{n}{2}} \Gamma(\frac{n-1}{2})}{n^{\frac{n}{2}-1}\sqrt{\pi}}
\end{equation}
and an application of Stirling's approximation confirms that this ratio tends to 1 as $n\to\infty$.

Note that the saddlepoint approximation did not give the exact density for $X$, notwithstanding the result of \refexample{Normal}.
In this case the ratio of the true and exact densities is a number depending on $n$ but not on $x$ or $\theta$.
The fact that the ratio does not depend on $\theta$ follows by the reasoning in \refsubsubsect{ExpFamily}; in particular, the saddlepoint MLE is exact, which we already knew from \refsubsubsect{ExpFamily}.
The fact that the ratio does not depend on $x$ is a particular feature of this example.
\end{example}

\begin{example}[Saddlepoint MLE may be incorrect globally]\lbexample{MLEWrong}
Let $Y_\theta = \frac{1}{2}e^{-\theta^2}Z + B$ where $Z\sim\mathcal{N}(0,1)$ and $\P(B=1)=\P(B=-1)=1/2$, with $B$ and $Z$ independent, and let $y_0=0$.
Thus $Y_\theta$ is a mixture of two normal distributions, with means $1$ and $-1$ and common variance $\frac{1}{4}e^{-2\theta^2}$.
We consider the case $y=0$, so that $x=0$ also.

For the saddlepoint approximation, we compute
\begin{equation}
\begin{gathered}
M_0(s;\theta) = e^{\frac{1}{8}e^{-2\theta^2}s^2}\cosh s, \quad K_0(s;\theta) = \frac{1}{8}e^{-2\theta^2}s^2 + \log\cosh s, 
\\
K_0''(s;\theta)=\frac{1}{4}e^{-2\theta^2}+\sech^2 s.
\end{gathered}
\end{equation}
Since $y=0$ coincides with the mean, we will have $\hat{s}_0(\theta,y)=\hat{s}(\theta,x)=0$ for all $\theta$ and
\begin{equation}
\hat{L}_n(\theta;0) = \frac{1}{\sqrt{2\pi n K''_0(0;\theta)}} = \frac{1}{\sqrt{2\pi n(1+e^{-2\theta^2}/4)}},
\end{equation}
which is equally easy to compute for all values of $n$.
In particular, we see that $\hat{L}_n(\theta;0)$ approaches its supremum in the limit where $\theta\to\pm\infty$.

We now turn to the true likelihood.
First consider $n=1$.
Then we can explicitly compute
\begin{equation}
L_1(\theta;0) = \frac{\exp\left( -2e^{2\theta^2}(\pm 1)^2 \right)}{\sqrt{2\pi(e^{-2\theta^2}/4)}}
, \qquad
\log L_1(\theta;0) = -2e^{2\theta^2} + \theta^2-\log\sqrt{\pi/2},
\end{equation}
and the unique global MLE is $\theta_\MLE(x=0,n=1) = 0$, the $\theta$ value that maximises the variance.
This makes intuitive sense: in each normally-distributed mixture component, the observed value $0$ is more than one standard deviation away from the mean, so increasing the variance will increase the density and hence the contribution to the overall likelihood.
We see also that $\log L_1(\theta;0)\to-\infty$ as $\theta\to\pm\infty$, corresponding to the observation that when $\theta$ is large the two mixture densities fall off quickly away from their respective means $1$ and $-1$.

In particular, for $n=1$, the true and saddlepoint likelihoods have completely different behaviour: the true likelihood has a unique maximum at $\theta=0$, whereas the saddlepoint approximation incorrectly predicts that the likelihood increases when $\theta\to\pm\infty$.
When $\abs{\theta}$ is large, the relative error in the saddlepoint approximation becomes large: the saddlepoint approximation predicts a bounded likelihood, whereas the true density at $y=0$ approaches zero.
Roughly speaking, this is because the distribution is bimodal with two increasingly narrow peaks.
The ``missing'' probability mass in the trough between the two peaks does not significantly affect the first or second moments, and the saddlepoint approximation does not detect it.
This applies equally to the usual normal approximation (which coincides with the saddlepoint approximation since $\hat{s}=0$) and indeed it is difficult to see how any general-purpose approximation based on moments could handle bimodality well.

For general $n$, the situation becomes more complicated, and changes considerably depending on the parity of $n$.
We can consider $X$ as a mixture of $n+1$ normally distributed components, each with variance $\frac{n}{4}e^{-2\theta^2}$, and with means and mixture probabilities corresponding to the values and probabilities for the sum of $n$ i.i.d.\ copies of $B$.
If $n$ is odd, the mean of each of these normally distributed components is an odd, and in particular non-zero, integer, so we can make the upper bound 
\begin{equation}
L_n(\theta;0) \leq \frac{\exp\left( -2e^{2\theta^2}(\pm 1)^2/n \right)}{\sqrt{2\pi(ne^{-2\theta^2}/4)}}, \qquad\text{$n$ odd.}
\end{equation}
In particular, similar to the case $n=1$, the likelihood tends to zero as $\theta\to\pm\infty$.
This is again different from the saddlepoint approximation, where $\hat{L}_n(\theta;0)$ continues to have its supremum in the limit $\theta\to\pm\infty$.

On the other hand, if $n$ is even, one of the mixture components has mean zero and we have the lower bound
\begin{equation}
L_n(\theta;0) \geq \binom{n}{n/2}2^{-n} \cdot \frac{1}{\sqrt{2\pi(ne^{-2\theta^2}/4)}} = c_n e^{\theta^2}, \qquad\text{$n$ even,}
\end{equation}
where $\binom{n}{n/2}2^{-n} \approx n^{-1/2}$ is the probability that the sum of $n$ i.i.d.\ copies of $B$ has value 0.
Thus the true likelihood grows rapidly as $\theta\to\pm\infty$.
The saddlepoint approximation misses this feature too (although, coincidentally, it ends up correctly predicting that the supremum occurs when $\theta\to\pm\infty$) this time because the true density has a high but narrow peak at $y=0$ when $\abs{\theta}$ is large.
In this case too, the ``extra'' probability mass in the peak does not greatly affect the variance (which arises primarily from variation of the mean across different mixture components) and the saddlepoint approximation does not detect it.
A similar effect occurs in \cite[sections~4.1 and 4.2]{KleSka2008}: the mixture distribution defined in \cite[equation~(5)]{KleSka2008} has a highly peaked component when $\omega$ is small (and even more so in the limit $\omega\decreasesto -\frac{1}{4}$, which they exclude from consideration) and the standard saddlepoint approximation does not handle the resulting MLE problems well.
\end{example}

\begin{example}[Global maximum for $\hat{L}^*$ need not give global maximum for $\hat{L}$]\lbexample{LhatVsL*Different}
Consider the univariate parametric model $Y\sim\mathcal{N}(\mu(\theta),\sigma(\theta)^2)$, $\theta\in\R=\thetadomain$.
We will choose the parametrising functions so that, for an observation with $y=0$ and hence $x=0$, the value $\theta=0$ is the unique global maximiser for $\theta\mapsto \hat{L}^*(\theta;0)$ but neither the true likelihood $L_n(\theta;0)$ nor its saddlepoint approximation $\hat{L}_n(\theta;0)$ have a global maximum, for any value of $n$.
Note that $L_n(\theta;x)=\hat{L}_n(\theta;x)$ already because the distribution is normal, so this example concerns the distinction between $L_n(\theta;0)$ and $\hat{L}^*(\theta;0)$ only.

Adapting the notation of \refexample{Normal} leads to
\begin{equation}
\begin{gathered}
K_0(s;\theta) = \mu s + \tfrac{1}{2}\sigma^2 s^2, \quad \hat{s}_0(\theta,0)=\hat{s}(\theta,0)=-\mu/\sigma^2, 
\\
\log\hat{L}^*(\theta;0) = n\log L^*_0(\hat{s}(\theta,0),\theta) = n\left( \mu(-\mu/\sigma^2) + \tfrac{1}{2}\sigma^2(-\mu/\sigma^2)^2 \right) = -\frac{n\mu^2}{2\sigma^2}
,
\end{gathered}
\end{equation}
where for convenience we have omitted the dependence of $\mu$ and $\sigma^2$ on $\theta$.
On the other hand, $X$ has the $\mathcal{N}(n\mu,n\sigma^2)$ distribution so
\begin{equation}
\log L_n(\theta;0) = -\frac{(0-n\mu)^2}{2(n\sigma^2)} - \frac{1}{2}\log(2\pi(n\sigma^2)) = -\frac{n\mu^2}{2\sigma^2} - \frac{1}{2}\log(2\pi n\sigma^2).
\end{equation}

We now make a choice of parametrisation to make the claimed properties hold.
We can make $\theta=0$ the unique global maximiser for $\theta\mapsto \hat{L}^*(\theta;0)$ if we arrange $\mu(0)=0$, $\mu(\theta)\neq 0$ for all $\theta\neq 0$ and $\mu(\theta)^2/\sigma(\theta)^2\to\infty$ as $\theta\to\pm\infty$.
On the other hand, we can make $\log L_n(\theta;0)\to\infty$ as $\theta\to\pm\infty$ if in addition we arrange $\sigma(\theta)\to 0$ in such a way that $-\log(\sigma(\theta)^2) \gg \mu(\theta)^2/\sigma(\theta)^2$ as $\theta\to\pm\infty$.
One choice that acheives this is
\begin{equation}
\mu(\theta) = \theta e^{-\theta^4}, \qquad \sigma(\theta)^2 = e^{-2\theta^4}.
\end{equation}
\end{example}

\begin{example}[Likelihoods may be ill-behaved for small $n$]\lbexample{QPlusGamma}
Let $Z=Z_{\alpha,r}$ have the $\mathrm{Gamma}(\alpha,r)$ distribution and let $G=G_p$ have the $\N$-valued $\Geometric_{\geq 1}(p)$ distribution (i.e., $\P(G=k)=p(1-p)^{k-1}$ for $k\in\N$), independently, where $\alpha,r>0$ and $0<p<1$.
Thus, writing $\theta=(\alpha,r,p)$, we are taking $\theta\in\thetadomain=\in(0,\infty)^2\times(0,1)$ .
Let $q(k)\colon\N\to\mathbb{Q}$ be an enumeration of the rational numbers $\mathbb{Q}$ such that $\abs{q(k)}\leq k$ for all $k\in\N$.
Set $Y=Y_\theta=Z_{\alpha,r}+q(G_p)$.
We can think of $Y$ as a mixture of shifted Gamma distributions $Z+q(k)$, with mixture probabilities $\P(G=k)$.

Each mixture distribution $Z+q(k)$ is absolutely continuous with respect to Lebesgue measure, so $Y_\theta$ has a density for any choice of $\theta$.
However, whenever $\alpha<1$, the density of $Z+q(k)$ is unbounded in a neighbourhood of $q(k)$, and it follows that, for such $\alpha$, $Y_\theta$ does not have a density function that is bounded (nor, \textit{a fortiori}, continuous) on any non-trivial interval.

Despite this, $Y_\theta$ and its MGF $M_0$ satisfy the assumptions of \refthm{GradientError}.
The assumption $\abs{q(k)}\leq k$ implies that $q(G)$ has exponentially bounded tails for any fixed $\theta$ and 
\begin{equation}
\bigabs{e^{sq(G)}} \leq e^{\abs{s} G} = \frac{p}{1-(1-p)e^{\abs{s}}},
\end{equation}
so that $M_{q(G)}(s;\theta)$ is defined at least in the interval $\abs{s}<\log(1/(1-p))$.
The identity
\begin{equation}\label{MMZqG}
M_0(s+\ii\phi;\theta) = M_Z(s+\ii\phi;\theta) M_{q(G)}(s+\ii\phi;\theta)
\end{equation}
holds for all $(s,\theta)\in\mathcal{S}$ and all $\phi\in\R$ (to see this, recall that $(s,\theta)\in\mathcal{S}$ requires $e^{sZ} e^{sq(G)}$ to be absolutely integrable, and for independent random variables $U$, $V$, the product $UV$ is absolutely integrable if and only if both $U$ and $V$ are absolutely integrable).
Hence
\begin{align}
\abs{\frac{M_0(s+\ii\phi;\theta)}{M_0(s;\theta)}} 
&= \abs{\frac{M_Z(s+\ii\phi;\theta) M_{q(G)}(s+\ii\phi;\theta)}{M_Z(s;\theta) M_{q(G)}(s;\theta)}}
\leq 
\abs{\frac{M_Z(s+\ii\phi;\theta)}{M_Z(s;\theta)}}
\end{align}
because of the trivial bound $\abs{\E(e^{(s+\ii\phi) q(G)})} \leq \E(\abs{e^{(s+\ii\phi) q(G)}}) = \E(e^{s q(G)})$ for all $\phi$.
So \eqref{DecayBound} holds because of the corresponding bound for the Gamma distribution, see \refexample{Gamma}.
Indeed, this argument shows generally that if $Y=Z_1+Z_2$ is the sum of two independent random variables, \eqref{DecayBound} follows if either $Z_1$ or $Z_2$ satisfies the corresponding bound.

Verifying \eqref{GrowthBound} is somewhat inconvenient because we lack an explicit formula and must use indirect arguments.
First consider $\frac{\partial}{\partial p}M_{q(G_p)}(s+\ii\phi;\theta)$.
Let $c_1(s,\theta)$ denote the minimum distance from $(s,\theta)$ to $\mathcal{S}^c$.
Then $c_1$ is positive on $\interior\mathcal{S}$ and continuous.
Consequently $C_2(s,\theta)=\sup_{s'\colon\abs{s'-s}\leq \frac{1}{2} c_1(s,\theta)} M_0(s';\theta)$ is finite and continuous on $\interior\mathcal{S}$, and we may choose $c_3(s,\theta)$ positive and continuous on $\interior\mathcal{S}$ such that $\abs{s}c_3(s,\theta)\leq\frac{1}{2}c_2(s,\theta)$ (for instance, $c_3(s,\theta)=\min(c_2(s,\theta)/2\abs{s}, 1)$).
Then for $(s,\theta)\in\interior\mathcal{S}$ we compute
\begin{align}
\abs{\frac{\partial}{\partial p}M_{q(G_p)}(s+\ii\phi;\theta)} &= \abs{\frac{\partial}{\partial p} \sum_{k=1}^\infty p(1-p)^{k-1} e^{(s+\ii\phi)q(k)}}
\notag\\&
= \abs{\sum_{k=1}^\infty \frac{\partial}{\partial p} \left( p(1-p)^{k-1} \right) e^{(s+\ii\phi)q(k)}}
\notag\\&
\leq \sum_{k=1}^\infty p(1-p)^{k-1} \abs{ \frac{1}{p}-\frac{k-1}{1-p} } \abs{e^{(s+\ii\phi)q(k)}}
\notag\\&
\leq \frac{1}{p(1-p)} \sum_{k=1}^\infty p(1-p)^{k-1} (1+k) e^{s q(k)}
\notag\\&
= \frac{\E((1+G_p)e^{s q(G_p)})}{p(1-p)}
\notag\\&
\leq \frac{\norm{1+G_p}_{\mathscr{L}^{1+1/c_3(s,\theta)}} \E(e^{s(1+c_3(s,\theta))q(G_p)})^{1/(1+c_3(s,\theta))}}{p(1-p)}
\notag\\&
\leq \frac{C_2(s,\theta)^{1/(1+c_3(s,\theta))} \norm{1+G_p}_{\mathscr{L}^{1+1/c_3(s,\theta)}} }{p(1-p)}
\label{partialpMqG}
\end{align}
by H\"older's inequality with dual indices $1+1/c_3$ and $1+c_3$.
The upper bound in \eqref{partialpMqG} is continuous as a function of $s$ and $\theta$.
(The interchange of differentiation and infinite summation in the second line of \eqref{partialpMqG} can be justified by the same chain of inequalities: given $\theta=(\alpha,r,p_0)$ with $p_0\in(0,1)$, one can find an interval $[p_1,p_2]\subset(0,1)$ with $p_0$ in its interior such that $(s,\theta')\in\mathcal{S}$ whenever $\theta'=(\alpha,r,p)$ with $p\in[p_1,p_2]$, and one can find a constant $C<\infty$ such that $\frac{\partial}{\partial p}(p(1-p)^{k-1}) \leq C(1+k)p_1 (1-p_1)^{k-1}$ for all $k\in\N$.
Then \eqref{partialpMqG} applied with $p=p_1$, together with \reflemma{DiffUnderInt} applied with $g(p,k)=p(1-p)^{k-1} e^{(s+\ii\phi)q(k)}$ and $\mu$ taken to be counting measure on $\N$, verify the assertion.)
A similar argument applies to other derivatives of $M_{q(G_p)}(s;\theta)$: instead of $(1+G_p)/[p(1-p)]$ we obtain other polynomials involving $G_p$ or $q(G_p)$, divided by a product of the form $p^i (1-p)^j$, all of which can be bounded by similar continuous functions.
Finally \eqref{GrowthBound} follows because \eqref{MMZqG} allows us to express derivatives of $M_0$ in terms of $M_Z$ and its derivatives, for which we have bounds from \refexample{Gamma}, and $M_{q(G)}$ and its derivatives.

We remark that \eqref{DecayBound}--\eqref{GrowthBound} imply that $X$ has a continuous and bounded density when $n$ is sufficiently large, even though the summands $Y$ had ill-behaved densities when $\alpha<1$.
We can also see this directly: note that $X$ is the sum of a $\mathrm{Gamma}(n\alpha,r)$ distribution plus $n$ copies of $q(G)$, all independent.
If $n\alpha>1$, the $\mathrm{Gamma}(n\alpha,r)$ term has a continuous and bounded density, so that $X$ must as well.
This is exactly the threshold at which $M_0(s+\ii\phi;\theta)^n$ decays quickly enough to be integrable as a function of $\phi$; and \eqref{DecayBound} is an obvious sufficient condition to ensure that this integrability must always occur for $n$ sufficiently large.
\end{example}

\begin{example}[A distribution for which \eqref{DecayBound} fails]\lbexample{DecayFails}
We can modify the preceding distribution to make \eqref{DecayBound} fail.
Let $Y$ have a mixture distribution with mixture probabilities $p_i>0$ for all $i\in\N$ and component distributions of the form $Z_{\alpha_i,r_i}+c_i$, where $Z_{\alpha_i,r_i}$ follows the $\Gamma(\alpha_i,r_i)$ distribution and $c_i$ are constants (all, optionally, depending on some underlying parameter $\theta$) so that
\begin{equation}
M_0(s;\theta) = \sum_{i=1} p_i e^{c_i s}\left( \frac{r_i}{r_i-s} \right)^{\alpha_i}
\end{equation}
whenever $M_0(s;\theta)$ is defined.

Now choose $\alpha_i$ so that $\inf_i \alpha_i = 0$.
(By imposing bounds on the sizes of $p_i,r_i,c_i$, we can still ensure that $M_0$ is defined for $s$ in some neighbourhood of 0: for instance, we may take $c_i=0$ and $r_i=1$ for all $i$, with the mixture probabilities chosen so that $p_i\leq 2^{-i-\alpha_i}$ for $i\geq 2$, in which case $M_0$ is defined whenever $s<1/2$.)
Then \eqref{DecayBound} fails.
Indeed, $M_0(\ii\phi;\theta)$ contains terms that decay as $\abs{\phi}^{-\alpha_i}$ when $\abs{\phi}\to\infty$, and since the $\alpha_i$ become arbitrarily small this precludes the bound \eqref{DecayBound}.
Another way to see this is to note that if \eqref{DecayBound} holds, then $X_\theta$ must have a bounded density whenever $n$ is large enough.
But for any $n$, $X$ will itself be a mixture model in which one component is a shifted Gamma$(n\alpha_i,r_i)$ distribution.
Since $\inf_i \alpha_i=0$ we can find $i=i(n)$ so that $n\alpha_i<1$, so that the Gamma$(n\alpha_i,r_i)$ density has a singularity and $X$ cannot have a bounded density.
\end{example}

\end{appendix}


\begin{thebibliography}{10}

\bibitem{B-NKlu1999}
O.~E. Barndorff-Nielsen and C.~Kl{\"u}ppelberg.
\newblock Tail exactness of multivariate saddlepoint approximations.
\newblock {\em Scandinavian Journal of Statistics}, 26(2):253--264, 1999.

\bibitem{B-N19782014InfoExpFamilies}
Ole~E. Barndorff-Nielsen.
\newblock {\em Information and Exponential Families}.
\newblock Wiley series in probability and statistics. John Wiley {\&} Sons,
  Ltd, 2014.

\bibitem{BleiNor1986}
Norman Bleistein and Richard~A. Handelsman.
\newblock {\em Asymptotic Expansions of Integrals}.
\newblock Dover Publications, New York, 1986.

\bibitem{Butler2007}
Ronald~W. Butler.
\newblock {\em Saddlepoint approximations with applications}.
\newblock Cambridge series on statistical and probabilistic mathematics.
  Cambridge University Press, Cambridge, 2007.

\bibitem{Daniels1954}
H.~E. Daniels.
\newblock Saddlepoint approximations in statistics.
\newblock {\em The Annals of Mathematical Statistics}, 25(4):631--650, 1954.

\bibitem{Daniels1982}
H.~E. Daniels.
\newblock The saddlepoint approximation for a general birth process.
\newblock {\em Journal of Applied Probability}, 19(1):20--28, 1982.

\bibitem{DavHauKraParameterLinearBirthDeath}
Anthony~C. Davison, Sophie Hautphenne, and Andrea Kraus.
\newblock Parameter estimation for discretely observed linear birth-and-death
  processes.
\newblock To appear in \textit{Biometrics}, 2020.

\bibitem{dH2000}
Frank den Hollander.
\newblock {\em Large Deviations}.
\newblock Fields Institute Monographs. American Mathematical Society, 2000.

\bibitem{GRToISP2007}
I.~S. Gradshteyn and M.~Ryzhik.
\newblock {\em Table of Integrals, Series and Products}.
\newblock Academic Press, 7th edition, 2007.
\newblock Edited by Alan Jeffrey and Daniel Zwillinger.

\bibitem{Jensen1988}
J.~L. Jensen.
\newblock Uniform saddlepoint approximations.
\newblock {\em Advances in Applied Probability}, 20(3):622--634, 1988.

\bibitem{Jensen1995Saddlepoint}
Jens~Ledet Jensen.
\newblock {\em Saddlepoint approximations}.
\newblock Oxford statistical science series 16. Clarendon Press, Oxford; New
  York, 1995.

\bibitem{KleSka2008}
Tore~Selland Kleppe and Hans~J. Skaug.
\newblock Building and fitting non-{G}aussian latent variable models via the
  moment-generating function.
\newblock {\em Scandinavian Journal of Statistics}, 35(4):664--676, 2008.

\bibitem{Kolassa2006SeriesApproxMethodsStats}
John~Edward Kolassa.
\newblock {\em Series approximation methods in statistics}, volume~88 of {\em
  Lecture notes in statistics}.
\newblock Springer-Verlag, New York, 3rd edition, 2006.

\bibitem{LunKleSkaPreprintSaddlepointInversion}
Berent {\r{A}}.~S. Lunde, Tore~S. Kleppe, and Hans~J. Skaug.
\newblock Saddlepoint adjusted inversion of characteristic functions.
\newblock arXiv:1811.05678[stat.CO], 2018.

\bibitem{MagnusNeudeckerMatrixDifferentialCalc}
Jan~R. Magnus and Magnus Neudecker.
\newblock {\em Matrix differential calculus with applications in statistics and
  econometrics}.
\newblock Wiley series in probability and statistics. Wiley, third edition,
  2019.

\bibitem{Ogden2017}
H.~E. Ogden.
\newblock On asymptotic validity of naive inference with an approximate
  likelihood.
\newblock {\em Biometrika}, 104(1):153--164, 02 2017.

\bibitem{OgdenErrorLaplaceHighDimension}
Helen~E. Ogden.
\newblock On the error in laplace approximations of high-dimensional integrals.
\newblock 2018.

\bibitem{PedDavFok2015}
Xanthi Pedeli, Anthony~C. Davison, and Konstantinos Fokianos.
\newblock Likelihood estimation for the {INAR(p)} model by saddlepoint
  approximation.
\newblock {\em Journal of the American Statistical Association},
  110(511):1229--1238, 2015.

\bibitem{Reid1988}
N.~Reid.
\newblock Saddlepoint methods and statistical inference.
\newblock {\em Statistical Science}, 3(2):213--227, 1988.

\bibitem{RudinPoMA}
Walter Rudin.
\newblock {\em Principles of mathematical analysis}.
\newblock McGraw-Hill New York, third edition, 1976.

\bibitem{ValFewCarPat2014}
R.~T.~R. Vale, R.~M. Fewster, E.~L. Carroll, and N.~J. Patenaude.
\newblock Maximum likelihood estimation for model {$M_{t,\alpha}$} for
  capture-recapture data with misidentification.
\newblock {\em Biometrics}, 70(4):962--971, 2014.

\bibitem{Wong2001}
R.~Wong.
\newblock {\em Asymptotic approximations of integrals}, volume~34 of {\em
  Classics in applied mathematics}.
\newblock Society for Industrial and Applied Mathematics, 2001.

\bibitem{WooBooBut1993}
Andrew T.~A. Wood, James~G. Booth, and Ronald~W. Butler.
\newblock Saddlepoint approximations to the {CDF} of some statistics with
  nonnormal limit distributions.
\newblock {\em Journal of the American Statistical Association},
  88(422):680--686, 1993.

\bibitem{ZhaBraFew2019}
W.~Zhang, M.~V. Bravington, and R.~M. Fewster.
\newblock Fast likelihood-based inference for latent count models using the
  saddlepoint approximation.
\newblock {\em Biometrics}, 75(3):723--733, 2019.

\end{thebibliography}
\end{document}